\documentclass[a4paper,reqno]{amsart}
\usepackage[utf8]{inputenc}
\usepackage{amssymb}
\usepackage{enumitem}
\usepackage{mathrsfs}
\usepackage{mathtools}
\usepackage{amsmath}
\usepackage [dvipsnames] { xcolor }
\usepackage{multirow}
\usepackage[colorlinks=true]{hyperref}

\hypersetup{linkcolor=Green,urlcolor=Cyan,citecolor=Cyan}

\setlist[enumerate]{label=\emph{(\roman*)}}

\usepackage{environ}

\newtheorem{theorem}{Theorem}[section]

\newtheorem{lemma}[theorem]{Lemma}
\newtheorem{proposition}[theorem]{Proposition}

\theoremstyle{definition}
\newtheorem{definition}[theorem]{Definition}
\newtheorem{remark}[theorem]{Remark}
\numberwithin{equation}{section}
\newcommand{\RR}{\mathbb{R}}

\newcommand{\e}{\varepsilon}
\newcommand{\dd}{{\rm d}}

\begin{document}
	
	\title[Two-bubble solutions for gKdV equations]{Construction of two-bubble blow-up solutions for the mass-critical gKdV equations}
	
	\author{Yang Lan}
	\address{Yau Mathematical Sciences Center, Tsinghua University, Beijing 100084, P. R. China.}
	\email{lanyang@mail.tsinghua.edu.cn}

	\author{Xu Yuan}
	\address{State Key Laboratory of Mathematical Sciences, Academy of Mathematics and Systems Science, Chinese Academy of Sciences, Beijing 100190, P. R. China.}
	\email{xu.yuan@amss.ac.cn}

	\subjclass[2020]{35B40 (primary), 35Q51, 37Q53}
	
	\thanks{Y. L. was partially funded by the New Cornerstone Investigator Program 100001127.}
	
	\begin{abstract}
		For the mass-critical generalized Korteweg-de Vries equation,
		\begin{equation*}
			\partial_{t}u+\partial_{x}\left(
			\partial_{x}^{2}u+u^{5}\right)=0,\quad (t,x)\in [0,\infty)\times \RR,
		\end{equation*}
		We prove the existence of a global solution that blows up in infinite time and approaches the sum of two decoupled bubbles with opposite signs.
		The proof is inspired by the techniques developed for the two-dimensional mass-critical NLS equation in a similar context by Martel-Rapha\"el~\cite{MartelRaphael}.
		
		The main difficulty originates from the fact that the unstable directions related to scaling are excited by the nonlinear interactions. To overcome this difficulty, a refined approximate solution that involves some non-localized profiles is needed. In particular, a sharp understanding for the interactions between solitons and such profiles is also required.
	\end{abstract}
	
	\maketitle
	
	\section{Introduction}
	\subsection{Main result}\label{SS:Main}
	In this article, we consider the dynamics of two-bubbles for the mass-critical generalized Korteweg-de Vries (gKdV) equation,
	\begin{equation}\label{equ:gKdV}
		\partial_{t}u+\partial_{x}\left(
		\partial_{x}^{2}u+u^{5}\right)=0,\quad (t,x)\in [0,\infty)\times \RR.
	\end{equation}
	The mass $M(u)$ and energy $E(u)$ are formally conserved by the flow of~\eqref{equ:gKdV} where 
	\begin{equation*}
		M(u)=\int_{\RR}u^{2}\dd x\quad \mbox{and}\quad 
		E(u)=\frac{1}{2}\int_{\RR}(\partial_{x}u)^{2}\dd x-\frac{1}{6}\int_{\RR}u^{6}\dd x.
	\end{equation*}
	Recall that, the Cauchy problem for equation~\eqref{equ:gKdV} is locally well-posed in the energy space $H^{1}(\RR)$ (see Kenig-Ponce-Vega~\cite{KPVgKdV1,KPVgKdV2}). More precisely, for any initial data $u_{0}\in H^{1}(\RR)$, there exists a unique (in a certain class) maximal solution $u$ of~\eqref{equ:gKdV} in 
	$C\left([0,T);H^{1}(\RR)\right)$. For this Cauchy problem, the following blow-up criterion holds:
	\begin{equation}\label{equ:blowCauchy}
		T<\infty\Longrightarrow 
		\| \partial_{x} u(t)\|_{L^{2}}=\infty\quad \mbox{as}\ t\uparrow T.
	\end{equation}
	For such $H^{1}$ solution $u$, the mass $M(u)$ and energy $E(u)$ are conserved on $[0,T)$.
	
	\smallskip
	Recall also that, if $u$ is a solution of~\eqref{equ:gKdV} then for any $\lambda>0$, the function 
	\begin{equation*}
		u_{\lambda}(t,x)=\lambda^{\frac{1}{2}} u\left(\lambda^{3}t,\lambda x\right),\quad \mbox{for}\ (t,x)\in [0,\infty)\times \RR,
	\end{equation*}
	is also a solution of~\eqref{equ:gKdV}. In addition, this scaling symmetry keeps the $L^{2}$-norm invariant so that the problem is \emph{mass-critical}. 
	
	\smallskip
	We recall the family of solitary wave solutions of~\eqref{equ:gKdV}.
	Let $Q(x)=\left(3{\mathrm{sech}}^{2}(2x)\right)^{\frac{1}{4}}$ be the unique (up to translation) positive $H^{1}$ solution of the equation
	\begin{equation*}
		-Q''+Q-Q^{5}=0,\quad \mbox{on}\ \RR.
	\end{equation*}
	Then, for any $(\lambda_{0},x_{0})\in (0,\infty)\times \RR$, the function 
	\begin{equation*}
		u(t,x)=\lambda_{0}^{\frac{1}{2}}Q\left(\lambda_{0}\left(x-\lambda_{0}^{2}t-x_{0}\right)\right),
	\end{equation*}
	is a solution of~\eqref{equ:gKdV} called \emph{solitary wave} or \emph{soliton}. Following a variational argument~\cite{Wein}, the conservation laws and the blow-up criterion~\eqref{equ:blowCauchy} imply that any initial data $u_{0}\in H^{1}(\RR)$ with subcritical mass, \emph{i.e.} satisfying $\|u_{0}\|_{L^{2}}<\|Q\|_{L^{2}}$, generates a \emph{global and bounded} solution in $H^{1}(\RR)$.
	
	\smallskip
	To go beyond the threshold mass $\|Q\|_{L^{2}}$, it is natural to restrict to solutions with small supercritical mass, \emph{i.e.} satisfying
	\begin{equation}\label{equ:massabove}
		\|Q\|_{L^{2}}\le \|u_{0}\|_{L^{2}}<(1+\delta)\|Q\|_{L^{2}},\quad
		\mbox{for}\ \ 0<\delta\ll 1.
	\end{equation}
	The study of singularity formation for such case was first developed in a series of works by Martel-Merle~\cite{MMJMPA,MMGAFA,MMANN,MMJAMS} and Merle~\cite{MerleJAMS} via rigidity Liouville-type property, monotonicity formula and localized Virial estimate. Later on, Martel-Merle-Rapha\"el~\cite{MMRACTA,MMRJEMS,MMRASNSP} revisited the blow-up analysis for~\eqref{equ:gKdV}, giving a comprehensive description of the flow near the soliton and thereby completing the results in the above-mentioned works. See Section~\ref{SS:BLOWUPKDV} for more discussion.
	
	\smallskip
	For the multi-bubble dynamics of~\eqref{equ:gKdV}, it is known from Combet-Martel~\cite{CombetMartel} that \emph{finite time} blow-up solutions exist with an arbitrary number of bubbles and any choice of signs.
	In this article, we construct the first example of \emph{infinite time} blow-up solution of~\eqref{equ:gKdV} related to the strong interactions of two bubbles with opposite signs. Our main result is formulated as follows.
	
	\begin{theorem}\label{thm:main}
		There exists a global-in-time solution $u\in C\left([0,\infty);H^{1}\right)$ of~\eqref{equ:gKdV} that decomposes asymptotically into a sum of two bubbles at infinity:
		\begin{equation*}
			\lim_{t\to \infty}\left\|u(t)-\frac{1}{\lambda^{\frac{1}{2}}(t)}\sum_{n=1,2}(-1)^{n+1}Q\left(\frac{\cdot-x_{n}(t)}{\lambda(t)}\right)\right\|_{H^{1}}=0.
		\end{equation*}
		Here, the position parameters $(x_{1},x_{2})$ satisfy
		\begin{equation*}
			x_{2}(t)-x_{1}(t)=(2+o(1))\log^{\frac{5}{6}}t,\quad  \mbox{as}\ t\uparrow\infty.
		\end{equation*}
		In addition, the scaling and position parameters $(\lambda(t),x_{1}(t))$ satisfy
		\begin{equation*}
			\begin{aligned}
				\left(\lambda(t),x_{1}(t)\right)&=\left(1+o(1)\right)\left(\frac{1}{\log^{\frac{1}{6}} t},t\log^{\frac{1}{3}}t\right),
				\  \mbox{as}\ t\uparrow \infty.
			\end{aligned}
		\end{equation*}
		In particular, the blow-up rate for $u(t)$ is
		\begin{equation*}
			\|\partial_{x}u(t)\|^{2}_{L^{2}}=(2+o(1))\|Q'\|^{2}_{L^{2}}\log^{\frac{1}{3}} t,\quad \mbox{as}\ \ t\uparrow \infty.
		\end{equation*}
	\end{theorem}
	
	\begin{remark}
		We mention here that, Theorem~\ref{thm:main} deals with~\emph{strong interactions} in the sense that the blow-up dynamics of each bubble is perturbed at the main order by the presence of another bubble (see Section~\ref{SS:ROAD} for more discussion).
	\end{remark}
	
	\begin{remark}
		The choice of signs and the number of bubbles are related to the solvability of a third-order ODE system. The source term of this system originates from the interaction between solitons and certain non-localized profiles (see Lemma~\ref{le:AB}). Indeed, the system for two bubbles with the same sign is not solvable, and the multi-bubble case is far more involved than the two-bubble case. This is the main reason why we restrict ourselves to the current case.
	\end{remark}
	
	\subsection{Blow-up description and classification for gKdV equation}\label{SS:BLOWUPKDV}
	Let us first consider the set of initial data
	\begin{equation*}
		\mathcal{A}=\left\{u_{0}=Q+\varepsilon_{0}:
		\|\varepsilon_{0}\|_{H^{1}}<\alpha_{0}\ \  \mbox{and}\ \
		\int_{0}^{\infty}y^{10}\varepsilon_{0}^{2}\dd y<1
		\right\},\ \  \mbox{with}\ \ 0<\alpha_{0}\ll 1.
	\end{equation*}
	For initial data in $H^{1}$ close to the soliton with a suitable space decay property,
	Martel-Merle-Rapha\"el~\cite{MMRACTA} fully described \emph{the blow-up dynamics} of the corresponding solution. More precisely, for initial data $u_{0}\in \mathcal{A}$, Martel-Merle-Rapha\"el~\cite{MMRACTA} proved that only three possible behaviors can occur:
	\begin{description}
		\item[Exit] The solution eventually exits any small neighborhood of the solitons.
		
		\item [Blow-up] The solution blows up stably in finite time with rate $(T-t)^{-1}$.
		
		\item [Soliton] The solution is global and locally converges to a soliton.
	\end{description}
	Later on, building on the work~\cite{MMRACTA}, Martel-Merle-Nakanishi-Rapha\"el~\cite{MMNR} showed that there exists a local $C^{1}$ co-dimension one manifold included in $\mathcal{A}$ which separates the stable
	blow-up behavior from solutions that eventually exit the soliton neighborhood. 
	In particular, the solutions on the manifold are \emph{global in time} and converge in a local norm to a soliton. 
	
	\smallskip
	On the other hand, Martel-Merle-Raphaël ~\cite{MMRJEMS} established the \emph{existence and description} of the minimal mass blow-up solution, a key step for the complete description of the flow around the soliton. Then, the sharp asymptotics in both time and space variables were derived by Combet-Martel ~\cite{Comkdv} for any order derivative of this solution. Building on these sharp properties and the techniques developed mainly in ~\cite{MMRACTA, MMRJEMS, MMRASNSP}, Combet-Martel ~\cite{CombetMartel} constructed \emph{finite time} multi-bubble blow-up solutions with an arbitrary number of bubbles. Such solutions are also related to the \emph{strong interaction} of bubbles as can be seen from the presence of the tail to the left of the minimal mass blow-up solution (see ~\cite[Section 1.3]{CombetMartel} for more discussion).
	
	\smallskip
	Since the assumption $u_{0}\in \mathcal{A}$, the above-mentioned stable blow-up result does not describe all blow-up solutions with the small supercritical mass~\eqref{equ:massabove}. In Martel-Merle-Rapha\"el~\cite{MMRASNSP}, the authors showed that there exists a large class of \emph{exotic} finite time blow-up solutions, close to the solitons, enjoying blow-up rates of the form $\|u(t)\|_{H^{1}}\sim (T-t)^{-\nu}$ for any $\nu>\frac{11}{13}$. Global solutions
	blowing up in \emph{infinite time} with $\|u(t)\|_{H^{1}}\sim e^{t}$ or $\|u(t)\|_{H^{1}}\sim t^{\nu}$ for any $\nu>0$ were also constructed in~\cite{MMRASNSP}. More recently, Manatova~\cite{Manatova} extended these results, proving the existence of exotic blow-up solutions for any rate
	$(T-t)^{-\nu}$ with $\nu\in \left(\frac{1}{2},1\right)$. We mention here that, the exponent $\nu=\frac{1}{2}$ is critical: for $\nu\ge \frac{1}{2}$, the position of blow-up soliton goes to $\infty$, whereas for $\nu<\frac{1}{2}$, such position converges to a finite point. Recently, Martel-Pilod~\cite{MDASNSP}
	constructed the first example of a finite-energy solution that blows up at a \emph{finite spatial point} with the rate $(T-t)^{-\frac{2}{5}}$. Then, the same authors in~\cite{MDPRINT} introduced a more general framework for the finite point blow-up and proved the existence of such solutions for any blow-up rate $(T-t)^{-\nu}$ with $\nu\in \left(\frac{3}{7},\frac{1}{2}\right)$.
	
	\smallskip
	The asymptotic behavior of solutions close to the soliton was also studied for \emph{saturated nonlinearities} of the form $u^{5}-\gamma|u|^{p-1}u$ for fixed $p>5$ in the limit $\gamma \downarrow 0$ in~\cite{LAN2}. In particular, \emph{flattening solitary waves} were constructed in~\cite[Theorem 1.5]{LAN2}. Then, a full family of such flattening waves was constructed in Martel-Pilod~\cite{MDCMP} for~\eqref{equ:gKdV}. Last, we refer to~\cite{BGM,CLY,CLYMINI,FHRY,LYMINI} for some recent advances in the study of blow-up dynamics of the 2D cubic Zakharov-Kuznetsov equation, which is a natural analogue of the mass-critical gKdV equation in two dimensions.

	\subsection{Previous results for other models}\label{SS:PREVIOUSRESULT}
	In the last twenty years, there has been significant progress in the study of the blow-up dynamics for the mass-critical NLS equation. The first example of multi-bubble blow-up solution was constructed in Merle~\cite{MerleCMP}. In particular, such solution blows up with the pseudo-conformal rate $\|\nabla u(t)\|_{L^2}\sim t^{-1}$ as $t\downarrow 0$. Recently, the uniqueness for such blow-up solution was obtained in~\cite{CAO}. Some other multi-bubble type blow-up solutions for the NLS equation were constructed by~\cite{FAN,PlanRaphael} in the context of the \emph{log-log} regime.
	We mention here that, these works deal with \emph{weak interactions} in the sense that the leading-order term of the soliton dynamics is not generated by the interactions.
	Last, we refer to~\cite{Bourgain,KSnls,MR2,MR4,MR5,Perelmannls} and references therein for the study of single bubble blow-up dynamics of the mass-critical NLS equation.
	
	\smallskip
	For the NLS equation, the first multi-bubble blow-up solution with \emph{strong interactions} was constructed by Martel-Raphaël~\cite{MartelRaphael}. In particular, they built an infinite-time blow-up solution involving the strong interaction of an arbitrary number $K\ge 2$ of bubbles. Through the pseudo-conformal transform, this also yielded the first example of a \emph{finite-time} blow-up solution with a rate strictly faster than the conformal one. Such a solution concentrates the $K$ bubbles at one point at the blow-up time.
	
	\smallskip
	Following the spirit of~\cite{MartelRaphael}, we study in this article the \emph{two-bubble dynamics} with strong interactions for the gKdV equation~\eqref{equ:gKdV}. However, at least two main difficulties arise, originating from structural differences between the two equations. The first difficulty originates from the linear growth of the blow-up profile on one side of $\RR$. This requires us to carefully trace both \emph{soliton-soliton} and \emph{profile-soliton} interactions, rather than only those of the \emph{soliton-soliton} type. A second, more significant difficulty arises from the presence of the parameter $\mu$, which relates to the difference between the scaling parameters of the two bubbles. Unlike in the NLS case, this extra parameter appears as a leading-order term in the ODE governing the distance parameter $z$. This requires us to determine the asymptotic behavior of this parameter, and thus, we must first obtain refined control for the parameters $(b_{1},b_{2})$ which are related to the blow-up profile. See more discussion in Section~\ref{SS:ROAD}. We refer to~\cite{CDMJEMS,DMWAPDE,JJ1,JJ2} for some
	related results of multi-bubble solutions with strong interactions for other critical models. We also refer to~\cite{ValetJFA,Gerard,JJKdV,JJ3,KMRCPAM,LANWANG,MartelNguyenNLS,NguyenKdV,NguyenNLS} for some related results of multi-soliton and multi-kink solutions in this regime for other dispersive models.
	
	\subsection{A Roadmap of the proof}\label{SS:ROAD}
	The blow-up solutions in Theorem~\ref{thm:main} are obtained
	using the study of interaction between two solitons.
	More precisely, We first consider a time dependent $C^{1}$ geometrical function $\mathcal{G}$ of the following form:
	\begin{equation*}
		\mathcal{G}=(\lambda,z,\mu,x_{1},b_{1},b_{2})\in (0,\infty)^{2}\times \RR^{4},
	\end{equation*}
	with $\lambda+|\mu|+|b_{1}|+|b_{2}|+|x_{1}|^{-1}+z^{-1}\ll 1$.
	Here, the function $\lambda$ is the scaling parameter, the function $\mu$ is the difference of the scaling for the two bubbles, the function $z$ is the distance between the two bubbles and the function $x_{1}$ is the position for the first bubble.
	As usual in investigating the blow-up phenomenon for the mass-critical problem, we
	introduce the following new space-time variables:
	\begin{equation}\label{def:dsdt}
		\frac{\dd s}{\dd t}=\frac{1}{\lambda^{3}(t)}\ \ \mbox{and}\ \ y=\frac{x-x_{1}(t)}{\lambda(t)}.
	\end{equation}
    From now on, we consider the solution related to the space-time variables $(s,y)$:
    \begin{equation*}
        w(s,y)=\lambda^{\frac{1}{2}}(s)u(t(s),\lambda(s)y+x_{1}(s)).
    \end{equation*}
	Here, $u(t,x)$ is a solution of~\eqref{equ:gKdV}. By an elementary computation, we find  
	\begin{equation}\label{equ:gKdv2}
		\partial_{s}w+\partial_{y}\left(\partial_{y}^{2}w-w+w^{5}\right)-\frac{\dot{\lambda}}{\lambda}\Lambda w-\left(\frac{\dot{x}_{1}}{\lambda}-1\right)\partial_{y}w=0.
	\end{equation}
	An approximate two-bubble solution $V$ of~\eqref{equ:gKdv2} is defined of the form
	\begin{equation*}
		V=Q_{1}-Q_{2}+b_{1}X_{1}\phi-b_{2}X_{2}\phi+e^{-z}\left(A_{1}+B_{2}\right)\phi+\rm{L.O.T}.
	\end{equation*}
	Here, $(Q_{1},Q_{2})$ are two solitons at different positions and $(X_{1},X_{2},A_{1},B_{2})$ are some suitable non-localized profiles.
    On the one hand, the profiles $(X_{1},X_{2})$ are related to the standard profile $P$ used in studying the blow-up dynamics of~\eqref{equ:gKdV} (see~\cite[Section 2.2]{MMRACTA} and Remark~\ref{re:P}).
	On the other hand, the leading-order terms of soliton-soliton interactions are generated by $e^{-z}(e^{\pm y}Q^{4})$. This is the main reason why we introduce the auxiliary functions $(A_{1},B_{2})$ to construct the approximate solution.  Due to the cancellation $(\Lambda Q,Q)=0$, the auxiliary functions $(A_{1},B_{2})$ should exhibit a linear growth on the right-hand side of $\RR$, which is essentially different from that in~\cite{MartelRaphael,NguyenKdV}. This difference leads to much more complicated interactions and ODE system related to parameters (see Lemma~\ref{le:PAB} and Section~\ref{SS:Approximate} for more details).
	
	\smallskip
	Due to the presence of the parameter $\mu$, the geometrical parameters satisfy a different ODE system from that of the NLS equation (see~\cite[Lemma 3]{MartelRaphael}).
	More precisely, the geometrical parameters $(z,\mu,b_{1},b_{2})$ satisfy
	\begin{equation}\label{equ:ODESYSTEM}
		\left\{
		\begin{aligned}
			\dot{z}-b_{1}z-\mu&=
			\dot{b}_{1}+\alpha e^{-z}+2b_{1}^{2} ={\rm{L.O.T}},\\
			\dot{\mu}-2(b_{2}-b_{1})&=\dot{b}_{2}+\alpha e^{-z}+2b_{1}^{2}
			={\rm{L.O.T}}.
		\end{aligned}
		\right.
	\end{equation}
	In order to treat the above system, we first investigate the asymptotic behavior of $\Theta=(b_{2}-b_{1})$. This step requires a refined estimate for the parameters $(b_{1},b_{2})$ (see Lemma~\ref{le:refinedTheta}).
	Together with an analogous refinement for $\mu$ (see Lemma~\ref{le:refinmu}), we deduce that the size of $\mu$ is of the same order as $b_{1}z$. Therefore, the ODE system~\eqref{equ:ODESYSTEM} admits a different structure, and thus, it exhibits a behavior markedly different from that for the NLS equation. This is the main reason why the blow-up rate in Theorem~\ref{thm:main} is different from that for the case of the NLS equation.
	
	\smallskip
	We mention here that, the refined control of
	$\Theta$ requires careful tracking of both soliton-soliton and profile-soliton interactions. For instance, while the leading-order soliton-soliton interaction is of order $e^{-z}$, the next order $\mu ze^{-z}$ also enters the refined estimate for $\Theta$ (see Lemma~\ref{le:nonlinear}). Moreover, due to the linear growth of $(A_{1},B_{2})$, there are some terms of same size that appear in the profile-soliton type interactions (see Lemma~\ref{le:Psi1}). Indeed, it turns out that at the leading order, the following equation holds for $(\mu,\Theta)$:
	\begin{equation}\label{equ:ODESYSTEM2}
		\dot{\mu}-2\Theta=\dot{\Theta}+\frac{5}{2}\alpha \frac{\dd}{\dd s}(ze^{-z})=\rm{L.O.T}.
	\end{equation}
	The ODE system~\eqref{equ:ODESYSTEM}--\eqref{equ:ODESYSTEM2} admits the following approximate solution:
	\begin{equation*}
		\left\{
		\begin{aligned}
			z_{\rm{app}}(s)\sim2\log s,\ \  b_{1{\rm{app}}}(s)\sim \frac{1}{6s\log s},\\
			\mu_{\rm{app}}(s)\sim \frac{5}{3s}\quad \mbox{and}\quad \Theta_{\rm{app}}(s)\sim-\frac{5}{6s^{2}},
		\end{aligned}
		\right.\Longrightarrow \lambda_{\rm{app}}(s)\sim \frac{1}{\log^{\frac{1}{6}}s}.
	\end{equation*}
    It follows directly from~\eqref{def:dsdt} that 
    \begin{equation*}
        t\sim \frac{s}{\log^{\frac{1}{2}}s}\ \ \mbox{and}\ \ 
        s\sim t\log^{\frac{1}{2}}t
        \Longrightarrow 
         \lambda_{\rm{app}}(t)\sim \frac{1}{\log^{\frac{1}{6}}t}.
    \end{equation*}
	The above estimates mean that the blow-up rate in Theorem~\ref{thm:main} is a reasonable estimate for the first order asymptotics of some particular solutions of~\eqref{equ:gKdV}.
	
	\smallskip
	Finally, we use a variant of an energy functional on $\varepsilon$ to control it in the soliton region. The key point here is to estimate the error term $\varepsilon$ backwards in time. On the one hand, from the Kato localization identity (see~\cite{Kato}), for any solution $u$ of~\eqref{equ:gKdV} and any decreasing $C^{\infty}$ function $\chi_{1}:\RR\to [0,1]$, we have 
    \begin{equation*}
        \frac{\dd}{\dd t}\int_{\RR}u^{2}\chi_{1}\dd y=-3\int_{\RR}(\partial_{x}u)^{2}\chi'_{1}\dd y+\rm{L.O.T}.
    \end{equation*}
  Following the spirit of the above identity, a functional that contains the global energy and a localized variant of the mass with monotonicity property was introduced to control the remainder term $\varepsilon$. On the other hand, from~\eqref{equ:ODESYSTEM}, a delicate scaling term related to $\mu$ does not have a favorable sign. To overcome such difficulty, we employ the idea introduced in~\cite[Section 3]{NguyenKdV} to construct a suitable localized refined term for the standard energy functional of $\varepsilon$. Based on such control and weak $H^{1}$ stability of gKdV equation~\eqref{equ:gKdV}, we argue by compactness and obtain the infinite time blow-up solution $u$ as the weak limit of a sequence of solutions $(u_{n})_{n\in \mathbb{N}}$.
	
	\subsection{Notation and conventions}\label{SS:Nota}
	We denote by $\mathcal{Y}$ the set of smooth function $f\in C^{\infty}(\RR)$, such that for all $k\in \mathbb{N}$, there exist $(c_{k},C_{k})\in (0,\infty)^{2}$ such that  
	\begin{equation*}
		\big|f^{(k)}(y)\big|\le C_{k}\left(1+|y|\right)^{c_{k}}e^{-|y|},\quad \mbox{on}\ \RR.
	\end{equation*}
	We also denote 
	\begin{equation*}
		\mathcal{Y}_{1}=\left\{f\in C^{\infty}(\RR): \mbox{there exists}\ g\in \mathcal{Y}\ \mbox{such that}\ f(y)=\int_{y}^{\infty}g(\sigma)\dd\sigma
		\right\}.
	\end{equation*}
	In addition, we denote 
	\begin{equation*}
		\begin{aligned}
			\mathcal{Z}_{1}=\left\{f\in C^{\infty}(\RR): \mbox{there exists}\ g\in \mathcal{Y}\ \mbox{such that}\ f(y)-f(0)=\int_{0}^{y}g(\sigma)\dd\sigma
			\right\},\\
			\mathcal{Z}_{2}=\left\{f\in C^{\infty}(\RR): \mbox{there exists}\ g\in \mathcal{Z}_{1}\ \mbox{such that}\ f(y)-f(0)=\int_{0}^{y}g(\sigma)\dd\sigma
			\right\}.
		\end{aligned}
	\end{equation*}
	
	For any $f_{1}\in L^{2}(\RR)$ and $f_{2}\in L^{2}(\RR)$, we denote the $L^{2}$-scalar product by 
	\begin{equation*}
		(f_{1},f_{2})=\int_{\RR}f_{1}(y)f_{2}(y)\dd y.
	\end{equation*}
	
	Let $\chi:\RR\mapsto [0,1]$ be a non-increasing $C^{\infty}$ function such that 
	\begin{equation*}
		\chi_{|(-\infty,1)}\equiv 1\quad \mbox{and}\quad \chi_{|(2,\infty)}\equiv 0.
	\end{equation*}
	
	We see from the explicit expression of $Q$ in Section~\ref{SS:Main} that, as $y\to \infty$,
	\begin{equation}\label{est:QQ'}
		Q(y)=c_{Q}e^{-y}+O\left(e^{-2y}\right)\quad \mbox{and}\quad 
		Q'(y)=-c_{Q}e^{-y}+O\left(e^{-2y}\right).
	\end{equation}
	We also see that, as $y\to -\infty$,
	\begin{equation}\label{est:QQ'-}
		Q(y)=c_{Q}e^{y}+O\left(e^{-2y}\right)\quad \mbox{and}\quad 
		Q'(y)=c_{Q}e^{y}+O\left(e^{-2y}\right).
	\end{equation}
	Here, we denote $c_{Q}=12^{\frac{1}{4}}>0$ by a universal positive constant. Moreover, we set 
	\begin{equation}\label{equ:defalpha}
		m_{0}=\frac{1}{4}\int_{\RR}Q(y)\dd y>0\quad \mbox{and}\quad 
		\alpha=\frac{c_{Q}}{m_{0}^{2}}\int_{\RR}e^{y}Q^{5}(y)\dd y>0.
	\end{equation}
	For any $n\in \mathbb{N}$ and smooth real-valued function $f:\RR\mapsto \RR$, we denote 
	\begin{equation*}
		\partial_{y}^{\le n}f=\left(f,\partial_{y}f,\dots,\partial_{y}^{n}f\right),\quad \mbox{on}\ \RR.
	\end{equation*}
	In addition, for any parameter $\mu\in (0,1)$ and such function $f$, we set 
	\begin{equation*}
		f_{1+\mu}(y)=(1+\mu)^{\frac{1}{4}}f\big((1+\mu)^{\frac{1}{2}}y\big),\quad \mbox{on} \ \RR.
	\end{equation*}
	\smallskip
	We now introduce the scaling operator and the linearized operator around $Q$:
	\begin{equation}\label{def:L}
		\Lambda f=\frac{1}{2}f+yf'\ \ \mbox{and}\ \
		\mathcal{L}f=-f''+f-5Q^{4}f,\quad 
		\mbox{for all}\ f\in H^{1}(\RR).
	\end{equation}
	From the Fundamental Theorem of Calculus, for any $-\infty<a<b<\infty$ and smooth real-valued function $f:\RR\mapsto \RR$, we find 
    \begin{equation*}
        f(b)=f(a)+f'(a)(b-a)+\int_{a}^{b}f''(s)(b-s)\dd s.
    \end{equation*}
    Therefore, for any parameter $\mu \in (0,1)$ and smooth real-valued function $f:\RR\mapsto \RR$, 
	\begin{equation}\label{equ:muf}
		f_{1+\mu}=f+\frac{\mu}{2} \Lambda f+\frac{\mu^{2}}{2}\int_{0}^{1}(1-s)\left(\frac{(\Lambda^{2}f)_{1+s\mu}}{2(1+s\mu)^{2}}-\frac{(\Lambda f)_{1+s\mu}}{(1+s\mu)^{2}}\right)\dd s.
	\end{equation}
	
	For a given small constant $\delta$, we denote by $\varrho(\delta)$ a generic small constant with
	\begin{equation*}
		\varrho(\delta)\to 0,\quad \mbox{as}\ \ \delta\to 0.
	\end{equation*}

	\section{Construction of the approximate solution}
	In this section, we construct the approximate solution for~\eqref{equ:gKdv2} with an explicit two-bubble asymptotic behavior and extract the evolution system of the geometrical parameters for the two bubbles. Let $I=[s_{0},s_{1}]\subset (0,\infty)$ be an time interval. Recall that, we consider the following geometrical parameters: 
	\begin{equation*}
		\mathcal{G}=(\lambda,z,\mu,x_{1},b_{1},b_{2})\in (0,\infty)^{2}\times \RR^{4}.
	\end{equation*}
	Here, the function $\lambda$ is the scaling parameter, the function $z$ is the distance between the two bubbles, the function $\mu$ is the difference of the scaling for the two bubbles, the function $x_{1}$ is the position of the first bubble and the functions $(b_{1},b_{2})$ are some parameters related to the non-localized profile $X$ (see Section~\ref{SS:Approximate} for more details).
	
	\smallskip
	To state the general results on the solutions of~\eqref{equ:gKdv2} close to the sum of two decoupled bubbles, we assume that the function $\mathcal{G}(s)$ satisfies the following estimates:
	\begin{equation}\label{est:para1}
		0<\lambda\ll 1,\quad 
		1\ll z<\infty,\quad |\mu|+|b_{1}|+|b_{2}|\ll 1\quad \mbox{and}\quad 
		|\mu| z^{3}\ll 1.
	\end{equation}
	
	\smallskip
	For any smooth real-valued function $f:I\times \RR\mapsto \RR$, we set 
	\begin{equation}\label{def:Gamma}
		\left( \Gamma f \right)(s,y)=f_{1+\mu}(s,y-z(s))=(1+\mu(s))^{\frac{1}{4}}f(s,(1+\mu(s))^{\frac{1}{2}}(y-z(s))).
	\end{equation}
	To simplify the notation, we denote 
	\begin{equation}\label{equ:defQ1Q2}
		Q_{1}(s,y)=Q(y)\quad \mbox{and}\quad Q_{2}(s,y)=(\Gamma Q)(s,y).
	\end{equation}
	In addition, we set
	\begin{equation*}
		S=Q_{1}-Q_{2},\quad G=S^{5}-Q_{1}^{5}+Q_{2}^{5} \quad \mbox{and}\quad H=5Q_{1}Q_{2}^{4}-5Q_{1}^{4}Q_{2}.
	\end{equation*}
	
	\subsection{The linearized operator}\label{SS:linear}
	In this subsection, we recall some standard properties for the linearized operator $\mathcal{L}$. We start with the following spectral theory.
	
	\begin{proposition}
		[\cite{MMRACTA}]\label{prop:Spectral}
		The operator $\mathcal{L}$ defined by~\eqref{def:L} is self-adjoint on $L^{2}$. Moreover, the operator $\mathcal{L}$ satisfies the following properties.
		
		\begin{enumerate}
			\item \emph{Spectrum of $\mathcal{L}$}. 
			The operator $\mathcal{L}$ has only one negative eigenvalue $-8$ with 
			an positive radially symmetric eigenfunction $Q^{3}$ and the kernel of $\mathcal{L}$ is given by ${\rm{Ker}}\mathcal{L}={\rm{Span}}\{Q'\}$.
			\item \emph{Scaling identities.} It holds 
			\begin{equation*}
				\mathcal{L}\Lambda Q=-2Q\quad \mbox{and}\quad 
				(\Lambda Q, Q)=0.
			\end{equation*}
			
			\item \emph{Coercivity of $\mathcal{L}$.}
			For all $f\in H^{1}$, we have 
			\begin{equation*}
				\left(f,Q'\right)=\left(f,Q^{3}\right)=0\Longrightarrow
				\left(\mathcal{L}f,f\right)\ge \|f\|_{L^{2}}^{2}.
			\end{equation*}
			In addition, there exists $\nu>0$ such that for all $f\in H^{1}$, 
			\begin{equation*}
				\left(\mathcal{L}f,f\right)\ge 
				\nu \|f\|_{H^{1}}^{2}-\frac{1}{\nu}
				\big(\left(f,Q\right)^{2}+\left(f,Q'\right)^{2}+\left(f,\Lambda Q\right)^{2}\big).
			\end{equation*}
			
			\item \emph{Inversion of $\mathcal{L}$.} 
			For any $g\in L^{2}$ with $(g,Q')=0$, there exists a unique function $f\in H^{2}$ with $(f,Q')=0$ such that $\mathcal{L}f=g$. Moreover, if $g\in \mathcal{Y}$, then we have $f\in \mathcal{Y}$.
		\end{enumerate}
	\end{proposition}
	
	\begin{proof}
		The proof is based on a standard argument related to the structure of  $\mathcal{L}$ and some results in Weinstein~\cite{Wein}. We refer to~\cite[Lemma 2]{MMGAFA} and~\cite[Lemma 2]{MMANN} for some details of the proof.
	\end{proof}
	
	\begin{remark}\label{re:Y}
		Let $g=Q^{4}$ in (iv) of Proposition~\ref{prop:Spectral}.
		It is easy to check that $(Q^{4},Q')=0$, and thus, 
		there exists a unique smooth function $Y\in \mathcal{Y}$ such that 
		\begin{equation*}
			\mathcal{L}Y=Q^{4}\quad \mbox{and}\quad (Y,Q')=0.
		\end{equation*}
		Moreover, from (ii) of Proposition~\ref{prop:Spectral}, we compute 
		\begin{equation*}
			(Y,Q)=-\frac{1}{2}(Y,\mathcal{L}\Lambda Q)=-\frac{1}{2}\left(Q^{4},\Lambda Q\right)=-\frac{3}{5}m_{0}.
		\end{equation*}
		The function $Y$ is related to the existence of the resonance for the operator $\partial_{y}\mathcal{L}$ of the form $1+f$ where $f$ is a Schwartz function. More precisely, we check that 
		\begin{equation*}
			\mathcal{L}(1+5Y)=1\Longrightarrow 
			(\mathcal{L}(1+5Y))'=0.
		\end{equation*}
		On the other hand, by integration by parts, we compute 
		\begin{equation*}
			Z=(yQ^{4})'\Longrightarrow (W,Q)=(Y,Q)+(Z,Q)=\frac{1}{5}m_{0}\quad \mbox{with}\  \ W=Y+Z.
		\end{equation*}
	\end{remark}
	To simplify the notation, we denote 
	\begin{equation*}
		\begin{aligned}
			Y_{1}(s,y)=Y(y)\quad \mbox{and}\quad Y_{2}(s,y)=(\Gamma Y)(s,y),\\
			Z_{1}(s,y)=Z(y)\quad \mbox{and}\quad Z_{2}(s,y)=(\Gamma Z)(s,y).
		\end{aligned}
	\end{equation*}
	
	We now introduce the existence of some non-localized profiles. We mention here that, the following non-localized profiles will be used to construct the approximate solution for~\eqref{equ:gKdv2} and the proof is inspired by~\cite[Proposition 2.2]{MMRACTA}.
	
	\begin{lemma}[Non-localized profiles]\label{le:non}
		\emph{(i)} Suppose that $g\in \mathcal{Y}$ with $(g,Q)=0$. Then there exists a unique smooth function $f_{1}\in \mathcal{Y}_{1}$ such that 
		\begin{equation*}
			\left(\mathcal{L}f_{1}\right)'=g,\quad (f_{1},Q')=0\quad \mbox{and}\quad 
			\lim_{y\to -\infty}f_{1}(y)=-\int_{\RR}g(\sigma)\dd \sigma.
		\end{equation*}
		
		\emph{(ii)}  Suppose that $g_{1}\in \mathcal{Z}_{1}$ with $(g_{1},Q)=0$. Then there exists a unique smooth function $f_{2}\in \mathcal{Z}_{2}$ such that 
		\begin{equation*}
			\begin{aligned}
				\left(\mathcal{L}f_{2}\right)'=g_{1},\quad   \left(f_{2},Q'\right)&=0,\\
				\lim_{y\to \infty}\left(f_{2}(y)-\left(\lim_{\sigma\to \infty}g_{1}(\sigma)\right) y\right)&=-\int_{0}^{\infty}\sigma g'_{1}(\sigma)\dd \sigma,\\
				\lim_{y\to -\infty}\left(f_{2}(y)-\left(\lim_{\sigma\to -\infty}g_{1}(\sigma)\right) y\right)&=\int_{-\infty}^{0}\sigma g'_{1}(\sigma)\dd \sigma.
			\end{aligned}
		\end{equation*}
	\end{lemma}
	
	\begin{proof}
		Proof of (i). We consider $f_{1}\in C^{\infty}(\RR)$ of the form
		\begin{equation*}
			f_{1}(y)=h(y)-\int_{y}^{\infty}g(\sigma)\dd \sigma,\quad \mbox{with}\ \ h\in \mathcal{Y}.
		\end{equation*}
		By an elementary computation, we see that $\left(\mathcal{L}f_{1}\right)'=g$ is equivalent to 
		\begin{equation*}
			\left( \mathcal{L}h\right)'=g+\left(\mathcal{L}\int_{y}^{\infty}g(\sigma)\dd\sigma\right)'=q',\quad \mbox{with}\ \ q=g'-5Q^{4}\int_{y}^{\infty}g(\sigma)\dd \sigma.
		\end{equation*}
		Note that, from $g\in \mathcal{Y}$ and $(g,Q)=0$, we find 
		\begin{equation*}
			q\in \mathcal{Y}\quad \mbox{and}\quad 
			\left(q,Q'\right)=-(q',Q)=-(g,Q)+\left(\int_{y}^{\infty}g(\sigma)\dd \sigma,\mathcal{L}Q'\right)=0.
		\end{equation*}
		Therefore, from (iv) of Proposition~\ref{prop:Spectral}, there exists a unique $h\in \mathcal{Y}$ such that 
		\begin{equation*}
			\mathcal{L}h=q\ \ \mbox{and}\ \ (h,Q')=0\Longrightarrow
			\left(\mathcal{L}h\right)'=q'\ \ \mbox{and}\ \ (h,Q')=0.
		\end{equation*}
		Then, from $h\in \mathcal{Y}$ and the definition of $f_{1}$, we find $f_{1}\in \mathcal{Y}_{1}$ and 
		\begin{equation*}
			\left(\mathcal{L}f_{1}\right)'=g,\quad (f_{1},Q')=0\quad \mbox{and}\quad 
			\lim_{y\to -\infty}f_{1}(y)=-\int_{\RR}g(\sigma)\dd \sigma.
		\end{equation*}
		Last, the uniqueness of $f_{1}$ is a consequence of the uniqueness of $h$.
		
		\smallskip
		Proof of (ii). Similarly, we consider $f_{2}\in C^{\infty}(\RR)$ of the form
		\begin{equation*}
			f_{2}(y)=h_{1}(y)+\int_{0}^{y}g_{1}(\sigma)\dd \sigma,\quad \mbox{with}\ \ h_{1}\in \mathcal{Y}.
		\end{equation*}
		By an elementary computation, we see that $\left(\mathcal{L}f_{2}\right)'=g_{1}$ is equivalent to 
		\begin{equation*}
			\left( \mathcal{L}h_{1}\right)'=g_{1}-\left(\mathcal{L}\int_{0}^{y}g_{1}(\sigma)\dd\sigma\right)'=q_{1}',\quad \mbox{with}\ \ q_{1}=g_{1}'+5Q^{4}\int_{0}^{y}g_{1}(\sigma)\dd \sigma.
		\end{equation*}
		Note that, from $g_{1}\in \mathcal{Z}_{1}$ and $(g_{1},Q)=0$, we find 
		\begin{equation*}
			q_{1}\in \mathcal{Y}\quad \mbox{and}\quad 
			\left(q_{1},Q'\right)=-(q_{1}',Q)=-(g_{1},Q)-\left(\int_{0}^{y}g_{1}(\sigma)\dd \sigma,\mathcal{L}Q'\right)=0.
		\end{equation*}
		Using again (iv) of Proposition~\ref{prop:Spectral}, there exists a unique $h_{1}\in \mathcal{Y}$ such that 
		\begin{equation*}
			\mathcal{L}h_{1}=q_{1}\ \ \mbox{and}\ \ (h_{1},Q')=0\Longrightarrow
			\left(\mathcal{L}h_{1}\right)'=q_{1}'\ \ \mbox{and}\ \ (h_{1},Q')=0.
		\end{equation*}
		Then, from $h_{1}\in \mathcal{Y}$ and the definition of $f_{2}$, we find 
		${f}_{2}\in \mathcal{Z}_{2}$, 
		\begin{equation*}
			\left(\mathcal{L}f_{2}\right)'=g_{1}\quad \mbox{and}\quad   \left(f_{2},Q'\right)=0.
		\end{equation*}
		In particular, from the Fundamental Theorem of Calculus, we obtain
		\begin{equation*}
			\int_{0}^{y}g_{1}(\sigma)\dd \sigma=\left.\sigma g_{1}(\sigma)\right|_{0}^{y}-\int_{0}^{y}\sigma g'_{1}(\sigma)\dd \sigma=g_{1}(y)y-\int_{0}^{y}\sigma g'_{1}(\sigma)\dd \sigma,
		\end{equation*}
		which directly implies
		\begin{equation*}
			\begin{aligned}
				\lim_{y\to \infty}\left(f_{2}(y)-\left(\lim_{\sigma\to \infty}g_{1}(\sigma)\right) y\right)&=-\int_{0}^{\infty}\sigma g'_{1}(\sigma)\dd \sigma,\\
				\lim_{y\to -\infty}\left(f_{2}(y)-\left(\lim_{\sigma\to -\infty}g_{1}(\sigma)\right) y\right)&=\int_{-\infty}^{0}\sigma g'_{1}(\sigma)\dd \sigma.
			\end{aligned}
		\end{equation*}
		Last, the uniqueness of $f_{2}$ is a consequence of the uniqueness of $h_{1}$.
	\end{proof}
	
	\begin{remark}\label{re:f1f2loca}
		Note that, from the construction of $f_{1}$ in (i) of Lemma~\ref{le:non}, for any $k\in \mathbb{N}$, there exist some constants $(c_{k},C_{k})\in (0,\infty)^{2}$ such that 
		\begin{equation*}
			\left| \left( f_{1}(y)+\int_{\RR}g(\sigma)\dd \sigma\right)^{(k)}\right|\textbf{1}_{(-\infty,0]}(y)
			+\big|f_{1}^{(k)}(y)\big|\textbf{1}_{[0,\infty)}(y)
			\le C_{k}\left(1+|y|\right)^{c_{k}}e^{-|y|}.
		\end{equation*}
		Similarly, for any $k\in \mathbb{N}$, there exist some constants $(c_{k},C_{k})\in (0,\infty)^{2}$ such that 
		\begin{equation*}
			\begin{aligned}
				\left|\left(f_{2}(y)-\left(\lim_{\sigma\to \infty}g_{1}(\sigma)\right) y
				+\int_{0}^{\infty}\sigma g'_{1}(\sigma)\dd \sigma
				\right)^{(k)}\right|\textbf{1}_{[0,\infty)}(y)&\le C_{k}\left(1+|y|\right)^{c_{k}}e^{-|y|},\\
				\left|\left(f_{2}(y)-\left(\lim_{\sigma\to -\infty}g_{1}(\sigma)\right) y
				-\int_{-\infty}^{0}\sigma g'_{1}(\sigma)\dd \sigma
				\right)^{(k)}\right|\textbf{1}_{(-\infty,0]}(y)&\le C_{k}\left(1+|y|\right)^{c_{k}}e^{-|y|}.
			\end{aligned}
		\end{equation*}
	\end{remark}
	
	\begin{remark}\label{re:P}
		Let $g=\Lambda Q$ in (i) of Lemma~\ref{le:non}. Then there exists a unique smooth function $P\in \mathcal{Y}_{1}$ such that 
		\begin{equation*}
			\left(\mathcal{L}P\right)'=\Lambda Q,\quad (P,Q')=0\quad \mbox{and}\quad 
			\lim_{y\to -\infty}P(y)=-\int_{\RR}\Lambda Q(\sigma)\dd \sigma.
		\end{equation*}
		In addition, from the definition of $\Lambda$ and $P$, we compute\footnote{See~\cite[Proposition 2.2]{MMRACTA} for more detail.}
		\begin{equation*}
			\lim_{y\to -\infty}P(y)=\frac{1}{2}\int_{\RR}Q(\sigma)\dd \sigma=2m_{0}\quad \mbox{and}\quad (P,Q)=\frac{1}{16}\left(\int_{\RR}Q(y)\dd y\right)^{2}=m_{0}^{2}.
		\end{equation*}
		Denote $X=P-2m_{0}(1+5Y)\in \mathcal{Z}_{1}$.
        It follows from Remark~\ref{re:Y} that 
        \begin{equation*}
			\left(\mathcal{L}X\right)'=\Lambda Q,\quad (X,Q')=0\quad \mbox{and}\quad 
			\lim_{y\to \infty}X(y)=-2m_{0}.
		\end{equation*}
		Using Remark~\ref{re:Y} and Remark~\ref{re:f1f2loca}, for any $k\in \mathbb{N}$, there exists  some constants $(c_{k},C_{k})\in (0,\infty)^{2}$ such that  
		\begin{equation*}
			\begin{aligned}
				\big| X^{(k)}(y)\big|\textbf{1}_{(-\infty,0]}(y)+ \big| \left( X(y)+2m_{0}\right)^{(k)}\big|\textbf{1}_{[0,\infty)}(y)&\le C_{k}\left(1+|y|\right)^{c_{k}}e^{-|y|}.
			\end{aligned}
		\end{equation*}
		Using again Remark~\ref{re:Y}, we compute 
		\begin{equation*}
			(X,Q)=(P,Q)-2m_{0}(1+5Y,Q)=-m_{0}^{2}.
		\end{equation*}
		Note that, from the uniqueness in (i) of Lemma~\ref{le:non}, we find 
		\begin{equation*}
			P(y)=-X(-y)\Longrightarrow X+P \ \ \mbox{is an odd function}\Longrightarrow
			(X+P,\Lambda Q)=0.
		\end{equation*}
		To simplify the notation, we denote 
		\begin{equation}\label{def:X1X2}
			\begin{aligned}
				P_{1}(s,y)=P(y)\quad \mbox{and}\quad P_{2}(s,y)=(\Gamma P)(s,y),\\
				X_{1}(s,y)=X(y)\quad \mbox{and}\quad X_{2}(s,y)=(\Gamma X)(s,y).
			\end{aligned}
		\end{equation}
	\end{remark}
	
	To compute the leading-order interaction terms between $X$ and $Q$, we rely on the following refined pointwise estimate for $X$ (see Lemma~\ref{le:refinedPAB} for more detail).
	\begin{lemma}\label{le:refindpointX}
		For any $k\in \mathbb{N}$, we have 
		\begin{equation*}
			\begin{aligned}
				\left| \left(X(y)+\frac{c_{Q}}{4}y^{2}e^{-|y|}\right)^{(k)}\right|\mathbf{1}_{(-\infty,0]}(y)&\lesssim (1+|y|)e^{-|y|},\\
				\left| \left(X(y)+2m_{0}-\frac{c_{Q}}{4}y^{2}e^{-|y|}\right)^{(k)}\right|\mathbf{1}_{[0,\infty)}(y)&\lesssim (1+|y|)e^{-|y|}.
			\end{aligned}
		\end{equation*}
	\end{lemma}
	
	\begin{proof}
		First, from~\eqref{est:QQ'} and Remark~\ref{re:Y}, we check that\footnote{Here, we use the fact that $-f''+f=g\Longrightarrow |f(y)|\lesssim \int_{\RR}e^{-|y-\sigma|}|g(\sigma)|\dd \sigma$.}
		\begin{equation}
			-Y''+Y=5Q^{4}Y+Q^{4}\Longrightarrow |Y^{(k)}(y)|\lesssim \int_{\RR}
			e^{-|y-\sigma|}e^{-4|\sigma|}\dd \sigma\lesssim e^{-|y|}.
		\end{equation}
		Second, by an elementary computation, we also check that 
		\begin{equation}\label{equ:defbarh}
			\mathcal{L}\left(yQ'\right)=-2Q''
			\Longrightarrow
			\mathcal{L}\bar{h}=yQ''+\frac{1}{2}Q'\ \ \mbox{with}\ \bar{h}=-\frac{1}{4}y^{2}Q'.
		\end{equation}
		Using again~\eqref{est:QQ'} and~\eqref{est:QQ'-}, we have 
		\begin{equation*}
			\begin{aligned}
				\left|\left(y^{2}Q'-c_{Q}y^{2}e^{-|y|}\right)^{(k)}\right|\mathbf{1}_{(-\infty,0]}&\lesssim e^{-|y|},\\
				\left|\left(y^{2}Q'+c_{Q}y^{2}e^{-|y|}\right)^{(k)}\right|\mathbf{1}_{[0,\infty)}&\lesssim e^{-|y|}.
			\end{aligned}
		\end{equation*}
		It follows from the definition of $\bar{h}$ that 
		\begin{equation}\label{est:barh}
			\begin{aligned}
				\left|\left(\bar{h}(y)+\frac{c_{Q}}{4}y^{2}e^{-|y|}\right)^{(k)}\right|\mathbf{1}_{(-\infty,0]}(y)&\lesssim (1+|y|)e^{-|y|},\\
				\left|\left(\bar{h}(y)-\frac{c_{Q}}{4}y^{2}e^{-|y|}\right)^{(k)}\right|\mathbf{1}_{[0,\infty)}(y)&\lesssim (1+|y|)e^{-|y|}.
			\end{aligned}
		\end{equation}
		Recall that, from the proof for (i) of Lemma~\ref{le:non}, we consider $P$ of the form
		\begin{equation*}
			P(y)=h(y)-\int_{y}^{\infty}\Lambda Q(\sigma)\dd \sigma,\quad \mbox{on}\   \RR.
		\end{equation*}
		Here, the real-valued function $h:\RR\to \RR$ is a solution of 
		\begin{equation*}
			\mathcal{L}h=(\Lambda Q)'-5Q^{4}\int_{y}^{\infty}\Lambda Q(\sigma)\dd \sigma,\quad \mbox{on}\ \RR.
		\end{equation*}
		Based on the above identity and~\eqref{equ:defbarh}, we have
		\begin{equation*}
			\mathcal{L}(h-\bar{h})=Q'-5Q^{4}\int_{y}^{\infty}\Lambda Q(\sigma)\dd \sigma,\quad \mbox{on}\ \RR,
		\end{equation*}
		which directly implies that
		\begin{equation}\label{est:hbarh}
			\left|(h(y)-\bar{h}(y))^{(k)}\right|
			\lesssim \int_{\RR}e^{-|y-\sigma|}e^{-|y|}\dd y
			\lesssim (1+|y|)e^{-|y|}.
		\end{equation}
		On the other hand, from the Fundamental Theorem of Calculus, 
		\begin{equation}\label{est:intLamQ}
			\begin{aligned}
				\left| \left(\int^{\infty}_{y}\Lambda Q(\sigma)\dd \sigma\right)^{(k)}\right|\mathbf{1}_{[0,\infty)}(y)
				&\lesssim (1+|y|)e^{-|y|},\\
				\left| \left(\int^{\infty}_{y}\Lambda Q(\sigma)\dd \sigma+2m_{0}\right)^{(k)}\right|\mathbf{1}_{(-\infty,0]}(y)&\lesssim (1+|y|)e^{-|y|}.
			\end{aligned}
		\end{equation}
		Combining~\eqref{est:barh}--\eqref{est:intLamQ} with the definition of $X$, we complete the proof.
	\end{proof}

	Based on Lemma~\ref{le:non}, we consider the following two auxiliary functions that will be used to construct the approximate solution for~\eqref{equ:gKdv2}. More precisely, these functions will help us to refine the leading-order term of the nonlinear interaction.
	
	\begin{lemma}\label{le:AB}
		There exist smooth functions $(A,B)\in \mathcal{Z}_{2}\times \mathcal{Z}_{2}$ such that 
		\begin{equation*}
			\begin{aligned}
				(\mathcal{L}A)'+5c_{Q}(e^{y}Q^{4})'+\alpha X=(\mathcal{L}B)'-5c_{Q}(e^{-y}Q^{4})'-\alpha X-10\alpha m_{0}W&=0.
			\end{aligned}
		\end{equation*}
		In addition, we have $(A,Q')=(B,Q')=0$ and there exist some universal constants $(a_{0},a_{1})\in \mathbb{\RR}^{2}$, depending only on $P$ and $Q$, such that 
		\begin{equation*}
			\begin{aligned}
				\lim_{y\to \infty}\left(A(y)-2\alpha m_{0}y-a_{0}\right)
				& = \lim_{y\to \infty}\left(B(y)+2\alpha m_{0}y+a_{1}\right)=0.
			\end{aligned}
		\end{equation*}
		Moreover, for any $k\in \mathbb{N}$, there exists a constant $c_{k}>0$ such that 
		\begin{equation*}
			\begin{aligned}
				\left( \left|A^{(k)}(y)\right|+
				\left|B^{(k)}(y)\right|
				\right)\mathbf{1}_{(-\infty,0]}(y)&\lesssim \left(1+|y|\right)^{c_{k}}e^{-|y|},\\
				\left|\left(A(y)-2\alpha m_{0} y
				-a_{0}
				\right)^{(k)}\right|\mathbf{1}_{[0,\infty)}(y)&\lesssim \left(1+|y|\right)^{c_{k}}e^{-|y|},\\
				\left|\left(B(y)+2\alpha m_{0} y+a_{1}
				\right)^{(k)}\right|\mathbf{1}_{[0,\infty)}(y)&\lesssim \left(1+|y|\right)^{c_{k}}e^{-|y|}.
			\end{aligned}
		\end{equation*}
	\end{lemma}

    \begin{remark}
        Indeed, the source terms involving $e^{\pm y}Q^{4}$ are introduced to refine the soliton-soliton interactions. On the other hand, the source term related to $W$ is used to refine the profile-soliton interaction between profile $A$ and soliton $Q$.
    \end{remark}
	
	\begin{proof}[Proof of Lemma~\ref{le:AB}]
		From Remark~\ref{re:Y}, Remark~\ref{re:P} and the definition of $\alpha$ in~\eqref{equ:defalpha}, 
		\begin{equation*}
			\left(
			5 c_{Q}(e^{y}Q^{4})'+\alpha X,Q\right)=\left(
			5c_{Q}(e^{-y}Q^{4})'+\alpha X+10\alpha m_{0}W,Q
			\right)=0.
		\end{equation*}
		Here, we use the fact that 
		\begin{equation*}
			\big(\big(e^{y}Q^{4}\big)',Q\big)=-\big(\big(e^{-y}Q^{4}\big)',Q\big)=\frac{1}{5}\int_{\RR}e^{y}Q^{5}(y)\dd y.
		\end{equation*}
		Therefore, the existence of $(A,B)\in \mathcal{Z}_{2}\times \mathcal{Z}_{2}$
		is a direct consequence for (ii) of Lemma~\ref{le:non}. In addition, using Remark~\ref{re:P}, we compute 
		\begin{equation*}
			\begin{aligned}
				\lim_{y\to \pm \infty}\left(  5c_{Q}(e^{y}Q^{4})'+\alpha X\right)&=-\alpha m_{0}\mp \alpha m_{0},\\
				\lim_{y\to \pm \infty}
				\left( 5c_{Q}(e^{-y}Q^{4})'+\alpha X+10\alpha m_{0}W\right)&=-\alpha m_{0}\mp \alpha m_{0}.
			\end{aligned}
		\end{equation*}
		On the other hand, the resonance $1+5Y$ satisfies $(\mathcal{L}(1+5Y))'=0$, and thus, we could use such resonance to adjust the asymptotic behavior of $A$ and $B$ on the left side of the real axis $\RR$.
		Therefore, using the above two identities and Remark~\ref{re:f1f2loca}, we complete the proof for the asymptotics behavior of $A$ and $B$ as $y\to \pm \infty$.
	\end{proof}
	
	Similar to before, we also denote 
	\begin{equation}\label{equ:defAB}
		\begin{aligned}
			A_{1}(s,y)=A(y)\quad \mbox{and}\quad 
			B_{2}(s,y)=(\Gamma B)(s,y).
		\end{aligned}
	\end{equation}
	
	In addition, we consider the following two auxiliary functions, which will be used to refine some leading-order terms at the $(b_{1}^{2},b_{2}^{2})$ level in the ODE system for $(b_{1},b_{2})$.
	\begin{lemma}\label{le:EF}
		There exist smooth functions $(E,F)\in \mathcal{Z}_{2}\times \mathcal{Z}_{2}$ such that 
		\begin{equation*}
			\begin{aligned}
				(\mathcal{L}E)'&=\Lambda X-2X+10(X^{2}Q^{3})',\\
				(\mathcal{L}F)'&=2X-\Lambda X+10m_{0}(2Y-\Lambda Y)
				-10(P^{2}Q^{3})'
				+15m_{0}Z.
			\end{aligned}
		\end{equation*}
		In addition, we have $(E,Q')=(F,Q')=0$ and there exist some universal constants $(a_{2},a_{3})\in \mathbb{\RR}^{2}$, depending only on $P$ and $Q$, such that 
		\begin{equation*}
			\begin{aligned}
				\lim_{y\to \infty}\left(E(y)-3m_{0}y-a_{2}\right)
				& = \lim_{y\to \infty}\left(F(y)+3m_{0}y+a_{3}\right)=0.
			\end{aligned}
		\end{equation*}
		Moreover, for any $k\in \mathbb{N}$, there exists a constant $c_{k}>0$ such that 
		\begin{equation*}
			\begin{aligned}
				\left( \left|E^{(k)}(y)\right|+
				\left|F^{(k)}(y)\right|
				\right)\mathbf{1}_{(-\infty,0]}(y)&\lesssim \left(1+|y|\right)^{c_{k}}e^{-|y|},\\
				\left|\left(E(y)-3m_{0} y
				-a_{2}
				\right)^{(k)}\right|\mathbf{1}_{[0,\infty)}(y)&\lesssim \left(1+|y|\right)^{c_{k}}e^{-|y|},\\
				\left|\left(F(y)+3m_{0} y+a_{3}
				\right)^{(k)}\right|\mathbf{1}_{[0,\infty)}(y)&\lesssim \left(1+|y|\right)^{c_{k}}e^{-|y|}.
			\end{aligned}
		\end{equation*}
	\end{lemma}
	
	\begin{proof}
		From Remark~\ref{re:P} and integration by parts, we check that 
		\begin{equation*}
			(\Lambda X,Q)=-(X,\Lambda Q)=(X,X'''-X')-10((X^{2}Q^{3})',Q).
		\end{equation*}
		Based on the above identity, we compute 
		\begin{equation*}
			(\Lambda X+10(X^{2}Q^{3})',Q)=-\frac{1}{2}\lim_{y\to \infty}X^{2}(y)=-2m_{0}^{2}=2(X,Q),
		\end{equation*}
		which implies that 
		\begin{equation}\label{equ:X10}
			\left(\Lambda X-2X+10(X^{2}Q^{3})',Q\right)=0.
		\end{equation}
		Similar as above, we check that 
		\begin{equation*}
			\left(\Lambda P+10(P^{2}Q^{3})',Q\right)=\frac{1}{2}\lim_{y\to -\infty}P^{2}(y)=2m_{0}^{2}.
		\end{equation*}
		Then, using the definition of $X=P-2m_{0}(1+5Y)$, we find 
		\begin{equation*}
			\left(\Lambda P,Q\right)=(\Lambda X+10m_{0}\Lambda Y,Q)+4m_{0}^{2},
		\end{equation*}
		which implies that 
		\begin{equation*}
			(\Lambda X+10m_{0}\Lambda Y+10(P^{2}Q^{3})',Q)=-2m_{0}^{2}.
		\end{equation*}
		On the other hand, using again Remark~\ref{re:Y} and Remark~\ref{re:P}, 
		\begin{equation*}
			(2X+20m_{0}Y+15m_{0}Z,Q)=-2m_{0}^{2}.
		\end{equation*}
		Combining the above two identities, we obtain 
		\begin{equation}\label{equ:X20}
			\left( 2X-\Lambda X+10m_{0}(2Y-\Lambda Y)
			-10(P^{2}Q^{3})'
			+15m_{0}Z,Q\right)=0.
		\end{equation}
		Note that, the proof of Lemma~\ref{le:EF} follows directly from equations~\eqref{equ:X10}--\eqref{equ:X20} and an argument similar to that used for the proof of Lemma~\ref{le:AB}.
	\end{proof}
	
	Similar to before, we also denote 
	\begin{equation}\label{equ:defEF}
		\begin{aligned}
			E_{1}(s,y)=E(y)\quad \mbox{and}\quad 
			F_{2}(s,y)=(\Gamma F)(s,y).
		\end{aligned}
	\end{equation}
	
	\subsection{Leading order of the nonlinear interactions}
	In this subsection, we study the interactions between the solitons and non-localized profiles which were introduced in Section~\ref{SS:linear}. We start with the following two technical Lemmas.
	\begin{lemma}\label{le:boundinter}
		Let $(\theta,K_{1},K_{2})\in (0,\infty)\times\mathbb{N}^{2}$ and $D\gg 1$. Assume that $f_{1}\in C^{\infty}(\RR;\RR)$ 
		and $f_{2}\in C^{\infty}(\RR;\RR)$
		satisfy
		\begin{equation}\label{est:f1f1'f2f2'}
			\begin{aligned}
				\left|f_{1}(y)\right|
				+\left|f_{1}'(y)\right|\lesssim \left(1+ |y|\right)^{K_{1}}e^{-\theta|y|},\quad \mbox{on} \ \RR,\\
				\left|f_{2}(y)\right|+
				\left|f_{2}'(y)\right|\lesssim 
				\left(1+ |y|\right)^{K_{2}}
				e^{-\theta|y|},\quad \mbox{on} \ \RR.
			\end{aligned}
		\end{equation}
		Then, for all $(z_{1},z_{2})\in \RR^{2}$ with $r=z_{2}-z_{1}>D$, we have 
		\begin{equation*}
			\begin{aligned}
				|\left(f_{1}(\cdot-z_{1}),f_{2}(\cdot-z_{2})\right)|&\lesssim
				\left(|z_{1}|^{K_{1}+K_{2}+1}+|z_{2}|^{K_{1}+K_{2}+1}\right)e^{-\theta r},\\
				\left\|f_{1}(\cdot-z_{1})f_{2}(\cdot-z_{2})\right\|_{H^{1}}&\lesssim 
				\left(|z_{1}|^{K_{1}+K_{2}+\frac{1}{2}}+|z_{2}|^{K_{1}+K_{2}+\frac{1}{2}}\right)e^{-\theta r}.
			\end{aligned}
		\end{equation*}
	\end{lemma}
	
	\begin{proof}
		First, using~\eqref{est:f1f1'f2f2'}, for any $y<z_{1}$, we find
		\begin{equation*}
			\begin{aligned}
				&\left| f_{1}(y-z_{1})f_{2}(y-z_{2})\right|+  \left|\left(f_{1}(y-z_{1})f_{2}(y-z_{2})\right)'\right|\\
				&\lesssim e^{-\theta r}
				\left(\langle y-z_{1}\rangle^{K_{2}}+|z_{2}|^{K_{2}}\right)\left(|f_{1}(y-z_{1})|+|f'_{1}(y-z_{1})|\right).
			\end{aligned}
		\end{equation*}
		Similarly, for any $y>z_{2}$, we find
		\begin{equation*}
			\begin{aligned}
				&\left| f_{1}(y-z_{1})f_{2}(y-z_{2})\right|+  \left|\left(f_{1}(y-z_{1})f_{2}(y-z_{2})\right)'\right|\\
				&\lesssim e^{-\theta r}
				\left(\langle y-z_{2}\rangle^{K_{1}}+|z_{1}|^{K_{1}}\right)\left(|f_{2}(y-z_{2})|+|f'_{2}(y-z_{2})|\right).
			\end{aligned}
		\end{equation*}
		Then, using again~\eqref{est:f1f1'f2f2'}, for any $z_{1}<y<z_{2}$, we deduce that
		\begin{equation*}
			\begin{aligned}
				&\left| f_{1}(y-z_{1})f_{2}(y-z_{2})\right|+\left| \left(f_{1}(y-z_{1})f_{2}(y-z_{2})\right)'\right|\\
				&\lesssim e^{-\theta r}
				\langle y-z_{1}\rangle ^{K_{1}}
				\langle y-z_{2}\rangle^{K_{2}}\lesssim \left(|z_{1}|^{K_{1}+K_{2}}+|z_{2}|^{K_{1}+K_{2}}\right)e^{-\theta r}.
			\end{aligned}
		\end{equation*}
		Combining above estimates and then integrating over $\RR$, we complete the proof.
	\end{proof}
	
	\begin{lemma}\label{le:interasym}
		Let $\left(\kappa_{1},\kappa_{2}\right)\in \RR^{2}$ and $D\gg 1$. Let $(z_{1},z_{2})\in \mathbb{R}^{2}$ with $r=z_{2}-z_{1}>D$.
		Assume that $f_{1}\in C^{\infty}(\RR;\RR)$ and $f_{2}\in C^{\infty}(\RR;\RR)$ satisfy
		\begin{equation}\label{est:ff'12}
			\left|f_{1}(y)\right|+|f_{1}'(y)|+|f_{2}(y)|+|f'_{2}(y)|\lesssim e^{-|y|},\quad \mbox{on}\ \RR.
		\end{equation}
		
		Then the following estimates hold.
		\begin{enumerate}
			\item {\emph{First-type estimate.}}
			Suppose that $f_{1}$ satisfies
			\begin{equation}\label{est:f1f1'}
				\begin{aligned}
					f_{1}(y)&=\kappa_{1}e^{-y}+O\left(e^{-2y}\right),\quad \  \ \mbox{as} \ y\to \infty,\\
					f'_{1}(y)&=-\kappa_{1}e^{-y}+O\left(e^{-2y}\right),\quad \mbox{as} \ y\to \infty.
				\end{aligned}
			\end{equation}
			Then, on $y>\frac{z_{1}+z_{2}}{2}$, we have 
			\begin{equation*}
				\begin{aligned}
					f_{1}(y-z_{1})f_{2}(y-z_{2})&=\kappa_{1}e^{-r}\left(e^{-(y-z_{2})}f_{2}(y-z_{2})\right)\\
					&+O_{H^{1}\left({z_{1}+z_{2}}<2y<\infty\right)}\left(e^{-\frac{3}{2}r}\right).
				\end{aligned}
			\end{equation*}
			
			\item {\emph{Second-type estimate.}}
			Suppose that $f_{2}$ satisfies
			\begin{equation}\label{est:f2f2'}
				\begin{aligned}
					f_{2}(y)&=\kappa_{2}e^{y}+O\left(e^{2y}\right),\quad \mbox{as} \ y\to -\infty,\\
					f'_{2}(y)&=\kappa_{2}e^{y}+O\left(e^{2y}\right),\quad \mbox{as} \ y\to -\infty.
				\end{aligned}
			\end{equation}
			Then, on $y<\frac{z_{1}+z_{2}}{2}$, we have 
			\begin{equation*}
				\begin{aligned}
					f_{1}(y-z_{1})f_{2}(y-z_{2})&=\kappa_{2}e^{-r}\left(e^{y-z_{1}}f_{1}(y-z_{1})\right)\\
					&+O_{H^{1}\left(-\infty<2y<z_{1}+z_{2}\right)}\left(e^{-\frac{3}{2}r}\right).
				\end{aligned}
			\end{equation*}
		\end{enumerate}
	\end{lemma}
	\begin{proof}
		Proof of (i). Using~\eqref{est:f1f1'}, for any $y>\frac{z_{1}+z_{2}}{2}$, we find
		\begin{equation*}
			y-z_{1}>\frac{z_{2}-z_{1}}{2}>\frac{D}{2}\Longrightarrow
			f_{1}(y-z_{1})=\kappa_{1}e^{-(y-z_{1})}+O\left(e^{-2(y-z_{1})}\right).
		\end{equation*}
		It follows from~\eqref{est:ff'12} that 
		\begin{equation*}
			\begin{aligned}
				f_{1}(y-z_{1})f_{2}(y-z_{2})&=\kappa_{1}e^{-r}\left(e^{-(y-z_{2})}f_{2}(y-z_{2})\right)\\
				&+O\left(\exp\left(-2(y-z_{1})-|y-z_{2}|\right)\right).
			\end{aligned}
		\end{equation*}
		Based on a similar argument as above, we also find
		\begin{equation*}
			\begin{aligned}
				\left(f_{1}(y-z_{1})f_{2}(y-z_{2})\right)'&=\kappa_{1}e^{-r}\left(e^{-(y-z_{2})}f_{2}(y-z_{2})\right)'\\
				&+O\left(\exp\left(-2(y-z_{1})-|y-z_{2}|\right)\right).
			\end{aligned}
		\end{equation*}
		By $r=z_{2}-z_{1}>D\gg 1$ and change of variable, we compute
		\begin{equation*}
			\begin{aligned}
				&\int_{\frac{z_{1}+z_{2}}{2}}^{\infty}\exp \left(-4(y-z_{1})-2|y-z_{2}|\right)\dd y\\
				&=\int_{-\frac{r}{2}}^{\infty}\exp \left(-4y-2|y|-4r\right)\dd y=O\left(e^{-3r}\right).
			\end{aligned}
		\end{equation*}
		Combining the above estimates, we complete the proof of first-type estimate.
		
		\smallskip
		Proof of (ii). Using~\eqref{est:f2f2'}, for any $y<\frac{z_{1}+z_{2}}{2}$, we find
		\begin{equation*}
			y-z_{2}<-\frac{z_{2}-z_{1}}{2}<-\frac{D}{2}\Longrightarrow
			f_{2}(y-z_{2})=\kappa_{2}e^{y-z_{2}}+O\left(e^{2(y-z_{2})}\right).
		\end{equation*}
		It follows from~\eqref{est:ff'12} that 
		\begin{equation*}
			\begin{aligned}
				f_{1}(y-z_{1})f_{2}(y-z_{2})&=\kappa_{2}e^{-r}\left(e^{y-z_{1}}f_{1}(y-z_{1})\right)\\
				&+O\left(\exp\left(2(y-z_{2})-|y-z_{1}|\right)\right).
			\end{aligned}
		\end{equation*}
		Based on a similar argument as above, we also find 
		\begin{equation*}
			\begin{aligned}
				\left(f_{1}(y-z_{1})f_{2}(y-z_{2})\right)'&=\kappa_{2}e^{-r}
				\left(e^{y-z_{1}}f_{1}(y-z_{1})\right)'\\
				&+O\left(\exp\left(2(y-z_{2})-|y-z_{1}|\right)\right).
			\end{aligned}
		\end{equation*}
		By $r=z_{2}-z_{1}>D\gg 1$ and change of variable, we compute
		\begin{equation*}
			\begin{aligned}
				&\int^{\frac{z_{1}+z_{2}}{2}}_{-\infty}
				\exp \left(4(y-z_{2})-2|y-z_{1}|\right)\dd y\\
				&=\int^{\frac{r}{2}}_{-\infty}\exp \left(4y-2|y|-4r\right)\dd y=O\left(e^{-3r}\right).
			\end{aligned}
		\end{equation*}
		Combining the above estimates, we complete the proof of second-type estimate.
	\end{proof}
	
	Recall that, we denote 
	\begin{equation*}
		S=Q_{1}-Q_{2},\quad G=S^{5}-Q_{1}^{5}+Q_{2}^{5} \quad \mbox{and}\quad H=5Q_{1}Q_{2}^{4}-5Q_{1}^{4}Q_{2}.
	\end{equation*}
	We now study the leading-order terms for the nonlinear interactions $G$ and $H$.
	\begin{lemma}\label{le:nonlinear}
		The following estimates hold.
		\begin{enumerate}
			\item {\emph{Bound}}. We have 
			\begin{equation*}
				\left\|\partial_{y}G-\partial_{y}H\right\|_{H^{1}}\lesssim 
				z^{\frac{1}{2}}e^{-2z}.
			\end{equation*}
			
			\item {\emph{Asymptotic}}. We have 
			\begin{equation*}
				\begin{aligned}
					\partial_{y}\left(Q_{1}^{4}Q_{2}\right)&=c_{Q}
					\left(1-\frac{1}{2}\mu z\right)e^{-z}\partial_{y}\left(e^{y}Q^{4}(y)\right)
					+O_{H^{1}}\left(e^{-\frac{3}{2}z}+ |\mu| e^{-z}\right),\\
					\partial_{y}\left(Q_{1}Q^{4}_{2}\right)&=c_{Q}e^{-z}\partial_{y}\left(e^{-(y-z)}Q^{4}(y-z)\right)
					+O_{H^{1}}\left(e^{-\frac{3}{2}z}+ |\mu| e^{-z}\right).
				\end{aligned}
			\end{equation*}
			In addition, we have 
			\begin{equation*}
				\begin{aligned}
					\partial_{y}H
					&=5c_{Q}e^{-z}
					\partial_{y}\left(e^{-(y-z)}Q^{4}(y-z)\right)
					+O_{H^{1}}\left(e^{-\frac{3}{2}z}\right)\\
					&-5c_{Q}\left(1-\frac{1}{2}\mu z\right)e^{-z}\partial_{y}\left(e^{y}Q^{4}(y)\right)+O_{H^{1}}\left(|\mu|  e^{-z}\right).
				\end{aligned}
			\end{equation*}
		\end{enumerate}
	\end{lemma}
	
	\begin{proof}
		Proof of (i). By an elementary computation, we find 
		\begin{equation*}
			G-H=10Q_{1}^{3}Q_{2}^{2}-10Q_{1}^{2}Q_{2}^{3}.
		\end{equation*}
		It follows directly that 
		\begin{equation*}
			\begin{aligned}
				\left|\partial_{y}G-\partial_{y}H\right|
				&\lesssim
				|\partial^{\le 1}_{y}Q_{1}|^{3}
				|\partial^{\le 1}_{y}Q_{2}|^{2}
				+|\partial^{\le 1}_{y}Q_{1}|^{2}
				|\partial^{\le 1}_{y}Q_{2}|^{3}
				,\\
				\left|\partial^{2}_{y}G-\partial_{y}^{2}H\right|
				&\lesssim
				|\partial^{\le 2}_{y}Q_{1}|^{3}
				|\partial^{\le 2}_{y}Q_{2}|^{2}
				+|\partial^{\le 2}_{y}Q_{1}|^{2}
				|\partial^{\le 2}_{y}Q_{2}|^{3}.
			\end{aligned}
		\end{equation*}
		Recall that, from the explicit expression of $Q$ in Section~\ref{SS:Main}, as $y\to \infty$,
		\begin{equation}\label{equ:estQQ'}
			Q(y)=c_{Q}e^{-y}+O\left(e^{-2y}\right)\quad \mbox{and}\quad 
			Q'(y)=-c_{Q}e^{-y}+O\left(e^{-2y}\right).
		\end{equation}
		Recall also that, from the explicit expression of $Q$ in Section~\ref{SS:Main}, as $y\to-\infty$,
		\begin{equation}\label{equ:estQQ'-}
			Q(y)=c_{Q}e^{y}+O\left(e^{2y}\right)\quad \mbox{and}\quad 
			Q'(y)=c_{Q}e^{y}+O\left(e^{-2y}\right).
		\end{equation}
		Then, from~\eqref{equ:muf}, we check that 
		\begin{equation*}
			\begin{aligned}
				Q_{2}(s,y)&=Q(y-z)+\frac{\mu_{}}{2}(\Lambda Q)(y-z)+O\left(\mu^{2}(y-z)^{2}e^{-(1-|\mu|)|y-z|}\right),\\
				\partial_{y}Q_{2}(s,y)&=Q'(y-z)+\frac{\mu}{2}(\Lambda Q)'(y-z)+O\left(\mu^{2}(y-z)^{2}e^{-(1-|\mu|)|y-z|}\right),\\
				\partial_{y}^{2}Q_{2}(s,y)&=Q''(y-z)+\frac{\mu}{2}(\Lambda Q)''(y-z)+O\left(\mu^{2}(y-z)^{2}e^{-(1-|\mu|)|y-z|}\right).
			\end{aligned}
		\end{equation*}
		Therefore, from Lemma~\ref{le:boundinter} and~\eqref{equ:estQQ'}--\eqref{equ:estQQ'-}, we obtain
		\begin{equation*}
			\begin{aligned}
				\left\|\partial_{y}G-\partial_{y}H\right\|_{H^{1}}
				&\lesssim 
				|\mu|z^{\frac{3}{2}}e^{-2z}+
				\mu^{2}z^{\frac{5}{2}}e^{-2z}+\mu^{4}z^{\frac{9}{2}}e^{-2(1-|\mu|)z}\\
				&+|\mu|^{3}z^{\frac{7}{2}}e^{-2z}+\mu^{6}z^{\frac{13}{2}}e^{-2(1-|\mu|)z}+ z^{\frac{1}{2}}e^{-2z}.
			\end{aligned}
		\end{equation*}
		Combining the above estimate with~\eqref{est:para1}, we complete the proof of (i).
		
		\smallskip
		Proof of (ii). Using again ~\eqref{equ:muf},~\eqref{est:para1} and Lemma~\ref{le:boundinter}, 
		\begin{equation*}
			\partial_{y}\left(Q_{1}^{4}Q_{2}\right)
			=
			\partial_{y}\left(Q^{4}(y)Q(y-z)\right)
			+\frac{\mu}{2}\partial_{y}(Q^{4}(y)(\Lambda Q)(y-z))+O_{H^{1}}\left(|\mu|e^{-z}\right).
		\end{equation*}
		Based on (ii) of Lemma~\ref{le:interasym} and~\eqref{equ:estQQ'}--\eqref{equ:estQQ'-}, for $y<\frac{z}{2}$, we find 
		\begin{equation*}
			\begin{aligned}
				\partial_{y}\left(Q^{4}(y)Q(y-z)\right)&=c_{Q}e^{-z}\partial_{y}\left(e^{y}Q^{4}(y)\right)+O_{H^{1}(-\infty<2y<z)}\left(e^{-\frac{3}{2}z}\right),\\
				\partial_{y}(Q^{4}(y)(\Lambda Q)(y-z))&=-c_{Q}ze^{-z}\partial_{y}(e^{y}Q^{4}(y))+O_{H^{1}(-\infty<2y<z)}\left(
				e^{-z}\right).
			\end{aligned}
		\end{equation*}
		Here, we use the fact that 
		\begin{equation*}
			\begin{aligned}
				(\Lambda Q)^{(k)}(y-z)
				&=-zQ^{(k+1)}(y-z)+yQ^{(k+1)}(y-z)+\frac{1}{2}Q^{(k)}(y-z)\\
				&=-c_{Q}ze^{y-z}+O\left(ze^{2(y-z)}+(1+|y|)e^{y-z}\right),\quad \mbox{for}\ \ y<\frac{z}{2}.
			\end{aligned}
		\end{equation*}
		On the other hand, from the exponential decay of $Q$, we directly have 
		\begin{equation*}
			\left\| \partial_{y}\left(Q_{1}^{4}Q_{2}\right)\right\|_{H^{1}(z<2y<\infty)}+
			\left\|e^{-z}\partial_{y}\left(e^{y}Q^{4}(y)\right)\right\|_{H^{1}(z<2y<\infty)}\lesssim e^{-\frac{3}{2}z}.
		\end{equation*}
		Combining the above estimates, we complete the proof for the first estimate.
		
		\smallskip
		Based on (i) of Lemma~\ref{le:interasym} and~\eqref{equ:estQQ'}, for $y>\frac{z}{2}$, we find 
		\begin{equation*}
			\begin{aligned}
				\partial_{y}\left(Q(y)Q^{4}(y-z)\right)&=c_{Q}e^{-z}\partial_{y}\left(e^{-(y-z)}Q^{4}(y-z)\right)+O_{H^{1}(z<2y<\infty)}\left(e^{-\frac{3}{2}z}\right).
			\end{aligned}
		\end{equation*}
		Thus, using a similar argument, we complete the proof for the second estimate. Combining the first and second estimates, we complete the proof for the third one, and thus, the proof of (ii) is complete.
	\end{proof}
	
	Since the functions $(X,A,B,E,F)$ are non-localized profiles, the following estimates related to the interactions between solitons and such profiles are also needed.
	
	\begin{lemma}\label{le:PAB}
		The following estimates hold for some constant $K>1$.
		\begin{enumerate}
			\item {\emph{Estimates related to $X$}.} We have 
			\begin{equation*}
				\begin{aligned}
					\left\|\partial_{y}\left((X_{1}+2m_{0})\right)Q_{2}^{4}\right\|_{H^{1}}&\lesssim z^{K}e^{-z},\\
					\left\|\partial_{y}\left((X_{1}+2m_{0})\right)Q_{2}^{3}\right\|_{H^{1}}&\lesssim z^{K}e^{-z},\\
					\left\|\partial_{y}\left(Q_{1}^{3}X_{2}\right)\right\|_{H^{1}}
					+
					\left\|\partial_{y}\left(Q_{1}^{4}X_{2}\right)\right\|_{H^{1}}&\lesssim z^{K}e^{-z}.
				\end{aligned}
			\end{equation*}
			In particular, we have 
			\begin{equation*}
				\begin{aligned}
					\left\|\left(Q_{1}^{3}X_{2}\right)\right\|_{H^{1}}
					+
					\left\|\left(Q_{1}^{4}X_{2}\right)\right\|_{H^{1}}&\lesssim z^{K}e^{-z},\\
                    |(Q_{1},X_{2})|+ |(\partial_{y}Q_{1},X_{2})|
                    +|(\Lambda Q_{1},X_{2})|
                    &\lesssim z^{K}e^{-z}.
				\end{aligned}
			\end{equation*}
			
			\item  {\emph{Estimates related to $A$ and $B$}.} We have 
			\begin{equation*}
				\begin{aligned}
					\left\|\partial_{y}\left(Q_{1}^{4}B_{2}\right)\right\|_{H^{1}}
					&\lesssim z^{K}e^{-z},\\
					\left\|\partial_{y}\left(\left(A_{1}-2\alpha m_{0}y-a_{0}\right)Q_{2}^{4}\right)\right\|_{H^{1}}
					&\lesssim z^{K}e^{-z}.
				\end{aligned}
			\end{equation*}
			
			\item {\emph{Estimates related to $E$ and $F$}.} We have 
			\begin{equation*}
				\begin{aligned}
					\left\|\partial_{y}\left(Q_{1}^{4}F_{2}\right)\right\|_{H^{1}}
					&\lesssim z^{K}e^{-z},\\
					\left\|\partial_{y}\left(\left(E_{1}-3m_{0}y-a_{2}\right)Q_{2}^{4}\right)\right\|_{H^{1}}
					&\lesssim z^{K}e^{-z}.
				\end{aligned}
			\end{equation*}
		\end{enumerate}
	\end{lemma}
	
	\begin{proof}
		Proof of (i). First, using~\eqref{est:para1} and Remark~\ref{re:P},
		for any $k\in \mathbb{N}$, there exists a constant $c_{k}>0$ such that for any $y<\frac{z}{2}$, we have 
		\begin{equation*}
        \begin{aligned}
			\big|X^{(k)}_{2}(y)\big|
            &\lesssim 
            \left(1+|y-z|^{c_{k}}\right)e^{-(1-|\mu|)^{\frac{1}{2}}|y-z|}
            \\
            &\lesssim 
			\left(
			(1+|y|)^{c_{k}}+z^{c_{k}}\right)e^{-z}e^{(1-|\mu|)^{\frac{1}{2}}y}.
            \end{aligned}
		\end{equation*}
		Then, using again the exponential decay of $Q$ in~\eqref{est:QQ'}, for any $k\in \mathbb{N}$ and $y>\frac{z}{2}$, 
		\begin{equation*}
        \begin{aligned}
			\big|Q^{(k)}_{1}(y)\big|&\lesssim \big|Q^{(k)}(y)\big|\lesssim e^{-y}\lesssim e^{-z}e^{-(y-z)},\\
             \big|X^{(k)}_{2}(y)\big|&\lesssim 
        |X((1+\mu)^{\frac{1}{2}}(y-z))|
        \lesssim e^{-(1-|\mu|)^{\frac{1}{2}}|y-z|}.
            \end{aligned}
		\end{equation*}
		Combining the above estimates and then integrating over $\RR$, we complete the proof for the estimates related to $Q_{1}$ and $X_{2}$.
		
		\smallskip
		On the other hand, using again~\eqref{est:QQ'} and~\eqref{est:para1}, for any $k\in \mathbb{N}$ and $y<\frac{z}{2}$, 
		\begin{equation*}
			\big|Q^{(k)}_{2}(y)\big|\lesssim e^{-(1-|\mu|)^{\frac{1}{2}}z}e^{(1-{|\mu|})^{\frac{1}{2}}y}
			\lesssim
			e^{-z}e^{(1-|\mu|)^{\frac{1}{2}}y}.
		\end{equation*}
		Then, using again Remark~\ref{re:P},
		for any $k\in \mathbb{N}$, there exists a constant $c_{k}>0$ such that for any $y>\frac{z}{2}$, we have
		\begin{equation*}
			\big|\left(X_{1}+2m_{0}\right)^{(k)}(y)\big|
			\lesssim (1+|y|)^{c_{k}}e^{-y}
			\lesssim 
			\left(
			(1+|y-z|)^{c_{k}}+z^{c_{k}}\right)e^{-z}e^{-(y-z)}.
		\end{equation*}
		Combining the above estimates and then integrating over $\RR$, we complete the proof for the estimates related to $Q_{2}$ and $X_{1}$.
		
		\smallskip
		Proof of (ii). Recall that, from Lemma~\ref{le:AB}, for any $k\in \mathbb{N}$, there exists a constant $c_{k}>0$ such that for any $y<\frac{z}{2}$, we have
		\begin{equation*}
			\big|B^{(k)}_{2}(y)\big|\lesssim 
			\left(
			(1+|y|)^{c_{k}}+z^{c_{k}}\right)e^{-z}e^{(1-|\mu|)^{\frac{1}{2}}y}.
		\end{equation*}
		Recall also that, from Lemma~\ref{le:AB}, for any $k\in \mathbb{N}$, there exists a constant $c_{k}>0$ such that for any $y>\frac{z}{2}$, we have
		\begin{equation*}
			\big|(A_{1}-2\alpha m_{0}y-a_{0})^{(k)}(y)\big|\lesssim (1+|y|)^{c_{k}}e^{-y}
			\lesssim 
			\left(
			(1+|y-z|)^{c_{k}}+z^{c_{k}}\right)e^{-z}e^{-(y-z)}.
		\end{equation*}
		Hence, using a similar argument as in the proof of (i), we complete the proof.

        \smallskip
        Proof of (iii). The proof is similar to (ii), and we omit it.
	\end{proof}
	
	\begin{lemma}\label{le:refinedPAB}
		The following refined estimates hold.
		\begin{enumerate}
			\item \emph{Refined estimate related to $X_{1}$}. We have 
			\begin{equation*}
				\begin{aligned}
					\partial_{y}(X_{1}Q_{2}^{4})&=\frac{c_{Q}}{4}z^{2}e^{-z}\partial_{y}\left(e^{-(y-z)}Q^{4}(y-z)\right)\\
					&-2m_{0}\partial_{y}(Q_{2}^{4})+O_{H^{1}}(ze^{-z}+|\mu|z^{3}e^{-z}).
				\end{aligned}
			\end{equation*}
			\item  \emph{Refined estimate related to $X_{2}$}. We have 
			\begin{equation*}
				\partial_{y}(Q_{1}^{4}X_{2})=-\frac{c_{Q}}{4}z^{2}e^{-z}\partial_{y}\left(e^{y}Q^{4}(y)\right)+O_{H^{1}}(ze^{-z}+|\mu|z^{3}e^{-z}).
			\end{equation*}
		\end{enumerate}
	\end{lemma}
	
	\begin{proof}
		Proof of (i). Recall that, from~\eqref{equ:muf}, we find  
		\begin{equation*}
			\begin{aligned}
				Q_{2}(s,y)&=Q(y-z)+\frac{\mu_{}}{2}(\Lambda Q)(y-z)+O\left(\mu^{2}(y-z)^{2}e^{-(1-|\mu|)|y-z|}\right),\\
				\partial_{y}Q_{2}(s,y)&=Q'(y-z)+\frac{\mu}{2}(\Lambda Q)'(y-z)+O\left(\mu^{2}(y-z)^{2}e^{-(1-|\mu|)|y-z|}\right),\\
				\partial_{y}^{2}Q_{2}(s,y)&=Q''(y-z)+\frac{\mu}{2}(\Lambda Q)''(y-z)+O\left(\mu^{2}(y-z)^{2}e^{-(1-|\mu|)|y-z|}\right).
			\end{aligned}
		\end{equation*}
		Then, from Lemma~\ref{le:refindpointX}, for any $k\in \mathbb{N}$ and $y>0$, 
		\begin{equation*}
			\begin{aligned}
				X_{1}^{(k)}(y)
				&=\left(-2m_{0}+\frac{c_{Q}}{4}z^{2}e^{-y}\right)^{(k)}
				+O\left((1+|y-z|+z)e^{-y}\right)\\
				&+\frac{c_{Q}}{4}\left(\left(y-z)^{2}e^{-y}\right)\right)^{(k)}
				+\frac{c_{Q}}{2}\left(z(y-z)e^{-y}\right)^{(k)}.
			\end{aligned}
		\end{equation*}
		Based on the above estimate and Lemma~\ref{le:interasym}, we obtain
		\begin{equation}\label{est:asymX1Q241}
			\begin{aligned}
				\partial_{y}(X_{1}Q_{2}^{4})&=\frac{c_{Q}}{4}z^{2}e^{-z}\partial_{y}\left(e^{-(y-z)}Q^{4}(y-z)\right)\\
				&-2m_{0}\partial_{y}(Q_{2}^{4})+O_{H^{1}(0,\infty)}(ze^{-z}+|\mu|z^{3}e^{-z}).
			\end{aligned}
		\end{equation}
		Here, we use the fact that 
		\begin{equation*}
			\int_{0}^{\infty}e^{-2y}\left(1+|y-z|^{2}+|y-z|^{4}\right)e^{-8|y-z|}\dd y\lesssim e^{-2z}.
		\end{equation*}
		On the other hand, using again the exponential decay of $Q$, 
		\begin{equation*}
			\begin{aligned}
				\partial_{y}(X_{1}Q_{2}^{4})&=\frac{c_{Q}}{4}z^{2}e^{-z}\partial_{y}\left(e^{-(y-z)}Q^{4}(y-z)\right)\\
				&-2m_{0}\partial_{y}(Q_{2}^{4})+O_{H^{1}(-\infty,0)}(ze^{-z}+|\mu|z^{3}e^{-z}).
			\end{aligned}
		\end{equation*}
		Combining the above estimate with~\eqref{est:asymX1Q241}, we complete the proof for (i).
		
		\smallskip
		Proof of (ii). Using again~\eqref{equ:muf} and Lemma~\ref{le:refindpointX}, for any $y<z$, we find  
		\begin{equation*}
			\begin{aligned}
				X_{2}(s,y)&=X(y-z)+\frac{\mu_{}}{2}(\Lambda X)(y-z)+O\left(\mu^{2}(y-z)^{4}e^{-(1-|\mu|)|y-z|}\right),\\
				\partial_{y}X_{2}(s,y)&=X'(y-z)+\frac{\mu}{2}(\Lambda X)'(y-z)+O\left(\mu^{2}(y-z)^{4}e^{-(1-|\mu|)|y-z|}\right),\\
				\partial_{y}^{2}X_{2}(s,y)&=X''(y-z)+\frac{\mu}{2}(\Lambda X)''(y-z)+O\left(\mu^{2}(y-z)^{4}e^{-(1-|\mu|)|y-z|}\right).
			\end{aligned}
		\end{equation*}
		Then, from Lemma~\ref{le:refindpointX}, for any $k\in \mathbb{N}$ and $y<z$,
		\begin{equation*}
			\begin{aligned}
				X^{(k)}(y-z)
				&=-\frac{c_{Q}}{4}\left(z^{2}e^{(y-z)}\right)^{(k)}
				+O\left((1+|y|+z)e^{(y-z)}\right)\\
				&-\frac{c_{Q}}{4}\left(y^{2}e^{(y-z)}\right)^{(k)}
				+\frac{c_{Q}}{2}\left(yz e^{(y-z)}\right)^{(k)}.
			\end{aligned}
		\end{equation*}
		Based on the above estimate and Lemma~\ref{le:interasym}, we obtain
		\begin{equation}\label{est:asymX2Q14}
			\partial_{y}(Q_{1}^{4}X_{2})=-\frac{c_{Q}}{4}z^{2}e^{-z}\partial_{y}\left(e^{y}Q^{4}(y)\right)+O_{H^{1}(-\infty,z)}(ze^{-z}+|\mu|z^{3}e^{-z}).
		\end{equation}
		On the other hand, using again the exponential decay of $Q$, 
		\begin{equation*}
			\partial_{y}(Q_{1}^{4}X_{2})=-\frac{c_{Q}}{4}z^{2}e^{-z}\partial_{y}\left(e^{y}Q^{4}(y)\right)+O_{H^{1}(z,\infty)}(ze^{-z}+|\mu|z^{3}e^{-z}).
		\end{equation*}
		Combining the above estimate with~\eqref{est:asymX2Q14}, we complete the proof for (ii).
	\end{proof}
	
	\subsection{Approximate solution}\label{SS:Approximate}
	In this subsection, we construct the approximate two-bubble solution for~\eqref{equ:gKdv2}. We now state the definition of admissible geometrical parameters $\mathcal{G}=(\lambda,z,\mu,x_{1},b_{1},b_{2})\in (0,\infty)^{2}\times \RR^{4}$. From now on, we denote by $\Theta = b_2 - b_1$ to keep track of the size difference between $b_1$ and $b_2$.
	\begin{definition}\label{def:Admissible}
		Let $I=[s_{0},s_{1}]\subset (0,\infty)$ be an time interval.
		We say that the function $\mathcal{G}(s):I\mapsto (0,\infty)^{2}\times \RR^{4}$ is admissible if it satisfies
		\begin{equation*}
			\begin{aligned}
				\frac{1}{7 s\log s}\le b_{1}\le \frac{1}{5s\log s},\quad 
				\frac{1}{7 s\log s}\le b_{2}\le \frac{1}{5s\log s},\\
				\frac{1}{7 s^{2}\log s}\le \alpha e^{-z}\le \frac{1}{5s^{2}\log s}
				\quad \mbox{and}\quad s|\Theta|+|\mu|\le \frac{10}{s}.
			\end{aligned}
		\end{equation*}
	\end{definition}
	Note that, from the definition of admissible geometrical parameters for $z$, we find 
	\begin{equation}\label{est:z1}
		e^{-z}\in\left[\frac{1}{7\alpha s^{2}\log s},\frac{1}{5\alpha s^{2}\log s}\right]\Longrightarrow |z-2\log s|\lesssim \log \log s.
	\end{equation}
	
	To simplify the notation, for any admissible geometrical function $\mathcal{G}(s)$, we denote\footnote{See~\eqref{def:X1X2} and~\eqref{equ:defAB} for the definition of $(X_{1},X_{2},A_{1},B_{2})$} 
	\begin{equation}\label{equ:defR}
		a=6b_{1}|\log b_{1}|\ \  \mbox{and}\ \ 
		R=b_{1}X_{1}-\zeta_{2} X_{2}+e^{-z}(A_{1}+B_{2}).
	\end{equation}
	Here, we denote $b_{2}=(1+\mu)^{\frac{3}{2}}\zeta_{2}$. From Definition~\ref{def:Admissible}, we compute 
	\begin{equation}\label{est:aa-1}
		a=6b_{1}|\log b_{1}|\in \left[\frac{4}{5s},\frac{2}{s}\right]\subseteq \Longrightarrow a^{-1}\in \left[\frac{s}{2},\frac{5s}{2}\right].
	\end{equation}
	Note that, the above-mentioned functions are non-localized profiles, and thus, we should consider the following suitable cut-off function
	\begin{equation}\label{equ:defphi}
		\phi(s,y)=\chi\left(a(s) y\right)\Longrightarrow \phi_{|(-\infty,a^{-1})}\equiv 1\quad \mbox{and}\quad 
		\phi_{|(2a^{-1},\infty)}\equiv 0.
	\end{equation}
	In addition, we require the refined term $U = U_1 + U_2$. More precisely, the term $U_1$ is introduced to refined the interactions at the $b_1^2$ level\footnote{See~\eqref{equ:defEF} for the definition of $(E_{1},F_{2})$.}:
	\begin{equation}\label{equ:defU1}
		U_{1}=b_{1}^{2}(E_{1}+F_{2})\phi+5(3m_{0}z+a_{2})b_{1}^{2}Y_{2}.
	\end{equation}
	We also need the following refined term $U_{2}$ which is related to the interactions between the soliton $Q_{2}$ and the non-localized profiles $X_{1}$ and $A_{1}$:
	\begin{equation}\label{equ:defU2}
		U_{2}=5\left(2\alpha m_{0}z+a_{0}\right)e^{-z}Y_{2}-\frac{10 m_{0}b_{1} Y_{2}}{(1+\mu)^{\frac{1}{4}}}.
	\end{equation}
	We now introduce the approximate solution and the related error term,
	\begin{equation}\label{equ:defV}
		V=S+R\phi+U \ \  \mbox{and} \ \ 
		\Psi(V)=\partial_{s}V+\partial_{y}\left(\partial_{y}^{2}V-V+V^{5}\right)+b_{1}\Lambda V.
	\end{equation}
	To avoid abuse of notation, we also denote $  V(s,y)=V(y;(z,\mu,b_{1},b_{2}))\in H^{1}$.
	
	\begin{remark}\label{re:SRU}
		On the one hand, the first two terms in $R$ are the standard blow-up profiles for the mass-critical gKdV equation, while the last two are related to the interaction between the two solitons. On the other hand, the first term in $U_{1}$ is used to refine the profile-soliton interaction at the $(b_{1}^{2},b^{2}_{2})$ level, while the second one is related to the interaction between the soliton $Q_{2}$ and the profile $E_{1}$. Last, the first term in $U_{2}$ is used to refine the interaction between the soliton $Q_{2}$ and the profile $A_{1}$, while the second one is related to the interaction between the soliton $Q_{2}$ and the profile $X_{1}$. See Proposition~\ref{prop:approx} for more detail.
	\end{remark}
	
	We now introduce the estimates related to non-localized profiles.
	\begin{lemma}\label{le:RAB}
		Let $I=[s_{0},s_{1}]\subset (0,\infty)$ be an time interval and 
		the function $\mathcal{G}(s):I\mapsto (0,\infty)^{2}\times \RR^{4}$ be admissible. Then the following estimates hold.
		\begin{enumerate}
			\item \emph{Estimate for $X_{1}$ and $X_{2}$.} We have 
			\begin{equation*}
				\|(X_{1}-X_{2})\phi\|_{H^{1}}\lesssim \log^{\frac{1}{2}}s.
			\end{equation*}
			
			\item \emph{Estimate for $R$.} We have 
			\begin{equation*}
				\left\|(\partial_{y}R)(\partial^{2}_{y}\phi)\right\|_{H^{1}}+ \left\|R\partial^{3}_{y}\phi\right\|_{H^{1}}
				+\left\|\left(\partial^{2}_{y}R\right)(\partial_{y}\phi)\right\|_{H^{1}}\lesssim \frac{1}{s^{3}\log s}.
			\end{equation*}
			\item \emph{Estimate for $A_{1}$ and $B_{2}$.} We have 
			\begin{equation*}
				\|(A_{1}+B_{2})\phi\|_{H^{1}}
				+\|y(\partial_{y}(A_{1}+B_{2}))\phi\|_{H^{1}}
				\lesssim s^{\frac{1}{2}}\log s.
			\end{equation*}
			
			\item \emph{First estimate for $E_{1}$ and $F_{2}$.} We have 
			\begin{equation*}
				\|(E_{1}+F_{2})\phi\|_{H^{1}}
				+\|y(\partial_{y}(E_{1}+F_{2}))\phi\|_{H^{1}}
				\lesssim s^{\frac{1}{2}}\log s.
			\end{equation*}
			
			\item \emph{Second estimates for $E_{1}$ and $F_{2}$.} We have 
			\begin{equation*}
				\begin{aligned}
					\left\|(\partial_{y}E_{1}+\partial_{y}F_{2})(\partial^{2}_{y}\phi)\right\|_{H^{1}}+ \left\|(E_{1}+F_{2})\partial^{3}_{y}\phi\right\|_{H^{1}}
					&\lesssim s^{-1},\\
					\left\|\left(\partial^{2}_{y}E_{1}+\partial_{y}^{2}F_{2}\right)(\partial_{y}\phi)\right\|_{H^{1}}+\|S^{4}(E_{1}+F_{2})\partial_{y}\phi\|_{H^{1}}&\lesssim s^{-1}.
				\end{aligned}
			\end{equation*}
			
			\item \emph{Third estimates for $E_{1}$ and $F_{2}$.} We have 
			\begin{equation*}
				\left\|(E_{1}+F_{2})\partial_{y}\phi\right\|_{H^{1}}+ \left\|y(\partial_{y}(E_{1}+F_{2}))\partial_{y}\phi\right\|_{H^{1}}\lesssim s^{-\frac{1}{2}}\log s.
			\end{equation*}
		\end{enumerate}
	\end{lemma}
	
	\begin{proof}
		Proof of (i). We first rewrite the term $X_{1}-X_{2}$ by 
		\begin{equation*}
			X_{1}-X_{2}=(X_{1}+2m_{0})-\big(X_{2}+2(1+\mu)^{\frac{1}{4}}m_{0}\big)+2\big((1+\mu)^{\frac{1}{4}}-1\big)m_{0}.
		\end{equation*}
		It follows from Remark~\ref{re:P} and Definition~\ref{def:Admissible} that 
		\begin{equation*}
			\begin{aligned}
				|\partial_{y}^{\le 1}(X_{1}-X_{2})|&\lesssim \big(e^{-\frac{|y|}{2}}+e^{-\frac{|y-z|}{2}}\big)\textbf{1}_{(-\infty,0)}(y)+\textbf{1}_{[0,z]}(y)\\
				&+\big(e^{-\frac{|y|}{2}}+e^{-\frac{|y-z|}{2}}+s^{-1}\big)\textbf{1}_{(z,\infty)}(y).
			\end{aligned}
		\end{equation*}
		Based on the above estimate and $\rm{supp}\phi\subset (-\infty,\frac{5s}{2})$, we complete the proof for (i).
		
		\smallskip
		Proof of (ii).
		We first rewrite the term $R$ by 
		\begin{equation*}
			\begin{aligned}
				R
				&=e^{-z}\big(B_{2}+2\alpha m_{0}(1+\mu)^{\frac{3}{4}}(y-z)\big)+2\alpha m_{0}(1+\mu)^{\frac{3}{4}}ze^{-z}\\
				&+2b_{1}\big((1+\mu)^{\frac{1}{4}}-1\big)m_{0}-b_{1}(X_{2}+2(1+\mu)^{\frac{1}{4}}m_{0})+(b_{1}-\zeta_{2})X_{2}\\
				&+b_{1}(X_{1}+2m_{0})+e^{-z}(A_{1}-2\alpha m_{0}y)+2\alpha m_{0}e^{-z}y\big(1-(1+\mu)^{\frac{3}{4}}\big).
			\end{aligned}
		\end{equation*}
		It follows from~\eqref{est:z1}, Remark~\ref{re:P}, Lemma~\ref{le:AB} and Definition~\ref{def:Admissible} that 
		\begin{equation*}
			|R|\textbf{1}_{[\frac{s}{2},\frac{5s}{2}]}(y)\lesssim \left(s^{-2}+ (s^{-3}\log^{-1}s)y\right)\textbf{1}_{[\frac{s}{2},\frac{5s}{2}]}(y)\lesssim s^{-2}.
		\end{equation*}
		Moreover, we check that 
		\begin{equation*}
			\left(|\partial_{y}R|+ |\partial^{2}_{y}R|+ |\partial^{3}_{y}R|\right)\textbf{1}_{[\frac{s}{2},\frac{5s}{2}]}(y)\lesssim \frac{1}{s^{3}\log s}.
		\end{equation*}
		Second, from~\eqref{est:aa-1} and the definition of $\phi$ in~\eqref{equ:defphi}, we find 
		\begin{equation}\label{est:pyphi}
			s|\partial_{y}\phi|+
			s^{2}|\partial_{y}^{2}\phi|+s^{3}|\partial_{y}^{3}\phi|+s^{4}|\partial_{y}^{4}\phi|\lesssim \textbf{1}_{[\frac{s}{2},\frac{5s}{2}]}(y).
		\end{equation}
		Combining the above estimates, we obtain 
		\begin{equation*}
			\begin{aligned}
				& \left\|(\partial_{y}R)(\partial^{2}_{y}\phi)\right\|_{H^{1}}+ \left\|R\partial^{3}_{y}\phi\right\|_{H^{1}}
				+\left\|\left(\partial^{2}_{y}R\right)(\partial_{y}\phi)\right\|_{H^{1}}\\
				&\lesssim
				\left \|\partial^{\le 3}_{y}R|\textbf{1}_{[\frac{s}{2},\frac{5s}{2}]}\right\|_{L^{\infty}}\|\partial^{2}_{y}\phi\|_{H^{1}}+
				\left \|\partial^{\le 3}_{y}R|\textbf{1}_{[\frac{s}{2},\frac{5s}{2}]}\right\|_{L^{\infty}}
				\|\partial^{3}_{y}\phi\|_{H^{1}}\\
				&+\left\|\partial_{y}^{2}R\right\|_{L^{\infty}}\|\partial_{y}\phi\|_{H^{1}}+\left\|\partial_{y}^{3}R\right\|_{L^{\infty}}\|\partial_{y}\phi\|_{H^{1}}
				\lesssim \frac{1}{s^{3}\log s},
			\end{aligned}
		\end{equation*}
		which directly completes the proof of (ii).
		
		\smallskip
		Proof of (iii).
		We first rewrite the term $A_{1}+B_{2}$ by 
		\begin{equation*}
			\begin{aligned}
				A_{1}+B_{2}&=(A_{1}-2\alpha m_{0}y)
				+2\alpha m_{0}y\big(1-(1+\mu)^{\frac{3}{4}}\big)\\
				&+(B_{2}+2\alpha m_{0}(1+\mu)^{\frac{3}{4}}(y-z))+2\alpha m_{0}(1+\mu)^{\frac{3}{4}}z.
			\end{aligned}
		\end{equation*}
		It follows from Lemma~\ref{le:AB} and Definition~\ref{def:Admissible} that 
		\begin{equation*}
			\begin{aligned}
				\left\|\left(A_{1}+B_{2}\right)\phi\right\|_{H^{1}(z,\infty)}\lesssim  \left\|\left(A_{1}+B_{2}\right)\phi\right\|_{H^{1}\left(z,\frac{5}{2}s\right)}\lesssim s^{\frac{1}{2}}\log s,\\
				\|y(\partial_{y}(A_{1}+B_{2}))\phi\|_{H^{1}(z,\infty)}\lesssim 
				\|y(\partial_{y}(A_{1}+B_{2}))\phi\|_{H^{1}\left(z,\frac{5}{2}s\right)}\lesssim s^{\frac{1}{2}}\log s.
			\end{aligned}
		\end{equation*}
		Second, using again Lemma~\ref{le:AB} and Definition~\ref{def:Admissible}, on $[0,z]$, we find
		\begin{equation*}
			\left|\partial_{y}^{\le 2}A_{1}\right|+ \left|\partial_{y}^{\le 2}B_{2}\right|\lesssim (1+|y|)+\log s,
		\end{equation*}
		which directly implies that
		\begin{equation*}
			\left\|\left(A_{1}+B_{2}\right)\phi\right\|_{H^{1}(0,z)}
			+\|y(\partial_{y}(A_{1}+B_{2}))\phi\|_{H^{1}(0,z)}
			\lesssim \log^{4}s.
		\end{equation*}
		Last, from the exponential decay of $A$ and $B$ on the left-hand side of $\RR$, 
		\begin{equation*}
			\left\|\left(A_{1}+B_{2}\right)\phi\right\|_{H^{1}(-\infty,0)}
			+ \|y(\partial_{y}(A_{1}+B_{2}))\phi\|_{H^{1}(-\infty,0)}
			\lesssim 1.
		\end{equation*}
		Combining the above estimates, we directly complete the proof of (iii).
		
		\smallskip
		Proof of (iv). We first rewrite the term $E_{1}+F_{2}$ by 
		\begin{equation}\label{equ:E1F21}
			\begin{aligned}
				E_{1}+F_{2}&=(E_{1}-3m_{0}y)
				+3m_{0}y\big(1-(1+\mu)^{\frac{3}{4}}\big)\\
				&+(F_{2}+3m_{0}(1+\mu)^{\frac{3}{4}}(y-z))+3m_{0}(1+\mu)^{\frac{3}{4}}z.
			\end{aligned}
		\end{equation}
		Hence, using an argument similar to that in (iii), we complete the proof for (iv).
		
		\smallskip
		Proof of (v). Note that, from~\eqref{equ:E1F21} and Lemma~\ref{le:EF}, we have 
		\begin{equation*}
			\left\|(E_{1}+F_{2})
			\textbf{1}_{[\frac{s}{2},\frac{5s}{2}]}
			\right\|_{L^{\infty}}\lesssim \log s.
		\end{equation*}
		In addition, we also have 
		\begin{equation*}
			\begin{aligned}
				\left( \left|\partial_{y}(E_{1}+F_{2})\right|+ \left|\partial^{2}_{y}(E_{1}+F_{2})\right|
				\right)
				\textbf{1}_{[\frac{s}{2},\frac{5s}{2}]}(y)&\lesssim s^{-1},\\
				\left|\partial^{3}_{y}(E_{1}+F_{2})\right|
				+\left|\partial_{y}^{\le 1}(S^{4}(E_{1}+F_{2}))\right|
				\textbf{1}_{[\frac{s}{2},\frac{5s}{2}]}(y)
				&\lesssim s^{-1}.
			\end{aligned}
		\end{equation*}
		Hence, using an argument similar to that in (ii), we complete the proof of (v).
		
		\smallskip
		Proof of (vi). The proof is similar to that of (iv), and we omit it.
	\end{proof}

	Consider the following functions which are related to the modulation equations:
	\begin{equation}\label{equ:defMod}
		\vec{{\rm{Mod}}}=\left(\dot{z}-b_{1} z-\mu,\vec{N},\dot{\mu}-2\Theta\right).
	\end{equation}
	Here, we denote $\Theta=b_{2}-b_{1}$ and 
	\begin{equation*}
		\vec{N}=\left(\dot{\zeta}_{2}+\alpha(1+\mu)^{\frac{3}{2}} e^{-z}+2(1+\mu)^{\frac{3}{2}}b_{1}^{2},
		\dot{b}_{1}+\alpha e^{-z}+2b_{1}^{2}\right).
	\end{equation*}
	\begin{definition}\label{def:Sfunction}
		Let $I=[s_{0},s_{1}]\subset (0,\infty)$ be a time interval and 
		the function $\mathcal{G}(s):I\mapsto (0,\infty)^{2}\times \RR^{3}$ be admissible.
		We denote by $\mathcal{S}$ the set of smooth functions $f:I\times \RR\to \RR$ such that 
		\begin{equation*}
			\|f(s)\|_{H^{1}}\lesssim 
			\frac{|\dot{z}|}{s^{2}\log s}
			+\frac{|\dot{\mu}|}{s^{\frac{3}{2}}\log s}+\frac{|\dot{b}_{1}|+|\dot{b}_{2}|}{s \log s}
			+\frac{1}{s^{3}\log s}.
		\end{equation*}
		In addition, for such smooth function $f\in \mathcal{S}$, we denote $f=O_{\mathcal{S}}(1)$.
	\end{definition}
    \begin{remark}
        Indeed, the fastest decaying term in the leading-order ODE system for the parameters is $s^{-3}$ (see Lemma~\ref{le:refinedTheta}). Therefore, any term that decays faster does not affect the evolution of the ODE system. This is the main reason why we assume a decay rate of $s^{-3}\log^{-1}s$ for functions in $\mathcal{S}$.
    \end{remark}
	To state the modulation equations, we also need the following specific functions:
	\begin{equation}\label{equ:defMQ}
		\vec{M}Q=(M_{1Q},M_{2Q},M_{3Q},M_{4Q}).
	\end{equation}
	Here, we set
	\begin{equation*}
		\begin{aligned}
			\left(M_{1Q},M_{2Q}\right)=\left(\partial_{y}Q_{2}, -X_{2}\phi\right) \ \ \mbox{and} \ \ M_{3Q}=X_{1}\phi-\frac{10m_{0}}{(1+\mu)^{\frac{1}{4}}}Y_{2}.
		\end{aligned}
	\end{equation*}
	In addition, we denote
	\begin{equation*}
		M_{4Q}=-\frac{1}{2(1+\mu)}\left(\Gamma (\Lambda Q)+\zeta_{2}\Gamma(\Lambda X)\phi-e^{-z}\Gamma (\Lambda B)\phi\right).
	\end{equation*}
	Recall that, from Definition~\ref{def:Admissible}, we find 
	\begin{equation*}
		e^{-z}\in\left[\frac{1}{7\alpha s^{2}\log s},\frac{1}{5\alpha s^{2}\log s}\right]\Longrightarrow (b_{1}z,b_{2}z)\in 
		\left[\frac{9}{35s},\frac{16}{35s}\right]^{2}.
	\end{equation*}
	The above estimates will be used frequently in the remainder of this article.
	
	\smallskip
	We now prove that the smooth real-valued function $V:I\times \RR\mapsto \RR$ is an approximate solution of the rescaled equation~\eqref{equ:gKdv2} in the following sense.
	\begin{proposition}
		\label{prop:approx}
		Let $I=[s_{0},s_{1}]\subset (0,\infty)$ be an time interval and 
		the function $\mathcal{G}(s):I\mapsto (0,\infty)^{2}\times \RR^{3}$ be admissible. Then we have
		\begin{equation*}
			\Psi (V)=\vec{{\rm{Mod}}}\cdot \vec{M}Q
			+\sum_{i=1}^{4}\Psi_{i}(V)
			+O_{\mathcal{S}}(1),
		\end{equation*}
		where
		\begin{equation*}
			\begin{aligned}
				\Psi_{1}(V)&=\left(\dot{z}-b_{1}z-\mu\right)(\zeta_{2}\partial_{y}X_{2}+10m_{0}b_{1}\partial_{y}Y_{2})\\
				&+b_{1}^{3}\Lambda \left((E_{1}+F_{2})\phi\right)
				-b_{1}(\zeta_{2}-b_{1})\Lambda X_{2}\phi-\Theta\zeta_{2}\Gamma(\Lambda X)\phi\\
                 &+b_{1}^{2}((1+\mu)^{\frac{3}{2}}-1)\Gamma(\Lambda X)\phi+\Theta (1-(1+\mu)^{-1})\Gamma (\Lambda Q)\\
				&+zb_{1}(\zeta_{2}-b_{1})\partial_{y}X_{2}\phi+b_{1}e^{-z}\left(\Lambda A_{1}+\Lambda B_{2}\right)\phi+\Theta e^{-z}\Gamma(\Lambda B)\phi,
			\end{aligned}
		\end{equation*}
		\begin{equation*}
			\begin{aligned}
				\Psi_{2}(V)&=-\dot{z}e^{-z}\left(A_{1}+B_{2}\right)\phi-10\alpha m_{0}\dot{z}ze^{-z}(Y_{2}+\partial_{y}Y_{2})\\
				&+\frac{5}{2}m_{0}b_{1}\dot{\mu}\left(Y_{2}-2\Gamma(\Lambda Y)\right) -\dot{z}e^{-z}\phi\partial_{y}B_{2}+30m_{0}z\dot{b}_{1}b_{1}Y_{2}\\
				&+\frac{\dot{\mu}b_{1}^{2}\phi}{2(1+\mu)}\Gamma (\Lambda F)
				+ 2\dot{b}_{1}b_{1}(E_{1}+F_{2})\phi-\dot{z}b_{1}^{2}\phi\partial_{y}F_{2}
				+\mu e^{-z}\phi\partial_{y}B_{2},
			\end{aligned}
		\end{equation*}
		\begin{equation*}
			\begin{aligned}
				\Psi_{3}(V)&=\frac{5c_{Q}}{4}b_{1}z^{2}e^{-z}\partial_{y}\left(e^{-(y-z)}Q^{4}(y-z)\right)
				+\mu b_{1}^{2}\phi \partial_{y}F_{2}\\
				& +\frac{5c_{Q}}{4}\zeta_{2}z^{2}e^{-z}\partial_{y}(e^{y}Q^{4}(y)) +10\alpha m_{0}\mu ze^{-z}\partial_{y}Y_{2}\\
				&+\frac{5c_{Q}}{2}\mu z e^{-z}\partial_{y}\left(e^{y}Q^{4}(y)\right)+10\alpha m_{0}z(1-(1+\mu)^{\frac{1}{4}})e^{-z}\partial_{y}\left(Q_{2}^{4}\right),
			\end{aligned}
		\end{equation*}
		\begin{equation*}
			\begin{aligned}
				\Psi_{4}(V)=b_{1}^{2}(E_{1}+F_{2})\partial_{s}\phi-b_{1}^{2}(E_{1}+F_{2})\partial_{y}\phi 
				+R\partial_{s}\phi
				+b_{1}yR\partial_{y}\phi-R\partial_{y}\phi.
			\end{aligned}
		\end{equation*}
		Here, the approximate solution $V$ and the error term $\Psi(V)$ are defined by~\eqref{equ:defV}.
	\end{proposition}
	
	\begin{proof}
		\textbf{Step 1.} General computation. First, from the definition of $\Gamma$ in~\eqref{def:Gamma}, 
		for any smooth real-valued function $f:I\times \RR\mapsto \RR$, 
		\begin{equation}\label{equ:computef}
			\partial_{s}\left(\Gamma f\right)=
			\Gamma\left(\partial_{s}f\right)-\dot{z}
			\partial_{y}\left(\Gamma f\right)
			+\frac{\dot{\mu}}{2(1+\mu)}\Gamma\left(\Lambda f\right).
		\end{equation}
		Based on the above identity and the definition of $Q_{2}$ in~\eqref{equ:defQ1Q2},  we compute
		\begin{equation}\label{est:dsS}
			\partial_{s}Q_{2}=-\dot{z}\partial_{y}Q_{2}+\frac{\dot{\mu}}{2(1+\mu)}\Gamma(\Lambda Q)\Longrightarrow 
			\partial_{s}S=\dot{z}\partial_{y}Q_{2}-\frac{\dot{\mu}}{2(1+\mu)}\Gamma(\Lambda Q).
		\end{equation}
		Then, using again~\eqref{equ:computef} and the definition of $R$ in~\eqref{equ:defR},
		\begin{equation}\label{est:dsR}
			\begin{aligned}
				\partial_{s}(R\phi)
				&=\dot{b}_{1}X_{1}\phi-\dot{\zeta}_{2}X_{2}\phi-\dot{z}e^{-z}(A_{1}+B_{2})\phi+R\partial_{s}\phi\\
				&+\dot{z}(\zeta_{2}\partial_{y}X_{2}-e^{-z}\partial_{y}B_{2})\phi
				+\frac{\dot{\mu}}{2(1+\mu)}\left(e^{-z}\Gamma(\Lambda B)-\zeta_{2}\Gamma(\Lambda X)\right)\phi.
			\end{aligned}
		\end{equation}
		In addition, from~\eqref{equ:defU1}, we check that 
		\begin{equation}\label{est:dsU1}
			\begin{aligned}
				\partial_{s}U_{1}
				&=2\dot{b}_{1}b_{1}(E_{1}+F_{2})\phi-\dot{z}b_{1}^{2}\phi\partial_{y}F_{2}+\frac{\dot{\mu}b_{1}^{2}\phi}{2(1+\mu)}\Gamma (\Lambda F)\\
				&-5(3m_{0}z+a_{2})\dot{z}b_{1}^{2}\partial_{y}Y_{2}+\frac{5(3m_{0}z+a_{2})}{2(1+\mu)}\dot{\mu}b_{1}^{2}\Gamma(\Lambda Y)\\
				&+b_{1}^{2}(E_{1}+F_{2})\partial_{s}\phi+15m_{0}\dot{z}b_{1}^{2}Y_{2}+10(3m_{0}z+a_{2})\dot{b}_{1}b_{1}Y_{2}.
			\end{aligned}
		\end{equation}
		Similarly, from~\eqref{equ:defU2}, we check that 
		\begin{equation}\label{est:dsU}
			\begin{aligned}
				\partial_{s}U_{2}
				&=-\frac{\dot{z}}{(1+\mu)^{\frac{1}{4}}}\left(5(1+\mu)^{\frac{1}{4}}\left(2\alpha m_{0}z+a_{0}\right)e^{-z}-10 m_{0}b_{1}\right)
				\partial_{y}Y_{2}\\
				&+\frac{\dot{\mu}}{2(1+\mu)^{\frac{5}{4}}}\left(5(1+\mu)^{\frac{1}{4}}\left(2\alpha m_{0}z+a_{0}\right)e^{-z}-10 m_{0}b_{1}\right)\Gamma (\Lambda Y)\\
				&-\frac{10m_{0}\dot{b}_{1}Y_{2}}{(1+\mu)^{\frac{1}{4}}}+\frac{5m_{0}b_{1}\dot{\mu}Y_{2}}{2(1+\mu)^{\frac{5}{4}}}-5a_{0}\dot{z}e^{-z}Y_{2}+10\alpha m_{0}\dot{z}e^{-z}(1-z)Y_{2}.
			\end{aligned}
		\end{equation}
		Combining identities~\eqref{est:dsS}--\eqref{est:dsU} with the definition of $V$, we obtain 
		\begin{equation}\label{equ:dsV}
			\begin{aligned}
				\partial_{s}V&=\dot{z}M_{1Q}+\dot{\zeta}_{2}M_{2Q}+\dot{b}_{1}{M_{3Q}}+\dot{\mu}M_{4Q}
				\\
				&+\dot{z}\left(\zeta_{2}\phi \partial_{y}X_{2}+{10m_{0}b_{1}}\partial_{y}Y_{2}\right)+\mathcal{H}_{1}+\mathcal{H}_{2}+\mathcal{H}_{3}+\mathcal{H}_{4}+\mathcal{H}_{5}.
			\end{aligned}
		\end{equation}
		Here, we denote 
		\begin{equation*}
			\begin{aligned}
				\mathcal{H}_{1}&=
				\frac{5m_{0}b_{1}\dot{\mu}}{2(1+\mu)^{\frac{5}{4}}}\left(Y_{2}-2\Gamma(\Lambda Y)\right)
				-5\dot{z}e^{-z}\left(2\alpha m_{0}z+a_{0}\right)(Y_{2}+\partial_{y}Y_{2}),\\
				\mathcal{H}_{2}&=\frac{5\dot{\mu}e^{-z}}{2(1+\mu)}(2\alpha m_{0}z+a_{0})\Gamma (\Lambda Y)-\dot{z}e^{-z}(A_{1}+B_{2})\phi
				-\dot{z}e^{-z}\phi\partial_{y}B_{2},\\
				\mathcal{H}_{3}&=b_{1}^{2}(E_{1}+F_{2})\partial_{s}\phi+10(3m_{0}z+a_{2})\dot{b}_{1}b_{1}Y_{2}+\frac{\dot{\mu}b_{1}^{2}\phi}{2(1+\mu)}\Gamma (\Lambda F)
				+R\partial_{s}\phi,\\
				\mathcal{H}_{4}&=2\dot{b}_{1}b_{1}(E_{1}+F_{2})\phi-\dot{z}b_{1}^{2}\phi\partial_{y}F_{2}+\frac{5(3m_{0}z+a_{2})}{2(1+\mu)}\dot{\mu}b_{1}^{2}\Gamma(\Lambda Y)+15m_{0}\dot{z}b_{1}^{2}Y_{2},\\
				\mathcal{H}_{5}&=\frac{10m_{0}b_{1}\dot{z}}{(1+\mu)^{\frac{1}{4}}}\left(1-(1+\mu)^{\frac{1}{4}}\right)\partial_{y}Y_{2}+10\alpha m_{0}\dot{z}e^{-z}Y_{2}
				-5(3m_{0}z+a_{2})\dot{z}b_{1}^{2}\partial_{y}Y_{2}.
			\end{aligned}
		\end{equation*}
		
		On the other hand, from the definition of $Q_{2}$ in~\eqref{equ:defQ1Q2}, we compute 
		\begin{equation*}
			-\partial_{y}^{2}Q_{2}+(1+\mu)Q_{2}-Q_{2}^{5}=0,\quad \mbox{on} \ \RR.
		\end{equation*}
		It follows directly that 
		\begin{equation}\label{equ:LV}
			\partial_{y}\left(\partial_{y}^{2}V-V+V^{5}\right)
			=-\mu M_{1Q}+\mathcal{H}_{6}+\mathcal{H}_{7}.
		\end{equation}
		Here, we denote
		\begin{equation*}
			\begin{aligned}
				\mathcal{H}_{6}&=\partial_{y}\left(\partial_{y}^{2}(R\phi+U)-(R\phi+U)+5S^{4}(R\phi+U)\right)+\partial_{y}H,\\
				\mathcal{H}_{7}&=\partial_{y}\left(
				(S+R\phi+U
				)^{5}-S^{5}-5S^{4}(R\phi+U)\right)+\partial_{y}G-\partial_{y}H.
			\end{aligned}
		\end{equation*}
		
		\smallskip
		\textbf{Step 2.} Estimate on $\mathcal{H}_{1}$. We claim that 
		\begin{equation}\label{est:H1}
			\mathcal{H}_{1}=  \frac{5}{2}m_{0}b_{1}\dot{\mu}\left(Y_{2}-2\Gamma(\Lambda Y)\right)
			-10\alpha m_{0}\dot{z}ze^{-z}(Y_{2}+\partial_{y}Y_{2})+O_{\mathcal{S}}(1).
		\end{equation}
		Indeed, we rewrite the term $\mathcal{H}_{2}$ by 
		\begin{equation*}
			\begin{aligned}
				\mathcal{H}_{1}&
				= \frac{5}{2}m_{0}b_{1}\dot{\mu}\left(Y_{2}-2\Gamma(\Lambda Y)\right)
				-10\alpha m_{0}\dot{z}ze^{-z}(Y_{2}+\partial_{y}Y_{2})
				\\
				&+\frac{5}{2}m_{0}b_{1}\dot{\mu}
				((1+\mu)^{-\frac{5}{4}}-1)\left(Y_{2}-2\Gamma(\Lambda Y)\right)
				-5a_{0}\dot{z}e^{-z}(Y_{2}+\partial_{y}Y_{2}).
			\end{aligned}
		\end{equation*}
		Combining the above identity with Definition~\ref{def:Admissible} and $Y\in \mathcal{Y}$, we find 
		\begin{equation*}
			\begin{aligned}
				\| \dot{z}e^{-z}(Y_{2}+\partial_{y}Y_{2})\|_{H^{1}}&
				\lesssim |\dot{z}e^{-z}|
				\lesssim \frac{|\dot{z}|}{s^{2}\log s},\\
				\| b_{1}\dot{\mu}
				((1+\mu)^{-\frac{5}{4}}-1)Y_{2}\|_{H^{1}}&\lesssim 
				|b_{1}\mu \dot{\mu}|\lesssim
				\frac{|\dot{\mu}|}{s^{2}\log s},\\
				\| b_{1}\dot{\mu}
				((1+\mu)^{-\frac{5}{4}}-1)\Gamma(\Lambda Y)\|_{H^{1}}&
				\lesssim |b_{1}\mu \dot{\mu}|
				\lesssim \frac{|\dot{\mu}|}{s^{2}\log s},
			\end{aligned}
		\end{equation*}
		which directly completes the proof for the estimate~\eqref{est:H1}.
		
		\smallskip
		\textbf{Step 3.} Estimate on $\mathcal{H}_{2}.$
		We claim that 
		\begin{equation}\label{est:H2}
			\mathcal{H}_{2}=-\dot{z}e^{-z}(A_{1}+B_{2})\phi
			-\dot{z}e^{-z}\phi\partial_{y}B_{2}+O_{\mathcal{S}}(1).
		\end{equation}
		Indeed, using~\eqref{est:z1}, Definition~\ref{def:Admissible} and $Y\in \mathcal{Y}$, we find 
		\begin{equation*}
			\|\dot{\mu}e^{-z}\left(|z|+1\right)\Gamma (\Lambda Y)\|\lesssim 
			|\dot{\mu}|e^{-z}(|z|+1)\lesssim \frac{|\dot{\mu}|}{s^{2}},
		\end{equation*}
		which directly completes the proof for the estimate~\eqref{est:H2}.
		
		\smallskip
		\textbf{Step 4.} Estimate on $\mathcal{H}_{3}$. 
		We claim that 
		\begin{equation}\label{est:H3}
			\begin{aligned}
				\mathcal{H}_{3}&=b_{1}^{2}(E_{1}+F_{2})\partial_{s}\phi+30m_{0}z\dot{b}_{1}b_{1}Y_{2}\\
				&+\frac{\dot{\mu}b_{1}^{2}\phi}{2(1+\mu)}\Gamma (\Lambda F)
				+R\partial_{s}\phi+O_{\mathcal{S}}(1).
			\end{aligned}
		\end{equation}
		Indeed, using an argument similar to that in Step 3, we find 
		\begin{equation*}
			\|a_{2}\dot{b}_{1}b_{1}Y_{2}\|_{H^{1}}\lesssim |\dot{b}_{1}b_{1}|\lesssim \frac{|\dot{b}_{1}|}{s\log s},
		\end{equation*}
		which directly completes the proof for the estimate~\eqref{est:H3}.
		
		\smallskip
		\textbf{Step 5.} Estimates on $\mathcal{H}_{4}$ and $\mathcal{H}_{5}$. 
		We claim that 
		\begin{equation}\label{est:H4}
			\mathcal{H}_{4}+\mathcal{H}_{5}=2\dot{b}_{1}b_{1}(E_{1}+F_{2})\phi-\dot{z}b_{1}^{2}\phi\partial_{y}F_{2}+O_{\mathcal{S}}(1).
		\end{equation}
		Indeed, using again~\eqref{est:z1}, Definition~\ref{def:Admissible} and $Y\in \mathcal{Y}$, we find 
		\begin{equation*}
			\|\dot{z}b_{1}^{2}Y_{2}\|_{H^{1}}+  \|\dot{\mu}b_{1}^{2}(|z|+1)\Gamma(\Lambda Y)\|_{H^{1}}\lesssim \frac{|\dot{z}|}{s^{2}\log^{2}s}
			+\frac{|\dot{\mu}|}{s^{2}\log s}.
		\end{equation*}
		Similarly, we also find 
		\begin{equation*}
			\begin{aligned}
				\|  b_{1}\dot{z}\big(1-(1+\mu)^{\frac{1}{4}}\big)\partial_{y}Y_{2}\|_{H^{1}}&\lesssim \frac{|\dot{z}|}{s^{2}\log s},\\
				\|\dot{z}e^{-z}Y_{2}\|_{H^{1}}+
				\| (|z|+1)\dot{z}b_{1}^{2}\partial_{y}Y_{2}\|_{H^{1}}
				&\lesssim \frac{|\dot{z}|}{s^{2}\log s}.
			\end{aligned}
		\end{equation*}
		We see that the estimate~\eqref{est:H4} follows from the above estimates.
		
		\smallskip
		\textbf{Step 6.} Estimate on $\mathcal{H}_{6}$. We claim that 
		\begin{equation}\label{est:H6}
			\begin{aligned}
				\mathcal{H}_{6}
				&=\frac{5c_{Q}}{4}b_{1}z^{2}e^{-z}\partial_{y}\left(e^{-(y-z)}Q^{4}(y-z)\right)-b_{1}^{2}(E_{1}+F_{2})\partial_{y}\phi\\
				&+b_{1}^{2}(2X_{1}-\Lambda X_{1})\phi-10b_{1}^{2}\partial_{y}(X_{1}^{2}Q_{1}^{3})-10\alpha m_{0}b_{1}\mu \partial_{y}Y_{2}\\
				&+10b_{1}^{2}\partial_{y}(P_{2}^{2}Q_{2}^{3})
				-10\alpha m_{0}e^{-z}Y_{2}+10m_{0}b_{1}^{2}\left(\Gamma(\Lambda Y)-2Y_{2}\right)\\
				&+10\alpha m_{0}z(1-(1+\mu)^{\frac{1}{4}})e^{-z}\partial_{y}\left(Q_{2}^{4}\right)
				+\mu e^{-z}\phi\partial_{y}B_{2}-R\partial_{y}\phi\\
				&+b_{1}^{2}(1+\mu)^{\frac{3}{2}}(\Gamma (\Lambda X)-2X_{2})\phi+\mu b_{1}^{2}\phi \partial_{y}F_{2}+\frac{5c_{Q}}{2}\mu z e^{-z}\partial_{y}\left(e^{y}Q^{4}(y)\right)\\
				&+\alpha e^{-z}X_{1}\phi -\alpha (1+\mu)^{\frac{3}{2}}e^{-z}X_{2}\phi
				+10\alpha m_{0}\mu ze^{-z}\partial_{y}Y_{2}-\mu\zeta_{2}\phi\partial_{y}X_{2}
				\\
				&+b_{2}\Lambda Q_{2}-b_{1}\Lambda Q_{1}
				+\frac{5c_{Q}}{4}\zeta_{2}z^{2}e^{-z}\partial_{y}(e^{y}Q^{4}(y))
				-b_{2}z\partial_{y}Q_{2}+O_{\mathcal{S}}(1).
			\end{aligned}
		\end{equation}
		Indeed, we decompose
		\begin{equation*}
			\mathcal{H}_{6}=\mathcal{H}_{6,1}+\mathcal{H}_{6,2}+\mathcal{H}_{6,3},
		\end{equation*}
		where
		\begin{equation*}
			\begin{aligned}
				\mathcal{H}_{6,1}&=\partial_{y}\left(\partial_{y}^{2}U-U+5S^{4}U\right)+R\partial_{y}^{3}\phi,\\
				\mathcal{H}_{6,2}&=\left(\partial_{y}\left(\partial_{y}^{2}R-R+5S^{4}R\right)\right)\phi+\partial_{y}H,\\
				\mathcal{H}_{6,3}&=\left(3\partial_{y}^{2}R-R+5S^{4}R\right)\partial_{y}\phi
				+3(\partial_{y}R)\partial_{y}^{2}\phi.
			\end{aligned}
		\end{equation*}
		
		\smallskip
		\emph{Estimate on $\mathcal{H}_{6,1}$.} We claim that 
		\begin{equation}\label{est:H61}
			\begin{aligned}
				\mathcal{H}_{6,1}
				&=\left(10m_{0}b_{1}-5(2\alpha m_{0} z+a_{0})e^{-z}\right)\partial_{y}\left(Q_{2}^{4}\right)\\
				&+10b_{1}^{2}\partial_{y}(P_{2}^{2}Q_{2}^{3})-b_{1}^{2}(E_{1}+F_{2})\partial_{y}\phi+\mu b_{1}^{2}\phi \partial_{y}F_{2}\\
				&+b_{1}^{2}(2X_{1}-\Lambda X_{1})\phi-10b_{1}^{2}\partial_{y}(X_{1}^{2}Q_{1}^{3})-10\alpha m_{0}b_{1}\mu \partial_{y}Y_{2}\\
                &+10m_{0}b_{1}^{2}\left(\Gamma(\Lambda Y)-2Y_{2}\right)+b_{1}^{2}(1+\mu)^{\frac{3}{2}}(\Gamma (\Lambda X)-2X_{2})\phi\\
				&+10\alpha m_{0}z(1-(1+\mu)^{\frac{1}{4}})e^{-z}\partial_{y}\left(Q_{2}^{4}\right)
				+10\alpha m_{0}\mu ze^{-z}\partial_{y}Y_{2}+O_{\mathcal{S}}(1).
			\end{aligned}
		\end{equation}
		Indeed, we first decompose
		\begin{equation*}
			\begin{aligned}
				\mathcal{H}_{6,1}=\mathcal{H}_{6,1,1}
				+\mathcal{H}_{6,1,2}+\mathcal{H}_{6,1,3},
			\end{aligned}
		\end{equation*}
		with 
		\begin{equation*}
			\begin{aligned}
				\mathcal{H}_{6,1,1}&=\partial_{y}\left(\partial_{y}^{2}U_{1}-U_{1}+5S^{4}U_{1}\right),\\
				\mathcal{H}_{6,1,2}&=\partial_{y}  \left(\partial_{y}^{2}U_{2}-(1+\mu)U_{2}+5Q_{2}^{4}U_{2}\right),\\
				\mathcal{H}_{6,1,3}&=5\partial_{y}((S^{4}-Q_{2}^{4})U_{2})+\mu \partial_{y}U_{2}+R\partial_{y}^{3}\phi.
			\end{aligned}
		\end{equation*}
		Note that, from the definition of $U_{1}$ in~\eqref{equ:defU1}, we rewrite 
		\begin{equation*}
			\begin{aligned}
				\mathcal{H}_{6,1,1}
				&=2b_{1}^{2}\left(\partial_{y}^{2}E_{1}+\partial_{y}^{2}F_{2}\right)\partial_{y}\phi+3b_{1}^{2}\left(\partial_{y}E_{1}+\partial_{y}F_{2}\right)\partial_{y}^{2}\phi\\
				&+b_{1}^{2}\left((\partial_{y}^{2}E_{1}+\partial_{y}^{2}F_{2})-(E_{1}+F_{2})+5S^{4}(E_{1}+F_{2})\right)\partial_{y}\phi\\
				&+b_{1}^{2}\left(\partial_{y}\left((\partial_{y}^{2}E_{1}+\partial_{y}^{2}F_{2})-(E_{1}+F_{2})+5S^{4}(E_{1}+F_{2})\right)\right)\phi\\
				&+5(3m_{0}z+a_{2})b_{1}^{2}\partial_{y}\left(\partial_{y}^{2}Y_{2}-Y_{2}+5S^{4}Y_{2}\right)+b_{1}^{2}(E_{1}+F_{2})\partial_{y}^{3}\phi.
			\end{aligned}
		\end{equation*}
		Using Definition~\ref{def:Admissible} and Lemma~\ref{le:RAB}, we check that 
		\begin{equation*}
			\begin{aligned}
				b_{1}^{2}
				\left(
				\|(E_{1}+F_{2})\partial_{y}^{3}\phi\|_{H^{1}}+
				\|\left(\partial_{y}^{2}E_{1}+\partial_{y}^{2}F_{2}\right)\partial_{y}\phi\|_{H^{1}}\right)&\lesssim \frac{1}{s^{3}\log s},\\
				b_{1}^{2}\left(\|\left(\partial_{y}E_{1}+\partial_{y}F_{2}\right)\partial_{y}^{2}\phi\|_{H^{1}}
				+\|S^{4}(E_{1}+F_{2})\partial_{y}\phi\|_{H^{1}}
				\right)
				&\lesssim \frac{1}{s^{3}\log s}.
			\end{aligned}
		\end{equation*}
		Then, from Lemma~\ref{le:EF}, Lemma~\ref{le:boundinter} and Lemma~\ref{le:PAB}, 
		\begin{equation*}
			\begin{aligned}
				b_{1}^{2}\left(\partial_{y}\left(\partial_{y}^{2}E_{1}-E_{1}+5S^{4}E_{1}\right)\right)\phi
				&=b_{1}^{2}(2X_{1}-\Lambda X_{1})\phi-10b_{1}^{2}\partial_{y}(X_{1}^{2}Q_{1}^{3})\\
				&+15m_{0}b_{1}^{2}\partial_{y}\left(yQ_{2}^{4}\right)+5a_{2}b_{1}^{2}\partial_{y}(Q_{2}^{4})+O_{\mathcal{S}}(1).
			\end{aligned}
		\end{equation*}
		Similarly, we compute 
		\begin{equation*}
			\begin{aligned}
				&b_{1}^{2}\left(\partial_{y}\left(\partial_{y}^{2}F_{2}-F_{2}+5S^{4}F_{2}\right)\right)\phi\\
				&=\mu b_{1}^{2}\phi \partial_{y}F_{2}+b_{1}^{2}(1+\mu)^{\frac{3}{2}}(\Gamma (\Lambda X)-2X_{2})\phi-15m_{0}b_{1}^{2}\Gamma Z\\
				&+10m_{0}b_{1}^{2}\left(\Gamma(\Lambda Y)-2Y_{2}\right)+10b_{1}^{2}\partial_{y}(P_{2}^{2}Q_{2}^{3})+O_{\mathcal{S}}(1).
			\end{aligned}
		\end{equation*}
		Moreover, from Lemma~\ref{le:boundinter} and Remark~\ref{re:Y}, 
		\begin{equation*}
			\begin{aligned}
				&5(3m_{0}z+a_{2})b_{1}^{2}\partial_{y}\left(\partial_{y}^{2}Y_{2}-Y_{2}+5S^{4}Y_{2}\right)\\
				&=-5(3m_{0}z+a_{2})b^{2}_{1}\partial_{y}(Q_{2}^{4})+O_{\mathcal{S}}(1).
			\end{aligned}
		\end{equation*}
		Combining the above estimates, we obtain 
		\begin{equation}\label{est:H611}
			\begin{aligned}
				\mathcal{H}_{6,1,1}
				&=\mu b_{1}^{2}\phi \partial_{y}F_{2}+10m_{0}b_{1}^{2}\left(\Gamma(\Lambda Y)-2Y_{2}\right)+O_{\mathcal{S}}(1)\\
				&+b_{1}^{2}(1+\mu)^{\frac{3}{2}}(\Gamma (\Lambda X)-2X_{2})\phi+10b_{1}^{2}\partial_{y}(P_{2}^{2}Q_{2}^{3})\\
				&+b_{1}^{2}(2X_{1}-\Lambda X_{1})\phi-10b_{1}^{2}\partial_{y}(X_{1}^{2}Q^{3}_{1})-b_{1}^{2}(E_{1}+F_{2})\partial_{y}\phi.
			\end{aligned}
		\end{equation}
		On the other hand, from Remark~\ref{re:Y} and Definition~\ref{def:Admissible}, we find 
		\begin{equation}\label{est:H612}
			\begin{aligned}
				\mathcal{H}_{6,1,2}
				&=\left(10m_{0}b_{1}-5(2\alpha m_{0} z+a_{0})e^{-z}\right)\partial_{y}\left(Q_{2}^{4}\right)\\
				&+10\alpha m_{0}z(1-(1+\mu)^{\frac{1}{4}})e^{-z}\partial_{y}\left(Q_{2}^{4}\right)+O_{\mathcal{S}}(1).
			\end{aligned}
		\end{equation}
		Then, by an elementary computation, we decompose
		\begin{equation*}
			S^{4}-Q_{2}^{4}=Q_{1}^{4}-4Q_{1}^{3}Q_{2}+6Q_{1}^{2}Q_{2}^{2}-4Q_{1}Q_{2}^{3}.
		\end{equation*}
		Combining the above estimate with Lemma~\ref{le:boundinter} and Lemma~\ref{le:RAB}, we find 
		\begin{equation}\label{est:H613}
			\mathcal{H}_{6,1,3}=10\alpha m_{0}\mu ze^{-z}\partial_{y}Y_{2}
			-10\alpha m_{0}b_{1}\mu \partial_{y}Y_{2}+O_{\mathcal{S}}(1).
		\end{equation}
		Here, we use the fact that 
		\begin{equation*}
        \begin{aligned}
			\|Q_{1}^{4}Y_{2}\|_{H^{1}}&\lesssim \left(\int_{\RR}e^{-8|y|}e^{-2(1-|\mu|)^{\frac{1}{2}}|y-z|}\dd y\right)^{\frac{1}{2}}\lesssim e^{-(1-|\mu|)^{\frac{1}{2}}z}\lesssim \frac{1}{s^{2}\log s},\\
            \|Q_{1}Q_{2}^{3}Y_{2}\|_{H^{1}}&\lesssim 
            \left(\int_{\RR}e^{-2|y|}e^{-8(1-|\mu|)^{\frac{1}{2}}|y-z|}\dd y\right)^{\frac{1}{2}}\lesssim e^{-z}\lesssim \frac{1}{s^{2}\log s}.
            \end{aligned}
		\end{equation*}
		We see that~\eqref{est:H61} follows directly from~\eqref{est:H611},~\eqref{est:H612} and~\eqref{est:H613}.
		
		\smallskip
		\emph{Estimate on $\mathcal{H}_{6,2}$.} We claim that 
		\begin{equation}\label{est:H62}
			\begin{aligned}
				\mathcal{H}_{6,2}
				&=\alpha e^{-z}X_{1}\phi -\alpha (1+\mu)^{\frac{3}{2}}e^{-z}X_{2}\phi+\mu e^{-z}\phi\partial_{y}B_{2}\\
				&+\frac{5c_{Q}}{4}b_{1}z^{2}e^{-z}\partial_{y}\left(e^{-(y-z)}Q^{4}(y-z)\right)-\mu\zeta_{2}\phi\partial_{y}X_{2}\\
				&+b_{2}\Lambda Q_{2}-b_{1}\Lambda Q_{1}-10m_{0}b_{1}\partial_{y}\left(Q_{2}^{4}\right)+5a_{0}e^{-z}\partial_{y}\left(Q_{2}^{4}\right)\\
				&+\frac{5c_{Q}}{4}\zeta_{2}z^{2}e^{-z}\partial_{y}(e^{y}Q^{4}(y))
				-b_{2}z\partial_{y}Q_{2}
				-10\alpha m_{0}e^{-z}(\Gamma W)\\
				&+\frac{5c_{Q}}{2}\mu z e^{-z}\partial_{y}\left(e^{y}Q^{4}(y)\right)
				+10\alpha m_{0}e^{-z}\partial_{y}\left(yQ_{2}^{4}\right)+O_{\mathcal{S}}(1).
			\end{aligned}
		\end{equation}
		Indeed, we first decompose
		\begin{equation*}
			\mathcal{H}_{6,2}=\partial_{y}H+\mathcal{H}_{6,2,1}+\mathcal{H}_{6,2,2}+\mathcal{H}_{6,2,3}+\mathcal{H}_{6,2,4},
		\end{equation*}
		where
		\begin{equation*}
			\begin{aligned}
				\mathcal{H}_{6,2,1}&=b_{1} \left(\partial_{y}\left(\partial^{2}_{y}X_{1}-X_{1}+5Q_{1}^{4}X_{1}\right)\right)\phi
				+5b_{1}\left(\partial_{y}\left(\left(S^{4}-Q_{1}^{4}\right)X_{1}\right)\right)\phi,\\
				\mathcal{H}_{6,2,2}&=-\zeta_{2}\left(\partial_{y}\left(\partial_{y}^{2}X_{2}-X_{2}+5Q_{2}^{4}X_{2}\right)\right)\phi
				-5\zeta_{2}\left(\partial_{y}\left(\left(S^{4}-Q_{2}^{4}\right)X_{2}\right)\right)\phi,\\
				\mathcal{H}_{6,2,3}&=e^{-z}
				\left(\partial_{y}\left(\partial_{y}^{2}A_{1}-A_{1}+5Q_{1}^{4}A_{1}\right)\right)\phi
				+5e^{-z}\left(\partial_{y}\left(\left(S^{4}-Q_{1}^{4}\right)A_{1}\right)\right)\phi,\\
				\mathcal{H}_{6,2,4}&=e^{-z}
				\left(\partial_{y}\left(\partial_{y}^{2}B_{2}-B_{2}+5Q_{2}^{4}B_{2}\right)\right)\phi
				+5e^{-z}\left(\partial_{y}\left(\left(S^{4}-Q_{2}^{4}\right)B_{2}\right)\right)\phi.
			\end{aligned}
		\end{equation*}
		First, from Remark~\ref{re:P}, we check that 
		\begin{equation*}
			\begin{aligned}
				\mathcal{H}_{6,2,1}&=-b_{1}(\Lambda Q)\phi
				+5b_{1}\left(\partial_{y}\left(X_{1}Q_{2}^{4}\right)\right)\phi\\
				&+5b_{1}\left(\partial_{y}\left(\left(S^{4}-Q_{1}^{4}-Q_{2}^{4}\right)X_{1}\right)\right)\phi.
			\end{aligned}
		\end{equation*}
		Note that, from Lemma~\ref{le:boundinter} and Definition~\ref{def:Admissible}, we find 
		\begin{equation*}
			\|\partial_{y}\left(\left(S^{4}-Q_{1}^{4}-Q_{2}^{4}\right)X_{1}\right)\|_{H^{1}}\lesssim \|Q_{1}^{3}Q_{2}\|_{H^{1}}
			+\|Q_{1}Q_{2}^{3}\|_{H^{1}}\lesssim s^{-2}.
		\end{equation*}
		It follows from Lemma~\ref{le:refinedPAB} and the exponential decay of $Q$ that 
		\begin{equation}\label{est:H621}
			\begin{aligned}
				\mathcal{H}_{6,2,1}&=\frac{5c_{Q}}{4}b_{1}z^{2}e^{-z}\partial_{y}\left(e^{-(y-z)}Q^{4}(y-z)\right)\\
				&-b_{1}\Lambda Q_{1}-10m_{0}b_{1}\partial_{y}\left(Q_{2}^{4}\right)+O_{\mathcal{S}}(1).
			\end{aligned}
		\end{equation}
		Second, using $\Lambda Q=\frac{1}{2}Q+yQ'$, we check that 
		\begin{equation*}
			\begin{aligned}
				&(1+\mu)^{\frac{1}{4}}\left[(\Lambda Q)((1+\mu)^{\frac{1}{2}}(y-z))\right]\\
				&=\Lambda Q_{2}-z(1+\mu)^{\frac{3}{4}}Q'((1+\mu)^{\frac{1}{2}}(y-z))=\Lambda Q_{2}-z\partial_{y}Q_{2}.
			\end{aligned}
		\end{equation*}
		On the other hand, we compute 
		\begin{equation*}
			\partial_{y}^{2}X_{2}-X_{2}+5Q^{4}_{2}X_{2}=\mu X_{2}-(1+\mu)\Gamma(\mathcal{L}X).
		\end{equation*}
		It follows from Remark~\ref{re:P} that\footnote{Here, we use the fact that $b_{2}=\zeta_{2}(1+\mu)^{\frac{3}{2}}$.}
		\begin{equation*}
			\begin{aligned}
				\mathcal{H}_{6,2,2}&=b_{2}\Lambda Q_{2}-b_{2}z\partial_{y}Q_{2}-\mu\zeta_{2}\phi\partial_{y}X_{2}\\
				&-5\zeta_{2}\left(\partial_{y}\left(\left(S^{4}-Q_{1}^{4}-Q_{2}^{4}\right)X_{2}\right)\right)\phi
				-5\zeta_{2}\left(\partial_{y}\left(X_{2}Q_{1}^{4}\right)\right)\phi.
			\end{aligned}
		\end{equation*}
		
		Similar as above, from Lemma~\ref{le:refinedPAB} and the exponential decay of $Q$,
		\begin{equation}\label{est:H622}
			\begin{aligned}
				\mathcal{H}_{6,2,2}
				&=b_{2}\Lambda Q_{2}-b_{2}z\partial_{y}Q_{2}-\mu\zeta_{2}\phi\partial_{y}X_{2}\\
				&+\frac{5c_{Q}}{4}\zeta_{2}z^{2}e^{-z}\partial_{y}(e^{y}Q^{4}(y))+O_{\mathcal{S}}(1).
			\end{aligned}
		\end{equation}
		Then, from  Lemma~\ref{le:AB}, Lemma~\ref{le:PAB} and Definition~\ref{def:Admissible}, 
		\begin{equation*}
			\begin{aligned}
				\mathcal{H}_{6,2,3}
				&=10\alpha m_{0}e^{-z}\partial_{y}\left(yQ_{2}^{4}\right)+5a_{0}e^{-z}\partial_{y}\left(Q_{2}^{4}\right)\\
				&+5c_{Q}e^{-z}\partial_{y}(e^{y}Q^{4})+\alpha e^{-z} X_{1}\phi+O_{\mathcal{S}}(1).
			\end{aligned}
		\end{equation*}
		Here, we use the fact that 
		\begin{equation*}
			\begin{aligned}
				& \|\partial_{y}\left(\left(S^{4}-Q_{1}^{4}-Q_{2}^{4}\right)A_{1}\right)\|_{H^{1}}\\
				&\lesssim \|(1+|y|)Q_{1}^{3}Q_{2}\|_{H^{1}}
				+\|(1+|y|)Q_{1}Q_{2}^{3}\|_{H^{1}}\lesssim s^{-\frac{3}{2}}.
			\end{aligned}
		\end{equation*}
		Similar as above, we also have
		\begin{equation*}
			\begin{aligned}
				\mathcal{H}_{6,2,4}
				&=-5c_{Q}e^{-z}\partial_{y}\big(e^{-(y-z)}Q^{4}(y-z)\big)-10\alpha m_{0}e^{-z}(\Gamma W)\\
				&-\alpha(1+\mu)^{\frac{3}{2}} e^{-z}X_{2}\phi+\mu e^{-z}\phi\partial_{y}B_{2}+O_{\mathcal{S}}(1).
			\end{aligned}
		\end{equation*}
		Combining the above two estimates with Lemma~\ref{le:interasym}, we obtain 
		\begin{equation}\label{est:H623624}
			\begin{aligned}
				\mathcal{H}_{6,2,3}+\mathcal{H}_{6,2,4}
				&=\alpha e^{-z}X_{1}\phi -\alpha (1+\mu)^{\frac{3}{2}}e^{-z}X_{2}\phi+\mu e^{-z}\phi\partial_{y}B_{2}\\
				&-\partial_{y}H+\frac{5c_{Q}}{2}\mu z e^{-z}\partial_{y}\left(e^{y}Q^{4}(y)\right)+5a_{0}e^{-z}\partial_{y}\left(Q_{2}^{4}\right)\\
				&+10\alpha m_{0}e^{-z}\partial_{y}\left(yQ_{2}^{4}\right)-10\alpha m_{0}e^{-z}(\Gamma W)+O_{\mathcal{S}}(1).
			\end{aligned}
		\end{equation}
		We see that~\eqref{est:H62} follows directly from~\eqref{est:H621},~\eqref{est:H622} and~\eqref{est:H623624}.
		
		\smallskip
		\emph{Estimate on $\mathcal{H}_{6,3}$.} 
		Note that, from the definition of $\phi$ in~\eqref{equ:defphi}, we directly have 
		\begin{equation*}
			\|(S^{4}R)\partial_{y}\phi\|_{H^{1}}\lesssim \left(\int_{\RR}
			\big(e^{-|y|}+e^{-|y-z|}\big)\textbf{1}_{[\frac{s}{2},\frac{5s}{2}]}(y)\dd y\right)^{\frac{1}{2}}\lesssim \frac{1}{s^{3}\log s}.
		\end{equation*}
		Based on the above estimate and Lemma~\ref{le:RAB}, we obtain \begin{equation}\label{est:H63}
			\mathcal{H}_{6,3}=-R\partial_{y}\phi+O_{\mathcal{S}}(1).
		\end{equation}
		
		Last, we see that~\eqref{est:H6} follows directly from~\eqref{est:H61},~\eqref{est:H62} and~\eqref{est:H63}.
		
		\smallskip
		\textbf{Step 7.} Estimate on $\mathcal{H}_{7}$. We claim that 
		\begin{equation}\label{est:H7}
			\mathcal{H}_{7}=10b_{1}^{2}\partial_{y}\left(Q_{1}^{3}X_{1}^{2}\right)-10b_{1}^{2}\partial_{y}\left(Q_{2}^{3}P_{2}^{2}\right)+O_{\mathcal{S}}(1).
		\end{equation}
		Indeed, we decompose 
		\begin{equation*}
			\mathcal{H}_{7}=\mathcal{H}_{7,1}+\mathcal{H}_{7,2},
		\end{equation*}
		where
		\begin{equation*}
			\begin{aligned}
				\mathcal{H}_{7,1}&=\partial_{y}\left(10S^{2}(R\phi+U)^{3}+5S(R\phi+U)^{4}\right),\\
				\mathcal{H}_{7,2}&=\partial_{y}\left(10S^{3}(R\phi+U)^{2}+(R\phi+U)^{5}\right)+\partial_{y}G-\partial_{y}H.
			\end{aligned}
		\end{equation*}
		
		\emph{Estimate on $\mathcal{H}_{7,1}$.} From Definition~\ref{def:Admissible} and $\rm{supp}\phi\subset (-\infty,\frac{5}{2}s)$, we find 
		\begin{equation}\label{est:H71}
			\begin{aligned}
				\|\mathcal{H}_{7,1}\|_{H^{1}}
				&\lesssim \|S^{2}(R\phi+U)^{3}\|_{H^{1}}+\|S(R\phi+U)^{4}\|_{H^{1}}\\
				&\lesssim b_{1}^{3}+\zeta_{2}^{3}+z^{4}\left(e^{-3z}+e^{-4z}+b_{1}^{6}+b_{1}^{8}\right)\lesssim 
				\frac{1}{s^{3}\log s}.
			\end{aligned}
		\end{equation}
		
		\emph{Estimate on $\mathcal{H}_{7,2}$.}
		First, we decompose 
		\begin{equation*}
			\begin{aligned}
				S^{3}(R\phi+U)^{2}
				&=Q_{1}^{3}(R\phi+U)^{2}-Q_{2}^{3}(R\phi+U)^{2}\\
				&-3Q_{1}^{2}Q_{2}\left(R\phi+U\right)^{2}+3Q_{1}Q_{2}^{2}(R\phi+U)^{2}.
			\end{aligned}
		\end{equation*}
		By an elementary computation and definition of $R$ in~\eqref{equ:defR}, 
		\begin{equation*}
			\begin{aligned}
				(R\phi+U)^{2}
				&=b_{1}^{2}X_{1}^{2}\phi^{2}+\left(R\phi-b_{1}X_{1}\phi\right)\left(R\phi+b_{1}X_{1}\phi\right)+2RU\phi+U^{2}.
			\end{aligned}
		\end{equation*}
		It follows from Lemma~\ref{le:boundinter}, Lemma~\ref{le:PAB} and Definition~\ref{def:Admissible} that 
		\begin{equation}\label{est:Q3R1}
			\partial_{y}(Q_{1}^{3}(R\phi+U)^{2})=b_{1}^{2}\partial_{y}(Q_{1}^{3}X_{1}^{2})+O_{\mathcal{S}}(1).
		\end{equation}
		Then, using the definition of $R$ and $U$ in~\eqref{equ:defR} and~\eqref{equ:defU1}--\eqref{equ:defU2}, we rewrite 
		\begin{equation*}
			\begin{aligned}
				(R\phi+U)^{2}&=\left(b_{1}X_{1}\phi-\zeta_{2}X_{2}\phi-{10 m_{0}b_{1}}(1+\mu)^{-\frac{1}{4}}Y_{2}\right)^{2}\\
				&+\left(U_{1}+e^{-z}(A_{1}+B_{2})\phi+5(2\alpha m_{0}z+a_{0})e^{-z}Y_{2}\right)^{2}\\
				&+2\left[\left(b_{1}X_{1}\phi-\zeta_{2}X_{2}\phi-{10 m_{0}b_{1}}(1+\mu)^{-\frac{1}{4}}Y_{2}\right)\right.\\
				&\left.\qquad \qquad \times\left(U_{1}+e^{-z}(A_{1}+B_{2})\phi+5(2\alpha m_{0}z+a_{0})e^{-z}Y_{2}\right)\right].
			\end{aligned}
		\end{equation*}
		Moreover, using the definition of $X$ in Remark~\ref{re:P}, 
		\begin{equation*}
			\begin{aligned}
				& \partial_{y}^{\le 2}\left[\left(b_{1}X_{1}\phi-\zeta_{2}X_{2}\phi-{10 m_{0}b_{1}}(1+\mu)^{-\frac{1}{4}}Y_{2}\right)^{2}\right]\\
				&=\partial_{y}^{\le 2}\left[b_{1}^{2}P_{2}^{2}-2b_{1}^{2}(X_{1}+2m_{0})P_{2}+b_{1}^{2}(X_{1}+2m_{0})^{2}\right]\\
                &+O_{L^{\infty}}(b_{1}(|\mu b_{1}|+|\Theta|))
                +O\left(|1-\phi|+|\partial_{y}\phi|+|\partial_{y}^{2}\phi|\right)
                .
			\end{aligned}
		\end{equation*}
		Similarly, we check that 
		\begin{equation*}
			\begin{aligned}
				&\partial_{y}^{\le 2}\left[\left(U_{1}+e^{-z}(A_{1}+B_{2})\phi+5(2\alpha m_{0}z+a_{0})e^{-z}Y_{2}\right)^{2}\right]\\
				&+2\partial_{y}^{\le 2}\left[\left(b_{1}X_{1}\phi-\zeta_{2}X_{2}\phi-{10 m_{0}b_{1}}(1+\mu)^{-\frac{1}{4}}Y_{2}\right)\right.\\
				&\qquad \qquad \times\left(U_{1}+e^{-z}(A_{1}+B_{2})\phi+5(2\alpha m_{0}z+a_{0})e^{-z}Y_{2}\right)\Big]\\
				&=O_{L^{\infty}}\left((1+z)(|b_{1}|+|\Theta|)(b_{1}^{2}+e^{-z})
				+z^{2}b_{1}^{4}+z^{2}e^{-2z}
				\right).
			\end{aligned}
		\end{equation*}
		Here, we use again the fact that 
		\begin{equation*}
			\begin{aligned}
				A_{1}+B_{2}&=(A_{1}-2\alpha m_{0}y)
				+2\alpha m_{0}y\big(1-(1+\mu)^{\frac{3}{4}}\big)\\
				&+(B_{2}+2\alpha m_{0}(1+\mu)^{\frac{3}{4}}(y-z))+2\alpha m_{0}(1+\mu)^{\frac{3}{4}}z,\\
				E_{1}+F_{2}&=(E_{1}-3m_{0}y)
				+3m_{0}y\big(1-(1+\mu)^{\frac{3}{4}}\big)\\
				&+(F_{2}+3m_{0}(1+\mu)^{\frac{3}{4}}(y-z))+3m_{0}(1+\mu)^{\frac{3}{4}}z.
			\end{aligned}
		\end{equation*}
		Combining the above estimates with Lemma~\ref{le:PAB} and Definition~\ref{def:Admissible}, we find
		\begin{equation}\label{est:Q3R2}
			\partial_{y}\left(Q_{2}^{3}(R\phi+U)^{2}\right)=b_{1}^{2}\partial_{y}\left(Q_{2}^{3}P_{2}^{2}\right)+O_{\mathcal{S}}(1).
		\end{equation}
		Next, using again Lemma~\ref{le:boundinter} and Definition~\ref{def:Admissible},
		\begin{equation*}
			\begin{aligned}
				\big\|
				\partial_{y}\big(Q_{1}^{2}Q_{2}\left(R\phi+U\right)^{2}\big)\big\|_{H^{1}}
				\lesssim \|\partial_{y}^{\le 2}(Q_{1}^{2}Q_{2})\|_{L^{2}}
				\|\partial_{y}^{\le 2}(R\phi+U)\|^{2}_{L^{\infty}}\lesssim \frac{1}{s^{3}\log s},
				\\
				\big\|
				\partial_{y}\big(Q_{1}Q^{2}_{2}\left(R\phi+U\right)^{2}\big)\big\|_{H^{1}}
				\lesssim \|\partial_{y}^{\le 2}(Q_{1}Q^{2}_{2})\|_{L^{2}}
				\|\partial_{y}^{\le 2}(R\phi+U)\|^{2}_{L^{\infty}}\lesssim \frac{1}{s^{3}\log s}.
			\end{aligned}
		\end{equation*}
        On the other hand, from Lemma~\ref{le:interasym} and Definition~\ref{def:Admissible},
        \begin{equation*}
        \begin{aligned}
            &\|\partial_{y}((R\phi+U)^{5})\|_{H^{1}}+\|\partial_{y}G-\partial_{y}H\|_{H^{1}}\\
            &\lesssim \|\partial_{y}^{\le 1}(R\phi+U)\|_{L^{\infty}}^{4}\|\partial_{y} (R\phi+U)\|_{H^{1}}+
            \|\partial_{y}G-\partial_{y}H\|_{H^{1}}\lesssim \frac{1}{s^{3}\log s}.
            \end{aligned}
        \end{equation*}
		Combining the above estimates with~\eqref{est:Q3R1} and~\eqref{est:Q3R2}, we obtain 
		\begin{equation*}
			\mathcal{H}_{7,2}=10b_{1}^{2}\partial_{y}\left(Q_{1}^{3}X_{1}^{2}\right)-10b_{1}^{2}\partial_{y}\left(Q_{2}^{3}P_{2}^{2}\right)+O_{\mathcal{S}}(1).
		\end{equation*}
		We see that~\eqref{est:H7} follows directly from the above estimate and~\eqref{est:H71}.

		\smallskip
		\textbf{Step 8.} Estimate on $b_{1}\Lambda V$. We claim that 
		\begin{equation}\label{est:bLambdaV}
			\begin{aligned}
				b_{1}\Lambda V&
				=b_{1}^{3}\Lambda \left((E_{1}+F_{2})\phi\right)
				-b_{1}(\zeta_{2}-b_{1})\Lambda X_{2}\phi\\
				&+b_{1}yR\partial_{y}\phi-zb_{1}^{2}\phi\partial_{y}X_{2}-10m_{0}zb_{1}^{2}\partial_{y}Y_{2}\\
				&+b_{1}\Lambda Q_{1}-b_{1}\Lambda Q_{2}+b_{1}^{2}(\Lambda X_{1})\phi-b_{1}^{2}\Gamma(\Lambda X)\phi\\
				&+b_{1}e^{-z}\left(\Lambda A_{1}+\Lambda B_{2}\right)\phi
				-10m_{0}b_{1}^{2}\Gamma (\Lambda Y)+O_{\mathcal{S}}(1).
			\end{aligned}
		\end{equation}
		Indeed, from the definition of $V$ in~\eqref{equ:defV}, we decompose 
		\begin{equation*}
			\begin{aligned}
				b_{1}\Lambda V&=b_{1}\Lambda Q_{1}-b_{1}\Lambda Q_{2}+b_{1}^{2}\Lambda X_{1}\phi-b_{1}\zeta_{2}\Lambda X_{2}\phi\\
				&+b_{1}e^{-z}\left(\Lambda A_{1}+\Lambda B_{2}\right)\phi+b_{1}yR\partial_{y}\phi+b_{1}\Lambda U_{1}\\
				&+5b_{1}(2\alpha m_{0}z+a_{0})e^{-z}\Lambda Y_{2}-10m_{0}b_{1}^{2}(1+\mu)^{-\frac{1}{4}}\Lambda Y_{2}.
			\end{aligned}
		\end{equation*}
		First, we rewrite the last term in the first line by 
		\begin{equation*}
			b_{1}\zeta_{2}\Lambda X_{2}\phi=b_{1}^{2}\Gamma(\Lambda X)\phi
			+zb_{1}^{2}\phi\partial_{y}X_{2}+b_{1}(\zeta_{2}-b_{1})\Lambda X_{2}\phi.
		\end{equation*}
		Similarly, we rewrite the last term in the third line by 
		\begin{equation*}
			\begin{aligned}
				10m_{0}b_{1}^{2}(1+\mu)^{-\frac{1}{4}}\Lambda Y_{2}
				&=10m_{0}b_{1}^{2}\Gamma (\Lambda Y)+10m_{0}zb_{1}^{2}\partial_{y}Y_{2}\\
				&+10m_{0}b_{1}^{2}\big((1+\mu)^{-\frac{1}{4}}-1\big)\Lambda Y_{2}.
			\end{aligned}
		\end{equation*}
		From Definition~\ref{def:Admissible} and $Y\in \mathcal{Y}$, we directly have 
		\begin{equation*}
			\|  b_{1}^{2}\big((1+\mu)^{-\frac{1}{4}}-1\big)\Lambda Y_{2}\|_{H^{1}}\lesssim |\mu z b_{1}^{2}|\lesssim \frac{1}{s^{3}\log s},
		\end{equation*}
		which implies that
		\begin{equation*}
			10m_{0}b_{1}^{2}(1+\mu)^{-\frac{1}{4}}\Lambda Y_{2}
			=10m_{0}b_{1}^{2}\Gamma (\Lambda Y)+10m_{0}zb_{1}^{2}\partial_{y}Y_{2}+O_{\mathcal{S}}(1).
		\end{equation*}
		On the other hand, from the definition of $U_{1}$ in~\eqref{equ:defU1}, we decompose 
		\begin{equation*}
			b_{1}\Lambda U_{1}=b_{1}^{3}\Lambda \left((E_{1}+F_{2})\phi\right)+5(3m_{0}z+a_{2})b_{1}^{3}\Lambda Y_{2}.
		\end{equation*}
		Using again Definition~\ref{def:Admissible} and $Y\in \mathcal{Y}$, we directly have 
		\begin{equation*}
			\|(3m_{0}z+a_{2})b_{1}^{3}\Lambda Y_{2}\|_{H^{1}}\lesssim z^{2}b_{1}^{3}\lesssim \frac{1}{s^{3}\log s},
		\end{equation*}
		which implies that 
		\begin{equation*}
			b_{1}\Lambda U_{1}=b_{1}^{3}\Lambda \left((E_{1}+F_{2})\phi\right)+O_{\mathcal{S}}(1).
		\end{equation*}
		Similarly, we also check that 
		\begin{equation*}
			\|b_{1}(2\alpha m_{0}z+a_{0})e^{-z}\Lambda Y_{2}\|_{H^{1}}\lesssim b_{1}z^{2}e^{-z}\lesssim \frac{1}{s^{3}\log s}.
		\end{equation*}
		We see that the estimate~\eqref{est:bLambdaV} follows directly from the above estimates.
		
		\smallskip
		\textbf{Step 9.} Conclusion. Combining~\eqref{equ:dsV},~\eqref{equ:LV},~\eqref{est:H1}--\eqref{est:H6},~\eqref{est:H7} and~\eqref{est:bLambdaV} with an elementary computation, we complete the proof for Proposition~\ref{prop:approx}.
	\end{proof}
	
	Last, we establish some technical estimates related to the error terms $(\Psi_{i})_{i=1}^{4}$.
	\begin{lemma}\label{le:Psi1}
		The following estimates hold true.
		\begin{enumerate}
			\item \emph{Estimates related to $\Psi_{1}(V)$.}
			We have 
			\begin{equation*}
				\|\Psi_{1}(V)\|_{H^{1}}\lesssim\frac{|\dot{z}-b_{1}z-\mu|}{s\log s}+\frac{1}{s^{\frac{5}{2}}\log s}.
			\end{equation*}
			In particular, we have 
			\begin{equation*}
				\left| \left(\Psi_{1}(V),Q_{1}\right)\right|+
				\left|\left(\Psi_{1}(V),Q_{2}\right)\right|\lesssim 
				\frac{|\dot{z}-b_{1}z-\mu|}{s^{2}\log s}+
				\frac{1}{s^{3}\log s}.
			\end{equation*}
			
			\item \emph{Estimates related to $\Psi_{2}(V)$.}
			We have 
			\begin{equation*}
				\begin{aligned}
					\Psi_{2}(V)&=-\dot{z}e^{-z}\left(A_{1}+B_{2}\right)\phi
					+\frac{\dot{\mu}}{2}b_{1}^{2}\phi\Gamma (\Lambda F)\\
					&+O_{H^{1}}\left(\frac{|\dot{z}|}{s^{\frac{3}{2}}\log s}+\frac{|\dot{\mu}|}{s\log s}+\frac{|\dot{b}_{1}|}{s^{\frac{1}{2}}}+\frac{1}{s^{\frac{5}{2}}\log s}\right).
				\end{aligned}
			\end{equation*}
			In particular, we have 
			\begin{equation*}
				\begin{aligned}
					\left|  \left(\Psi_{2}(V),Q_{1}\right)\right|
					&\lesssim
					\frac{|\dot{z}|}{s^{2}\log s}+\frac{|\dot{\mu}|}{s^{2}\log s}+\frac{|\dot{b}_{1}|}{s}+\frac{1}{s^{3}\log s},\\
					\left(\Psi_{2}(V),Q_{2}\right)&=-2\alpha m_{0}^{2}\dot{z}ze^{-z}+\frac{5}{2}m_{0}b_{1}\dot{\mu}\left(Y-2\Lambda Y,Q\right)\\
					&+O\left( \frac{|\dot{z}|}{s^{2}\log s}+\frac{|\dot{\mu}|}{s^{2}\log s}+\frac{|\dot{b}_{1}|}{s}+\frac{1}{s^{3}\log s}\right).
				\end{aligned}
			\end{equation*}
			
			\item \emph{Estimates related to $\Psi_{3}(V)$.}
			We have 
			\begin{equation*}
				\|\Psi_{3}(V)\|_{H^{1}}\lesssim \frac{1}{s^{\frac{5}{2}}\log s}.
			\end{equation*}
			In particular, we have 
			\begin{equation*}
				\begin{aligned}
					(\Psi_{3}(V),Q_{1})
					&=\frac{\alpha}{4}m_{0}^{2}\zeta_{2}z^{2}e^{-z}+\frac{\alpha}{2}m_{0}^{2}\mu ze^{-z}+O\left(\frac{1}{s^{3}\log s}\right),\\
					(\Psi_{3}(V),Q_{2})
					&=-\frac{\alpha}{4}m_{0}^{2}b_{1}z^{2}e^{-z}+O\left(\frac{1}{s^{3}\log s}\right).
				\end{aligned}
			\end{equation*}
			
			\item  \emph{Estimates related to $\Psi_{4}(V)$.}
			We have 
			\begin{equation*}
				\begin{aligned}
					|\Psi_{4}(V)|&\lesssim \frac{|\dot{b}_{1}|}{s}\log s \mathbf{1}_{[\frac{s}{2},\frac{5s}{2}]}(y)
					+\frac{1}{s^{3}}\mathbf{1}_{[\frac{s}{2},\frac{5s}{2}]}(y),\\
					|\partial_{y}\Psi_{4}(V)|&\lesssim
					\frac{|\dot{b}_{1}|}{s} \mathbf{1}_{[\frac{s}{2},\frac{5s}{2}]}(y)
					+\frac{1}{s^{3}\log s}\mathbf{1}_{[\frac{s}{2},\frac{5s}{2}]}(y).
				\end{aligned}
			\end{equation*}
			In particular, we have 
			\begin{equation*}
				\left| \left(\Psi_{4}(V),Q_{1}\right)\right|+\left|\left(\Psi_{4}(V),Q_{2}\right)\right|\lesssim \frac{|\dot{b}_{1}|}{s}+\frac{1}{s^{3}\log s}.
			\end{equation*}
		\end{enumerate}
	\end{lemma}
	
	\begin{proof}
		Proof of (i). The proof for the $H^{1}$ norm estimate of $\Psi_{1}(V)$ is a direct consequence of Definition~\ref{def:Admissible} and Lemma~\ref{le:RAB}. On the other hand, using the exponential decay of $Q$ and Lemma~\ref{le:boundinter}, we find 
		\begin{equation*}
			|(\partial_{y}X_{2},Q_{1})|+ |(\partial_{y}Y_{2},Q_{1})|
            +|(\Gamma(\Lambda Q),Q_{1})|
            \lesssim z^{5}e^{-z}\lesssim s^{-1}.
		\end{equation*}
		In addition, from Remark~\ref{re:Y} and Remark~\ref{re:P}, we check that 
		\begin{equation*}
			(\partial_{y}X_{2},Q_{2})=(\partial_{y}Y_{2},Q_{2})
            =(\Gamma(\Lambda Q),Q_{2})
            =0\ \ \mbox{and}\ \ \left(\Lambda X_{2},Q_{2}\right)=\frac{1}{2}(X,Q)+(yX',Q).
		\end{equation*}
		Combining the above estimates with Lemma~\ref{le:boundinter} and Definition~\ref{def:Admissible}, we complete the proof for the estimates in (i).
		
		\smallskip
		Proof of (ii). The proof for the $H^{1}$ norm estimate of $\Psi_{2}(V)$ is also a direct consequence of Definition~\ref{def:Admissible} and Lemma~\ref{le:RAB}. On the other hand, from Lemma~\ref{le:AB} and the exponential decay of $Q$, we directly have 
        \begin{equation*}
            \left(  -\dot{z}e^{-z}\left(A_{1}+B_{2}\right)\phi,Q_{1}\right)=O\left(\frac{|\dot{z}|}{s^{2}\log s}\right).
        \end{equation*}
         In addition, we rewrite 
		\begin{equation*}
			\begin{aligned}
				A_{1}+B_{2}&=(A_{1}-2\alpha m_{0}y)
				+2\alpha m_{0}y\big(1-(1+\mu)^{\frac{3}{4}}\big)\\
				&+(B_{2}+2\alpha m_{0}(1+\mu)^{\frac{3}{4}}(y-z))+2\alpha m_{0}(1+\mu)^{\frac{3}{4}}z,
			\end{aligned}
		\end{equation*}
		which implies that\footnote{Here, we use the fact that $m_{0}=\frac{1}{4}\int_{\RR}Q(y)\dd y$.}
		\begin{equation*}
			\begin{aligned}
				\left(  -\dot{z}e^{-z}\left(A_{1}+B_{2}\right)\phi,Q_{2}\right)=
				-8\alpha m_{0}^{2}\dot{z}ze^{-z}+
				O\left(\frac{|{\dot{z}}|}{s^{2}\log s}\right).
			\end{aligned}
		\end{equation*}
		Then, from Remark~\ref{re:Y}, we compute 
		\begin{equation*}
			\begin{aligned}
				\left(-10\alpha m_{0}\dot{z}ze^{-z}(Y_{2}+\partial_{y}Y_{2}),Q_{1}\right)
				&=O\left(\frac{|{\dot{z}}|}{s^{2}\log s}\right),\\
				\left(-10\alpha m_{0}\dot{z}ze^{-z}(Y_{2}+\partial_{y}Y_{2}),Q_{2}\right)
				&=
				6\alpha m_{0}^{2}\dot{z}ze^{-z}+
				O\left(\frac{|{\dot{z}}|}{s^{2}\log s}\right).
			\end{aligned}
		\end{equation*}
		Combining the above estimates with Lemma~\ref{le:boundinter} and Definition~\ref{def:Admissible}, we complete the proof for the estimates in (ii).
		
		\smallskip
		Proof of (iii). The proof for the $H^{1}$ norm estimate of $\Psi_{3}(V)$ is also a direct consequence of Definition~\ref{def:Admissible} and Lemma~\ref{le:RAB}. On the other hand, using again Lemma~\ref{le:boundinter} and Definition~\ref{def:Admissible}, we compute\footnote{Here, we use the fact that $\alpha=\frac{c_{Q}}{m_{0}^{2}}\int_{\RR}e^{y}Q^{5}(y)\dd y$.}
		\begin{equation*}
			\begin{aligned}
				\frac{5c_{Q}}{4}b_{1}z^{2}e^{-z}\left(\partial_{y}\left(e^{-(y-z)}Q^{4}(y-z)\right),Q_{1}\right)&=O\left(\frac{1}{s^{3}\log s}\right),\\
				\frac{5c_{Q}}{4}b_{1}z^{2}e^{-z}\left(\partial_{y}\left(e^{-(y-z)}Q^{4}(y-z)\right),Q_{2}\right)&=
				-\frac{\alpha}{4}m_{0}^{2}b_{1}z^{2}e^{-z}+
				O\left(\frac{1}{s^{3}\log s}\right).
			\end{aligned}
		\end{equation*}
		Similar as above, we also compute 
		\begin{equation*}
			\begin{aligned}
				\frac{5c_{Q}}{4}\zeta_{2}z^{2}e^{-z}
				\left(\partial_{y}(e^{y}Q^{4}(y)),Q_{1}\right)&=\frac{\alpha}{4}m_{0}^{2}\zeta_{2}z^{2}e^{-z},\\
				\frac{5c_{Q}}{4}\zeta_{2}z^{2}e^{-z}
				\left(\partial_{y}(e^{y}Q^{4}(y)),Q_{2}\right)&=
				O\left(\frac{1}{s^{3}\log s}\right),\\
				\frac{5c_{Q}}{2}\mu z e^{-z}
				\left(\partial_{y}\left(e^{y}Q^{4}(y)\right),Q_{1}\right)
				&=\frac{\alpha}{2}m_{0}^{2}\mu z e^{-z},\\
				\frac{5c_{Q}}{2}\mu z e^{-z}
				\left(\partial_{y}\left(e^{y}Q^{4}(y)\right),Q_{2}\right)
				&=O\left(\frac{1}{s^{3}\log s}\right).
			\end{aligned}
		\end{equation*}
		Then, from Remark~\ref{re:Y}, Lemma~\ref{le:EF} and Lemma~\ref{le:boundinter},
		\begin{equation*}
			(\partial_{y}F_{2},Q_{2})=(\partial_{y}Y_{2},Q_{2})=(\partial_{y}(Q_{2}^{4}),Q_{2})=0,
		\end{equation*}
		\begin{equation*}
			|(\partial_{y}F_{2},Q_{1})|+|(\partial_{y}Y_{2},Q_{1})|
			+|(\partial_{y}(Q_{2}^{4}),Q_{1})|\lesssim s^{-1}.
		\end{equation*}
		Combining above estimates with Definition~\ref{def:Admissible}, we complete the proof for (iii).
		
		\smallskip 
		Proof of (iv). 
		From Definition~\ref{def:Admissible} and the definition of $\phi$ in~\eqref{equ:defphi}, we check that 
		\begin{equation*}
			\begin{aligned}
				\left|\partial_{s}\phi\right|
				&\lesssim \bigg|\bigg(\frac{\dot{a}}{a}\bigg)\chi'(ay)\bigg|\lesssim 
				(s\log s)|\dot{b}_{1}|\textbf{1}_{[\frac{s}{2},\frac{5s}{2}]}(y),\\
				\left|\partial_{y}\partial_{s}\phi\right|
				&\lesssim |\dot{a}|\left(\left|\chi'(ay)\right|+
                \left|\chi''(ay)\right|
                \right)
                \lesssim 
				(\log s)|\dot{b}_{1}|\textbf{1}_{[\frac{s}{2},\frac{5s}{2}]}(y).
			\end{aligned}
		\end{equation*}
		Similar as above, we also check that 
		\begin{equation*}
			\begin{aligned}
				\left| \partial_{y}\phi\right|\lesssim s^{-1}\textbf{1}_{[\frac{s}{2},\frac{5s}{2}]}(y)\quad \mbox{and}\quad 
				\left| \partial^{2}_{y}\phi\right|\lesssim s^{-2}\textbf{1}_{[\frac{s}{2},\frac{5s}{2}]}(y).
			\end{aligned}
		\end{equation*}
		Recall that, we rewrite the term $R$ by
		\begin{equation*}
			\begin{aligned}
				R
				&=e^{-z}\big(B_{2}+2\alpha m_{0}(1+\mu)^{\frac{3}{4}}(y-z)\big)+2\alpha m_{0}(1+\mu)^{\frac{3}{4}}ze^{-z}\\
				&+2b_{1}\big((1+\mu)^{\frac{1}{4}}-1\big)m_{0}-b_{1}(X_{2}+2(1+\mu)^{\frac{1}{4}}m_{0})+(b_{1}-\zeta_{2})X_{2}\\
				&+b_{1}(X_{1}+2m_{0})+e^{-z}(A_{1}-2\alpha m_{0}y)+2\alpha m_{0}e^{-z}y\big(1-(1+\mu)^{\frac{3}{4}}\big),
			\end{aligned}
		\end{equation*}
		which implies that 
		\begin{equation*}
			|R|\textbf{1}_{[\frac{s}{2},\frac{5s}{2}]}(y)\lesssim s^{-2}\textbf{1}_{[\frac{s}{2},\frac{5s}{2}]}(y) \ \ \mbox{and} \ \ 
			|\partial_{y}R|\textbf{1}_{[\frac{s}{2},\frac{5s}{2}]}(y)\lesssim s^{-2}\log^{-1}s\textbf{1}_{[\frac{s}{2},\frac{5s}{2}]}(y).
		\end{equation*}
		In addition, we rewrite the term $E_{1}+F_{2}$ by 
		\begin{equation*}
			\begin{aligned}
				E_{1}+F_{2}&=(E_{1}-3 m_{0}y)
				+3 m_{0}y\big(1-(1+\mu)^{\frac{3}{4}}\big)\\
				&+(F_{2}+3 m_{0}(1+\mu)^{\frac{3}{4}}(y-z))+3 m_{0}(1+\mu)^{\frac{3}{4}}z,
			\end{aligned}
		\end{equation*}
		which implies that 
		\begin{equation*}
			\begin{aligned}
				| E_{1}+F_{2}|\textbf{1}_{[\frac{s}{2},\frac{5s}{2}]}(y)
				& \lesssim (1+z+|\mu y|)\textbf{1}_{[\frac{s}{2},\frac{5s}{2}]}(y)
				\lesssim (\log s) \textbf{1}_{[\frac{s}{2},\frac{5s}{2}]}(y),\\
				|\partial_{y} E_{1}+\partial_{y}F_{2}|\textbf{1}_{[\frac{s}{2},\frac{5s}{2}]}(y)
				&\lesssim (1+|\mu y|)\textbf{1}_{[\frac{s}{2},\frac{5s}{2}]}(y)
				\lesssim  \textbf{1}_{[\frac{s}{2},\frac{5s}{2}]}(y).
			\end{aligned}
		\end{equation*}
		It follows from Definition~\ref{def:Admissible} and the definition of $\Psi_{4}(V)$ in Proposition~\ref{prop:approx} that 
		\begin{equation*}
			\begin{aligned}
				\left|\Psi_{4}(V)\right|
				&\lesssim 
				\left(\frac{|E_{1}+F_{2}|}{s^{2}\log^{2}s}\right)|\partial_{s}\phi|+\left(\frac{|yR|}{s\log s}\right)|\partial_{y}\phi|\\
				&+\left(\frac{|E_{1}+F_{2}|}{s^{2}\log^{2}s}\right)|\partial_{y}\phi|+|R|(\partial_{s}\phi|+\partial_{y}\phi|)\\
				&\lesssim \frac{|\dot{b}_{1}|}{s}\log s \textbf{1}_{[\frac{s}{2},\frac{5s}{2}]}(y)
				+\frac{1}{s^{3}\log s}\textbf{1}_{[\frac{s}{2},\frac{5s}{2}]}(y)
				+\frac{1}{s^{3}}\textbf{1}_{[\frac{s}{2},\frac{5s}{2}]}(y).
			\end{aligned}
		\end{equation*}
		Similarly, we also have 
		\begin{equation*}
			\begin{aligned}
				\left|\partial_{y}\Psi_{4}(V)\right|
				&\lesssim \left(\frac{|E_{1}+F_{2}|}{s^{2}\log^{2}s}+|R|\right)|\partial_{s}\partial_{y}\phi|+\left(\frac{|yR|}{s\log s}+|R|\right)|\partial^{2}_{y}\phi|\\
				&+ 
				\left(\frac{|\partial_{y}E_{1}+\partial_{y}F_{2}|}{s^{2}\log^{2}s}+|\partial_{y}R|\right)|\partial_{s}\phi|+\left(\frac{|R|+|y\partial_{y}R|}{s\log s}\right)|\partial_{y}\phi|\\
				&+\left(\frac{|E_{1}+F_{2}|}{s^{2}\log^{2}s}\right)|\partial^{2}_{y}\phi|
				+\left(\frac{|\partial_{y}(E_{1}+F_{2})|}{s^{2}\log^{2}s}\right)|\partial_{y}\phi|+|\partial_{y}R||\partial_{y}\phi|
				\\
				&\lesssim \frac{|\dot{b}_{1}|}{s} \textbf{1}_{[\frac{s}{2},\frac{5s}{2}]}(y)
				+\frac{1}{s^{3}\log^{2} s}\textbf{1}_{[\frac{s}{2},\frac{5s}{2}]}(y)
				+\frac{1}{s^{3}\log s}\textbf{1}_{[\frac{s}{2},\frac{5s}{2}]}(y).
			\end{aligned}
		\end{equation*}
		Based on the above estimates and the exponential decay of $Q$, we obtain 
		\begin{equation*}
			\begin{aligned}
				& |(\Psi_{4}(V),Q_{1})|+|(\Psi_{4}(V),Q_{2})|\\
				& \lesssim 
				\frac{|\dot{b}_{1}|}{s}\log s \big\|Q_{1}\textbf{1}_{[\frac{s}{2},\frac{5s}{2}]}(y)\big\|_{L^{1}}
				+  \frac{|\dot{b}_{1}|}{s}\log s \big\|Q_{2}\textbf{1}_{[\frac{s}{2},\frac{5s}{2}]}(y)\big\|_{L^{1}}\\
				& +\frac{1}{s^{3}}\big\|Q_{1}\textbf{1}_{[\frac{s}{2},\frac{5s}{2}]}(y)\big\|_{L^{1}}
				+\frac{1}{s^{3}}\big\|Q_{2}\textbf{1}_{[\frac{s}{2},\frac{5s}{2}]}(y)\big\|_{L^{1}}\lesssim  \frac{|\dot{b}_{1}|}{s}+\frac{1}{s^{3}\log s}.
			\end{aligned}
		\end{equation*}
		Combining the above estimates, we complete the proof for (iv).
	\end{proof}
	
	\begin{lemma}\label{le:Psi2}
		The following estimates hold true.
		\begin{enumerate}
			\item \emph{Estimate related to $\Psi_{1}(V)$}. We have 
			\begin{equation*}
				\begin{aligned}
					\left|(\Psi_{1}(V),\Lambda Q)\right|
					+\left|(\Psi_{1}(V),\Gamma \Lambda Q)\right|
					&\lesssim \frac{|\dot{z}-b_{1}z-\mu|}{s\log s}+\frac{1}{s^{3}},\\
					\left|(\Psi_{1}(V),\partial_{y}Q_{1})\right|
					+\left|(\Psi_{1}(V),\partial_{y}Q_{2})\right|
					&\lesssim \frac{|\dot{z}-b_{1}z-\mu|}{s\log s}+\frac{1}{s^{3}}.
				\end{aligned}
			\end{equation*}
			
			\item \emph{Estimate related to $\Psi_{2}(V)$}. We have 
			\begin{equation*}
				\begin{aligned}
					\left|(\Psi_{2}(V),\Lambda Q)\right|
					+\left|(\Psi_{2}(V),\Gamma \Lambda Q)\right|
					&\lesssim \frac{|\dot{z}|}{s^{2}}+\frac{|\dot{\mu}|}{s\log s}+\frac{|\dot{b}_{1}|}{s}+\frac{1}{s^{3}},\\
					\left|(\Psi_{2}(V),\partial_{y}Q_{1})\right|
					+\left|(\Psi_{2}(V),\partial_{y}Q_{2})\right|
					&\lesssim \frac{|\dot{z}|}{s^{2}}+\frac{|\dot{\mu}|}{s\log s}+\frac{|\dot{b}_{1}|}{s}+\frac{1}{s^{3}}.
				\end{aligned}
			\end{equation*}
			
			\item \emph{Estimate related to $\Psi_{3}(V)$}. We have 
			\begin{equation*}
				\begin{aligned}
					\left|(\Psi_{3}(V),\Lambda Q)\right|
					+\left|(\Psi_{3}(V),\Gamma \Lambda Q)\right|\lesssim \frac{1}{s^{3}},
					\\
					\left|(\Psi_{3}(V), \partial_{y}Q_{1})\right|
					+\left|(\Psi_{3}(V),\partial_{y}Q_{2})\right|
					\lesssim \frac{1}{s^{3}}.
				\end{aligned}
			\end{equation*}
			
			\item \emph{Estimate related to $\Psi_{4}(V)$}. We have 
			\begin{equation*}
				\begin{aligned}
					\left|(\Psi_{4}(V),\Lambda Q)\right|
					+\left|(\Psi_{4}(V),\Gamma \Lambda Q)\right|
					\lesssim \frac{|\dot{b}_{1}|}{s}+\frac{1}{s^{3}},\\
					\left|(\Psi_{4}(V),\partial_{y}Q_{1})\right|
					+\left|(\Psi_{4}(V),\partial_{y}Q_{2})\right|
					\lesssim \frac{|\dot{b}_{1}|}{s}+\frac{1}{s^{3}}.
				\end{aligned}
			\end{equation*}
		\end{enumerate}
	\end{lemma}
	\begin{proof}
		The proof is based on an argument similar to the proof for Lemma~\ref{le:Psi1} and the definition of $(\Psi_{i})_{i=1}^{4}$ in Proposition~\ref{prop:approx}, and we omit it.
	\end{proof}
	
	\section{Dynamics close to the two-bubble solution}

	\subsection{Modulation of the approximate solution}
	In this subsection, we state a standard modulation result around the approximate two-bubble solution $V$. The proof of such result relies on a standard argument based on the Implicit Function Theorem~(see \emph{e.g.}~\cite[Lemma 2.5]{MMRACTA}).  For the sake of completeness and the reader's convenience, we provide the details for the proof here.
	
	\begin{proposition}\label{prop:modu}
		Let $u(t)$ be a solution of~\eqref{equ:gKdV} on a time interval $[t_{0},t_{1}]$. Suppose that at some time  $\bar{t}\in [t_{0},t_{1}]$, the solution $u(\bar{t})$ admits a geometrical decomposition in the following sense: there exist some geometrical parameters $\bar{\mathcal{G}}=(\bar{\lambda},\bar{z},\bar{\mu},\bar{x}_{1},\bar{b}_{1},\bar{b}_{2})\in (0,\infty)^{2}\times \mathbb{R}^{4}$ such that, defining
		\begin{equation*}
			\bar{\varepsilon}(y)=\bar{\lambda}^{\frac{1}{2}}u\big(\bar{t},\bar{\lambda}y+\bar{x}_{1}\big)-V
			\big(y;(\bar{z},\bar{\mu},\bar{b}_{1},\bar{b}_{2})\big)\in H^{1},
		\end{equation*}
		the function $\bar{\varepsilon}$ satisfies the following orthogonality conditions 
		\begin{equation}\label{equ:orthmodu1}
			\begin{aligned}
				(\bar{\varepsilon},\Lambda Q)
				&=(\bar{\varepsilon},Q')
				=(\bar{\varepsilon},Q)
				=0,\\
				(\bar{\varepsilon},\overline{\Gamma}\Lambda Q)
				& =(\bar{\varepsilon},\overline{\Gamma}Q')
				=(\bar{\varepsilon},\overline{\Gamma}Q)=0.
			\end{aligned}
		\end{equation}
		Here, we denote 
		\begin{equation*}
			\begin{aligned}
				(\overline{\Gamma}Q)(y)&=(1+\bar{\mu})^{\frac{1}{4}}\big[Q((1+\bar{\mu})^{\frac{1}{2}}(y-\bar{z}))\big],\\
				(\overline{\Gamma}Q')(y)&=(1+\bar{\mu})^{\frac{1}{4}}\big[Q'((1+\bar{\mu})^{\frac{1}{2}}(y-\bar{z}))\big],\\
				( \overline{\Gamma}\Lambda Q)(y)&=(1+\bar{\mu})^{\frac{1}{4}}\big[(\Lambda Q)((1+\bar{\mu})^{\frac{1}{2}}(y-\bar{z}))\big].
			\end{aligned}
		\end{equation*}
		We assume moreover that there exists a small enough constant $\delta>0$ such that 
		\begin{equation}\label{equ:modu2}
			(\bar{\lambda},\bar{z})\in (0,\infty)\times (\delta^{-1},\infty)
			\quad \mbox{and}\quad \|\bar{\varepsilon}\|_{H^{1}}+|\bar{b}_{1}|+|\bar{b}_{2}|+|\bar{\mu}|< \bar{z}^{-3}<\delta^{3}.
		\end{equation}
		Then there exist a time $\tau>0$ small enough and a  unique $C^{1}$ geometrical function $\mathcal{G}(t)=(\lambda(t),z(t),\mu(t),x_{1}(t),b_{1}(t),b_{2}(t))\in (0,\infty)^{2}\times \RR^{4}$ on $[\bar{t}-\tau,\bar{t}+\tau]$ such that, for all $t\in [\bar{t}-\tau,\bar{t}+\tau]$, $\varepsilon(t)$ being defined by 
		\begin{equation}\label{equ:defe}
			\varepsilon(t,y)=\lambda^{\frac{1}{2}}(t)u
			\left(t,\lambda(t)y+x_{1}(t)\right)-
			V\left(y;(z,\mu,b_{1},b_{2})\right),
		\end{equation}
		it satisfies the orthogonality conditions
		\begin{equation}\label{equ:orth}
			\begin{aligned}
				(\varepsilon(t),\Lambda Q)&
				=(\varepsilon(t),\partial_{y}Q_{1})
				=(\varepsilon(t),Q_{1})=0,\\
				(\varepsilon(t),\Gamma\Lambda Q)&=
				(\varepsilon(t),\partial_{y} Q_{2})= (\varepsilon(t),Q_{2})=0.
			\end{aligned}
		\end{equation}
		In addition, we have $\mathcal{G}(\bar{t})=\bar{\mathcal{G}}$ and $(\varepsilon(t),\mathcal{G}(t))$ satisfies
		\begin{equation*}
			\|\varepsilon\|_{H^{1}}+|b_{1}|+|b_{2}|
			+|\mu|
			+|z|^{-1}
			+|x_{1}-\bar{x}_{1}|
			+\left|\frac{\lambda}{\bar{\lambda}}-1\right|\lesssim \varrho (\delta).
		\end{equation*}
	\end{proposition}
	
	\begin{proof}
		Using the scaling and translation invariances of~\eqref{equ:gKdV}, we can assume that $(\bar{\lambda},\bar{x}_{1})=(1,0)$. From now on, we denote $(\bar{u},\overline{\mathcal{G}})=(u(\bar{t}),1,\bar{z},\bar{\mu},0,\bar{b}_{1},\bar{b}_{2})$. 
		
		\smallskip
		Fix $t\in [t_{0},t_{1}]$. We consider the map
		\begin{equation*}
			\begin{aligned}
				\Phi: (u,\mathcal{G})\mapsto (\Phi_{1},\Phi_{2},\Phi_{3},\Phi_{4},\Phi_{5},\Phi_{6})\in \RR^{6}.
			\end{aligned}
		\end{equation*}
		Here, we denote 
		\begin{equation*}
			\begin{aligned}
				\Phi_{1}&=(\varepsilon(t),\Lambda Q), \quad \ \Phi_{2}=(\varepsilon(t),\partial_{y}Q_{2}),\ \ \Phi_{3}=(\varepsilon(t),\Gamma \Lambda Q),\\
				\Phi_{4}&=(\varepsilon(t),\partial_{y}Q_{1}),\ \ \Phi_{5}=(\varepsilon(t),Q_{1}),\quad \ \ \Phi_{6}=(\varepsilon(t),Q_{2}).
			\end{aligned}
		\end{equation*}
		Note that, the set of conditions in~\eqref{equ:orth} is equivalent to $\Phi(u,\mathcal{G})=0$. In what follows, we solve this nonlinear system around $(\bar{u},\overline{\mathcal{G}})$ by the Implicit Function Theorem. 
		
		\smallskip
		First, it is easy to check that 
		$\Phi(\bar{u},\overline{\mathcal{G}})=0$. Second, by~\eqref{equ:modu2} and~\eqref{equ:defe}, we find
		\begin{equation*}
			\begin{aligned}
				\left. \frac{\partial \varepsilon}{\partial x_{1}} \right|_{(u,\mathcal{G})=(\bar{u},\overline{\mathcal{G}})}
				&=Q'-Q'(\cdot-\bar{z})+O_{H^{1}_{loc}}\left(\bar{z}^{-2}\right),\\
				\left. \frac{\partial \varepsilon}{\partial b_{1}} \right|_{(u,\mathcal{G})=(\bar{u},\overline{\mathcal{G}})}
				&=10m_{0}Y(\cdot-\bar{z})-X
				+O_{H^{1}_{loc}}\left(\bar{z}^{-2}\right),\\
				\left. \frac{\partial \varepsilon}{\partial \lambda} \right|_{(u,\mathcal{G})=(\bar{u},\overline{\mathcal{G}})}
				&=\Lambda Q-(\Lambda Q)(\cdot-\bar{z})-\bar{z}Q'(\cdot-\bar{z})+O_{H^{1}_{loc}}\left(\bar{z}^{-2}\right).
			\end{aligned}
		\end{equation*}
		Using again~\eqref{equ:modu2} and~\eqref{equ:defe}, we also find 
		\begin{equation*}
			\begin{aligned}
				\left. \frac{\partial \varepsilon}{\partial b_{2}} \right|_{(u,\mathcal{G})=(\bar{u},\overline{\mathcal{G}})}
				&=X(\cdot-\bar{z})+O_{H^{1}_{loc}}\left(\bar{z}^{-2}\right),\\
				\left. \frac{\partial \varepsilon}{\partial z} \right|_{(u,\mathcal{G})=(\bar{u},\overline{\mathcal{G}})}
				&=-Q'(\cdot-\bar{z})+O_{H^{1}_{loc}}\left(\bar{z}^{-2}\right),\\
				\left. \frac{\partial \varepsilon}{\partial \mu} \right|_{(u,\mathcal{G})=(\bar{u},\overline{\mathcal{G}})}&=\frac{1}{2}(\Lambda Q)(\cdot-\bar{z})+O_{H^{1}_{loc}}\left(\bar{z}^{-2}\right).
			\end{aligned}
		\end{equation*}
		Therefore, using integration by parts, for any $(u,\mathcal{G})$ around $(\bar{u},\overline{\mathcal{G}})$, we find 
		\begin{equation*}
         \begin{aligned}
			D_{\mathcal{G}}\Phi&=
			\begin{pmatrix}
				\|\Lambda Q\|_{L^{2}}^{2}
				& 0 & 0 & 0 & O(1) & 0\\
				- \bar{z}\|Q'\|_{L^{2}}^{2} &-\|Q'\|_{L^{2}}^{2} & 0 & -\|Q'\|_{L^{2}}^{2} & 0 &0 \\
				-\|\Lambda Q\|_{L^{2}}^{2} & 0 & \frac{1}{2}\|\Lambda Q\|^{2}_{L^{2}}& 0 &O(1) & O(1)\\
				0& 0 & 0 & \|Q'\|^{2}_{L^{2}} &0 &0 \\
				0 & 0 & 0 & 0 & m_{0}^{2} & 0\\
				0 & 0 & 0 & 0 & 2m_{0}^{2} & -m_{0}^{2}
			\end{pmatrix}
            \\
			&+O\left(\bar{z}^{-2}\right)+O\left(\|u-\bar{u}\|_{H^{1}}\right)+O\left(|\bar{\mathcal{G}}-\mathcal{G}|\right).
            \end{aligned}
		\end{equation*}
		Here, we use the fact that\footnote{See Proposition~\ref{prop:Spectral}, Remark~\ref{re:Y} and Remark~\ref{re:P}.}
		\begin{equation*}
			\begin{aligned}
				(\Lambda Q,Q)=(\Lambda Q,Q')=(X,Q)+m_{0}^{2}&=0,\\
				(Q,Q')=(X,Q')=(Y,Q')=(Y,Q)+\frac{3}{5}m_{0}^{2}&=0.
			\end{aligned}
		\end{equation*}
		Thus, from the smallness condition~\eqref{equ:modu2}, $D_{\mathcal{G}}\Phi(u,\mathcal{G})$ is an invertible matrix for $\delta>0$ small enough, with a lower bound uniform around $(\bar{u},\overline{\mathcal{G}})$.
		Therefore, by the uniform variant of the Implicit
		Function Theorem, there exist $\delta_{1}\in (0,\infty)$ small enough (independent with $\bar{z}$ and $\delta$) and a $C^{1}$ map 
		\begin{equation*}
			\Pi : B_{H^{1}}(\bar{u},\delta_{1}\bar{z}^{-3})\to B_{\RR^{6}}(\overline{\mathcal{G}},\bar{z}^{-3}),
		\end{equation*}
		such that for all $(u,\mathcal{G})\in B_{H^{1}}(\bar{u},\delta_{1}\bar{z}^{-3})\times B_{\RR^{6}}(\overline{\mathcal{G}},\bar{z}^{-3})$,
		\begin{equation*}
			\Phi (u,\mathcal{G})=0  \ \ 
			\mbox{if and only if}\ \ 
			\mathcal{G}=\Pi (u).
		\end{equation*}
		Moreover, taking $\delta>0$ small enough, for any $\bar{\varepsilon}$ satisfying~\eqref{equ:modu2}, we have 
		\begin{equation*}
			\|\varepsilon\|_{H^{1}}
			+|\mathcal{G}-\overline{\mathcal{G}}|
			\lesssim \|\bar{\varepsilon}\|_{H^{1}}+  |\mathcal{G}-\overline{\mathcal{G}}|\lesssim \|\overline{\varepsilon}\|_{H^{1}}
			+
			|\Pi (u)-\Pi(\bar{u})| \lesssim \delta,
		\end{equation*}
		which directly completes the proof for Proposition~\ref{prop:modu}.
	\end{proof}
	
	\subsection{Backwards uniform estimates}
	In this subsection, we state the uniform estimates for particular backwards solutions. As usual in the construction of solutions with explicit asymptotic behavior, the key point is to carefully adjust their final data to obtain the estimates that are uniform in the special regime of Theorem~\ref{thm:main}.
	
	\smallskip
	Let $\mathcal{G}^{in}=(\lambda^{in},z^{in},\mu^{in},x_{1}^{in},b_{1}^{in},b_{2}^{in})\in (0,\infty)^{2}\times \RR^{4}$ to be chosen with $0<\lambda^{in}\ll 1$, $1\ll z^{in}<\infty$, $1\ll x_{1}^{in}<\infty$ and $|\mu^{in}|+|b_{1}^{in}|+|b_{2}^{in}|\ll (z^{in})^{-3}\ll 1$. Let $u(t)$ for $t\in [T^*,T_n]$ be the solution of~\eqref{equ:gKdV} with the following finial data 
	\begin{equation}\label{equ:deffinaldata}
		u(T_n,x)=\frac{1}{(\lambda^{in})^{\frac{1}{2}}}V^{in}\left(\frac{x-x_{1}^{in}}{\lambda^{in}}\right)\in H^{1}.
	\end{equation}
	Here, we denote 
	\begin{equation*}
		V^{in}(y)=V(y;(z^{in},\mu^{in},b_{1}^{in},b_{2}^{in}))\in H^{1}.
	\end{equation*}
	Note that, the finial data $u(T_n)$ satisfies~\eqref{equ:orthmodu1}--\eqref{equ:modu2}, and thus, by Proposition~\ref{prop:modu} and continuity of the solution flow for~\eqref{equ:gKdV} in $H^{1}$, it exists and admits decomposition which satisfies~\eqref{equ:modu2}--\eqref{equ:orth} on some time interval $[T^{*},T_{n}]$. In addition, we denote $(\varepsilon,\mathcal{G})$ by the error term and geometrical parameters related to the decomposition in Proposition~\ref{prop:modu}.
	
	\smallskip
	For $s^{in}\gg 1$, we normalize the rescaled time $s$ as follows: for $t\in [T^{*},T_{n}]$, we set 
	\begin{equation}\label{equ:defs}
		s(t)=s^{in}-\int_{t}^{T_{n}}\frac{\dd \sigma}{\lambda^{3}(\sigma)}\in (-\infty,s^{in}].
	\end{equation}
	Observe from Proposition~\ref{prop:modu} and~\eqref{equ:deffinaldata} that 
	\begin{equation}\label{equ:moduinit}
		(\varepsilon(s^{in}),\mathcal{G}(s^{in}))=(0,\lambda^{in},z^{in},\mu^{in},x_{1}^{in},b_{1}^{in},b_{2}^{in}).
	\end{equation}
	
	From now on, we fix $T_{n}=n$ and choose $s^{in}$ such that 
	\begin{equation}\label{est:defsin}
		\left|T_{n}-\frac{s^{in}}{\log^{\frac{1}{2}}s^{in}}\right|\le \frac{s^{in}}{\log^{\frac{3}{4}} s^{in}}.
	\end{equation}
	
	The following Proposition is the main part of the proof of Theorem~\ref{thm:main}.
	
	\begin{proposition}
		[Uniform estimates]\label{prop:uni}
		There exists $s_{0}>1$ such that for all $s^{in}>s_{0}$ satisfying~\eqref{est:defsin}, there exist some choice of geometrical parameters $\mathcal{G}^{in}=(\lambda^{in},z^{in},\mu^{in},x_{1}^{in},b_{1}^{in},b_{2}^{in})\in (0,\infty)^{2}\times \RR^{4}$ with 
		\begin{equation}\label{est:finialdatawell}
			\begin{aligned}
				\lambda^{in}=\log^{-\frac{1}{6}}(s^{in}),\quad 
				b_2^{in}=b_1^{in}-\frac{5\alpha}{2}z^{in}e^{-z^{in}},\\
				\mu^{in}=5b_{1}^{in}z^{in},\quad b_1^{in}=\left(\frac{\alpha}{3}\right)^{\frac{1}{2}}\left(z^{in}\right)^{-\frac{1}{2}}e^{-\frac{1}{2}z^{in}},\\
				\left|x_{1}^{in}-\frac{s^{in}}{\log^{\frac{1}{6}} s^{in}}\right|^{2}+ \left|\left(\frac{1}{3\alpha}\right)^{\frac{1}{2}}(z^{in})^{-\frac{1}{2}}e^{\frac{1}{2}z^{in}}-s^{in}\right|^{2}\le\frac{(s^{in})^{2}}{\log s^{in}}.
			\end{aligned}
		\end{equation}
		such that the solution $u$ of~\eqref{equ:gKdV} corresponding to~\eqref{equ:deffinaldata} exists and satisfies~\eqref{equ:orthmodu1}--\eqref{equ:orth} on the rescaled time interval $[s_{0},s^{in}]$. Here, the rescaled time $s$ is defined in~\eqref{equ:defs}. Moreover, the decomposition of $u$ given by Proposition~\ref{prop:modu} on $[s_{0},s^{in}]$ 
		\begin{equation*}
			\varepsilon(s,y)=\lambda^{\frac{1}{2}}(s)u(s,\lambda(s)y+x_{1}(s))-V(y;(z,\mu,b_{1},b_{2})),
		\end{equation*}
		satisfies the following uniform estimates: for all $s\in [s_{0},s^{in}]$, we have
		\begin{equation}\label{est:unisolution}
			\begin{aligned}
				|z(s)-2\log s|\lesssim \log \log s,\quad 
				|b_{1}(s)|\lesssim \frac{1}{s\log s},\\
				\Big|x_{1}(s)-\frac{s}{\log^{\frac{1}{6}} s}\Big|\lesssim \frac{s}{\log^{\frac{1}{2}}s},\quad 
				\|\varepsilon(s)\|_{H^{1}}\lesssim \frac{1}{s},\\
				\Big|\lambda (s)-\frac{1}{\log^{\frac{1}{6}} s}\Big|\lesssim \frac{1}{\log^{\frac{2}{3}}s},\quad 
				|\mu(s)|\lesssim \frac{1}{s},\quad |b_{2}(s)|\lesssim \frac{1}{s\log s}.
			\end{aligned}
		\end{equation}
	\end{proposition}
	
	The key point in Proposition~\ref{prop:uni} is that $s_{0}>1$ and the implied constants in~\eqref{est:unisolution} are independent of $s^{in}$ as $s^{in}\to \infty$. The rest of the section is organized as follows. First, in Section~\ref{SS:Boot}, we introduce the bootstrap assumption which related to the above uniform estimates. Then, in Section \ref{SS:Energy}--\ref{SS:Proofuni}, we devote to the proof of Proposition~\ref{prop:uni} based on a standard energy estimate. More precisely, we control the remainder term $\varepsilon$ via energy estimate in the framework of the two-bubble solutions and estimate the geometrical parameters $\mathcal{G}$ via some ODE arguments.
	
	\subsection{Bootstrap bounds}\label{SS:Boot} 
	For $t\le T_n$ (\emph{i.e.} $s\le s^{in}$), as long as $u(t)$ is well-defined in $H^{1}$ and satisfies~\eqref{equ:orthmodu1}--\eqref{equ:modu2}, we decompose $u(t)$ as in Proposition~\ref{prop:modu}. In particular, we denote by $(\varepsilon(s),\mathcal{G}(s))$ the remainder term and the geometrical parameters of the decomposition for $u(t)$.  Let $\chi_{1}:\RR\to [0,1]$ be decreasing $C^{\infty}$ function such that
	\begin{equation}\label{equ:defchi1}
		\chi_{1}(y)
		=\left\{
		\begin{aligned}
			& 1-e^{y} \ \ \mbox{on} \ (-\infty,-1], \\ 
			& e^{-y}\quad \ \ \mbox{on}\ [2,\infty),
		\end{aligned}
		\right.
		\quad \mbox{and}\quad 
		\chi'_{1}<0\ \ \mbox{on}\ \RR.
	\end{equation}
	For $B>1$ large to be chosen later, we define
	\begin{equation}\label{equ:defphi1}
		\mathcal{N}_{B}(\varepsilon)=\left(\int_{\RR}\left((\partial_{y}\varepsilon)^{2}+\varepsilon^{2}\phi_{1}\right)\dd y\right)^{\frac{1}{2}}\quad \mbox{with}\ \ \phi_{1}(y)=\chi_{1}\left(\frac{y}{B}-s^{\frac{1}{2}}\right).
	\end{equation}

	The proof of Proposition~\ref{prop:uni} is based on the following bootstrap estimates: for $C_{0}>1$ to be chosen,
	\begin{equation}\label{est:Boot}
		\left\{
		\begin{aligned}
			|\mu(s)|\le \frac{5}{s},\quad   \frac{1}{7 s\log s}\le b_{1}(s)\le \frac{1}{5s \log s},\\
             |\Theta(s)|\le \frac{5}{s^{2}},\quad 
            \frac{1}{7 s\log s}\le b_{2}(s)\le \frac{1}{5s \log s},
            \\
            \frac{1}{2\log^{\frac{1}{6}}s}\le \lambda(s)\le \frac{3}{2\log^{\frac{1}{6}}s},\quad \mathcal{N}_{B}(\varepsilon)\le \frac{C_{0}}{s^{\frac{3}{2}}\log s},\\
			\Big|x_{1}(s)-\frac{s}{\log^{\frac{1}{6}} s}\Big|^{2}+  \left|\left(\frac{1}{3\alpha}\right)^{\frac{1}{2}}z^{-\frac{1}{2}}(s)e^{\frac{z(s)}{2}}-s\right|^{2}\le \frac{s^{2}}{ \log s}.
		\end{aligned}
		\right.
	\end{equation}
	Note that, the estimate on $z$ in~\eqref{est:Boot} immediately implies that, for $s$ large, 
	\begin{equation}\label{est:bootz}
		0<e^{-z(s)}\lesssim s^{-2}\log^{-1}s\quad \mbox{and}\quad 
		|z(s)-2\log s|\lesssim \log \log s.
	\end{equation}
	For $s_{0}>1$ to be chosen large enough (independent of $s^{in}$), and all $s^{in}\gg s_{0}$, we set 
	\begin{equation}\label{equ:defss}
		s^{*}=\inf\left\{s\in [s_{0},s^{in}]:u(s)\ \mbox{satisfies}\ \eqref{equ:orthmodu1}-\eqref{equ:modu2} \ \mbox{and}\ \eqref{est:Boot}\ \mbox{on}\ [s,s^{in}]\right\}.
	\end{equation}
	Observe that the choice of final data in~\eqref{equ:deffinaldata} and the continuity of the flow for~\eqref{equ:gKdV} directly imply $s^* < s^{in}<\infty$. Moreover, for any $s^{*}<s<s^{in}<\infty$, the bootstrap assumptions~\eqref{est:Boot}--\eqref{est:bootz} guarantee that~\eqref{equ:orthmodu1}--\eqref{equ:modu2} hold on $[s, s^{in}]$ whenever~\eqref{equ:orth} and~\eqref{est:Boot} hold on this interval. Therefore, to close the bootstrap argument concerning $s^*$, it suffices to justify the estimates in~\eqref{est:Boot} rigorously.
	
	\smallskip
	We further observe that, for $s_{0}>1$ large enough, the geometrical parameters $\mathcal{G}(s)$ satisfying~\eqref{est:Boot} are admissible in the sense of Definition~\ref{def:Admissible}. It follows that Proposition~\ref{prop:approx} remains valid under the bootstrap assumption~\eqref{est:Boot}.
	
	\smallskip
	In the rest of the section, the implied constant in $\lesssim$ and $O$ do not depend on the constant $C_{0}$ appearing in the bootstrap assumption~\eqref{est:Boot}.
	
	\smallskip
	We first establish the $L^{2}$ and $L^{\infty}$ estimates of the remainder term $\varepsilon$.
	\begin{lemma}\label{le:L2e}
		We have
		\begin{equation*}
			\|\varepsilon(s)\|_{L^{2}}\lesssim s^{-1}\quad \mbox{and}\quad \|\varepsilon(s)\|_{L^{\infty}}\lesssim s^{-\frac{5}{4}}.
		\end{equation*}
	\end{lemma}
	\begin{proof}
		First, from the orthogonality condition~\eqref{equ:orth}, we compute 
		\begin{equation*}
			\|u(s)\|^{2}_{L^{2}}=\|\varepsilon(s)\|_{L^{2}}^{2}+\|V(s)\|_{L^{2}}^{2}+2(\varepsilon,R\phi)+2(\varepsilon,U).
		\end{equation*}
		Note that, using the bootstrap assumption~\eqref{est:Boot} and Lemma~\ref{le:boundinter}, 
		\begin{equation*}
			\|V(s)\|^{2}_{L^{2}}=2\|Q\|_{L^{2}}^{2}+2(S,R\phi+U)+\|R\phi+U\|_{L^{2}}^{2}+
			O\left(s^{-2}\right).
		\end{equation*}
		In addition, using again the bootstrap assumption~\eqref{est:Boot}, 
		\begin{equation*}
			(S,R\phi+U)=(S,b_{1}X_{1}-\zeta_{2}X_{2}-10m_{0}b_{1}Y_{2})+O\left(s^{-2}\right).
		\end{equation*}
		From Remark~\ref{re:Y} and Remark~\ref{re:P}, we have 
		\begin{equation*}
			\begin{aligned}
				&(S,b_{1}X_{1}-\zeta_{2}X_{2}-10m_{0}b_{1}Y_{2})\\
				&=b_{1}(Q_{1},X_{1}-X_{2}-10m_{0}Y_{2})+
				\left(1-(1+\mu)^{-\frac{3}{2}}\right)b_{1}(S,X_{2})\\
				&-b_{1}(Q_{2},X_{1}-X_{2}-10m_{0}Y_{2})-\Theta(1+\mu)^{-\frac{3}{2}}(S,X_{2})=O\left(s^{-2}\right).
			\end{aligned}
		\end{equation*}
		Here, we use the fact that 
		\begin{equation*}
			\begin{aligned}
				(Q_{1},X_{1}-X_{2}-10m_{0}Y_{2})&=-m_{0}^{2}+O\left(s^{-1}\right),\\
				(Q_{2},X_{1}-X_{2}-10m_{0}Y_{2})&=-m_{0}^{2}+O\left(s^{-1}\right).
			\end{aligned}
		\end{equation*}
		Then, using Lemma~\ref{le:RAB} and the bootstrap assumption~\eqref{est:Boot},
		\begin{equation*}
			\begin{aligned}
				\|R\phi\|_{L^{2}}^{2}+\|U\|_{L^{2}}^{2}
				&\lesssim\left(b_{1}^{4}+e^{-2z}\right)\|(A_{1}+B_{2})\phi\|_{L^{2}}^{2}
				+b_{1}^{2}\|Y_{2}\|_{L^{2}}^{2}
				\\
				&+\left(b_{1}^{4}+e^{-2z}\right)\|(E_{1}+F_{2})\phi\|_{L^{2}}^{2}
				+z^{2}e^{-2z}\|Y_{2}\|_{L^{2}}^{2}
				\\
				&+b_{1}^{2}\|(X_{1}-X_{2})\phi\|_{L^{2}}^{2}+|\zeta_{2}-b_{1}|^{2}\|X_{2}\phi\|_{L^{2}}^{2}\lesssim s^{-2}.
			\end{aligned}
		\end{equation*}
		Similar as above, using the choice of final data~\eqref{est:finialdatawell}, we also have 
		\begin{equation*}
			\|V^{in}\|_{L^{2}}^{2}=2\|Q\|_{L^{2}}^{2}+O\left((s^{in})^{-2}\right)=2\|Q\|_{L^{2}}^{2}+O\left(s^{-2}\right).
		\end{equation*}
		Combining the above estimates with the conservation law of mass, we complete the proof for the $L^{2}$ estimate. The $L^{\infty}$ estimate then follows directly from the bootstrap assumption~\eqref{est:Boot} and the Fundamental Theorem of Calculus.
	\end{proof}
	
	\smallskip
	We now deduce the equation of the remainder term $\varepsilon$ from~\eqref{equ:gKdv2} and Proposition~\ref{prop:modu}.
	
	\begin{lemma}[Equation of $\varepsilon$]\label{le:eque}
		It holds 
		\begin{equation}\label{equ:pse}
			\begin{aligned}    &\partial_{s}\varepsilon+\partial_{y}\left(\partial_{y}^{2}\varepsilon-\varepsilon+(V+\varepsilon)^{5}-V^{5}\right)+\Psi(V)\\
				&=\frac{\dot{\lambda}}{\lambda}\Lambda \varepsilon+\bigg(\frac{\dot{\lambda}}{\lambda}+b_{1}\bigg)\Lambda V+\left(\frac{\dot{x}_{1}}{\lambda}-1\right)\left(\partial_{y}V+\partial_{y}\varepsilon\right).
			\end{aligned}
		\end{equation}
	\end{lemma}
	\begin{proof}
		First, from~\eqref{equ:gKdv2} and~\eqref{equ:defe}, we find 
		\begin{equation*}
			\begin{aligned}
				& \partial_{s}(V+\varepsilon)+\partial_{y}\left(\partial_{y}^{2}(V+\varepsilon)-(V+\varepsilon)+(V+\varepsilon)^{5}\right)\\
				&=\frac{\dot{\lambda}}{\lambda} (\Lambda V+\Lambda \varepsilon)+\left(\frac{\dot{x}_{1}}{\lambda}-1\right)(\partial_{y}V+\partial_{y}\varepsilon),
			\end{aligned}
		\end{equation*}
		which directly implies 
		\begin{equation*}
			\begin{aligned}
				&\partial_{s}\varepsilon+\partial_{y}\left(\partial_{y}^{2}\varepsilon-\varepsilon+(V+\varepsilon)^{5}-V^{5}\right)+\Psi(V)\\
				&+\partial_{s}V+\partial_{y}\left(\partial_{y}^{2}V-V+V^{5}\right)+b_{1}\Lambda V-\Psi(V)\\
				&= \frac{\dot{\lambda}}{\lambda}\Lambda \varepsilon
				+\bigg(\frac{\dot{\lambda}}{\lambda}+b_{1}\bigg)\Lambda V
				+\left(\frac{\dot{x}_{1}}{\lambda}-1\right)\left(\partial_{y}V+\partial_{y}\varepsilon\right).
			\end{aligned}
		\end{equation*}
		Combining the above identity with~\eqref{equ:defV}, we complete the proof for~\eqref{equ:pse}.
	\end{proof}
	
	\smallskip
	Note that, from~\eqref{est:Boot} and the definition $b_{2}=\zeta_{2}(1+\mu)^{\frac{3}{2}}$, we compute 
	\begin{equation}\label{est:dotzeta}
		\dot{\zeta}_{2}=\frac{\dot{b}_{2}}{(1+\mu)^{\frac{3}{2}}}-\frac{3\dot{\mu}b_{1}}{2(1+\mu)^{\frac{5}{2}}}
		\Longrightarrow
		\dot{\zeta}_{2}=\dot{b}_{2}+O\left(\frac{|\dot{b}_{2}|}{s}+\frac{|\dot{\mu}|}{s\log s}\right).
	\end{equation}
	The above estimate will be used frequently in the proof of the following Lemma.
	
	\smallskip
	Now, we derive the control for $\mathcal{G}(s)$ from the orthogonality conditions~\eqref{equ:orth}.
	\begin{lemma}[Control for parameters]\label{le:controlpara}
		The following estimates hold on $[s^{*},s^{in}]$.
		\begin{enumerate}
			\item \emph{Estimate on $(\lambda,z,\mu,x_{1})$}. We have 
			\begin{equation*}
				\begin{aligned}
					\Big|\dot{z}-b_{1}z-\mu+z\Big(\frac{\dot{\lambda}}{\lambda}+b_{1}\Big)\Big|
					&\lesssim
					\frac{C_{0}}{s^{\frac{3}{2}}\log s},\\
					\bigg|\frac{\dot{\lambda}}{\lambda}+b_{1}\bigg|+\bigg|\frac{\dot{x}_{1}}{\lambda}-1\bigg|+|\dot{\mu}-2\Theta|&\lesssim
					\frac{C_{0}}{s^{\frac{3}{2}}\log s}.
				\end{aligned}
			\end{equation*}
			\item \emph{Estimate on $(b_{1},\zeta_{2})$.} We have 
			\begin{equation*}
				\begin{aligned}
					|\dot{b}_{1}+\alpha e^{-z}+2b_{1}^{2}|&\lesssim \frac{C_{0}}{s^{\frac{5}{2}}\log^{2}s}
					+\frac{C_{0}^{2}}{s^{3}\log^{2}s},\\
					|\dot{\zeta}_{2}+\alpha e^{-z}+2b_{1}^{2}|&\lesssim \frac{C_{0}}{s^{\frac{5}{2}}\log^{2}s}
					+\frac{C_{0}^{2}}{s^{3}\log^{2}s}.
				\end{aligned}
			\end{equation*}
		\end{enumerate}
	\end{lemma}
	
	\begin{proof}
		\textbf{Step 1.} Estimate on $(\lambda,x_{1})$. 
		From Lemma~\ref{le:eque} and the orthogonality condition $(\varepsilon(s),\Lambda Q)=0$, we check that 
		\begin{equation*}
			\begin{aligned}
				&\left(\partial_{y}\left(\partial_{y}^{2}\varepsilon-\varepsilon+(V+\varepsilon)^{5}-V^{5}\right),\Lambda Q\right)+\left(\Psi(V),\Lambda Q\right)\\
				&=\frac{\dot{\lambda}}{\lambda}\left(\Lambda \varepsilon,\Lambda Q\right)+
				\bigg(\frac{\dot{\lambda}}{\lambda}+b_{1}\bigg)\left(\Lambda V,\Lambda Q\right)+\left(\frac{\dot{x}_{1}}{\lambda}-1\right)\left(\left(\partial_{y}V+\partial_{y}\varepsilon\right),\Lambda Q\right).
			\end{aligned}
		\end{equation*}
		First, from the bootstrap assumption~\eqref{est:Boot} and Lemma~\ref{le:L2e}, we find 
		\begin{equation*}
			\begin{aligned}
				& \left| \left(\partial_{y}\left(\partial_{y}^{2}\varepsilon-\varepsilon+(V+\varepsilon)^{5}-V^{5}\right),\Lambda Q\right)\right|\\
		&\lesssim\|V\|_{L^{\infty}}^{2}\|\varepsilon\|_{L^{\infty}}^{3}\|(\Lambda Q)'\|_{L^{1}}+\|V\|_{L^{\infty}}\|\varepsilon\|_{L^{\infty}}^{4}
        \|(\Lambda Q)'\|_{L^{1}}+\|\varepsilon\|_{L^{\infty}}^{5}\|(\Lambda Q)'\|_{L^{1}}\\
				&+
				(1+\|V\|_{L^{\infty}}^{4})
				\left|\left(\varepsilon,\partial^{\le 3}_{y}\Lambda Q\right)\right|+\|V\|_{L^{\infty}}^{3}\|\varepsilon\|_{L^{\infty}}^{2}\|(\Lambda Q)'\|_{L^{1}}\lesssim \frac{C_{0}}{s^{\frac{3}{2}}\log s}.
			\end{aligned}
		\end{equation*}
		
		Then, using $(\Lambda Q,Q')=0$ the definition of $V$ in~\eqref{equ:defV}, we check that 
		\begin{equation*}
			\begin{aligned}
				(\partial_{y}V,\Lambda Q)&=O\left(|b_{1}|+|\Theta|+zb_{1}^{2}+z^{5}e^{-z}\right),\\
				(\Lambda V,\Lambda Q)&=\|\Lambda Q\|^{2}_{L^{2}}+O\left(|b_{1}|+|\Theta|+zb_{1}^{2}+z^{5}e^{-z}\right).
			\end{aligned}
		\end{equation*}
		On the other hand, from Lemma~\ref{le:Psi2} and bootstrap assumption~\eqref{est:Boot}, 
		\begin{equation*}
			\begin{aligned}
				\left(\Psi(V),\Lambda Q\right)
				&=\left(\vec{{\rm{Mod}}}\cdot \vec{M}Q+\sum_{i=1}^{4}\Psi_{i}(V)+O_{\mathcal{S}}(1),\Lambda Q\right)\\
				&=O \left(\frac{|\vec{{\rm{Mod}}}|}{s \log s}
				+|\dot{b}_{1}+\alpha e^{-z}+2b_{1}^{2}|
				+ \frac{1}{s^{3}}\right).
			\end{aligned}
		\end{equation*}
		Here, we use the fact that 
		\begin{equation*}
			f\in \mathcal{S}\Longrightarrow \|f\|_{H^{1}}\lesssim \frac{|\vec{{\rm{Mod}}}|}{s\log s}+\frac{1}{s^{3}\log s}.
		\end{equation*}
		Next, we rewrite 
		\begin{equation*}
			\begin{aligned}
				\frac{\dot{\lambda}}{\lambda}\left(\Lambda \varepsilon,\Lambda Q\right)
				&=
				\bigg( \frac{\dot{\lambda}}{\lambda}+b_{1}\bigg)\left(\Lambda \varepsilon,\Lambda Q\right)
				-b_{1}(\Lambda \varepsilon,\Lambda Q)\\
				&=O\bigg(\frac{C_{0}}{s^{\frac{3}{2}}\log s}
				\bigg|\frac{\dot{\lambda}}{\lambda}+b_{1}\bigg|
				+\frac{C_{0}}{s^{\frac{5}{2}}\log^{2}s}\bigg).
			\end{aligned}
		\end{equation*}
		
		It follows from the bootstrap assumption \eqref{est:Boot} that
		\begin{equation*}
			\bigg|\frac{\dot{\lambda}}{\lambda}+b_{1}\bigg|\lesssim 
			\frac{|\vec{{\rm{Mod}}}|}{s \log s}
			+ \frac{1}{s \log s}\left|\frac{\dot{x}_{1}}{\lambda}-1\right|
			+|\dot{b}_{1}+\alpha e^{-z}+2b_{1}^{2}|
			+\frac{C_{0}}{s^{\frac{3}{2}}\log s}.
		\end{equation*}
		Similarly, using the orthogonality condition $(\varepsilon(s),\partial_{y}Q_{1})=0$, we find 
		\begin{equation*}
			\bigg|\frac{\dot{x}_{1}}{\lambda}-1\bigg|\lesssim 
			\frac{|\vec{{\rm{Mod}}}|}{s \log s}
			+\frac{1}{s\log s}\bigg|\frac{\dot{\lambda}}{\lambda}+b_{1}\bigg|
			+|\dot{b}_{1}+\alpha e^{-z}+2b_{1}^{2}|
			+\frac{C_{0}}{s\log^{\frac{3}{2}}s}.
		\end{equation*}
		Combining the above estimates, we deduce that 
		\begin{equation}\label{est:lambdax1s}
			\bigg|\frac{\dot{\lambda}}{\lambda}+b_{1}\bigg|
			+ \left|\frac{\dot{x}_{1}}{\lambda}-1\right|
			\lesssim 
			\frac{|\vec{{\rm{Mod}}}|}{s \log s}
			+|\dot{b}_{1}+\alpha e^{-z}+2b_{1}^{2}|
			+\frac{C_{0}}{s^{\frac{3}{2}}\log s}.
		\end{equation}
		
		\textbf{Step 2.} Estimate on $\mu$. From Lemma~\ref{le:eque} and the orthogonality condition $(\varepsilon(s),\Gamma \Lambda Q)=0$, we check that
		\begin{equation*}
			\begin{aligned}
				&\left(\partial_{y}\left(\partial_{y}^{2}\varepsilon-\varepsilon+(V+\varepsilon)^{5}-V^{5}\right),\Gamma \Lambda Q\right)+\left(\Psi(V),
				\Gamma\Lambda Q\right)\\
				&=
				\bigg(\frac{\dot{\lambda}}{\lambda}+b_{1}\bigg)\left(\Lambda V,
				\Gamma\Lambda Q\right)+\left(\frac{\dot{x}_{1}}{\lambda}-1\right)\left(\left(\partial_{y}V+\partial_{y}\varepsilon\right),\Gamma\Lambda Q\right)\\
				&+\frac{\dot{\lambda}}{\lambda}\left(\Lambda \varepsilon,\Gamma \Lambda Q\right)-\dot{z}\left(\varepsilon,\partial_{y}\Gamma\Lambda Q\right)
				+\frac{\dot{\mu}}{2(1+\mu)}(\varepsilon,\Gamma (\Lambda \Lambda Q)).
			\end{aligned}
		\end{equation*}
		First, from the bootstrap assumption~\eqref{est:Boot} and Lemma~\ref{le:L2e},
		\begin{equation*}
			\left| \left(\partial_{y}\left(\partial_{y}^{2}\varepsilon-\varepsilon+(V+\varepsilon)^{5}-V^{5}\right),\Gamma \Lambda Q\right)\right|\lesssim \frac{C_{0}}{s^{\frac{3}{2}}\log s}.
		\end{equation*}
		Then, using again $(\Lambda Q,Q')=0$ and 
		the definition of $V$ in~\eqref{equ:defV}, 
		\begin{equation*}
			\begin{aligned}
				(\partial_{y}V,\Gamma \Lambda Q)&=O\left(|b_{1}|+|\Theta|+zb_{1}^{2}+z^{5}e^{-z}\right),\\
				(\Lambda V,\Gamma \Lambda Q)&=-\|\Lambda Q\|^{2}_{L^{2}}+O\left(z|b_{1}|+z|\Theta|+zb_{1}^{2}+z^{5}e^{-z}\right).
			\end{aligned}
		\end{equation*}
		On the other hand, from Lemma~\ref{le:Psi2} and bootstrap assumption~\eqref{est:Boot}, 
		\begin{equation*}
			\begin{aligned}
				\left(\Psi(V),\Gamma\Lambda Q\right)
				&=\left(\vec{{\rm{Mod}}}\cdot \vec{M}Q+\sum_{i=1}^{4}\Psi_{i}(V)+O_{\mathcal{S}}(1),\Gamma\Lambda Q\right)\\
				&=-\frac{1}{2}(\dot{\mu}-2\Theta)\left(\|\Lambda Q\|_{L^{2}}^{2}+O(s^{-1})\right)+O\left(\frac{|\vec{\rm{Mod}}|}{s\log s}+\frac{1}{s^{3}}\right)\\
				&+O\left(|\dot{b}_{1}+\alpha e^{-z}+2b_{1}^{2}|+|\dot{\zeta}_{2}+\alpha e^{-z}+2b_{1}^{2}|\right).
			\end{aligned}
		\end{equation*}
		Next, using again the bootstrap assumption~\eqref{est:Boot}, we find 
		\begin{equation*}
			\begin{aligned}
				& \bigg| \frac{\dot{\lambda}}{\lambda}\left(\Lambda \varepsilon,\Gamma \Lambda Q\right)\bigg|+
				\left|\dot{z}\left(\varepsilon,\partial_{y}\Gamma\Lambda Q\right)\right|
				+\left|\frac{\dot{\mu}}{1+\mu}(\varepsilon,\Gamma (\Lambda \Lambda Q))\right|\\
				&\lesssim \frac{C_{0}}{s^{\frac{3}{2}}}\bigg|\frac{\dot{\lambda}}{\lambda}+b_{1}\bigg|
				+\frac{|\vec{\rm{Mod}}|}{s\log s}
				+\frac{C_{0}}{s^{\frac{5}{2}}\log s}
				+\frac{C_{0}}{s^{\frac{5}{2}}\log^{2}s}.
			\end{aligned}
		\end{equation*}
		Combining the above estimates with~\eqref{est:lambdax1s}, we deduce that 
		\begin{equation}\label{est:dotmu}
			\begin{aligned}
				|\dot{\mu}-2\Theta|&\lesssim
				\frac{|\vec{{\rm{Mod}}}|}{s \log s}
				+|\dot{b}_{1}+\alpha e^{-z}+2b_{1}^{2}|\\
				&+\frac{C_{0}}{s^{\frac{3}{2}}\log s}
				+|\dot{\zeta}_{2}+\alpha e^{-z}+2b_{1}^{2}|.
			\end{aligned}
		\end{equation}
		
		\textbf{Step 3.} Estimate on $z$. From Lemma~\ref{le:eque} and the orthogonality condition $(\varepsilon(s),\partial_{y}Q_{2})=0$, we check that
		\begin{equation*}
			\begin{aligned}
				&\left(\partial_{y}\left(\partial_{y}^{2}\varepsilon-\varepsilon+(V+\varepsilon)^{5}-V^{5}\right),\partial_{y}Q_{2}\right)+\left(\Psi(V),
				\partial_{y}Q_{2}\right)\\
				&=
				\bigg(\frac{\dot{\lambda}}{\lambda}+b_{1}\bigg)\left(\Lambda V,
				\partial_{y}Q_{2}\right)+\left(\frac{\dot{x}_{1}}{\lambda}-1\right)\left(\left(\partial_{y}V+\partial_{y}\varepsilon\right),\partial_{y}Q_{2}\right)\\
				&+\frac{\dot{\lambda}}{\lambda}\left(\Lambda \varepsilon,\partial_{y}Q_{2}\right)-\dot{z}\left(\varepsilon,\partial^{2}_{y}Q_{2}\right)
				+\frac{\dot{\mu}}{2(1+\mu)^{\frac{1}{2}}}(\varepsilon,\Gamma (\Lambda Q')).
			\end{aligned}
		\end{equation*}
		First, from the bootstrap assumption~\eqref{est:Boot} and Lemma~\ref{le:L2e}, we find 
		\begin{equation*}
			\left| \left(\partial_{y}\left(\partial_{y}^{2}\varepsilon-\varepsilon+(V+\varepsilon)^{5}-V^{5}\right),\partial_{y}Q_{2}\right)\right|\lesssim \frac{C_{0}}{s^{\frac{3}{2}}\log s}.
		\end{equation*}
		Then, using again $(\Lambda Q,Q')=0$ and 
		the definition of $V$ in~\eqref{equ:defV}, 
		\begin{equation*}
			\begin{aligned}
				(\partial_{y}V,\partial_{y}Q_{2})&=-\|Q'\|^{2}_{L^{2}}+O\left(|b_{1}|+|\Theta|+|\mu|+zb_{1}^{2}+z^{5}e^{-z}\right),\\
				(\Lambda V,\partial_{y}Q_{2})&=-z\| Q'\|^{2}_{L^{2}}+O\left(z|b_{1}|+z|\Theta|+|\mu|z+zb_{1}^{2}+z^{5}e^{-z}\right).
			\end{aligned}
		\end{equation*}
		On the other hand, from Lemma~\ref{le:Psi2} and bootstrap assumption~\eqref{est:Boot}, 
		\begin{equation*}
			\begin{aligned}
				\left(\Psi(V),\partial_{y}Q_{2}\right)
				&=\left(\vec{{\rm{Mod}}}\cdot \vec{M}Q+\sum_{i=1}^{4}\Psi_{i}(V)+O_{\mathcal{S}}(1),\partial_{y}Q_{2}\right)\\
				&=(\dot{z}-b_{1}z-\mu)\|Q'\|_{L^{2}}^{2}
				+O\left(\frac{|\vec{\rm{Mod}}|}{s\log s}+\frac{1}{s^{3}}\right)\\
				&+O\left(|\dot{b}_{1}+\alpha e^{-z}+2b_{1}^{2}|+|\dot{\zeta}_{2}+\alpha e^{-z}+2b_{1}^{2}|\right).
			\end{aligned}
		\end{equation*}
		Next, using again the bootstrap assumption~\eqref{est:Boot}, we find 
		\begin{equation*}
			\begin{aligned}
				& \bigg| \frac{\dot{\lambda}}{\lambda}\left(\Lambda \varepsilon,\partial_{y}Q_{2}\right)\bigg|+
				\left|\dot{z}\left(\varepsilon,\partial^{2}_{y} Q_{2}\right)\right|
				+\left|\frac{\dot{\mu}}{(1+\mu)^{\frac{1}{2}}}(\varepsilon,\Gamma (\Lambda Q'))\right|\\
				&\lesssim \frac{C_{0}}{s^{\frac{3}{2}}}\bigg|\frac{\dot{\lambda}}{\lambda}+b_{1}\bigg|
				+\frac{|\vec{\rm{Mod}}|}{s\log s}
				+\frac{C_{0}}{s^{\frac{5}{2}}\log s}
				+\frac{C_{0}}{s^{\frac{5}{2}}\log^{2}s}.
			\end{aligned}
		\end{equation*}
		Combining the above estimates with~\eqref{est:lambdax1s}, we deduce that 
		\begin{equation}\label{est:dotzbz}
			\begin{aligned}
				\bigg|\dot{z}-b_{1}z-\mu+z\bigg(\frac{\dot{\lambda}}{\lambda}+b_{1}\bigg)\bigg|
				&\lesssim
				\frac{|\vec{{\rm{Mod}}}|}{s \log s}
				+|\dot{b}_{1}+\alpha e^{-z}+2b_{1}^{2}|\\
				&+\frac{C_{0}}{s^{\frac{3}{2}}\log s}
				+|\dot{\zeta}_{2}+\alpha e^{-z}+2b_{1}^{2}|.
			\end{aligned}
		\end{equation}
		
		\textbf{Step 4.} Estimates on $(b_{1},\zeta_{2})$. From Lemma~\ref{le:eque} and the orthogonality condition $(\varepsilon(s), Q)=0$, we check that 
		\begin{equation*}
			\begin{aligned}
				&\left(\partial_{y}\left(\partial_{y}^{2}\varepsilon-\varepsilon+(V+\varepsilon)^{5}-V^{5}\right),Q\right)+\left(\Psi(V), Q\right)\\
				&=\frac{\dot{\lambda}}{\lambda}\left(\Lambda \varepsilon,Q\right)+
				\bigg(\frac{\dot{\lambda}}{\lambda}+b_{1}\bigg)\left(\Lambda V, Q\right)+\left(\frac{\dot{x}_{1}}{\lambda}-1\right)\left(\left(\partial_{y}V+\partial_{y}\varepsilon\right),Q\right).
			\end{aligned}
		\end{equation*}
		First, we rewrite 
		\begin{equation*}
			\begin{aligned}
				\partial_{y}^{2}\varepsilon-\varepsilon+(V+\varepsilon)^{5}-V^{5}
				&=-\mathcal{L}\varepsilon+5(V^{4}-Q^{4})\varepsilon\\
				&+\left((V+\varepsilon)^{5}-V^{5}-5V^{4}\varepsilon\right).
			\end{aligned}
		\end{equation*}
		It follows from~\eqref{est:Boot}, Lemma~\ref{le:L2e} and ${\rm{Ker}}\mathcal{L}={\rm{Span}}\{Q'\}$ that 
		\begin{equation*}
			\begin{aligned}
				&\left|\left(\partial_{y}\left(\partial_{y}^{2}\varepsilon-\varepsilon+(V+\varepsilon)^{5}-V^{5}\right),Q\right)\right|\\
				&\lesssim\left(\|V\|_{L^{\infty}}^{3}+\|\varepsilon\|_{L^{\infty}}^{3}\right)(\varepsilon^{2},|Q'|)+|((S^{4}-Q^{4})\varepsilon,Q')|\\
				&+\left(\|R\phi\|_{L^{\infty}}+\|U\|_{L^{\infty}}\right)|\|\varepsilon Q'\|_{L^{1}}\lesssim \frac{C_{0}}{s^{\frac{5}{2}}\log^{2}s}+\frac{C_{0}^{2}}{s^{3}\log^{2}s}.
			\end{aligned}
		\end{equation*}
		Here, we use the fact that 
		\begin{equation*}
			\|\varepsilon Q'\|_{L^{1}}
			+\|\varepsilon Q_{2}\|_{L^{1}}
			+|(\varepsilon^{2},Q')|^{\frac{1}{2}}\lesssim \mathcal{N}_{B}(\varepsilon).
		\end{equation*}
		
		Then, using again $(Q',Q)=(\Lambda Q,Q)=0$ and 
		the definition of $V$ in~\eqref{equ:defV},
		\begin{equation*}
			\begin{aligned}
				(\partial_{y}V, Q)&=O\left(|b_{1}|+|\Theta|+zb_{1}^{2}+z^{5}e^{-z}\right),\\
				(\Lambda V,Q)&=O\left(|b_{1}|+|\Theta|+zb_{1}^{2}+z^{5}e^{-z}\right).
			\end{aligned}
		\end{equation*}
		On the other hand, from Lemma~\ref{le:Psi1} and bootstrap assumption~\eqref{est:Boot}, 
		\begin{equation*}
			\begin{aligned}
				\left(\Psi(V),Q\right)
				&=\left(\vec{{\rm{Mod}}}\cdot \vec{M}Q+\sum_{i=1}^{4}\Psi_{i}(V)+O_{\mathcal{S}}(1),Q\right)\\
				&=\left(-m_{0}^{2}+O(s^{-1})\right)\big(\dot{b}_{1}+\alpha e^{-z}+2b_{1}^{2}\big)+O\left(s^{-3}\right)\\
				&+O\left(\frac{1}{s}|\dot{\zeta}_{2}+\alpha e^{-z}+2b_{1}^{2}|
				+\frac{|\dot{z}-b_{1}z-\mu|}{s^{\frac{3}{2}} \log s}+\frac{|\dot{\mu}-2\Theta|}{s\log s}\right).
			\end{aligned}
		\end{equation*}
		Next, using integration by parts,
		\begin{equation*}
			(\varepsilon,\Lambda Q)=(\varepsilon,Q')=0\Longrightarrow
			\left(\Lambda \varepsilon,Q\right)=(\partial_{y}\varepsilon,Q)=0.
		\end{equation*}
		Combining the above estimates with~\eqref{est:lambdax1s}, we have 
		\begin{equation*}
			\begin{aligned}
				|\dot{b}_{1}+\alpha e^{-z}+2b_{1}^{2}|
				&\lesssim
			\frac{1}{s}|\dot{\zeta}_{2}+\alpha e^{-z}+2b_{1}^{2}|
				+\frac{|\dot{z}-b_{1}z-\mu|}{s^{\frac{3}{2}} \log s}\\
                &+\frac{|\dot{\mu}-2\Theta|}{s\log s}
                + \frac{C_{0}}{s^{\frac{5}{2}}\log^{2}s}+\frac{C_{0}^{2}}{s^{3}\log^{2}s}+\frac{1}{s^{3}}.
			\end{aligned}
		\end{equation*}
		Based on a similar argument as above, we also obtain
		\begin{equation*}
			\begin{aligned}
				|\dot{\zeta}_{2}+\alpha e^{-z}+2b_{1}^{2}|
				&\lesssim
			\frac{1}{s}|\dot{b}_{1}+\alpha e^{-z}+2b_{1}^{2}|
				+\frac{|\dot{z}-b_{1}z-\mu|}{s^{\frac{3}{2}} \log s}\\
                &+\frac{|\dot{\mu}-2\Theta|}{s\log s}
                + \frac{C_{0}}{s^{\frac{5}{2}}\log^{2}s}+\frac{C_{0}^{2}}{s^{3}\log^{2}s}+\frac{1}{s^{3}}.
			\end{aligned}
		\end{equation*}
		Combining the above two estimates, we deduce that 
		\begin{equation}\label{est:dotb12}
			\begin{aligned}
				|\dot{b}_{1}+\alpha e^{-z}+2b_{1}^{2}|
				&\lesssim
				\frac{|\dot{z}-b_{1}z-\mu|}{s^{\frac{3}{2}} \log s}+\frac{|\dot{\mu}-2\Theta|}{s\log s}
                + \frac{C_{0}}{s^{\frac{5}{2}}\log^{2}s}+\frac{C_{0}^{2}}{s^{3}\log^{2}s}+\frac{1}{s^{3}},
                \\
				|\dot{\zeta}_{2}+\alpha e^{-z}+2b_{1}^{2}|
				&\lesssim
				\frac{|\dot{z}-b_{1}z-\mu|}{s^{\frac{3}{2}} \log s}+\frac{|\dot{\mu}-2\Theta|}{s\log s}
                + \frac{C_{0}}{s^{\frac{5}{2}}\log^{2}s}+\frac{C_{0}^{2}}{s^{3}\log^{2}s}+\frac{1}{s^{3}}.
			\end{aligned}
		\end{equation}
		
		\smallskip
		\textbf{Step 5.} Conclusion. First, from~\eqref{est:lambdax1s}--\eqref{est:dotb12}, we have
		\begin{equation}\label{est:modu1}
			\begin{aligned}
				&\bigg|\dot{z}-b_{1}z-\mu+z\bigg(\frac{\dot{\lambda}}{\lambda}+b_{1}\bigg)\bigg|+ |\dot{\mu}-2\Theta|\\
				&\lesssim \frac{|\vec{{\rm{Mod}}}|}{s \log s}
				+\frac{C_{0}}{s^{\frac{3}{2}}\log s}+\frac{1}{s^{3}}+\frac{C_{0}^{2}}{s^{3}\log^{2}s},
			\end{aligned}
		\end{equation}
		which directly implies that 
		\begin{equation}\label{est:modu2}
			\begin{aligned}
				&\bigg|\dot{z}-b_{1}z-\mu\bigg|+ |\dot{\mu}-2\Theta|\lesssim \frac{|\vec{{\rm{Mod}}}|}{s}
				+\frac{C_{0}}{s^{\frac{3}{2}}}+\frac{\log s}{s^{3}}+\frac{C_{0}^{2}}{s^{3}\log s}.
			\end{aligned}
		\end{equation}
		Combining the above estimate with~\eqref{est:dotb12}, we deduce that 
		\begin{equation}\label{est:MOD1}
			|\vec{{\rm{Mod}}}|\lesssim \frac{C_{0}}{s^{\frac{3}{2}}}+\frac{\log s}{s^{3}}+\frac{C_{0}^{2}}{s^{3}\log s}\lesssim \frac{C_{0}}{s^{\frac{3}{2}}}.
		\end{equation}
		We see that (i) and (ii) follows directly from~\eqref{est:lambdax1s}--\eqref{est:MOD1} and $s_{0}$ large enough.
	\end{proof}
	
	Based on Lemma~\ref{le:Psi1} and Lemma~\ref{le:controlpara}, we directly obtain the following estimates.
	
	\begin{lemma}\label{le:Psinorm}
		The following estimate hold.
		\begin{enumerate}
			\item \emph{Estimate for $\Psi_{1}(V)$.} We have 
			\begin{equation*}
				\Psi_{1}(V)=\left(\dot{z}-b_{1}z-\mu\right)(\zeta_{2}\partial_{y}X_{2}+10m_{0}b_{1}\partial_{y}Y_{2})+O_{H^{1}}\left(\frac{1}{s^{\frac{5}{2}}\log s}\right).
			\end{equation*}
			
			\item \emph{Standard estimate for $\Psi_{2}(V)-\Psi_{4}(V)$.} We have 
			\begin{equation*}
            \begin{aligned}
            \|\Psi_{2}(V)\|_{{H}^{1}} 
+\|\Psi_{4}(V)\|_{{H}^{1}}
				\lesssim \frac{1}{s^{2}\log s},\\
				\|\Psi_{2}(V)\|_{\dot{H}^{1}} +
				\|\Psi_{3}(V)\|_{H^{1}}
				+\|\Psi_{4}(V)\|_{\dot{H}^{1}}
				\lesssim \frac{1}{s^{\frac{5}{2}}\log s}.
                \end{aligned}
			\end{equation*}
			
			\item \emph{Weighted $L^{2}$ estimate for $\Psi_{2}(V)$ and $\Psi_{4}(V)$.} We have 
			\begin{equation*}
				\left(\int_{\RR}|\Psi_{2}(V)|^{2}\phi_{1}\dd y\right)^{\frac{1}{2}} +\left(\int_{\RR}|\Psi_{4}(V)|^{2}\phi_{1}\dd y\right)^{\frac{1}{2}}\lesssim \frac{1}{s^{\frac{5}{2}}\log s}.
			\end{equation*}
		\end{enumerate}
	\end{lemma}
	
	\begin{proof}
		Proof of (i) and (ii). These estimates follow directly from 
        Lemma~\ref{le:RAB}, Lemma~\ref{le:Psi1}, Lemma~\ref{le:controlpara}, and the expansion of $\Psi_1(V)$ in Proposition~\ref{prop:approx}.
		
		\smallskip
		Proof of (iii). First, from Lemma~\ref{le:Psi1} and Lemma~\ref{le:controlpara}, we have 
		\begin{equation*}
			\Psi_{2}(V)=-\dot{z}e^{-z}\left(A_{1}+B_{2}\right)\phi
			+\frac{\dot{\mu}}{2}b_{1}^{2}\phi\Gamma(\Lambda F) 
			+O_{L^{2}}\left(\frac{1}{s^{\frac{5}{2}}\log s}\right).
		\end{equation*}
		Then, we rewrite again 
		\begin{equation*}
			\begin{aligned}
				A_{1}+B_{2}&=(A_{1}-2\alpha m_{0}y)
				+2\alpha m_{0}y\big(1-(1+\mu)^{\frac{3}{4}}\big)\\
				&+(B_{2}+2\alpha m_{0}(1+\mu)^{\frac{3}{4}}(y-z))+2\alpha m_{0}(1+\mu)^{\frac{3}{4}}z,
			\end{aligned}
		\end{equation*}
		which implies that 
		\begin{equation*}
			|(A_{1}+B_{2})\phi|\lesssim z\textbf{1}_{\left(-\infty,\frac{5}{2}s\right]}(y)\lesssim (\log s) \textbf{1}_{\left(-\infty,\frac{5}{2}s\right]}(y).
		\end{equation*}
		Combining the above estimate with~\eqref{equ:defchi1}--\eqref{equ:defphi1} and Lemma~\ref{le:controlpara}, we obtain
		\begin{equation*}
			\begin{aligned}
				\int_{\RR}|\dot{z}e^{-z}\left(A_{1}+B_{2}\right)\phi|^{2}\phi_{1}\dd y&\lesssim s^{-6}\log^{-2}s\int_{-\infty}^{B(s^{\frac{1}{2}}+2)}
				|A_{1}+B_{2}|^{2}\phi^{2}
				\dd y\\
				&+s^{-6}\int_{B(s^{\frac{1}{2}}+2)}^{\frac{5s}{2}}\phi_{1}\dd y\lesssim s^{-\frac{11}{2}}.
			\end{aligned}
		\end{equation*}
		Similar as above, we also obtain
		\begin{equation*}
			\begin{aligned}
				\int_{\RR}|\dot{\mu}b_{1}^{2}\phi\Gamma(\Lambda F)|^{2}\phi_{1}\dd y&\lesssim \frac{C_{0}^{2}}{s^{7}\log^{6}s}\int_{\RR}|\phi\Gamma(\Lambda F)|^{2}\phi_{1}\dd y\\
				&\lesssim\frac{C_{0}^{2}}{s^{7}\log^{6}s}\int^{\frac{5s}{2}}_{B(s^{\frac{1}{2}}+2)}\left(\left|\frac{y}{B}-s^\frac{1}{2}\right|^{2}+s\right)\phi_{1}
				\dd y\\
				&+\frac{C_{0}^{2}}{s^{7}\log^{6}s}\int_{-\infty}^{B(s^{\frac{1}{2}}+2)}|\Gamma(\Lambda F)|^{2}\dd y\lesssim \frac{C_{0}^{2}}{s^{\frac{11}{2}}\log^{6}s}.
			\end{aligned}
		\end{equation*}
		In addition, from (iv) of Lemma~\ref{le:Psi1} and Lemma~\ref{le:controlpara}, we obtain 
		\begin{equation*}
			\int_{\RR}|\Psi_{4}(V)|^{2}\phi_{1}\dd y\lesssim s^{-6}\int_{\frac{s}{2}}^{\frac{5s}{2}}\phi_{1}\dd y\lesssim s^{-6}.
		\end{equation*}
		We see that (ii) follows directly from the above estimates.
	\end{proof}
	
	Last, we derive the refined control for the evolution of parameters $(\mu,\Theta)$, which is inspired by the previous work of~\cite[(iv) of Lemma 2.7]{MMRACTA} on the study of near soliton dynamics of the mass-critical gKdV equation. We denote 
	\begin{equation}\label{equ:defrho}
		\mathcal{J}_{1}(s)=\frac{1}{m_{0}^{2}}\left(\varepsilon(s),\rho_{1}(s)\right)\ \  \mbox{with}\ \ \rho_{1}(s,y)=\int_{y}^{\infty}\left(\Lambda Q(\sigma)-(\Lambda Q)(\sigma-z)\right)\dd \sigma.
	\end{equation}
	Thanks to the exponential decay of $\rho_{1}$ on both sides of real axis $\mathbb{R}$, the quantity $\mathcal{J}_{1}(s)$ is well-defined for any $\varepsilon(s) \in H^1$ over the time interval $[s^{*}, s^{in}]$.
	More precisely, from the exponential decay of $Q$, for $y>z$, we have 
	\begin{equation}\label{est:rho1}
		|\rho_{1}(s,y)|\lesssim \int_{y}^{\infty}|\Lambda Q(\sigma)|\dd \sigma+\int_{y-z}^{\infty}|\Lambda Q(\sigma)|\dd \sigma\lesssim
		e^{-\frac{1}{2}|y-z|}.
	\end{equation}
	On the other hand, using again the exponential decay of $Q$, for $y<0$, we have 
	\begin{equation}\label{est:rho2}
		|\rho_{1}(s,y)|\lesssim \int_{-\infty}^{y}|\Lambda Q(\sigma)|\dd \sigma+\int_{-\infty}^{y-z}|\Lambda Q(\sigma)|\dd \sigma\lesssim e^{-\frac{1}{2}|y|}.
	\end{equation}
	Combining the above two estimates with~\eqref{equ:defchi1}--\eqref{equ:defphi1}, we obtain
	\begin{equation}\label{est:rho3}
		|\rho_{1}(s,y)|\lesssim e^{-\frac{1}{2}|y|}\textbf{1}_{(-\infty,0)}+\textbf{1}_{[0,z]}+e^{-\frac{1}{2}|y-z|}\textbf{1}_{(z,\infty)}
		\lesssim
		\phi_{1}(s,y).
	\end{equation}
	\begin{lemma}
		[Refined control of $\mu$]
		\label{le:refinmu}
		On $[s^{*},s^{in}]$, we have 
		\begin{equation*}
			\big|\dot{\mu}-2\Theta+\dot{\mathcal{J}}_{1}\big|\lesssim 
			\frac{1}{s^{\frac{5}{2}}\log^{\frac{1}{2}} s}
			+\frac{C_{0}}{s^{\frac{5}{2}}\log s}
			+\frac{C_{0}^{2}}{s^{3}\log s}.
		\end{equation*}
	\end{lemma}
	
	To complete the proof of above Lemma, we need the following technical estimates.
	\begin{lemma}\label{le:estrho}
		The following estimates hold on $[s^{*},s^{in}]$.
		\begin{enumerate}
			\item \emph{Estimates related to $\varepsilon$.} We have
			\begin{equation*}
				\left|   (\partial_{y}\varepsilon,\rho_{1}) \right|
				+
				\left|(\Lambda \varepsilon,\rho_{1})\right|
				\lesssim \frac{C_{0}}{s^{\frac{3}{2}}}.
			\end{equation*}
			
			\item \emph{Estimates related to $Q$.} We have 
			\begin{equation*}
				\left(\Lambda Q,\rho_{1}\right)=(\Gamma (\Lambda Q),\rho_{1})=-2m_{0}^{2}+O\left(s^{-1}\right).
			\end{equation*}
			
			\item \emph{First estimate related to $V$.}
			We have 
			\begin{equation*}
				\left|(\partial_{y}V,\rho_{1})\right|\lesssim s^{-1}.
			\end{equation*}
			
			\item \emph{Second estimate related to $V$.}
			We have 
			\begin{equation*}
				(\Lambda V,\rho_{1})=z(Q_{2},\partial_{y}\rho_{1})+O\left(s^{-1}\right).
			\end{equation*}
			
			\item \emph{Estimates related to $\Psi(V)$.} We have 
			\begin{equation*}
				\begin{aligned}
					\left(\Psi(V),\rho_{1}\right)&=m_{0}^{2}(\dot{\mu}-2\Theta)-(\dot{z}-b_{1}z-\mu)\left(Q_{2},\partial_{y}\rho_{1}\right)\\
					&+O\left(\frac{1}{s^{\frac{5}{2}}\log^{\frac{1}{2}} s}
					+\frac{C_{0}}{s^{\frac{5}{2}}\log s}
					+\frac{C_{0}^{2}}{s^{3}\log s}
					\right).
				\end{aligned}
			\end{equation*}
		\end{enumerate}
	\end{lemma}
	
	\begin{proof}
		Proof of (i). From integration by parts and the bootstrap assumption~\eqref{est:Boot}, 
		\begin{equation*}
			\begin{aligned}
				(\partial_{y}\varepsilon,\rho_{1})&=(\varepsilon,\Lambda Q-(\Lambda Q)(\cdot-z))=O\left(\frac{C_{0}}{s^{\frac{3}{2}}}\right),\\
				(\Lambda\varepsilon,\rho_{1})&=-\frac{1}{2}(\varepsilon,\rho_{1})-(\varepsilon,y\partial_{y}\rho_{1})
				=-\frac{1}{2}(\varepsilon,\rho_{1})+
				O\left(\frac{C_{0}}{s^{\frac{3}{2}}}\right).
			\end{aligned}
		\end{equation*}
		Here, we use the fact that 
		\begin{equation*}
			|(1+|y|)\Lambda Q|+|(1+|y-z|)\Gamma \Lambda Q|\lesssim \phi_{1},\quad \mbox{on}\ \RR.
		\end{equation*}
		Then, from~\eqref{est:Boot} and~\eqref{est:rho1}--\eqref{est:rho3}, we check that 
		\begin{equation*}
			\begin{aligned}
				\left| \left(\varepsilon,\rho_{1}\right)\right|
				&\lesssim \left(\int_{\RR}\varepsilon^{2}\rho_{1}\dd y\right)^{\frac{1}{2}}\left(\int_{\RR}\rho_{1}\dd y\right)^{\frac{1}{2}}\lesssim \frac{C_{0}}{s^{\frac{3}{2}}\log^{\frac{1}{2}}s}.
			\end{aligned}
		\end{equation*}
		Combining the above estimates, we complete the proof for (i).
		
		\smallskip
		Proof of (ii). First, from the Fubini's Theorem and $\int_{\RR}\Lambda Q(y)\dd y=-2m_{0}$, 
		\begin{equation*}
			\begin{aligned}
				\Big(\Lambda Q,\int_{y}^{\infty}\Lambda Q(\sigma)\dd \sigma\Big)
				&=
				\int_{\RR^{2}}\Lambda Q(y)\Lambda Q(\sigma)
				\textbf{1}_{\{y\le \sigma\}}
				\dd y\dd \sigma\\
				&=\frac{1}{2}\left(\int_{\RR}\Lambda Q(y)\dd y\right)^{2}=2m_{0}^{2}.
			\end{aligned}
		\end{equation*}
		Second, from the exponential decay of $Q$, we check that 
		\begin{equation*}
			\begin{aligned}
				\int_{y}^{\infty}(\Lambda Q)(\sigma-z)\dd \sigma&=O\left(\langle y-z\rangle e^{-|y-z|}\right),\qquad \qquad \ \ \  \mbox{for} \ \  y>z,\\
				\int_{y}^{\infty}(\Lambda Q)(\sigma-z)\dd \sigma&=-2m_{0}+O\left(\langle y-z\rangle e^{-|y-z|}\right),\quad \mbox{for}\ \  y<z.
			\end{aligned}
		\end{equation*}
		Based on the above estimate and the bootstrap assumption~\eqref{est:bootz}, we find  
		\begin{equation*}
			\begin{aligned}
				\int_{z}^{\infty}\int_{y}^{\infty}
				\left(\Lambda Q(y)\right)
				\left((\Lambda Q)(\sigma-z)\right)\dd \sigma\dd y&=O\Big(s^{-\frac{3}{2}}\Big),\\
				\int_{-\infty}^{z}\int_{y}^{\infty}
				\left(\Lambda Q(y)\right)
				\left((\Lambda Q)(\sigma-z)\right)\dd \sigma\dd y&=4m_{0}^{2}+O\Big(s^{-\frac{3}{2}}\Big),
			\end{aligned}
		\end{equation*}
		which directly implies 
		\begin{equation*}
			\Big(\Lambda Q,\int_{y}^{\infty}\Lambda Q(\sigma-z)\dd \sigma\Big)
			=4m_{0}^{2}+O\Big(s^{-\frac{3}{2}}\Big).
		\end{equation*}
		Here, we use the fact that 
		\begin{equation*}
			\begin{aligned}
				\int_{z}^{\infty}\langle y\rangle \langle y-z \rangle e^{-|y|-|y-z|}\dd y\lesssim e^{-\frac{3}{4}z}&\lesssim s^{-\frac{3}{2}},\\
				\int_{-\infty}^{z}\langle y\rangle \langle y-z \rangle e^{-|y|-|y-z|}\dd y\lesssim e^{-\frac{3}{4}z}&\lesssim s^{-\frac{3}{2}}.
			\end{aligned}
		\end{equation*}
		Combining the above estimates, we obtain 
		\begin{equation*}
			\begin{aligned}
				(\Lambda Q ,\rho_{1})&= \Big(\Lambda Q,\int_{y}^{\infty}\Lambda Q(\sigma)\dd \sigma\Big)\\
				&-   \Big(\Lambda Q,\int_{y}^{\infty}\Lambda Q(\sigma-z)\dd \sigma\Big)=-2m_{0}^{2}+O\Big(s^{-\frac{3}{2}}\Big).
			\end{aligned}
		\end{equation*}
		Similar as above, we also obtain 
		\begin{equation*}
			\begin{aligned}
				(\Gamma(\Lambda Q),\rho_{1})&= \Big(\Gamma(\Lambda Q),\int_{y}^{\infty}\Lambda Q(\sigma)\dd \sigma\Big)\\
				&-   \Big(\Gamma\left(\Lambda Q\right),\int_{y}^{\infty}\Lambda Q(\sigma-z)\dd \sigma\Big)=-2m_{0}^{2}+O\Big(s^{-1}\Big),
			\end{aligned}
		\end{equation*}
		which directly completes the proof for (ii).
		
		\smallskip
		Proof of (iii). Using integration by parts, we compute
		\begin{equation*}
			(\partial_{y}V,\rho)=(V,\Lambda Q)-(V,(\Lambda Q)(\cdot-z)).
		\end{equation*}
		Based on the above identity, $(\Lambda Q,Q)=0$ and Lemma~\ref{le:boundinter}, we obtain 
		\begin{equation*}
			\left|(\partial_{y}V,\rho)\right|\lesssim |\mu|+|b_{1}|+|\Theta|+z^{5}e^{-z}\lesssim s^{-1},
		\end{equation*}
		which directly completes the proof for (iii).
		
		\smallskip
		Proof of (iv). First, we rewrite 
		\begin{equation*}
			\Lambda V=\Lambda Q-\Gamma(\Lambda Q)-z\partial_{y}Q_{2}+\Lambda (R\phi)+\Lambda U.
		\end{equation*}
		On the one hand, from (ii) of Lemma~\ref{le:estrho}, we check that 
		\begin{equation*}
			\left( \Lambda Q-\Gamma(\Lambda Q)-z\partial_{y}Q_{2},\rho_{1}\right)=z(Q_{2},\partial_{y}\rho_{1})+O(s^{-1}).
		\end{equation*}
		On the other hand, from the bootstrap assumption~\eqref{est:Boot}, 
		\begin{equation*}
			\Lambda (R\phi)+\Lambda U=z\zeta_{2}\partial_{y}X_{2}+10m_{0}zb_{1}\partial_{y}Y_{2}+O_{L^{\infty}}\left(\frac{1}{s\log s}\right).
		\end{equation*}
		Based on the above estimate and~\eqref{est:rho1}--\eqref{est:rho2}, we obtain 
		\begin{equation*}
			\begin{aligned}
				\left| \left(
				\Lambda (R\phi)+\Lambda U,\rho_{1}
				\right)\right|
				&\lesssim s^{-1}\|\partial_{y}X_{2}\|_{L^{1}}+s^{-1}\|\partial_{y}Y_{2}\|_{L^{1}}
				+\frac{\|\rho_{1}\|_{L^{1}}}{s\log s}\lesssim s^{-1}.
			\end{aligned}
		\end{equation*}
		
		Combining the above estimates, we complete the proof for (iv).
		
		\smallskip
		Proof of (v). Note that, from Definition~\ref{def:Admissible}, Proposition~\ref{prop:approx}, Lemma~\ref{le:controlpara} and Lemma~\ref{le:Psinorm}, we find 
		\begin{equation*}
			\begin{aligned}
				\Psi (V)&=
				\left(\dot{z}-b_{1}z-\mu\right)(\zeta_{2} \partial_{y}X_{2}+10m_{0}b_{1}\partial_{y}Y_{2})\\
				&+\vec{{\rm{Mod}}}\cdot \vec{M}Q
				+O_{L^{2}_{\phi_{1}}}\left(\frac{1}{s^{\frac{5}{2}}\log s}\right).
			\end{aligned}
		\end{equation*}
		Here, we denote 
		\begin{equation*}
			f\in O_{L^{2}_{\phi_{1}}}\left(\frac{1}{s^{\frac{5}{2}}\log s}\right)
			\Longleftrightarrow \left(\int_{\RR}|f(y)|^{2}\phi_{1}\dd y\right)^{\frac{1}{2}}\lesssim \frac{1}{s^{\frac{5}{2}}\log s}.
		\end{equation*}
		It follows from~\eqref{est:rho1}--\eqref{est:rho3} that 
		\begin{equation*}
			\begin{aligned}
				\left(\Psi (V),\rho_{1}\right)&=
				\left(\dot{z}-b_{1}z-\mu\right)
				\left((\zeta_{2} \partial_{y}X_{2}+10m_{0}b_{1}\partial_{y}Y_{2}),\rho_{1}\right)\\
				&+\left(\vec{{\rm{Mod}}}\cdot \vec{M}Q,\rho_{1}\right)
				+O\left(\frac{1}{s^{\frac{5}{2}}\log^{\frac{1}{2}} s}\right).
			\end{aligned}
		\end{equation*}
		In addition, using again Lemma~\ref{le:controlpara} and integration by parts, we obtain 
		\begin{equation}\label{est:Psirho1}
			\begin{aligned}
				\left(\Psi (V),\rho_{1}\right)=\left(\vec{{\rm{Mod}}}\cdot \vec{M}Q,\rho_{1}\right)
				+O\left(\frac{1}{s^{\frac{5}{2}}\log^{\frac{1}{2}} s}
				+\frac{C_{0}}{s^{\frac{5}{2}}\log s}
				\right).
			\end{aligned}
		\end{equation}
		Using integration by parts and the definition of $\vec{M}Q$ in~\eqref{equ:defMQ}, 
		\begin{equation*}
			(\dot{z}-b_{1}z-\mu)\left(M_{1Q},\rho_{1}\right)=-(\dot{z}-b_{1}z-\mu)(Q_{2},\partial_{y}\rho_{1}).
		\end{equation*}
		Then, using again the bootstrap assumption~\eqref{est:Boot} and
		~\eqref{est:rho1}--\eqref{est:rho2}, we have 
		\begin{equation*}
			\left|(M_{2Q},\rho_{1})\right|+\left|(M_{3Q},\rho_{1})\right|\lesssim 
			1+z\lesssim
			\log s,
		\end{equation*}
		which implies that 
		\begin{equation*}
			\begin{aligned}
				\left| (\dot{b}_{1}+\alpha e^{-z}+2b_{1}^{2})(M_{3Q},\rho_{1})\right|
				&\lesssim 
				\frac{C_{0}}{s^{\frac{5}{2}}\log s}
				+\frac{C_{0}^{2}}{s^{3}\log s},\\
				\left| (\dot{\zeta}_{2}+\alpha (1+\mu)^{\frac{3}{2}}e^{-z}+2(1+\mu)^{\frac{3}{2}}b_{1}^{2})(M_{2Q},\rho_{1})\right|
				&\lesssim 
				\frac{C_{0}}{s^{\frac{5}{2}}\log s}
				+\frac{C_{0}^{2}}{s^{3}\log s}.
			\end{aligned}
		\end{equation*}
		Last, from (ii) of Lemma~\ref{le:estrho} and~\eqref{est:Boot}, we find
		\begin{equation*}
			(M_{4Q},\rho_{1})=-\frac{1}{2}\left(\Gamma (\Lambda Q),\rho_{1}\right)+O(s^{-1})=m_{0}^{2}+O(s^{-1}),
		\end{equation*}
		which implies that 
		\begin{equation*}
			(\dot{\mu}-2\Theta)\left(M_{4Q},\rho_{1}\right)=m_{0}^{2} (\dot{\mu}-2\Theta)+\left(\frac{C_{0}}{s^{\frac{5}{2}}\log s}\right).
		\end{equation*}
		Combining the above estimates with~\eqref{est:Psirho1}, we complete the proof for (v).
	\end{proof}
	
	We are in a position to complete the proof of Lemma~\ref{le:refinmu}.
	\begin{proof}[Proof of Lemma~\ref{le:refinmu}]
		From the definition of $\mathcal{J}_{1}$ in~\eqref{equ:defrho}, we compute 
		\begin{equation*}
			\dot{\mathcal{J}_{1}}=\frac{1}{m_{0}^{2}}(\partial_{s}\varepsilon,\rho_{1})-\frac{\dot{z}}{m_{0}^{2}}(\varepsilon,(\Lambda Q)(\cdot-z)).
		\end{equation*}
		Using Lemma~\ref{le:eque}, we decompose 
		\begin{equation*}
			\begin{aligned}
				\left(\partial_{s}\varepsilon,\rho_{1}\right)
				&=\left(\left(\partial_{y}^{2}\varepsilon-\varepsilon+(V+\varepsilon)^{5}-V^{5}\right),\partial_{y}\rho_{1}\right)
				+\frac{\dot{\lambda}}{\lambda}\left(\Lambda \varepsilon,\rho_{1}\right)\\
				&-\left(\Psi(V),\rho_{1}\right)
				+\bigg(\frac{\dot{\lambda}}{\lambda}+b_{1}\bigg)
				\left(\Lambda V,\rho_{1}\right)+\left(\frac{\dot{x}_{1}}{\lambda}-1\right)\left(\left(\partial_{y}V+\partial_{y}\varepsilon\right),\rho_{1}\right).
			\end{aligned}
		\end{equation*}
		First, we compute
		\begin{equation*}
			\begin{aligned}
				\partial_{y}^{2}\varepsilon-\varepsilon +(V+\varepsilon)^{5}-V^{5}&= \partial_{y}^{2}\varepsilon-\varepsilon
				+5(Q_{1}^{4}+Q_{2}^{4})\varepsilon\\
				&
				+5(V^{4}-Q_{1}^{4}-Q_{2}^{4})\varepsilon+10V^{3}\varepsilon^{2}\\
				&+10V^{2}\varepsilon^{3}+5V\varepsilon^{4}+\varepsilon^{5}.
			\end{aligned}
		\end{equation*}
		Based on~\eqref{est:Boot}, $\mathcal{L}\Lambda Q=-2Q$ and $(\varepsilon,Q_{1})=(\varepsilon,Q_{2})=0$, we compute 
		\begin{equation*}
			\begin{aligned}
				&\left| \left(  \partial_{y}^{2}\varepsilon-\varepsilon
				+5(Q_{1}^{4}+Q_{2}^{4})\varepsilon,\partial_{y}\rho_{1}\right)\right|\\
				&\lesssim |\mu|\mathcal{N}_{B}(\varepsilon)+ze^{-z}\mathcal{N}_{B}(\varepsilon)\lesssim \frac{C_{0}}{s^{\frac{5}{2}}\log s}.
			\end{aligned}
		\end{equation*}
		In addition, using again the bootstrap assumption~\eqref{est:Boot} and Lemma~\ref{le:L2e},
		\begin{equation*}
			\begin{aligned}
				\left| \left((V^{4}-Q_{1}^{4}-Q_{2}^{4})\varepsilon,\partial_{y}\rho_{1}\right)\right|
				+\left|\left(V^{3}\varepsilon^{2},\partial_{y}\rho_{1}\right)\right|&\lesssim \frac{C_{0}}{s^{\frac{5}{2}}\log^{2}s}, \\
				\left(V^{2}\varepsilon^{3},\partial_{y}\rho_{1}\right)+
				\left|\left(V\varepsilon^{4},\partial_{y}\rho_{1}\right)\right|
				+\left|\left(\varepsilon^{5},\partial_{y}\rho_{1}\right)\right|
				&\lesssim \frac{C_{0}}{s^{\frac{5}{2}}\log^{2}s}.
			\end{aligned}
		\end{equation*}
		Then, from (iv)--(v) of Lemma~\ref{le:estrho} and Lemma~\ref{le:controlpara}
		\begin{equation*}
			\begin{aligned}
				& -\left(\Psi(V),\rho_{1}\right)
				+\bigg(\frac{\dot{\lambda}}{\lambda}+b_{1}\bigg)
				\left(\Lambda V,\rho_{1}\right)\\
				&=-m_{0}^{2}(\dot{\mu}-2\Theta)+O\left(\frac{1}{s^{\frac{5}{2}}\log^{\frac{1}{2}} s}
				+\frac{C_{0}}{s^{\frac{5}{2}}\log s}
				+\frac{C_{0}^{2}}{s^{3}\log s}
				\right).
			\end{aligned}
		\end{equation*}
		Here, we use the fact that 
		\begin{equation*}
			(Q_{2},\partial_{y}\rho_{1})=
			(Q_{2},(\Lambda Q)(\cdot-z))+
			(Q_{2},\Lambda Q)=O\left(s^{-1}\right).
		\end{equation*}
		Last, from Lemma~\ref{le:controlpara} and the bootstrap assumption~\eqref{est:Boot},
		\begin{equation*}
			\begin{aligned}
				\bigg|\left(\frac{\dot{x}_{1}}{\lambda}-1\right)\left(\left(\partial_{y}V+\partial_{y}\varepsilon\right),\rho_{1}\right)\bigg|&\lesssim \frac{C_{0}}{s^{\frac{5}{2}}\log s},\\\
				\bigg|\frac{\dot{\lambda}}{\lambda}\left(\Lambda \varepsilon,\rho_{1}\right)\bigg|
				+\bigg|
				\dot{z}(\varepsilon,(\Lambda Q)(\cdot-z))
				\bigg|&\lesssim \frac{C_{0}}{s^{\frac{5}{2}}\log s}.
			\end{aligned}
		\end{equation*}
		Combining the above estimates, we complete the proof of Lemma~\ref{le:refinmu}.
	\end{proof}
	
	We denote 
	\begin{equation}\label{equ:defJ2}
		\mathcal{J}_{2}(s)=\frac{1}{m_{0}^{2}}(\varepsilon(s),\rho_{2}(s))\quad \mbox{with}\ \ \rho_{2}(s,y)=X-\bar{P}_{2}+2m_{0}.
	\end{equation}
	Here, we set $\bar{P}_{2}(y)=P(y-z)$. Note that, from Remark~\ref{re:P}, for $y>z$, we have 
	\begin{equation}\label{est:rho21}
		|\rho_{2}(s,y)|\lesssim |X(y)+2m_{0}|+|\bar{P}_{2}(y)|\lesssim e^{-\frac{1}{2}|y|}+e^{-\frac{1}{2}|y-z|}.
	\end{equation}
	Similarly, for $y<0$, we also have 
	\begin{equation}\label{est:rho22}
		|\rho_{2}(s,y)|\lesssim |X(y)|+|\bar{P}_{2}(y)-2m_{0}|\lesssim e^{-\frac{1}{2}|y|}+e^{-\frac{1}{2}|y-z|}.
	\end{equation}
	It follows directly that 
	\begin{equation}\label{est:pointrho2}
		|\rho_{2}(s,y)|\lesssim \left(e^{-\frac{1}{2}|y|}+e^{-\frac{1}{2}|y-z|}\right)\left(\textbf{1}_{(-\infty,0)}+\textbf{1}_{(z,\infty)}\right)
		+\textbf{1}_{[0,z]}\lesssim \phi_{1}(s,y).
	\end{equation}
	
	\begin{lemma}[Refined control of $\Theta$]\label{le:refinedTheta}
		On $[s^{*},s^{in}]$, we have
		\begin{equation}\label{est:refinedTheta}
			\begin{aligned}
				&\left|\dot{\Theta}+\frac{5}{2}\alpha \frac{\dd}{\dd s}(ze^{-z})+c_{0}b_{1}\dot{\mu}-b_{1}\dot{\mathcal{J}_{2}}\right|\lesssim \frac{1}{s^{3}\log s}+\frac{C_{0}^{2}}{s^{3}\log^{2}s}.
			\end{aligned}
		\end{equation}
		Here, $c_{0}\in \RR$ is a constant depending only on $(P,Q,Y)$.
	\end{lemma}
	
	\begin{proof}
		\textbf{Step 1.} Refined estimate related to $b_{1}$. We claim that 
		\begin{equation}\label{est:refinedb1}
			\begin{aligned}
				\dot{b}_{1}+\alpha e^{-z}+2b_{1}^{2}
				&=\frac{\alpha}{4}b_{1}z^{2}e^{-z}+\frac{\alpha}{2}\mu ze^{-z}\\
				&-20\frac{b_{1}}{m^{2}_{0}}\left(\varepsilon,Q^{3}XQ'\right)-\frac{b_{1}}{m_{0}^{2}}\left(\frac{\dot{\lambda}}{\lambda}+b_{1}\right)(\Lambda X,Q)\\
				&+O\left(\frac{1}{s^{3}\log s}+\frac{C_{0}^{2}}{s^{3}\log^{2}s}+\frac{C_{0}}{s^{\frac{7}{2}}}
				+\frac{C_{0}}{s^{\frac{7}{2}}\log s}
				\right).
			\end{aligned}
		\end{equation}
		Indeed, from Lemma~\ref{le:eque} and the orthogonality condition $(\varepsilon,Q)=0$, 
		\begin{equation*}
			\begin{aligned}    
				&\left(\partial_{y}\left(\partial_{y}^{2}\varepsilon-\varepsilon+(V+\varepsilon)^{5}-V^{5}\right), Q\right)+\left(\Psi(V), Q\right)\\
				&=\frac{\dot{\lambda}}{\lambda}
				\left(\Lambda \varepsilon, Q\right)+\bigg(\frac{\dot{\lambda}}{\lambda}+b_{1}\bigg)\left(\Lambda V,Q\right)+\left(\frac{\dot{x}_{1}}{\lambda}-1\right)
				\left(\left(\partial_{y}V+\partial_{y}\varepsilon\right), Q\right).
			\end{aligned}
		\end{equation*}
		By an elementary computation, we decompose
		\begin{equation*}
			\begin{aligned}
				\partial_{y}^{2}\varepsilon-\varepsilon +(V+\varepsilon)^{5}-V^{5}&= -\mathcal{L}\varepsilon+5Q_{2}^{4}\varepsilon
				+5V\varepsilon^{4}+\varepsilon^{5}
				\\
				&
				+5(V^{4}-Q_{1}^{4}-Q_{2}^{4})\varepsilon+10V^{3}\varepsilon^{2}+10V^{2}\varepsilon^{3}.
			\end{aligned}
		\end{equation*}
		It follows from Lemma~\ref{le:boundinter} and Lemma~\ref{le:L2e} that\footnote{Here, we use the fact that ${\rm{Ker}}\mathcal{L}={{\rm{Span}}}\{Q'\}$.}
		\begin{equation*}
			\begin{aligned}
				&\left(\partial_{y}\left(\partial_{y}^{2}\varepsilon-\varepsilon+(V+\varepsilon)^{5}-V^{5}\right), Q\right)\\
				&=
				5\left(\partial_{y}\left((V^{4}-Q_{1}^{4}-Q_{2}^{4})\varepsilon\right),Q\right)+O\left(
				ze^{-z}\mathcal{N}_{B}(\varepsilon)
				+\mathcal{N}^{2}_{B}(\varepsilon)
				\right)
				\\
				&=
				5\left(\partial_{y}\left((V^{4}-Q_{1}^{4}-Q_{2}^{4})\varepsilon\right),Q\right)+
				O\left(\frac{C_{0}^{2}}{s^{3}\log^{2}s}+\frac{C_{0}}{s^{\frac{7}{2}}\log s}\right).
			\end{aligned}
		\end{equation*}
		Moreover, from the definition of $V$ in~\eqref{equ:defV} and the bootstrap assumption~\eqref{est:Boot}, 
		\begin{equation}\label{est:V4Q4}
			\begin{aligned}
				V^{4}-Q_{1}^{4}-Q_{2}^{4}
				&=(S^{4}-Q_{1}^{4}-Q_{2}^{4})+4(S^{3}-Q_{1}^{3}+Q_{2}^{3})(R\phi+U)\\
				&+4(Q_{1}^{3}-Q_{2}^{3})(b_{1}X_{1}-\zeta_{2}X_{2}-10m_{0}b_{1}Y_{2})+O_{L^{\infty}}\left(s^{-2}\right).
			\end{aligned}
		\end{equation}
		Therefore, from Lemma~\ref{le:boundinter}, Lemma~\ref{le:PAB} and the bootstrap assumption~\eqref{est:Boot}, 
		\begin{equation*}
			\begin{aligned}
				&\left(\partial_{y}\left(\partial_{y}^{2}\varepsilon-\varepsilon+(V+\varepsilon)^{5}-V^{5}\right), Q\right)\\
				&=-20b_{1}\left(\varepsilon,Q^{3}XQ'\right)+
				O\left(\frac{C_{0}^{2}}{s^{3}\log^{2}s}+\frac{C_{0}}{s^{\frac{7}{2}}\log s}\right).
			\end{aligned}
		\end{equation*}
		Then, using again Remark~\ref{re:P}, Lemma~\ref{le:boundinter}, Lemma~\ref{le:PAB}, Lemma~\ref{le:controlpara} and the definition of $\vec{{\rm{Mod}}}\cdot \vec{M}Q$ in~\eqref{equ:defMod}--\eqref{equ:defMQ}, we find
		\begin{equation*}
			\left( {\vec{\rm{Mod}}}\cdot\vec{M}Q,Q\right)=-m_{0}^{2}
			\left(\dot{b}_{1}+\alpha e^{-z}+2b_{1}^{2}\right)
			+O\left(\frac{C_{0}}{s^{\frac{7}{2}}}\right).
		\end{equation*}
		Based on the above estimate and Lemma~\ref{le:Psi1}, we obtain 
		\begin{equation*}
			\begin{aligned}
				(\Psi(V),Q)&=-m_{0}^{2}\left(
				\dot{b}_{1}+\alpha e^{-z}+2b_{1}^{2}
				-\frac{\alpha}{4}\zeta_{2}z^{2}e^{-z}-\frac{\alpha}{2}\mu ze^{-z}
				\right)\\
				&+O\left(\frac{1}{s^{3}\log s}+\frac{C_{0}}{s^{\frac{7}{2}}\log^{2}s}
				+\frac{C_{0}^{2}}{s^{3}\log^{2}s}
				+\frac{C_{0}}{s^{\frac{7}{2}}}
				\right).
			\end{aligned}
		\end{equation*}
		Next, from $(\varepsilon,\Lambda Q)=(\varepsilon,Q')=(Q,Q')=(X,Q')=0$, Lemma~\ref{le:boundinter} and Lemma~\ref{le:PAB},
		\begin{equation*}
			\begin{aligned}
				& \bigg|  \frac{\dot{\lambda}}{\lambda}
				\left(\Lambda \varepsilon, Q\right)\bigg|
				+\bigg|\left(\frac{\dot{x}_{1}}{\lambda}-1\right)
				\left(\left(\partial_{y}V+\partial_{y}\varepsilon\right), Q\right)\bigg|\\
				&\lesssim
				\bigg|\frac{\dot{x}_{1}}{\lambda}-1\bigg|
				\left(ze^{-z}+zb_{1}^{2}\right)
				\lesssim \frac{C_{0}}{s^{\frac{7}{2}}\log s}.
			\end{aligned}
		\end{equation*}
		Last, using again the definition of $V$ in~\eqref{equ:defV}, we decompose 
		\begin{equation*}
			\begin{aligned}
				\Lambda V
				&=\Lambda Q-\Gamma (\Lambda Q)-z\partial_{y}Q_{2}+b_{1}(\Lambda X)\phi\\
				&+b_{1}yX\partial_{y}\phi-\zeta_{2}\Lambda (X_{2}\phi)
				-10m_{0}b_{1}\Lambda Y_{2}
				+O_{L^{\infty}}\left(\frac{\log s}{s^{2}}\right).
			\end{aligned}
		\end{equation*}
		It follows from $(\Lambda Q,Q)=0$ and Lemma~\ref{le:boundinter} that 
		\begin{equation*}
			\left(\frac{\dot{\lambda}}{\lambda}+b_{1}\right)\left(\Lambda V,Q\right)=b_{1}\left(\frac{\dot{\lambda}}{\lambda}+b_{1}\right)(\Lambda X,Q)+O\left(\frac{C_{0}}{s^{\frac{7}{2}}}\right).
		\end{equation*}
		Combining the above estimates with~\eqref{est:Boot}, we complete the proof for~\eqref{est:refinedb1}.
		
		\smallskip
		\textbf{Step 2.} Refined estimate related to $\Theta$. We claim that 
		\begin{equation}\label{est:refinedb1b2}
			\begin{aligned}
				\dot{\Theta}&=20\frac{b_{1}}{m_{0}^{2}}\left(\varepsilon,Q_{2}^{3}P_{2}\partial_{y}Q_{2}\right)
				-20\frac{b_{1}}{m_{0}^{2}}\left(\varepsilon,Q^{3}XQ'\right)\\
				&+\frac{b_{1}\dot{\mu}}{2m_{0}}
				\left(3m_{0}+\frac{1}{m_{0}}(\Lambda X,Q)-5(Y-2\Lambda Y,Q)\right)
				+2\alpha \dot{z}ze^{-z}\\
				&+\frac{\alpha}{2}\left(bz+\mu\right)ze^{-z}
				+O\left(\frac{1}{s^{3}\log s}+\frac{C_{0}^{2}}{s^{3}\log^{2}s}+\frac{C_{0}}{s^{\frac{7}{2}}}+\frac{C_{0}}{s^{\frac{7}{2}}\log s}\right).
			\end{aligned}
		\end{equation}
		Indeed, from Lemma~\ref{le:eque} and the orthogonality conditions $(\varepsilon,Q_{2})=(\varepsilon,\partial_{y}Q_{2})=(\varepsilon,\Gamma(\Lambda Q))=0$, we find 
		\begin{equation*}
			\begin{aligned}    
				&\left(\partial_{y}\left(\partial_{y}^{2}\varepsilon-\varepsilon+(V+\varepsilon)^{5}-V^{5}\right), Q_{2}\right)+\left(\Psi(V), Q_{2}\right)\\
				&=\frac{\dot{\lambda}}{\lambda}
				\left(\Lambda \varepsilon, Q_{2}\right)+\bigg(\frac{\dot{\lambda}}{\lambda}+b_{1}\bigg)\left(\Lambda V,Q_{2}\right)+\left(\frac{\dot{x}_{1}}{\lambda}-1\right)
				\left(\left(\partial_{y}V+\partial_{y}\varepsilon\right), Q_{2}\right).
			\end{aligned}
		\end{equation*}
		Based on a similar argument as in Step 1, we have 
		\begin{equation*}
			\begin{aligned}
				&\left(\partial_{y}\left(\partial_{y}^{2}\varepsilon-\varepsilon+(V+\varepsilon)^{5}-V^{5}\right), Q_{2}\right)\\
				&=
				5\left(\partial_{y}\left((V^{4}-Q_{1}^{4}-Q_{2}^{4})\varepsilon\right),Q_{2}\right)+
				O\left(\frac{C_{0}^{2}}{s^{3}\log^{2}s}+\frac{C_{0}}{s^{\frac{7}{2}}\log s}\right).
			\end{aligned}
		\end{equation*}
		Note that, from Lemma~\ref{le:refindpointX} and the definition of $X$ in Remark~\ref{re:P},
		\begin{equation*}
			\left(b_{1}X_{1}-\zeta_{2}X_{2}-10m_{0}b_{1}Y_{2}\right)
			\textbf{1}_{\left[\frac{z}{2},\infty\right)}=-b_{1}P_{2}\textbf{1}_{[\frac{z}{2},\infty)}+O_{L^{\infty}}\left(\frac{\log s}{s^{2}}\right).
		\end{equation*}
		Therefore, using again Lemma~\ref{le:boundinter}, Lemma~\ref{le:PAB}, Lemma~\ref{le:controlpara}, the definition of $\vec{{\rm{Mod}}}\cdot \vec{M}Q$ in~\eqref{equ:defMod}--\eqref{equ:defMQ} and~\eqref{est:V4Q4}, we find
		\begin{equation*}
			\begin{aligned}
				&\left(\partial_{y}\left(\partial_{y}^{2}\varepsilon-\varepsilon+(V+\varepsilon)^{5}-V^{5}\right), Q_{2}\right)\\
				&=-20b_{1}\left(\varepsilon,Q_{2}^{3}P_{2}\partial_{y}Q_{2}\right)
				+
				O\left(\frac{C_{0}^{2}}{s^{3}\log^{2}s}+\frac{C_{0}}{s^{\frac{7}{2}}}\right).
			\end{aligned}
		\end{equation*}
		Then, using the definition of $\vec{{\rm{Mod}}}\cdot \vec{M}Q$ in~\eqref{equ:defMod}--\eqref{equ:defMQ} and Lemma~\ref{le:controlpara},
		we have 
		\begin{equation*}
			\begin{aligned}
				\left( {\vec{\rm{Mod}}}\cdot\vec{M}Q,Q_{2}\right)
				&=m_{0}^{2}\left(\dot{\zeta}_{2}+\alpha(1+\mu)^{\frac{3}{2}}e^{-z}+2(1+\mu)^{\frac{3}{2}}b_{1}^{2}\right)\\
				&-2m_{0}^{2}
				\left(\dot{b}_{1}+\alpha e^{-z}+2b_{1}^{2}\right)
				-\frac{b_{1}\dot{\mu}}{2}(\Lambda X,Q)
				\\
				&+O\left(\frac{1}{s^{\frac{7}{2}}}+\frac{1}{s^{3}\log s}+\frac{C_{0}}{s^{\frac{7}{2}}\log s}+\frac{C_{0}^{2}}{s^{3}\log^{2}s}\right),
			\end{aligned}
		\end{equation*}
		which implies that 
		\begin{equation*}
			\begin{aligned}
				\left( {\vec{\rm{Mod}}}\cdot\vec{M}Q,Q_{2}\right)
				&=-m_{0}^{2}
				\left(\dot{b}_{1}+\alpha e^{-z}+2b_{1}^{2}\right)\\
				&+m_{0}^{2}\dot{\Theta}-\frac{3}{2}m_{0}^{2}b_{1}\dot{\mu}
				-\frac{b_{1}\dot{\mu}}{2}(\Lambda X,Q)\\
				&+O\left(\frac{1}{s^{3}\log s}+\frac{C_{0}}{s^{\frac{7}{2}}\log s}+\frac{C_{0}^{2}}{s^{3}\log^{2}s}\right).
			\end{aligned}
		\end{equation*}
		Here, we use the fact that 
		\begin{equation*}
			\dot{\zeta}_{2}=\frac{\dot{b}_{2}}{(1+\mu)^{\frac{3}{2}}}-\frac{3\dot{\mu}b_{2}}{2(1+\mu)^{\frac{5}{2}}}
			\Longrightarrow
			\dot{\zeta}_{2}=\dot{b}_{2}-\frac{3}{2}b_{1}\dot{\mu}+O\left(\frac{1}{s^{3}\log s}\right).
		\end{equation*}
		Based on the above estimate and Lemma~\ref{le:Psi1}, we obtain 
		\begin{equation*}
			\begin{aligned}
				(\Psi(V),Q_{2})
				&=-m_{0}^{2}
				\left(\dot{b}_{1}+\alpha e^{-z}+2b_{1}^{2}\right)
				+m_{0}^{2}\dot{\Theta}-\frac{3}{2}m_{0}^{2}b_{1}\dot{\mu}   -\frac{b_{1}\dot{\mu}}{2}(\Lambda X,Q)
				\\
				&+\frac{5}{2}m_{0}b_{1}\dot{\mu}(Y-2\Lambda Y,Q)-\frac{\alpha}{4}m_{0}^{2}b_{1}z^{2}e^{-z}
				-2\alpha m_{0}^{2}\dot{z}ze^{-z}\\
				& +O\left(\frac{1}{s^{3}\log s}+\frac{C_{0}}{s^{\frac{7}{2}}\log^{2}s}
				+\frac{C_{0}^{2}}{s^{3}\log^{2}s}
				+\frac{1}{s^{\frac{7}{2}}}
				\right).
			\end{aligned}
		\end{equation*}
		From $(\varepsilon,\Gamma(\Lambda Q))=(\varepsilon,\partial_{y}Q_{2})=(Q,Q')=(X,Q')=(Y,Q')=0$ and Lemma~\ref{le:boundinter},
		\begin{equation*}
			\begin{aligned}
				& \bigg|  \frac{\dot{\lambda}}{\lambda}
				\left(\Lambda \varepsilon, Q_{2}\right)\bigg|
				+\bigg|\left(\frac{\dot{x}_{1}}{\lambda}-1\right)
				\left(\left(\partial_{y}V+\partial_{y}\varepsilon\right), Q_{2}\right)\bigg|\\
				&\lesssim
				\bigg|\frac{\dot{x}_{1}}{\lambda}-1\bigg|
				\left(ze^{-z}+zb_{1}^{2}\right)
				\lesssim \frac{C_{0}}{s^{\frac{7}{2}}\log s}.
			\end{aligned}
		\end{equation*}
		Last, using again the definition of $V$ in~\eqref{equ:defV}, we decompose 
		\begin{equation*}
			\begin{aligned}
				\Lambda V
				&=\Lambda Q-\Gamma (\Lambda Q)-z\partial_{y}Q_{2}+\frac{1}{2}b_{1}X\phi
				-\zeta_{2}yX_{2}\partial_{y}\phi
				\\
				&+b_{1}y(\partial_{y}X)\phi
				+b_{1}yX\partial_{y}\phi-\zeta_{2} \Gamma(\Lambda X)\phi
				-z\zeta_{2}(\partial_{y}X_{2})\phi\\
				&-10m_{0}b_{1}\Gamma(\Lambda Y)
				-10m_{0}b_{1}z\partial_{y}Y_{2}
				+O_{L^{\infty}}\left(\frac{\log s}{s^{2}}\right).
			\end{aligned}
		\end{equation*}
		It follows from $(\Lambda Q,Q)=(Q,Q')=(X,Q')=(Y,Q')=0$ and Lemma~\ref{le:boundinter} that 
		\begin{equation*}
			\left(\frac{\dot{\lambda}}{\lambda}+b_{1}\right)\left(\Lambda V,Q_{2}\right)=-b_{1}\left(\frac{\dot{\lambda}}{\lambda}+b_{1}\right)(\Lambda P,Q)+O\left(\frac{C_{0}}{s^{\frac{7}{2}}}\right).
		\end{equation*}
		Here, we use the fact that 
		\begin{equation*}
			\left( \frac{1}{2}b_{1}X-\zeta_{2}\Gamma(\Lambda X)
			-10m_{0}b_{1}\Gamma(\Lambda Y)
			\right)\textbf{1}_{[\frac{z}{2},\infty)}=-b_{1}\Gamma(\Lambda P)\textbf{1}_{[\frac{z}{2},\infty)}+O_{L^{\infty}}\left(\frac{\log s}{s^{2}}\right).
		\end{equation*}
		Combining the above estimates with~\eqref{est:refinedb1}, we obtain 
		\begin{equation*}
			\begin{aligned}
				\dot{\Theta}
				-2\alpha \dot{z}ze^{-z}
				&=20\frac{b_{1}}{m_{0}^{2}}\left(\varepsilon,Q_{2}^{3}P_{2}\partial_{y}Q_{2}\right)
				-20\frac{b_{1}}{m_{0}^{2}}\left(\varepsilon,Q^{3}XQ'\right)\\
				&+\frac{b_{1}\dot{\mu}}{2m_{0}}\left(3m_{0}+\frac{1}{m_{0}}(\Lambda X,Q)-5(Y-2\Lambda Y,Q)\right)\\
				&+\frac{\alpha}{2}\left(bz+\mu\right)ze^{-z}
				+\frac{b_{1}}{m_{0}^{2}}\left(\frac{\dot{\lambda}}{\lambda}+b_{1}\right)\left(X+P,\Lambda Q\right)\\
				&+O\left(\frac{1}{s^{3}\log s}+\frac{C_{0}^{2}}{s^{3}\log^{2}s}+\frac{C_{0}}{s^{\frac{7}{2}}}+\frac{C_{0}}{s^{\frac{7}{2}}\log s}\right).
			\end{aligned}
		\end{equation*}
		We see that~\eqref{est:refinedb1b2} follows directly from the above estimate and Remark~\ref{re:P}.
		
		\smallskip
		\textbf{Step 3.} Conclusion. We claim that 
		\begin{equation}\label{est:dsJ2}
			\begin{aligned}
				\dot{\mathcal{J}}_{2}
				&=-\frac{20}{m_{0}^{2}}\left(\varepsilon,Q^{3}XQ'\right)+\frac{20}{m_{0}^{2}}(\varepsilon,Q_{2}^{3}P_{2}\partial_{y}Q_{2})\\
				&+O\left(
				\frac{C_{0}}{s^{\frac{5}{2}}\log^{\frac{1}{2}} s}+
				\frac{C^{2}_{0}}{s^{3}\log^{2} s}
				\right).
			\end{aligned}
		\end{equation}
		
		Indeed, from the definition of $\mathcal{J}_{2}$ in~\eqref{equ:defJ2}, we compute 
		\begin{equation*}
			\dot{\mathcal{J}_{2}}=\frac{1}{m_{0}^{2}}(\partial_{s}\varepsilon,\rho_{2})+\frac{\dot{z}}{m_{0}^{2}}(\varepsilon,(\partial_{y}P(\cdot-z))).
		\end{equation*}
		Using Lemma~\ref{le:eque}, we decompose 
		\begin{equation*}
			\begin{aligned}
				\left(\partial_{s}\varepsilon,\rho_{2}\right)
				&=\left(\left(\partial_{y}^{2}\varepsilon-\varepsilon+(V+\varepsilon)^{5}-V^{5}\right),\partial_{y}\rho_{2}\right)
				+\frac{\dot{\lambda}}{\lambda}\left(\Lambda \varepsilon,\rho_{2}\right)\\
				&-\left(\Psi(V),\rho_{2}\right)
				+\bigg(\frac{\dot{\lambda}}{\lambda}+b_{1}\bigg)
				\left(\Lambda V,\rho_{2}\right)+\left(\frac{\dot{x}_{1}}{\lambda}-1\right)\left(\left(\partial_{y}V+\partial_{y}\varepsilon\right),\rho_{2}\right).
			\end{aligned}
		\end{equation*}
		Recall that, we compute
		\begin{equation*}
			\begin{aligned}
				\partial_{y}^{2}\varepsilon-\varepsilon +(V+\varepsilon)^{5}-V^{5}&= \partial_{y}^{2}\varepsilon-\varepsilon
				+5(Q_{1}^{4}+Q_{2}^{4})\varepsilon\\
				&
				+5(V^{4}-Q_{1}^{4}-Q_{2}^{4})\varepsilon+10V^{3}\varepsilon^{2}\\
				&+10V^{2}\varepsilon^{3}+5V\varepsilon^{4}+\varepsilon^{5}.
			\end{aligned}
		\end{equation*}
		Recall also that, from Remark~\ref{re:P},
		\begin{equation*}
			\begin{aligned}
				\mathcal{L}(X')&=(\mathcal{L}X)'+20Q^{3}XQ'=\Lambda Q+20Q^{3}XQ',\\
				\mathcal{L}(P')&=(\mathcal{L}P)'+20Q^{3}PQ'=\Lambda Q+20Q^{3}PQ'.
			\end{aligned}
		\end{equation*}
		It follows from Lemma~\ref{le:boundinter} and the bootstrap assumption~\eqref{est:Boot} that 
		\begin{equation*}
			\begin{aligned}
				&\left(\left(\partial_{y}^{2}\varepsilon-\varepsilon+(V+\varepsilon)^{5}-V^{5}\right),\partial_{y}\rho_{2}\right)\\
				&=-20\left(\varepsilon,Q^{3}XQ'\right)+20(\varepsilon,Q_{2}^{3}P_{2}\partial_{y}Q_{2})+O\left(\frac{C_{0}}{s^{\frac{5}{2}}\log s}\right).
			\end{aligned}
		\end{equation*}
		Then, from Lemma~\ref{le:controlpara}, Lemma~\ref{le:Psinorm},~\eqref{est:pointrho2} and the bootstrap assumption~\eqref{est:Boot},
		\begin{equation*}
			\begin{aligned}
				\bigg|\left(\frac{\dot{x}_{1}}{\lambda}-1\right)\left(\partial_{y}\varepsilon,\rho_{2}\right)\bigg|
				&\lesssim
				\mathcal{N}^{2}_{B}(\varepsilon)\lesssim
				\frac{C^{2}_{0}}{s^{3}\log^{2} s},\\
				\left| \left(\Psi(V),\rho_{2}\right)\right|
				+\bigg|\frac{\dot{\lambda}}{\lambda}\left(\Lambda \varepsilon,\rho_{2}\right)\bigg|
				&\lesssim
				s^{-1}\mathcal{N}_{B}(\varepsilon)+\frac{1}{s^{\frac{5}{2}}\log ^{\frac{1}{2}}s}
				\lesssim
				\frac{C_{0}}{s^{\frac{5}{2}}\log^{\frac{1}{2}} s}.
			\end{aligned}
		\end{equation*}
		Next, using again Remark~\ref{re:P} and Lemma~\ref{le:boundinter}, we check that 
		\begin{equation*}
			\begin{aligned}
				\left(\Lambda Q,\rho_{2}\right)=(\Lambda Q,X)+O\left(s^{-1}\right),\\
				\left(\Gamma (\Lambda Q),\rho_{2}\right)=-(\Lambda Q,P)+O\left(s^{-1}\right).
			\end{aligned}
		\end{equation*}
		In particular, we also check that 
		\begin{equation*}
			(\partial_{y}Q_{1},\rho_{2})=O\left(s^{-1}\right)\quad \mbox{and}\quad (\partial_{y}Q_{2},\rho_{2})=O\left(s^{-1}\right).
		\end{equation*}
		Therefore, from Remark~\ref{re:P}, Lemma~\ref{le:controlpara}, bootstrap assumption~\eqref{est:Boot} and~\eqref{est:pointrho2}, 
		\begin{equation*}
			\begin{aligned}
				&\bigg| \bigg(\frac{\dot{\lambda}}{\lambda}+b_{1}\bigg)
				\left(\Lambda V,\rho_{2}\right)\bigg|+\bigg|\left(\frac{\dot{x}_{1}}{\lambda}-1\right)\left(\partial_{y}V,\rho_{2}\right)\bigg|\\
				&\lesssim \left(\|R\phi+U\|_{L^{\infty}}\|\rho_{2}\|_{L^{1}}+s^{-1}\right)\mathcal{N}_{B}(\varepsilon)\lesssim \frac{C_{0}}{s^{\frac{5}{2}}\log s}.
			\end{aligned}
		\end{equation*}
		Last, from Lemma~\ref{le:controlpara} and the Cauchy-Schwarz inequality, 
		\begin{equation*}
			\bigg| \frac{\dot{z}}{m_{0}^{2}}(\varepsilon,(\partial_{y}P(\cdot-z)))\bigg|
			\lesssim s^{-1}\mathcal{N}_{B}(\varepsilon)\lesssim \frac{C_{0}}{s^{\frac{5}{2}}\log s}.
		\end{equation*}
		Combining the above estimates, we complete the proof of~\eqref{est:dsJ2}.
		
		\smallskip
		We see that~\eqref{est:refinedTheta} follows directly from~\eqref{est:refinedb1b2},~\eqref{est:dsJ2} and Lemma~\ref{le:controlpara}.
	\end{proof}
	
	\subsection{Energy functional}\label{SS:Energy}
	Let $\chi_{2}:\RR\to [0,1]$ be $C^{\infty}$ function such that 
	\begin{equation*}
		\chi_{2}\equiv 1 \ \ \mbox{on}\ \left[\frac{1}{2},\frac{3}{2}\right]\quad \mbox{and}\quad \chi_{2}\equiv 0 \ \ \mbox{on}\ \left(-\infty,\frac{1}{4}\right]\cup \left[\frac{7}{4},\infty\right).
	\end{equation*}
	Denote 
	\begin{equation}\label{equ:defphi2}
		\phi_{2}(y)=\chi_{2}\left(\frac{y}{z(s)}\right)\Longrightarrow
		\phi_{2}\equiv 1 \ \ \mbox{on}\ \left[\frac{z}{2},\frac{3z}{2}\right].
	\end{equation}
	We consider the nonlinear energy functional for $\varepsilon$:
	\begin{equation*}
		\mathcal{F}(s,\varepsilon)=\mathcal{F}_{1}(s,\varepsilon)+\mathcal{F}_{2}(s,\varepsilon),
	\end{equation*}
	where
	\begin{equation*}
		\begin{aligned}
			\mathcal{F}_{1}(s,\varepsilon)&=\int_{\RR}
			\left((\partial_{y}\varepsilon)^{2}+\varepsilon^{2}\phi_{1}
			-\frac{1}{3}\left((V+\varepsilon)^{6}-V^{6}-6V^{5}\varepsilon\right)\right)\dd y,\\
			\mathcal{F}_{2}(s,\varepsilon)&=\mu\int_{\RR}\varepsilon^{2}\phi_{2}\dd y-2b_{1}\left(\mathcal{J}_{2}+\left(c_{0}+\frac{5}{4}\right)\mathcal{J}_{1}\right)(\varepsilon\phi_{1},X_{1}\phi).
		\end{aligned}
	\end{equation*}
	
	We mention here that, the function $\mathcal{F}$ is coercive in $\varepsilon$ at the main order and enjoys a monotonicity property adapted to this problem.
	
	\begin{proposition}\label{prop:energy}
		The following estimates hold on $[s^{*},s^{in}]$.
		\begin{enumerate}
			\item \emph{Coercivity of $\mathcal{F}$.}
			We have 
			\begin{equation}\label{est:coerF}
				\mathcal{N}_{B}^{2}(\varepsilon)\lesssim \mathcal{F}+s^{-\frac{13}{4}}.
			\end{equation}
			
			\item \emph{Time variation of $\mathcal{F}$.} We have 
			\begin{equation}\label{est:dsF}
				-\frac{\dd \mathcal{F}}{\dd s}\lesssim \frac{C_{0}^{2}}{s^{4}\log^{\frac{5}{2}}s}+\frac{C_{0}}{s^{4}\log^{2}s}.
			\end{equation}
		\end{enumerate}
	\end{proposition}
	
	\begin{proof}
		Proof of (i). The proof of the coercivity for $\mathcal{F}$ is a standard consequence of the coercivity property in Proposition~\ref{prop:Spectral} around single soliton with the orthogonality conditions~\eqref{equ:orth} and an elementary localization argument. Below, we provide
		a sketch of the proof for the sake of completeness and for the readers’ convenience. 
		
		\smallskip
		For some $0<\delta\ll1$, we decompose 
		\begin{equation*}
			\mathcal{F}_{1}=\mathcal{F}_{1,1}+\mathcal{F}_{1,2}+\mathcal{F}_{1,3}+\mathcal{F}_{1,4}+\mathcal{F}_{1,5},
		\end{equation*}
		where
		\begin{equation*}
			\begin{aligned}
				\mathcal{F}_{1,1}&=5\int_{\RR}(S^{4}-V^{4})\varepsilon^{2}\dd y-5\int_{\RR}S^{4}\varepsilon^{2}(1-\phi_{1})\dd y,\\
				\mathcal{F}_{1,2}&=-\frac{1}{3}\int_{\RR}\varepsilon^{3}\left(20V^{3}+15V^{2}\varepsilon+6V\varepsilon^{2}+\varepsilon^{3}\right)\dd y,\\
				\mathcal{F}_{1,3}&=(1-\delta)\int_{\RR}\left((\partial_{y}(\varepsilon\sqrt{\phi_{1}}))^{2}+(\varepsilon\sqrt{\phi_{1}})^{2}-
				5S^{4}(\varepsilon\sqrt{\phi_{1}})^{2}
				\right)\dd y,\\
				\mathcal{F}_{1,4}&=\delta\int_{\RR}(\partial_y\e)^2+\delta\int_{\RR}\left((\varepsilon\sqrt{\phi_{1}})^{2}-
				5S^{4}(\varepsilon\sqrt{\phi_{1}})^{2}
				\right)\dd y,\\
				\mathcal{F}_{1,5}&=(1-\delta)\int_{\RR}(\partial_{y}\varepsilon)^{2}(1-\phi_{1})\dd y
				-\frac{(1-\delta)}{2}\int_{\RR}\varepsilon^{2}
                \left(\frac{(\partial_{y}\phi_{1})^{2}}{2\phi_{1}}
                -\partial_{y}^{2}\phi_{1}
                \right)\dd y.
			\end{aligned}
		\end{equation*}
		
		\emph{Estimates on $\mathcal{F}_{1,1}$ and $\mathcal{F}_{1,2}$.} We claim that 
		\begin{equation}\label{est:F1112}
			\mathcal{F}_{1,1}+\mathcal{F}_{1,2}=O\left(s^{-\frac{13}{4}}\right).
		\end{equation}
		Indeed, from Lemma~\ref{le:L2e} and the bootstrap assumption~\eqref{est:Boot}, we directly have 
		\begin{equation*}
			\begin{aligned}
				\left| \int_{\RR}(S^{4}-V^{4})\varepsilon^{2}\dd y\right|
				&\lesssim \left(\|R\phi\|_{L^{\infty}}+\|U\|_{L^{\infty}}\right)\mathcal{N}_{B}^{2}(\varepsilon)\\
				&+\left(\|R\phi\|^{4}_{L^{\infty}}+\|U\|^{4}_{L^{\infty}}\right)\|\varepsilon\|_{L^{2}}^{2}\lesssim s^{-4}.
			\end{aligned}
		\end{equation*}
		Then, using the exponential decay of $Q$ and the definition of $\phi_{1}$ in~\eqref{equ:defphi1}, 
		\begin{equation*}
			\left| \int_{\RR}S^{4}\varepsilon^{2}(1-\phi_{1})\dd y\right|\lesssim \|S^{4}(1-\phi_{1})\|_{L^{\infty}}\|\varepsilon\|_{L^{2}}^{2}\lesssim s^{-4}.
		\end{equation*}
		Based on a similar argument as above, we deduce that 
		\begin{equation*}
			\left|\mathcal{F}_{1,2}\right|\lesssim \left(\|\varepsilon\|_{L^{\infty}}
			+\|\varepsilon\|_{L^{\infty}}^{4}
			\right)
			\|\varepsilon\|_{L^{2}}^{2}\lesssim s^{-\frac{13}{4}}.
		\end{equation*}
		We see that~\eqref{est:F1112} follows directly from the above estimates.
		
		\smallskip
		\emph{Estimate on $\mathcal{F}_{1,3}$.} We claim that 
		\begin{equation}\label{est:F13}
			\mathcal{F}_{1,3}\ge \frac{\nu}{2} \left\|\varepsilon\sqrt{\phi_{1}}\right\|_{L^2}^{2}
			+O\left(s^{-4}\right).
		\end{equation}
		Indeed, from the orthogonality condition~\eqref{equ:orth}, we find 
		\begin{equation*}
			\begin{aligned}
				(\varepsilon\sqrt{\phi_{1}},Q_{1})^{2}+
				(\varepsilon\sqrt{\phi_{1}},\partial_{y}Q_{1})^{2}+
				(\varepsilon\sqrt{\phi_{1}},\Lambda Q)^{2}&=O\left(s^{-4}\right),\\
				(\varepsilon\sqrt{\phi_{1}},Q_{2})^{2}+
				(\varepsilon\sqrt{\phi_{1}},\partial_{y}Q_{2})^{2}+ 
				(\varepsilon\sqrt{\phi_{1}},\Gamma(\Lambda Q))^{2}&=O\left(s^{-4}\right).
			\end{aligned}
		\end{equation*}
		It follows from Proposition~\ref{prop:Spectral} and an standard localization argument that 
		\begin{equation*}
			\int_{\RR}\left((\partial_{y}(\varepsilon\sqrt{\phi_{1}}))^{2}+(\varepsilon\sqrt{\phi_{1}})^{2}-
			5(Q_{1}^{4}+Q_{2}^{4})(\varepsilon\sqrt{\phi_{1}})^{2}
			\right)\dd y
			\ge \nu\left\|\varepsilon\sqrt{\phi_{1}}\right\|^2_{H^{1}}+O(s^{-4}).
		\end{equation*}
		We see that~\eqref{est:F13} follows from the above estimates and Lemma~\ref{le:boundinter}.
		
		\smallskip
		\emph{Estimate on $\mathcal{F}_{1,4}$.} It holds 
		\begin{equation}\label{est:F14}
			\mathcal{F}_{1,4}=\delta \int_{\RR}(\partial_y\e)^2\dd y+O( \delta)\int_{\RR} \e^2\phi_1\dd y.
		\end{equation}
		
		\smallskip
		\emph{Estimate on $\mathcal{F}_{1,5}$.} We claim that 
		\begin{equation}\label{est:F15}
			\mathcal{F}_{1,5}\gtrsim -B^{-2}\int_{\RR}\varepsilon^{2}\phi_{1}\dd y.
		\end{equation}
		Indeed, from the definition of $\phi_{1}$ in~\eqref{equ:defphi1}, we have $1-\phi_1\geq 0$ and
		\begin{equation*}
			\begin{aligned}
				|\partial_{y}\phi_{1}|\lesssim B^{-1}\left|\chi_{1}'\left(\frac{y}{B}-s^{\frac{1}{2}}\right)\right|\lesssim B^{-1}\phi_{1},\\
				|\partial^{2}_{y}\phi_{1}|\lesssim B^{-2}\left|\chi_{1}''\left(\frac{y}{B}-s^{\frac{1}{2}}\right)\right|\lesssim B^{-2}\phi_{1},
			\end{aligned}
		\end{equation*}
		which directly completes the proof for~\eqref{est:F15}.
		
		\smallskip
        \emph{Estimate on $\mathcal{F}_{2}$.}
		Using~\eqref{est:rho3} and~\eqref{est:pointrho2}, we find \begin{equation}\label{est:J1J2}
			|\mathcal{J}_{1}|+|\mathcal{J}_{2}|\lesssim \left(\int_{\RR}
			(|\rho_{1}|+|\rho_{2}|)\dd y\right)^{\frac{1}{2}}\mathcal{N}_{B}(\varepsilon)
			\lesssim (\log^{\frac{1}{2}}s) \mathcal{N}_{B}(\varepsilon).
		\end{equation}
		Then, from Lemma~\ref{le:refindpointX} and the definition of $\phi_{1}$ in~\eqref{equ:defphi1}, 
		\begin{equation}\label{est:ephi1Xphi}
			\begin{aligned}
				|(\varepsilon\phi_{1},X_{1}\phi)|
				&\lesssim \mathcal{N}_{B}(\varepsilon)
				\left(   \int_{-\infty}^{0}y^{4}e^{-2|y|}\dd y+\int_{0}^{B(s^{\frac{1}{2}}+2)}1\dd y\right)^{\frac{1}{2}}\\
				&+\mathcal{N}_{B}(\varepsilon)\left(\int_{B(s^{\frac{1}{2}}+2)}^{\infty}\phi_{1}\dd y\right)^{\frac{1}{2}}\lesssim \left(1+s^{\frac{1}{4}}\right)\mathcal{N}_{B}(\varepsilon).
			\end{aligned}
		\end{equation}
		Based on the above estimate and the bootstrap assumption~\eqref{est:Boot}, we obtain
		\begin{equation}\label{est:F2}
			|\mathcal{F}_{2}|\lesssim |\mu|\mathcal{N}_{B}^{2}(\varepsilon)+
			|b_{1}|(|\mathcal{J}_{1}|+|\mathcal{J}_{2}|)|(\varepsilon\phi_{1},X_{1}\phi)|
			\lesssim s^{-\frac{7}{2}}.
		\end{equation}
		Here, we use the fact that 
		\begin{equation*}
			\phi_{2}\lesssim \phi_{1}\Longrightarrow 
			\int_{\RR}\varepsilon^{2}\phi_{2}\dd y 
			\lesssim \int_{\RR}\varepsilon^{2}\phi_{1}\dd y 
			\lesssim \mathcal{N}_{B}^{2}(\varepsilon).
		\end{equation*}
		Combining the estimates~\eqref{est:F1112}--\eqref{est:F15} with~\eqref{est:F2}, we complete the proof for~\eqref{est:coerF}.
		
		\smallskip
		Proof of (ii). From now on, we use the following notation:
		\begin{equation*}
			f\simeq g\Longrightarrow f=g+O\left(\frac{C_{0}^{2}}{s^{4}\log^{\frac{5}{2}}s}+\frac{C_{0}}{s^{4}\log^{2}s}\right).
		\end{equation*}
		
		\textbf{Step 1.} Estimate for the time variation of $\mathcal{F}_{1}$. We claim that 
		\begin{equation}\label{est:dsF1}
			\begin{aligned}
				\frac{\dd \mathcal{F}_{1}}{\dd s}&\ge 
				2\dot{\Theta}(\varepsilon\phi_{1},X_{1}\phi)
				-\frac{5}{2}b_{1}\dot{\mu}(\varepsilon\phi_{1},X_{1}\phi)
				-\frac{1}{8}\int_{\RR}\varepsilon^{2}\partial_{y}\phi_{1}\dd y\\
				&+20\mu\int_{\RR}(\partial_{y}Q_{2})(Q_{2}^{3}\varepsilon^{2})\dd y
				+
				O\left(\frac{C_{0}^{2}}{s^{4}\log^{\frac{5}{2}}s}+\frac{C_{0}}{s^{4}\log^{2}s}\right).
			\end{aligned}
		\end{equation}
		
		By an elementary computation, we decompose
		\begin{equation*}
			\frac{\dd \mathcal{F}_{1}}{\dd s}=\mathcal{I}_{1}+\mathcal{I}_{2},
		\end{equation*}
		where 
		\begin{equation*}
			\begin{aligned}
				\mathcal{I}_{1}&=2\int_{\RR}(\partial_{s}\varepsilon)
				\left(-\partial_{y}^{2}\varepsilon+\varepsilon\phi_{1}-(V+\varepsilon)^{5}+V^{5}\right)\dd y,\\
				\mathcal{I}_{2}&=-2\int_{\RR}(\partial_{s}V)\left((V+\varepsilon)^{5}-V^{5}-5V^{4}\varepsilon\right)\dd y+\int_{\RR}\varepsilon^{2}\partial_{s}\phi_{1}\dd y.
			\end{aligned}
		\end{equation*}
		
		\emph{Estimate on $\mathcal{I}_{1}$.}
		We claim that 
		\begin{equation}\label{est:I1}
			\begin{aligned}
				\mathcal{I}_{1}&\ge 2\dot{\Theta}(\varepsilon\phi_{1},X_{1}\phi)
				-\frac{5}{2}b_{1}\dot{\mu}(\varepsilon\phi_{1},X_{1}\phi)
				-\frac{1}{4}\int_{\RR}\varepsilon^{2}\partial_{y}\phi_{1}\dd y\\
				&+20z\frac{\dot{\lambda}}{\lambda}\int_{\RR}(\partial_{y}Q_{2})(Q_{2}^{3}\varepsilon^{2})\dd y
				+
				O\left(\frac{C_{0}^{2}}{s^{4}\log^{\frac{5}{2}}s}+\frac{C_{0}}{s^{4}\log^{2}s}\right).
			\end{aligned}
		\end{equation}
		Indeed, using Lemma~\ref{le:eque}, we decompose 
		\begin{equation*}
			\mathcal{I}_{1}=\mathcal{I}_{1,1}+\mathcal{I}_{1,2}+\mathcal{I}_{1,3}+\mathcal{I}_{1,4}+\mathcal{I}_{1,5},
		\end{equation*}
		where
		\begin{equation*}
			\begin{aligned}
				\mathcal{I}_{1,1}&=2\int_{\RR}\bar{\Psi}(V)
				\left(\partial_{y}^{2}\varepsilon-\varepsilon \phi_{1}+(V+\varepsilon)^{5}-V^{5}\right)\dd y,\\
				\mathcal{I}_{1,2}&=2\frac{\dot{\lambda}}{\lambda}\int_{\RR}\Lambda \varepsilon \left(-\partial_{y}^{2}\varepsilon+\varepsilon\phi_{1}-(V+\varepsilon)^{5}+V^{5}\right)\dd y,\\
				\mathcal{I}_{1,3}&=2\int_{\RR}\varepsilon(1-\phi_{1})\partial_{y}\left(\partial_{y}^{2}\varepsilon-\varepsilon +(V+\varepsilon)^{5}-V^{5}\right)\dd y,\\
				\mathcal{I}_{1,4}&=2\Big(\frac{\dot{x}_{1}}{\lambda}-1\Big)
				\int_{\RR}\partial_{y}V\left(-\partial_{y}^{2}\varepsilon+\varepsilon\phi_{1}-(V+\varepsilon)^{5}+V^{5}\right)\dd y,\\
				\mathcal{I}_{1,5}&=2\Big(\frac{\dot{x}_{1}}{\lambda}-1\Big)
				\int_{\RR}\partial_{y}\varepsilon\left(-\partial_{y}^{2}\varepsilon+\varepsilon\phi_{1}-(V+\varepsilon)^{5}+V^{5}\right)\dd y.
			\end{aligned}
		\end{equation*}
		Here, we denote 
		\begin{equation*}
			\bar{\Psi}(V)=\Psi(V)-\left(\frac{\dot{\lambda}}{\lambda}+b_{1}\right)\Lambda V.
		\end{equation*}
		
		\emph{Estimate on $\mathcal{I}_{1,1}$.} 
		We claim that 
		\begin{equation}\label{est:I11}
			\begin{aligned}
				\mathcal{I}_{1,1}\simeq 
				2\dot{\Theta}(\varepsilon\phi_{1},X_{1}\phi)
				-\frac{5}{2}b_{1}\dot{\mu}(\varepsilon\phi_{1},X_{1}\phi).
			\end{aligned}
		\end{equation}
		Indeed, we decompose 
		\begin{equation*}
			\begin{aligned}
				\partial_{y}^{2}\varepsilon-\varepsilon \phi_{1}+(V+\varepsilon)^{5}-V^{5}&= \partial_{y}^{2}\varepsilon-\varepsilon
				+5(Q_{1}^{4}+Q_{2}^{4})\varepsilon\\
				&+(1-\phi_{1})\varepsilon
				+5(V^{4}-Q_{1}^{4}-Q_{2}^{4})\varepsilon\\
				&+10V^{3}\varepsilon^{2}+10V^{2}\varepsilon^{3}+5V\varepsilon^{4}+\varepsilon^{5}. \end{aligned}
		\end{equation*}
		Based on the above identity, ${\rm{Ker}}\mathcal{L}={\rm{Span}}\{Q'\}$
        and $(\varepsilon,\partial_{y}Q_{2})=0$, we compute 
		\begin{equation*}
			\begin{aligned}
				&(\dot{z}-b_{1}z-\mu)\int_{\RR}M_{1Q}\left(
				\partial_{y}^{2}\varepsilon-\varepsilon \phi_{1}+(V+\varepsilon)^{5}-V^{5}
				\right)\dd y\\
				&=5(\dot{z}-b_{1}z-\mu)\int_{\RR}(\partial_{y}Q_{2})(V^{4}-Q^{4}_{1}-Q^{4}_{2})\varepsilon\dd y\\
				&+O\left(|\dot{z}-b_{1}z-\mu|
				\left(\|Q_{1}Q_{2}^{3}\|_{L^{2}}\mathcal{N}_{B}(\varepsilon)
				+\|Q^{3}_{1}Q_{2}\|_{L^{2}}\mathcal{N}_{B}(\varepsilon)
				\mathcal{N}_{B}^{2}(\varepsilon)\right)\right)\\
				&+O\left(
				|\dot{z}-b_{1}z-\mu|\left(\|Q_{1}^{3}\partial_{y}Q_{2}\|_{L^{2}}
				+\||\partial_{y}Q_{2}|^{\frac{1}{2}}(1-\phi_{1})\|_{L^{2}}\right)\mathcal{N}_{B}(\varepsilon)\right).
			\end{aligned}
		\end{equation*}
		It follows from~\eqref{est:Boot}, Lemma~\ref{le:boundinter} and Lemma~\ref{le:controlpara} that 
		\begin{equation}\label{est:I111}
			\begin{aligned}
				&(\dot{z}-b_{1}z-\mu)\int_{\RR}M_{1Q}\left(
				\partial_{y}^{2}\varepsilon-\varepsilon \phi_{1}+(V+\varepsilon)^{5}-V^{5}
				\right)\dd y\\
				&\simeq 5(\dot{z}-b_{1}z-\mu)\int_{\RR}(\partial_{y}Q_{2})(V^{4}-Q^{4}_{1}-Q^{4}_{2})\varepsilon\dd y.
			\end{aligned}
		\end{equation}
		Then, from the bootstrap assumption~\eqref{est:Boot} and Lemma~\ref{le:L2e}, 
		\begin{equation*}
			\begin{aligned}
				&\int_{\RR}M_{2Q}\left(
				\partial_{y}^{2}\varepsilon-\varepsilon \phi_{1}+(V+\varepsilon)^{5}-V^{5}
				\right)\dd y\\
				&=-\left(\varepsilon\phi_{1},M_{2Q}\right)+
				O\left(\|M_{2Q}\|_{L^{2}}\left(\|R\phi\|_{L^{\infty}}^{4}+\|U\|_{L^{\infty}}^{4}\right)\|\varepsilon\|_{L^{2}}+\mathcal{N}_{B}(\varepsilon)\right)\\
				&+
				O\left( \|M_{2Q}\|_{\dot{H}^{1}}\mathcal{N}_{B}(\varepsilon)
				+\|\varepsilon\|_{L^{\infty}}^{3}\|\varepsilon\|^{2}_{L^{2}}\right)
				=-\left(\varepsilon\phi_{1},M_{2Q}\right)+
				O\left( \frac{C_{0}}{s^{\frac{3}{2}}\log s}\right),
			\end{aligned}
		\end{equation*}
		\begin{equation*}
			\begin{aligned}
				&\int_{\RR}M_{3Q}\left(
				\partial_{y}^{2}\varepsilon-\varepsilon \phi_{1}+(V+\varepsilon)^{5}-V^{5}
				\right)\dd y\\
				&=-\left(\varepsilon\phi_{1},M_{3Q}\right)+
				O\left(\|M_{3Q}\|_{L^{2}}\left(\|R\phi\|_{L^{\infty}}^{4}+\|U\|_{L^{\infty}}^{4}\right)\|\varepsilon\|_{L^{2}}+\mathcal{N}_{B}(\varepsilon)\right)\\
				&+
				O\left( \|M_{3Q}\|_{\dot{H}^{1}}\mathcal{N}_{B}(\varepsilon)
				+\|\varepsilon\|_{L^{\infty}}^{3}\|\varepsilon\|^{2}_{L^{2}}\right)
				=-\left(\varepsilon\phi_{1},M_{3Q}\right)+
				O\left( \frac{C_{0}}{s^{\frac{3}{2}}\log s}\right).
			\end{aligned}
		\end{equation*}
		Based on the above estimates and Lemma~\ref{le:controlpara}, we find 
		\begin{equation*}
			\begin{aligned}
				&\int_{\RR}\vec{N}\cdot (M_{2Q},M_{3Q})\left(
				\partial_{y}^{2}\varepsilon-\varepsilon \phi_{1}+(V+\varepsilon)^{5}-V^{5}
				\right)\dd y\\
				&\simeq (\dot{\zeta}_{2}-\dot{b}_{1}+\alpha ((1+\mu)^{\frac{3}{2}}-1) e^{-z}
                +2((1+\mu)^{\frac{3}{2}}-1)b_{1}^{2}
                )\left(\varepsilon\phi_{1},X_{1}\phi\right)\\
				&+(\dot{\zeta}_{2}+\alpha(1+\mu)^{\frac{3}{2}}e^{-z}+2(1+\mu)^{\frac{3}{2}}b_{1}^{2})\left(\varepsilon\phi_{1},(X_{2}-X_{1})\phi\right).
			\end{aligned}
		\end{equation*}
		Note that, from $b_{2}=(1+\mu)^{\frac{3}{2}}\zeta_{2}$ and Lemma~\ref{le:controlpara}, we compute 
		\begin{equation*}
			\dot{\zeta}_{2}-\dot{b}_{1}+\alpha ((1+\mu)^{\frac{3}{2}}-1) e^{-z}
                +2((1+\mu)^{\frac{3}{2}}-1)b_{1}^{2}
            =\dot{\Theta}
			-\frac{3}{2}b_{1}\dot{\mu}
			+O\left(\frac{1}{s^{3}\log s}\right).
		\end{equation*}
		Combining the above estimates with Lemma~\ref{le:RAB} and Lemma~\ref{le:controlpara}, we obtain 
		\begin{equation}\label{est:I112}
			\begin{aligned}
				&\int_{\RR}\vec{N}\cdot (M_{2Q},M_{3Q})\left(
				\partial_{y}^{2}\varepsilon-\varepsilon \phi_{1}+(V+\varepsilon)^{5}-V^{5}
				\right)\dd y\\
				&\simeq
				\dot{\Theta}(\varepsilon\phi_{1},X_{1}\phi)-\frac{3}{2}b_{1}\dot{\mu}(\varepsilon\phi_{1},X_{1}\phi).
			\end{aligned}
		\end{equation}
		Similar as above, from $\mathcal{L}\Lambda Q=-2Q$, $(\varepsilon,Q_{2})=(\varepsilon,\Gamma(\Lambda Q))=0$ and (i) of Lemma~\ref{le:RAB}, we deduce that 
		\begin{equation}\label{est:I113}
			\begin{aligned}
				& (\dot{\mu}-2\Theta)\int_{\RR}M_{4Q}\left(
				\partial_{y}^{2}\varepsilon-\varepsilon \phi_{1}+(V+\varepsilon)^{5}-V^{5}
				\right)\dd y\simeq \frac{1}{4}b_{1}\dot{\mu}(\varepsilon\phi_{1},X_{1}\phi).
			\end{aligned}
		\end{equation}
		Here, we use the fact that 
		\begin{equation*}
        \begin{aligned}
        \left|\left(\varepsilon\phi_{1},\Gamma(\Lambda X)\phi\right)\right|\lesssim \mathcal{N}_{B}(\varepsilon)\left(\int_{\RR}|\Gamma(\Lambda X)|^{2}\phi_{1}\dd y\right)^{\frac{1}{2}}\lesssim \frac{C_{0}}{s^{\frac{5}{4}}\log s},\\
			\left|\left(\varepsilon\phi_{1},\Gamma(\Lambda B)\phi\right)\right|\lesssim \mathcal{N}_{B}(\varepsilon)\left(\int_{\RR}|\Gamma(\Lambda B)|^{2}\phi_{1}\dd y\right)^{\frac{1}{2}}\lesssim \frac{C_{0}}{s^{\frac{3}{4}}\log s}.
            \end{aligned}
		\end{equation*}
		
		On the other hand, from Lemma~\ref{le:L2e} and Lemma~\ref{le:Psinorm}, we find 
		\begin{equation}\label{est:I114}
			\begin{aligned}
				& \left| \int_{\RR}\Psi_{2}(V) \left(\partial_{y}^{2}\varepsilon-\varepsilon \phi_{1}+(V+\varepsilon)^{5}-V^{5}\right)\dd y\right|\\
				&\lesssim \left(\|\Psi_{2}(V)\|_{\dot{H}^{1}}+\left(\int_{\RR}|\Psi_{2}(V)|^{2}\phi_{1}\dd y\right)^{\frac{1}{2}}\right)\mathcal{N}_{B}(\varepsilon)\\
				&+\|\Psi_{2}(V)\|_{L^{2}}\left(\|R\phi\|_{L^{\infty}}^{4}+\|U\|_{L^{\infty}}^{4}+\|\varepsilon\|_{L^{\infty}}^{4}\right)\|\varepsilon\|_{L^{2}}\lesssim 
				\frac{C_{0}}{s^{4}\log^{2}s}.
			\end{aligned}
		\end{equation}
		Similarly, we also find 
		\begin{equation}\label{est:I115}
			\begin{aligned}
				\left| \int_{\RR}\Psi_{3}(V) \left(\partial_{y}^{2}\varepsilon-\varepsilon \phi_{1}+(V+\varepsilon)^{5}-V^{5}\right)\dd y\right|
				&\lesssim 
				\frac{C_{0}}{s^{4}\log^{2}s},\\
				\left| \int_{\RR}\Psi_{4}(V) \left(\partial_{y}^{2}\varepsilon-\varepsilon \phi_{1}+(V+\varepsilon)^{5}-V^{5}\right)\dd y\right|
				&\lesssim 
				\frac{C_{0}}{s^{4}\log^{2}s}.
			\end{aligned}
		\end{equation}
		In particular, using again (i) of Lemma~\ref{le:Psinorm}, we have 
		\begin{equation}\label{est:I116}
			\begin{aligned}
				& \int_{\RR}\Psi_{1}(V) \left(\partial_{y}^{2}\varepsilon-\varepsilon \phi_{1}+(V+\varepsilon)^{5}-V^{5}\right)\dd y\\
                &\simeq\zeta_{2}(\dot{z}-b_{1}z-\mu)\int_{\RR}
				\partial_{y}X_{2}\left(\partial_{y}^{2}\varepsilon-\varepsilon \phi_{1}+(V+\varepsilon)^{5}-V^{5}\right)\dd y\\
				&+10m_{0}b_{1}(\dot{z}-b_{1}z-\mu)\int_{\RR}
				\partial_{y}Y_{2}\left(\partial_{y}^{2}\varepsilon-\varepsilon \phi_{1}+(V+\varepsilon)^{5}-V^{5}\right)\dd y.
			\end{aligned}
		\end{equation}
		Last, from the definition of $V$ in~\eqref{equ:defV}, we rewrite 
		\begin{equation*}
			\begin{aligned}
				\Lambda V
				&=
				\Lambda Q-
				\Gamma(\Lambda Q)
				-z\partial_{y}Q_{2}+b_{1}(\Lambda X)\phi+b_{1}yX\partial_{y}\phi\\
				&-\zeta_{2}\Gamma(\Lambda X)\phi-\zeta_{2}yX_{2}\partial_{y}\phi-z\zeta_{2}(\partial_{y}X_{2})\phi
				+e^{-z}\Lambda ((A_{1}+B_{2})\phi)\\
				&-\frac{10m_{0}b_{1}\Gamma (\Lambda Y)}{(1+\mu)^{\frac{1}{4}}}-\frac{10m_{0}b_{1}z\partial_{y}Y_{2}}{(1+\mu)^{\frac{1}{4}}}+5(2\alpha m_{0}z+a_{0})e^{-z}\Lambda Y_{2}+\Lambda U_{1}.
			\end{aligned}
		\end{equation*}
		Note that, from the bootstrap assumption~\eqref{est:Boot} and Lemma~\ref{le:RAB}, 
		\begin{equation*}
			\begin{aligned}
				b_{1}(\Lambda X)\phi-\zeta_{2}(\Gamma(\Lambda X))\phi&=O_{H^{1}}\left(\frac{1}{s\log^{\frac{1}{2}}s}\right),\\
				b_{1}yX_{1}\partial_{y}\phi-\zeta_{2}yX_{2}\partial_{y}\phi&=O_{H^{1}}\left(\frac{1}{s\log^{\frac{1}{2}}s}\right).
			\end{aligned}
		\end{equation*}
		It follows from Lemma~\ref{le:RAB} that 
		\begin{equation}\label{est:LambdaVenergy}
			\begin{aligned}
				\Lambda V&=\Lambda Q-\Gamma (\Lambda Q)-z\partial_{y}Q_{2}-z\zeta_{2}(\partial_{y}X_{2})\phi\\
				&-10m_{0}b_{1}z\partial_{y}Y_{2}+O_{H^{1}}\left(\frac{1}{s\log^{\frac{1}{2}}s}\right).
			\end{aligned}
		\end{equation}
		Based on the above estimate and Lemma~\ref{le:controlpara}, we obtain
		\begin{equation*}
			\begin{aligned}
				\left(\frac{\dot{\lambda}}{\lambda}+b_{1}\right)\Lambda V&=
				\left(\frac{\dot{\lambda}}{\lambda}+b_{1}\right)
				\Lambda Q-
				\left(\frac{\dot{\lambda}}{\lambda}+b_{1}\right)
				\Gamma (\Lambda Q)\\
				&-z \left(\frac{\dot{\lambda}}{\lambda}+b_{1}\right)
				\partial_{y}Q_{2}-z\zeta_{2} \left(\frac{\dot{\lambda}}{\lambda}+b_{1}\right)(\partial_{y}X_{2})\phi\\
				&-10m_{0}b_{1}z \left(\frac{\dot{\lambda}}{\lambda}+b_{1}\right)\partial_{y}Y_{2}+O_{H^{1}}\left(\frac{C_{0}}{s^{\frac{5}{2}}\log^{\frac{3}{2}}s}\right).
			\end{aligned}
		\end{equation*}
		Therefore, using a similar argument as in the proof for~\eqref{est:I111} and \eqref{est:I114}, 
		\begin{equation}\label{est:I117}
			\begin{aligned}
				&\left(\frac{\dot{\lambda}}{\lambda}+b_{1}\right)\int_{\RR}\Lambda V
				\left(\partial_{y}^{2}\varepsilon-\varepsilon \phi_{1}+(V+\varepsilon)^{5}-V^{5}\right)\dd y\\
				&\simeq -5z \left(\frac{\dot{\lambda}}{\lambda}+b_{1}\right)
				\int_{\RR}\partial_{y}Q_{2} \left(V^{4}-Q_{1}^{4}-Q_{2}^{4}\right)\varepsilon\dd y
				\\
				&-z\zeta_{2} \left(\frac{\dot{\lambda}}{\lambda}+b_{1}\right)\int_{\RR}\partial_{y}X_{2} \left(\partial_{y}^{2}\varepsilon-\varepsilon \phi_{1}+(V+\varepsilon)^{5}-V^{5}\right)\dd y\\
				&-10m_{0}b_{1}z \left(\frac{\dot{\lambda}}{\lambda}+b_{1}\right)\int_{\RR}\partial_{y}Y_{2}\left(\partial_{y}^{2}\varepsilon-\varepsilon \phi_{1}+(V+\varepsilon)^{5}-V^{5}\right)\dd y.
			\end{aligned}
		\end{equation}
		Combining the estimates~\eqref{est:I111}--\eqref{est:I116} with~\eqref{est:I117}, we obtain
		\begin{equation*}
			\begin{aligned}
				\mathcal{I}_{1,1}&\simeq 
				2\dot{\Theta}(\varepsilon\phi_{1},X_{1}\phi)
				-\frac{5}{2}b_{1}\dot{\mu}(\varepsilon\phi_{1},X_{1}\phi)
				\\
				&+10\mathcal{Z}\int_{\RR}\partial_{y}Q_{2}(V^{4}-Q^{4}_{1}-Q^{4}_{2})\varepsilon\dd y\\
				&+2\zeta_{2}
				\mathcal{Z}
				\int_{\RR}
				\partial_{y}X_{2}\left(\partial_{y}^{2}\varepsilon-\varepsilon \phi_{1}+(V+\varepsilon)^{5}-V^{5}\right)\dd y\\
				&+20m_{0}b_{1}
				\mathcal{Z}
				\int_{\RR}
				\partial_{y}Y_{2}\left(\partial_{y}^{2}\varepsilon-\varepsilon \phi_{1}+(V+\varepsilon)^{5}-V^{5}\right)\dd y.
			\end{aligned}
		\end{equation*}
		Here, we denote 
		\begin{equation*}
			\mathcal{Z}=\dot{z}-b_{1}z-\mu+z\left(\frac{\dot{\lambda}}{\lambda}+b_{1}\right).
		\end{equation*}
		Then, from the Cauchy-Schwarz inequality, we directly have 
		\begin{equation*}
			\begin{aligned}
				\left|
				\int_{\RR}
				\partial_{y}X_{2}\left(\partial_{y}^{2}\varepsilon-\varepsilon \phi_{1}+(V+\varepsilon)^{5}-V^{5}\right)\dd y\right|&\lesssim \mathcal{N}_{B}(\varepsilon)\lesssim \frac{C_{0}}{s^{\frac{3}{2}}\log s},\\
				\left|
				\int_{\RR}
				\partial_{y}Y_{2}\left(\partial_{y}^{2}\varepsilon-\varepsilon \phi_{1}+(V+\varepsilon)^{5}-V^{5}\right)\dd y\right|&\lesssim 
				\mathcal{N}_{B}(\varepsilon)\lesssim
				\frac{C_{0}}{s^{\frac{3}{2}}\log s}.
			\end{aligned}
		\end{equation*}
        In addition, from Lemma Lemma~\ref{le:boundinter} and Definition~\ref{def:Admissible}, 
        \begin{equation*}
            \left|\int_{\RR}\partial_{y}Q_{2}(V^{4}-Q^{4}_{1}-Q^{4}_{2})\varepsilon\dd y\right|\lesssim \left(ze^{z}+\|R\phi\|_{L^{\infty}}+\|U\|_{L^{\infty}}\right)\mathcal{N}_{B}(\varepsilon)\lesssim \frac{C_{0}}{s^{\frac{5}{2}}\log^{2} s}.
        \end{equation*}
		We see that~\eqref{est:I11} follows directly from the above estimate and Lemma~\ref{le:controlpara}.

		\smallskip
		\emph{Estimate on $\mathcal{I}_{1,2}$.} We claim that 
		\begin{equation}\label{est:I12}
			\mathcal{I}_{1,2}\simeq 20z\frac{\dot{\lambda}}{\lambda}\int_{\RR}(\partial_{y}Q_{2})(Q_{2}^{3}\varepsilon^{2})\dd y+O\left(s^{-\frac{1}{2}}\int_{\RR}\varepsilon^{2}\partial_{y}\phi_{1}\dd y\right).
		\end{equation}
		Indeed, from integration by parts, we find 
		\begin{equation*}
			\begin{aligned}
				\int_{\RR}\Lambda \varepsilon(-\partial_{y}^{2}\varepsilon+\varepsilon\phi_{1})\dd y
				&=\int_{\RR}(\partial_{y}\varepsilon)^{2}\dd y-\frac{B}{2}s^{\frac{1}{2}}\int_{\RR}\varepsilon^{2}\partial_{y}\phi_{1}\dd y\\
				&-\frac{B}{2}\int_{\RR}\varepsilon^{2}\left(\frac{y}{B}-s^{\frac{1}{2}}\right)\partial_{y}\phi_{1}\dd y.
			\end{aligned}
		\end{equation*}
		Note that, from the definition of $\phi_{1}$ in~\eqref{equ:defphi1},
		\begin{equation*}
			\left| \left(\frac{y}{B}-s^{\frac{1}{2}}\right)\partial_{y}\phi_{1}\right|\lesssim |\partial_{y}\phi_{1}|^{\frac{99}{100}}.
		\end{equation*}
		It follows from Lemma~\ref{le:controlpara} and H\"older inequality that 
		\begin{equation*}
			\begin{aligned}
				&\left|\frac{\dot{\lambda}}{\lambda}\int_{\RR}\varepsilon^{2}\left(\frac{y}{B}-s^{\frac{1}{2}}\right)\partial_{y}\phi_{1}\dd y\right|\\
				&\lesssim \left(\int_{\RR}\varepsilon^{2}|\partial_{y}\phi_{1}|\dd y\right)^{\frac{99}{100}}\left(\frac{\|\varepsilon\|^{\frac{1}{50}}_{L^{2}}}{s\log s}\right)\lesssim s^{-\frac{1}{2}}\int_{\RR}\varepsilon^{2}|\partial_{y}\phi_{1}|\dd y+s^{-10}.
			\end{aligned}
		\end{equation*}
		Combining the above estimates with Lemma~\ref{le:controlpara}, we find 
		\begin{equation*}
			\begin{aligned}
				2\frac{\dot{\lambda}}{\lambda}\int_{\RR}\Lambda \varepsilon(-\partial_{y}^{2}\varepsilon+\varepsilon\phi_{1})\dd y
				&=O\left(s^{-\frac{1}{2}}\int_{\RR}\varepsilon^{2}|\partial_{y}\phi_{1}|\dd y\right)\\
				&+O\left(\frac{C_{0}^{2}}{s^{4}\log^{3}s}+\frac{1}{s^{10}}\right).
			\end{aligned}
		\end{equation*}
		On the other hand, we decompose 
		\begin{equation*}
			\begin{aligned}
				(V+\varepsilon)^{5}-V^{5}
				&=5(Q_{1}^{4}+Q_{2}^{4})\varepsilon+5(V^{4}-Q_{1}^{4}-Q_{2}^{4})\varepsilon\\
				&+10V^{3}\varepsilon^{2}+5V^{2}\varepsilon^{3}+10V\varepsilon^{4}+\varepsilon^{5}.
			\end{aligned}
		\end{equation*}
		Using integration by parts, we directly have 
		\begin{equation*}
			\begin{aligned}
				-2\int_{\RR}\Lambda \varepsilon \left(5(Q_{1}^{4}+Q_{2}^{4})\varepsilon\right)\dd y
				&=
				20z\int_{\RR}(\partial_{y}Q_{2})
				\left( Q_{2}^{3} \varepsilon^{2}\right)\dd y\\
				&+20\int_{\RR}(y\partial_{y}Q_{1})(Q_{1}^{3}\varepsilon^{2})\dd y\\
				&+20\int_{\RR}((y-z)\partial_{y}Q_{2})(Q_{2}^{3}\varepsilon^{2})\dd y,
			\end{aligned}
		\end{equation*}
		which implies that 
		\begin{equation*}
			-2\frac{\dot{\lambda}}{\lambda}\int_{\RR}\Lambda \varepsilon \left(5(Q_{1}^{4}+Q_{2}^{4})\varepsilon\right)\dd y\simeq20z\frac{\dot{\lambda}}{\lambda}\int_{\RR}(\partial_{y}Q_{2})
			\left( Q_{2}^{3} \varepsilon^{2}\right)\dd y.
		\end{equation*}
		Then, using an elementary computation, we have 
		\begin{equation*}
			V^{4}-Q_{1}^{4}-Q_{2}^{4}=\left(S^{4}-Q_{1}^{4}-Q_{2}^{4}\right)
			+O\left(|S|^{3}|R\phi+U|+|R\phi+U|^{4}\right),
		\end{equation*}
		which implies that 
		\begin{equation*}
			\begin{aligned}
				&\bigg|\frac{\dot{\lambda}}{\lambda}\int_{\RR}\Lambda \varepsilon 
				\left(5(V^{4}-Q_{1}^{4}-Q_{2}^{4})\varepsilon\right)
				\dd y\bigg|\\
				&\lesssim\frac{1}{s\log s}\left(\|yQ_{1}\|_{L^{\infty}}+\|yQ_{2}\|_{L^{\infty}}\right)\|R\phi+U\|_{L^{\infty}}\mathcal{N}_{B}^{2}(\varepsilon)\\
				&+ \frac{1}{s\log s}\left(
				\mathcal{N}_{B}^{2}(\varepsilon)+\|\langle y\rangle(R\phi+U)^{4}\|_{L^{\infty}}\|\varepsilon\|_{H^{1}}^{2}
				\right)\lesssim \frac{C_{0}^{2}}{s^{4}\log^{3}s}.
			\end{aligned}
		\end{equation*}
		Based on a similar argument, we also have 
		\begin{equation*}
			\begin{aligned}
				&\bigg|\frac{\dot{\lambda}}{\lambda}\int_{\RR}\Lambda \varepsilon 
				\left(
				10V^{3}\varepsilon^{2}+5V^{2}\varepsilon^{3}+10V\varepsilon^{4}+\varepsilon^{5}
				\right)
				\dd y\bigg|\\
				&\lesssim \frac{\mathcal{N}_{B}^{2}(\varepsilon)\|\varepsilon\|_{L^{\infty}}}{s\log s}
				\left(1+\|yQ_{1}\|_{L^{\infty}}+\|yQ_{2}\|_{L^{\infty}}+\|\varepsilon\|_{L^{\infty}}^{2}\right)\\
				&+\frac{\mathcal{N}_{B}(\varepsilon)\|\varepsilon\|_{L^{\infty}}}{s\log s}
				\|\varepsilon\|_{L^{2}}
				\left(\|y(R\phi+U)^{3}\|_{L^{\infty}}+\|\varepsilon\|_{L^{\infty}}^{2}\|yV\|_{L^{\infty}}\right)\\
				&+\frac{\|\varepsilon\|_{L^{2}}^{2}\|\varepsilon\|_{L^{\infty}}}{s\log s}\left(\|R\phi\|_{L^{\infty}}^{3}+\|U\|_{L^{\infty}}^{3}+\|\varepsilon\|_{L^{\infty}}^{3}\right)\\
				&+\frac{1}{s\log s}\|R\phi+U\|_{L^{\infty}}\|\varepsilon\|_{L^{\infty}}^{3}\|\varepsilon\|^{2}_{L^{2}}\lesssim \frac{C_{0}^{2}}{s^{4}\log^{3}s}.
			\end{aligned}
		\end{equation*}
		We see that~\eqref{est:I12} follows directly from the above estimates.
		
		\smallskip
		\emph{Estimate on $\mathcal{I}_{1,3}$.} We claim that 
		\begin{equation}\label{est:I13}
			\mathcal{I}_{1,3}\ge -\frac{1}{2}\int_{\RR}\varepsilon^{2}\partial_{y}\phi_{1}\dd y+
			O\left(\frac{C_{0}^{2}}{s^{4}\log^{3}s}\right).
		\end{equation}
		Indeed, from integration by parts, we compute 
		\begin{equation*}
			\begin{aligned}
				\mathcal{I}_{1,3}&=-3\int_{\RR}(\partial_{y}\varepsilon)^{2}\partial_{y}\phi_{1}\dd y
				-\int_{\RR}\varepsilon^{2}\partial_{y}\phi_{1}\dd y\\
				&-2\int_{\RR}\partial_{y}\varepsilon\left((V+\varepsilon)^{5}-V^{5}\right)
				(1-\phi_{1})
				\dd y\\
				&+2\int_{\RR}\varepsilon\left((V+\varepsilon)^{5}-V^{5}\right)\partial_{y}\phi_{1}\dd y
				+\int_{\RR}\varepsilon^{2}\partial_{y}^{3}\phi_{1}\dd y.
			\end{aligned}
		\end{equation*}
		First, from the definition of $\phi_{1}$ in~\eqref{equ:defphi1},
		\begin{equation*}
			\begin{aligned}
				|\partial_{y}\phi_{1}|\lesssim s^{-2},\ \ \mbox{on}\ \left(-\infty,s^{\frac{1}{2}}\right)&\Longrightarrow
				S^{4}|\partial_{y}\phi_{1}|\lesssim s^{-2}(Q_{1}+Q_{2}),\\
				|1-\phi_{1}|\lesssim s^{-2},\ \ \mbox{on}\ \left(-\infty,s^{\frac{1}{2}}\right)&\Longrightarrow
				S^{4}|1-\phi_{1}|\lesssim s^{-2}(Q_{1}+Q_{2}).
			\end{aligned}
		\end{equation*}
		It follows from the bootstrap assumption~\eqref{est:Boot} and Lemma~\ref{le:L2e} that 
		\begin{equation*}
			\begin{aligned}
				&\bigg|\int_{\RR}\varepsilon\left((V+\varepsilon)^{5}-V^{5}\right)\partial_{y}\phi_{1}\dd y\bigg|\\
				&+\bigg|\int_{\RR}\partial_{y}\varepsilon\left((V+\varepsilon)^{5}-V^{5}\right)
				(1-\phi_{1})
				\dd y\bigg|\\
				&\lesssim s^{-2}\mathcal{N}_{B}^{2}(\varepsilon)
				+\|\varepsilon\|_{L^{2}}^{2}\left(\|R\phi+U\|_{L^{\infty}}^{4}+\|\varepsilon\|_{L^{\infty}}^{4}\right)
				\\
				&+\mathcal{N}_{B}(\varepsilon)\|\varepsilon\|_{L^{2}}\left(\|R\phi+U\|_{L^{\infty}}^{4}
				+\|\varepsilon\|_{L^{\infty}}^{4}\right)\lesssim
				\frac{C_{0}^{2}}{s^{4}\log^{3}s}.
			\end{aligned}
		\end{equation*}
		Second, using again the definition of $\phi_{1}$ in~\eqref{equ:defphi1}, 
		\begin{equation*}
			| \partial_{y}^{3}\phi_{1}|\lesssim B^{-2}|\partial_{y}\phi_{1}|\Longrightarrow
			\int_{\RR}\varepsilon^{2}\partial_{y}^{3}\phi_{1}\dd y=O\left(B^{-2}\int_{\RR}\varepsilon^{2}|\partial_{y}\phi_{1}|\dd y\right).
		\end{equation*}
		We see that~\eqref{est:I13} follows directly from the above estimates and $\partial_{y}\phi_{1}<0$.
		
		\smallskip
		\emph{Estimates on $\mathcal{I}_{1,4}$ and $\mathcal{I}_{1,5}$.}
		We claim that 
		\begin{equation}\label{est:I14I15}
			\mathcal{I}_{1,4}+\mathcal{I}_{1,5}=O\left(\frac{C_{0}^{2}}{s^{4}\log^{3}s}\right).
		\end{equation}
		Indeed, we decompose again
		\begin{equation*}
			\begin{aligned}
				\partial_{y}^{2}\varepsilon-\varepsilon \phi_{1}+(V+\varepsilon)^{5}-V^{5}&= \partial_{y}^{2}\varepsilon-\varepsilon
				+5(Q_{1}^{4}+Q_{2}^{4})\varepsilon\\
				&+(1-\phi_{1})\varepsilon
				+5(V^{4}-Q_{1}^{4}-Q_{2}^{4})\varepsilon\\
				&+10V^{3}\varepsilon^{2}+10V^{2}\varepsilon^{3}+5V\varepsilon^{4}+\varepsilon^{5}.
			\end{aligned}
		\end{equation*}
		Based on the above identity and ${\rm{Ker}}\mathcal{L}={\rm{Span}}\{Q'\}$, we compute 
		\begin{equation*}
			\begin{aligned}
				& \Big|\Big(\frac{\dot{x}_{1}}{\lambda}-1\Big)
				\int_{\RR}\partial_{y}S\left(-\partial_{y}^{2}\varepsilon+\varepsilon\phi_{1}-(V+\varepsilon)^{5}+V^{5}\right)\dd y\Big|\\
				&\lesssim\frac{C_{0}}{s^{\frac{3}{2}}\log s}\left(\|R\phi+U\|_{L^{\infty}}\mathcal{N}_{B}(\varepsilon)+\|R\phi+U\|_{L^{\infty}}^{4}\mathcal{N}_{B}(\varepsilon)
				+\mathcal{N}_{B}^{2}(\varepsilon)
				\right)\\
				&+ \frac{C_{0}}{s^{\frac{3}{2}}\log s}\left(\|\partial_{y}Q_{1}|^{\frac{1}{2}}(1-\phi_{1})\|_{L^{\infty}}+\||\partial_{y}Q_{2}|^{\frac{1}{2}}(1-\phi_{1})\|_{L^{2}}
                +\|Q_{2}Q^{2}_{1}\|_{L^{2}}
                \right)\mathcal{N}_{B}(\varepsilon)\\
				&+\frac{C_{0}}{s^{\frac{3}{2}}\log s}\left(\|(\partial_{y}Q_{1})Q^{2}_{2}\|_{L^{2}}+\|(\partial_{y}Q_{2})Q^{2}_{1}\|_{L^{2}}
                +\|Q_{1}Q^{2}_{2}\|_{L^{2}}
                \right)\mathcal{N}_{B}(\varepsilon)\lesssim \frac{C_{0}^{2}}{s^{4}\log^{3}s}.
			\end{aligned}
		\end{equation*}
		On the other hand, using an argument similar to that in the proof for Lemma~\ref{le:RAB}, 
		\begin{equation*}
			\|\partial_{y}(R\phi)\|_{H^{1}}+\|\partial_{y}U\|_{H^{1}}
			\lesssim \|(\partial_{y}R)\phi\|_{H^{1}}+\|\partial_{y}U\|_{H^{1}}
			+\|R\partial_{y}\phi\|_{H^{1}}\lesssim \frac{1}{s\log s}.
		\end{equation*}
		It follows directly from the bootstrap assumption~\eqref{est:Boot} and Lemma~\ref{le:L2e} that 
		\begin{equation*}
			\begin{aligned}
				&\Big|\Big(\frac{\dot{x}_{1}}{\lambda}-1\Big)
				\int_{\RR}(\partial_{y}(R\phi+U))\left(-\partial_{y}^{2}\varepsilon+\varepsilon\phi_{1}-(V+\varepsilon)^{5}+V^{5}\right)\dd y\Big|\\
				&\lesssim \frac{C_{0}}{s^{\frac{3}{2}}\log s}\|\partial_{y}(R\phi+U)\|_{H^{1}}\left(\|R\phi+U\|_{L^{\infty}}^{4}+\|\varepsilon\|_{L^{\infty}}^{4}\right)\|\varepsilon\|_{L^{2}}\\
				&+ \frac{C_{0}}{s^{\frac{3}{2}}\log s}\|\partial_{y}(R\phi+U)\|_{H^{1}}\mathcal{N}_{B}(\varepsilon)\lesssim \frac{C_{0}^{2}}{s^{4}\log^{3}s}.
			\end{aligned}
		\end{equation*}
		Similar as above, we also check that 
		\begin{equation*}
			\begin{aligned}
				\left|   \mathcal{I}_{1,5}\right|&\lesssim
				\frac{C_{0}}{s^{\frac{3}{2}}\log s}\mathcal{N}_{B}(\varepsilon)\|\varepsilon\|_{L^{2}}
				\left(\|R\phi+U\|_{L^{\infty}}^{4}
				+\|\varepsilon\|_{L^{\infty}}^{4}
				\right)
				\\
				&+ \frac{C_{0}}{s^{\frac{3}{2}}\log s}\mathcal{N}_{B}^{2}(\varepsilon)
				\lesssim \frac{C_{0}^{3}}{s^{\frac{9}{2}}\log^{3}s}+\frac{C^{2}_{0}}{s^{5}\log^{2}s}.
			\end{aligned}
		\end{equation*}
		We see that~\eqref{est:I14I15} follows directly from the above estimates.
		
		\smallskip
		Combining ~\eqref{est:I11},~\eqref{est:I12},~\eqref{est:I13} with \eqref{est:I14I15}, we complete the proof for~\eqref{est:I1}.

		\smallskip
		\emph{Estimate on $\mathcal{I}_{2}$.}
		We claim that 
		\begin{equation}\label{est:I2}
			\mathcal{I}_{2}\simeq 20\dot{z}\int_{\RR}(\partial_{y}Q_{2})\left(Q_{2}^{3}\varepsilon^{2}\right)\dd y
			+O\left(\frac{B}{s^{\frac{1}{2}}}\int_{\RR}\varepsilon^{2}|\partial_{y}\phi_{1}|\dd y\right).
		\end{equation}
		Indeed, from the definition of $V$ in~\eqref{equ:defV}, we compute\footnote{See the proof of~\eqref{equ:dsV}--\eqref{est:H4} for more detail.} 
		\begin{equation}\label{equ:dsV2}
			\begin{aligned}
				\partial_{s}V&=\dot{z}M_{1Q}+\dot{\zeta}_{2}M_{2Q}+\dot{b}_{1}{M_{3Q}}+\dot{\mu}M_{4Q}
				\\
				&+2\dot{b}_{1}b_{1}(E_{1}+F_{2})\phi-\dot{z}b_{1}^{2}\phi\partial_{y}F_{2} +R\partial_{s}\phi\\
				&+b_{1}^{2}(E_{1}+F_{2})\partial_{s}\phi+30m_{0}z\dot{b}_{1}b_{1}Y_{2}+\frac{\dot{\mu}b_{1}^{2}\phi}{2(1+\mu)}\Gamma (\Lambda F)
				\\
				&-10\alpha m_{0}\dot{z}ze^{-z}(Y_{2}+\partial_{y}Y_{2})
				-\dot{z}e^{-z}(A_{1}+B_{2})\phi
				-\dot{z}e^{-z}\phi\partial_{y}B_{2}\\
                &+\dot{z}\left(\zeta_{2}\phi \partial_{y}X_{2}+{10m_{0}b_{1}}\partial_{y}Y_{2}\right)+\frac{5}{2}m_{0}b_{1}\dot{\mu}\left(Y_{2}-2\Gamma(\Lambda Y)\right)
                +O_{\mathcal{S}}(1).
			\end{aligned}
		\end{equation}
		It follows from~\eqref{est:Boot} and Lemma~\ref{le:controlpara} that 
		\begin{equation}\label{est:psV}
			\|\partial_{s}V\|_{L^{\infty}}\lesssim s^{-1}\quad \mbox{and}\quad 
			\partial_{s}V=\dot{z}\partial_{y}Q_{2}+O_{L^{\infty}}\left(\frac{1}{s\log s}\right).
		\end{equation}
		On the one hand, using again~\eqref{est:Boot} and Lemma~\ref{le:L2e}, we have 
		\begin{equation*}
			\begin{aligned}
				\int_{\RR}|\partial_{s}V||(R\phi+U)|^{3}\varepsilon^{2}\dd y
				&\lesssim \|\partial_{s}V\|_{L^{\infty}}\|R\phi+U\|_{L^{\infty}}^{3}\|\varepsilon\|^{2}_{L^{2}}\lesssim \frac{C_{0}^{2}}{s^{5}\log^{3}s},\\
				\int_{\RR}|\partial_{s}V||S^{2}(R\phi+U)|\varepsilon^{2}\dd y&\lesssim \|\partial_{s}V\|_{L^{\infty}}\|R\phi+U\|_{L^{\infty}}\mathcal{N}_{B}^{2}(\varepsilon)\lesssim \frac{C_{0}^{2}}{s^{5}\log^{3}s}.
			\end{aligned}
		\end{equation*}
		It follows from~\eqref{est:psV} and Lemma~\ref{le:boundinter} that  
		\begin{equation*}
			\begin{aligned}
				\int_{\RR}(\partial_{s}V)V^{3}\varepsilon^{2}\dd y
				&=\int_{\RR}(\partial_{s}V)S^{3}\varepsilon^{2}\dd y+O\left(\frac{C_{0}^{2}}{s^{5}\log^{3}s}\right)\\
				&=-\dot{z}\int_{\RR}(\partial_{y}Q_{2})\left(Q_{2}^{3}\varepsilon^{2}\right)\dd y+O\left(\frac{C_{0}^{2}}{s^{4}\log^{3}s}\right).
			\end{aligned}
		\end{equation*}
		On the other hand, using again~\eqref{est:psV} and Lemma~\ref{le:L2e}, 
		\begin{equation*}
			\begin{aligned}
				& \int_{\RR}|\partial_{s}V|
				\left(|V|^{2}|\varepsilon|^{3}
				+|V||\varepsilon|^{4}+|\varepsilon|^{5}
				\right)
				\dd y\\
				&\lesssim \|\partial_{s}V\|_{L^{\infty}}
				\left(  \|V\|_{L^{\infty}}^{2}
				\|\varepsilon\|_{L^{\infty}}
				+\|\varepsilon\|_{L^{\infty}}^{3}
				\right)
				\| \varepsilon\|_{L^{2}}^{2}\lesssim s^{-\frac{17}{4}}.
			\end{aligned}
		\end{equation*}
		Last, from the definition of $\phi_{1}$ in~\eqref{equ:defphi1},
		\begin{equation*}
			|\partial_{s}\phi_{1}|\lesssim \frac{B}{s^{\frac{1}{2}}}|\partial_{y}\phi_{1}|
			\Longrightarrow
			\int_{\RR}\varepsilon^{2}\partial_{s}\phi_{1}\dd y
			=O\left(\frac{B}{s^{\frac{1}{2}}}\int_{\RR}\varepsilon^{2}|\partial_{y}\phi_{1}|\dd y\right).
		\end{equation*}
		Combining the above estimates,
		we directly complete the proof for~\eqref{est:I2}.
		
		\smallskip
		Note that, from~\eqref{est:I1} and~\eqref{est:I2}, we directly have 
		\begin{equation*}
			\begin{aligned}
				\frac{\dd \mathcal{F}_{1}}{\dd s}&\ge 
				2\dot{\Theta}(\varepsilon\phi_{1},X_{1}\phi)-\frac{5}{2}b_{1}\dot{\mu}(\varepsilon\phi_{1},X_{1}\phi)
				-\frac{1}{8}\int_{\RR}\varepsilon^{2}\partial_{y}\phi_{1}\dd y\\
				&+20\left(\dot{z}+z\frac{\dot{\lambda}}{\lambda}\right)
				\int_{\RR}(\partial_{y}Q_{2})(Q_{2}^{3}\varepsilon^{2})\dd y
				+
				O\left(\frac{C_{0}^{2}}{s^{4}\log^{\frac{5}{2}}s}+\frac{C_{0}}{s^{4}\log^{2}s}\right).
			\end{aligned}
		\end{equation*}
		Note also that, from Lemma~\ref{le:controlpara}, we obtain 
		\begin{equation*}
			20\left(\dot{z}+z\frac{\dot{\lambda}}{\lambda}\right)
			\int_{\RR}(\partial_{y}Q_{2})(Q_{2}^{3}\varepsilon^{2})\dd y
			\simeq 20\mu \int_{\RR}(\partial_{y}Q_{2})(Q_{2}^{3}\varepsilon^{2})\dd y.
		\end{equation*}
		Combining the above two estimates, we complete the proof for~\eqref{est:dsF1}.
		
		\smallskip
		\textbf{Step 2.} Estimate for the time variation of $\mathcal{F}_{2}$. We claim that 
		\begin{equation}\label{est:dsF2}
			\begin{aligned}
				\frac{\dd \mathcal{F}_{2}}{\dd s}&\simeq
				-20\mu \int_{\RR}(\partial_{y}Q_{2})(Q_{2}^{3}\varepsilon^{2})\dd y
				\\
				&-2\dot{\Theta}(\varepsilon\phi_{1},X_{1}\phi)+\frac{5}{2}b_{1}\dot{\mu}(\varepsilon\phi_{1},X_{1}\phi).
			\end{aligned}
		\end{equation}
		Indeed, we decompose 
		\begin{equation*}
			\begin{aligned}
				\frac{\dd \mathcal{F}_{2}}{\dd s}
				&=\frac{\dd }{\dd s}\left(\mu \int_{\RR}\varepsilon^{2}\phi_{2}\dd y\right)
				-
				2\left(b_{1}
				\left(\mathcal{J}_{2}+
				\left(c_{0}+\frac{5}{4}\right)\mathcal{J}_{1}\right)\right)
				\left( \frac{\dd }{\dd s}(\varepsilon \phi_{1},X_{1}\phi)\right)\\
				&-2\left(\frac{\dd }{\dd s}\left(b_{1}
				\left(\mathcal{J}_{2}+
				\left(c_{0}+\frac{5}{4}\right)\mathcal{J}_{1}\right)\right)\right)
				(\varepsilon \phi_{1},X_{1}\phi)
				=\mathcal{I}_{3}+\mathcal{I}_{4}+\mathcal{I}_{5}.
			\end{aligned}
		\end{equation*}
		
		\emph{Estimate on $\mathcal{I}_{3}$.} We claim that 
		\begin{equation}\label{est:I3}
			\mathcal{I}_{3}\simeq-20\mu \int_{\RR}(\partial_{y}Q_{2})(Q_{2}^{3}\varepsilon^{2})\dd y.
		\end{equation}
		Indeed, by an elementary computation, we find 
		\begin{equation*}
			\mathcal{I}_{3}=2\mu \int_{\RR}(\partial_{s}\varepsilon)\varepsilon\phi_{2}\dd y+\mu\int_{\RR}\varepsilon^{2}\partial_{s}\phi_{2}\dd y+\dot{\mu}\int_{\RR}\varepsilon^{2}\phi_{2}\dd y. 
		\end{equation*}
		First, from Lemma~\ref{le:eque}, we compute 
		\begin{equation*}
			\begin{aligned}
				\int_{\RR}(\partial_{s}\varepsilon)\varepsilon\phi_{2}\dd y
				&=\int_{\RR}\left(\partial_{y}^{2}\varepsilon-\varepsilon+(V+\varepsilon)^{5}-V^{5}\right)(\partial_{y}\varepsilon)\phi_{2}\dd y\\
				&+\int_{\RR}\left(\partial_{y}^{2}\varepsilon-\varepsilon+(V+\varepsilon)^{5}-V^{5}\right)\varepsilon(\partial_{y}\phi_{2})\dd y\\
				&+\bigg(\frac{\dot{\lambda}}{\lambda}+b_{1}\bigg)
				\int_{\RR}(\Lambda V)\varepsilon\phi_{2}\dd y+
				\left(\frac{\dot{x}_{1}}{\lambda}-1\right)\int_{\RR}\left(\partial_{y}V\right)\varepsilon\phi_{2}\dd y\\
				&-\int_{\RR}\Psi(V)\varepsilon\phi_{2}\dd y+\frac{\dot{\lambda}}{\lambda}\int_{\RR}(\Lambda \varepsilon)\varepsilon \phi_{2}\dd y
				+
				\left(\frac{\dot{x}_{1}}{\lambda}-1\right)\int_{\RR}\left(\partial_{y}\varepsilon\right)\varepsilon\phi_{2}\dd y.
			\end{aligned}
		\end{equation*}
		Recall that, we decompose 
		\begin{equation*}
			\begin{aligned}
				\partial_{y}^{2}\varepsilon-\varepsilon +(V+\varepsilon)^{5}-V^{5}&= \partial_{y}^{2}\varepsilon-\varepsilon
				+5(Q_{1}^{4}+Q_{2}^{4})\varepsilon\\
				&
				+5(V^{4}-Q_{1}^{4}-Q_{2}^{4})\varepsilon+\varepsilon^{5}\\
				&+10V^{3}\varepsilon^{2}+10V^{2}\varepsilon^{3}+5V\varepsilon^{4}. 
			\end{aligned}
		\end{equation*}
		Using integration by parts and Lemma~\ref{le:boundinter}, we check that 
		\begin{equation*}
			\begin{aligned}
				&\int_{\RR}\left(\partial_{y}^{2}\varepsilon-\varepsilon+5(Q_{1}^{4}+Q_{2}^{4})\varepsilon\right)(\partial_{y}\varepsilon)\phi_{2}\dd y\\
				&=-\frac{1}{2}\int_{\RR}(\partial_{y}\varepsilon)^{2}\partial_{y}\phi_{2}\dd y+\frac{1}{2}\int_{\RR}\varepsilon^{2}\partial_{y}\phi_{2}\dd y\\
				&-10\int_{\RR}(\partial_{y}Q_{2})(Q_{2}^{3}\varepsilon^{2})\dd y+O\left(\frac{C_{0}^{2}}{s^{4}\log^{2}s}\right).
			\end{aligned}
		\end{equation*}
		Note that, from the definition of $\phi_{2}$ in~\eqref{equ:defphi2}, we find 
		\begin{equation}\label{est:pointpy2}
			\partial_{y}\phi_{2}=\frac{1}{z}\chi'_{2}\left(\frac{y}{z}\right)\Longrightarrow |\partial_{y}\phi_{2}|\lesssim \frac{\phi_{1}}{\log s}.
		\end{equation}
		It follows from the bootstrap assumption~\eqref{est:Boot} that 
		\begin{equation*}
			\begin{aligned}
				&\int_{\RR}\left(\partial_{y}^{2}\varepsilon-\varepsilon+5(Q_{1}^{4}+Q_{2}^{4})\varepsilon\right)(\partial_{y}\varepsilon)\phi_{2}\dd y\\
				&=-10\int_{\RR}(\partial_{y}Q_{2})(Q_{2}^{3}\varepsilon^{2})\dd y
				+O\left(\frac{C_{0}^{2}}{s^{4}\log^{2}s}+\frac{C_{0}^{2}}{s^{3}\log^{3}s}\right).
			\end{aligned}
		\end{equation*}
		Then, using again Lemma~\ref{le:L2e} and the bootstrap assumption~\eqref{est:Boot},
		\begin{equation*}
			\begin{aligned}
				&\int_{\RR}  \left(|V^{4}-Q_{1}^{4}-Q_{2}^{4}||\varepsilon|+|\varepsilon|^{5}\right)|\partial_{y}\varepsilon|\phi_{2}\dd y\\
				&+\int_{\RR}\left(|V|^{3}\varepsilon^{2}+V^{2}|\varepsilon|^{3}+|V|\varepsilon^{4}\right) |\partial_{y}\varepsilon|\phi_{2}\dd y\\
				&\lesssim \left(\|Q_{1}^{3}Q_2\|_{L^{2}}+\|Q_{1}Q^{3}_2\|_{L^{2}}\right)\|\varepsilon\|_{L^{\infty}}\mathcal{N}_{B}(\varepsilon)\\
				&+\left(\|R\phi+U\|_{L^{\infty}}+\|\varepsilon\|_{L^{\infty}}\right)\mathcal{N}_{B}^{2}(\varepsilon)\lesssim \frac{C_{0}^{2}}{s^{4}\log^{2}s}.
			\end{aligned}
		\end{equation*}
		Combining the above two estimates, we deduce that 
		\begin{equation}\label{est:I31}
			\begin{aligned}
				& \int_{\RR}\left(\partial_{y}^{2}\varepsilon-\varepsilon+(V+\varepsilon)^{5}-V^{5}\right)(\partial_{y}\varepsilon)\phi_{2}\dd y\\
				&=-10\int_{\RR}(\partial_{y}Q_{2})(Q_{2}^{3}\varepsilon^{2})\dd y
				+O\left(\frac{C_{0}^{2}}{s^{4}\log^{2}s}+\frac{C_{0}^{2}}{s^{3}\log^{3}s}\right).
			\end{aligned}
		\end{equation}
		Using again~\eqref{est:pointpy2}, the bootstrap assumption~\eqref{est:Boot} and Lemma~\ref{le:L2e},
		\begin{equation}\label{est:I32}
			\begin{aligned}
				&\left|\int_{\RR}\left(\partial_{y}^{2}\varepsilon-\varepsilon+(V+\varepsilon)^{5}-V^{5}\right)\varepsilon(\partial_{y}\phi_{2})\dd y\right|\\
				&\lesssim\left(z^{-1}+\|R\phi+U\|_{L^{\infty}}+\|\varepsilon\|_{L^{\infty}}\right)\mathcal{N}_{B}^{2}(\varepsilon)\\
				&\lesssim \frac{C_{0}^{2}}{s^{3}\log^{3}s}+\frac{C_{0}^{2}}{s^{4}\log^{2}s}+\frac{C_{0}^{2}}{s^{\frac{17}{4}}\log^{2}s}.
			\end{aligned}
		\end{equation}
		Then, from $(\varepsilon,\Lambda Q)=(\varepsilon,\Gamma(\Lambda Q))=(\varepsilon,\partial_{y}Q_{1})=(\varepsilon,\partial_{y}Q_{2})=0$ and~\eqref{est:LambdaVenergy}, 
		\begin{equation}\label{est:I33}
			\begin{aligned}
				\bigg|\bigg(\frac{\dot{\lambda}}{\lambda}+b_{1}\bigg)
				\int_{\RR}(\Lambda V)\varepsilon\phi_{2}\dd y\bigg|
				\lesssim \frac{C_{0}^{2}}{s^{3}\log^{3}s},
				\\
				\bigg| \left(\frac{\dot{x}_{1}}{\lambda}-1\right)\int_{\RR}\left(\partial_{y}V\right)\varepsilon\phi_{2}\dd y\bigg|\lesssim \frac{C_{0}^{2}}{s^{3}\log^{3}s}.
			\end{aligned}
		\end{equation}
		Next, using again $(\varepsilon,\Gamma(\Lambda Q))=(\varepsilon,\partial_{y}Q_{2})=0$ and Lemma~\ref{le:controlpara}, we find 
		\begin{equation*}
			\bigg|\int_{\RR}({\vec{{\rm{Mod}}}}\cdot\vec{M}Q)\varepsilon\phi_{2}\dd y\bigg|\lesssim s^{-\frac{1}{2}}\mathcal{N}_B^{2}(\varepsilon)\lesssim \frac{C_{0}^{2}}{s^{\frac{7}{2}}\log^{2}s}.
		\end{equation*}
		In addition, from Lemma~\ref{le:controlpara}, Lemma~\ref{le:Psinorm} and $\phi_{2}\lesssim \phi_{1}$,
		\begin{equation*}
			\sum_{i=1}^{4}\bigg|\int_{\RR}\Psi_{i}({V})\varepsilon\phi_{2}\dd y\bigg|\lesssim s^{-\frac{5}{2}}\mathcal{N}_{B}(\varepsilon)\lesssim \frac{C_{0}}{s^{4}\log s}.
		\end{equation*}
		Combining the above estimates with Definition~\ref{def:Admissible} and Proposition~\ref{prop:approx}, we obtain
		\begin{equation}\label{est:I34}
			\bigg|\int_{\RR}\Psi({V})\varepsilon\phi_{2}\dd y\bigg|\lesssim 
			\frac{C_{0}^{2}}{s^{\frac{7}{2}}\log^{2}s}+
			\frac{C_{0}^{2}}{s^{3}\log^{3} s}.
		\end{equation}
		By integration by parts, we compute 
		\begin{equation*}
			\int_{\RR}(\Lambda \varepsilon)\varepsilon\phi_{2}\dd y=\frac{1}{2}\int_{\RR}\varepsilon^{2}\phi_{2}\dd y-\frac{1}{2}\int_{\RR}(\partial_{y}\varepsilon)^{2}\phi_{2}\dd y-\frac{1}{2}(\partial_{y}\varepsilon)^{2}y\partial_{y}\phi_{2}\dd y,
		\end{equation*}
		which implies that 
		\begin{equation}\label{est:I35}
			\begin{aligned}
				&\bigg|\frac{\dot{\lambda}}{\lambda}\int_{\RR}(\Lambda \varepsilon)\varepsilon \phi_{2}\dd y\bigg|
				+
				\bigg|\left(\frac{\dot{x}_{1}}{\lambda}-1\right)\int_{\RR}\left(\partial_{y}\varepsilon\right)\varepsilon\phi_{2}\dd y\bigg|\\
				&\lesssim s^{-1}\mathcal{N}_{B}^{2}(\varepsilon)+\mathcal{N}_{B}^{3}(\varepsilon)\lesssim \frac{C_{0}^{2}}{s^{4}\log^{2}s}+\frac{C_{0}^{3}}{s^{\frac{9}{2}}\log^{3}s}.
			\end{aligned}
		\end{equation}
		Last, from Lemma~\ref{le:controlpara}, the definition of $\phi_{2}$ in~\eqref{equ:defphi2} and~\eqref{est:Boot},
		\begin{equation}\label{est:I36}
			\bigg|  \mu\int_{\RR}\varepsilon^{2}\partial_{s}\phi_{2}\dd y\bigg|+\bigg|\dot{\mu}\int_{\RR}\varepsilon^{2}\phi_{2}\dd y\bigg|\lesssim
			\frac{C_{0}^{2}}{s^{5}\log s}+\frac{C_{0}^{3}}{s^{\frac{9}{2}}\log^{3}s}.
		\end{equation}
		We see that~\eqref{est:I3} follows directly from estimates~\eqref{est:I31}--\eqref{est:I36}.
		
		\smallskip
		\emph{Estimate on $\mathcal{I}_{4}$.} We claim that 
		\begin{equation}\label{est:I4}
			\mathcal{I}_{4}=O\left(\frac{C_{0}^{2}}{s^{4}\log^{\frac{5}{2}}s}\right).
		\end{equation}
		Indeed, we first decompose 
		\begin{equation*}
			\frac{\dd }{\dd s}\left(\varepsilon\phi_{1},X_{1}\phi\right)
			=\left((\partial_{s}\varepsilon)\phi_{1},X_{1}\phi\right)+\left(\varepsilon
			(\partial_{s}\phi_{1}),X_{1}\phi\right)
			+\left(\varepsilon\phi_{1},X_{1}\partial_{s}\phi\right).
		\end{equation*}
		From Lemma~\ref{le:eque} and integration by parts,
		\begin{equation*}
			\begin{aligned}
				\left((\partial_{s}\varepsilon)\phi_{1},X_{1}\phi\right)
				&=\int_{\RR}\left(
				\partial_{y}^{2}\varepsilon-\varepsilon+(V+\varepsilon)^{5}-V^{5}\right)
				\left(\partial_{y}(X_{1}\phi\phi_{1})\right)\dd y\\
				&+\left(\frac{\dot{\lambda}}{\lambda}+b_{1}\right)\left((\Lambda V)\phi_{1},X_{1}\phi\right)+\left(\frac{\dot{x}_{1}}{\lambda}-1\right)
				\left((\partial_{y}V)\phi_{1},X_{1}\phi\right)\\
				&-(\Psi(V)\phi_{1},X_{1}\phi)+\frac{\dot{\lambda}}{\lambda}\left((\Lambda \varepsilon)\phi_{1},X_{1}\phi\right)
				+\left(\frac{\dot{x}_{1}}{\lambda}-1\right)
				\left((\partial_{y}\varepsilon)\phi_{1},X_{1}\phi\right).
			\end{aligned}
		\end{equation*}
		By an elementary computation, we find 
		\begin{equation*}
			\partial_{y}(X_{1}\phi\phi_{1})=(\partial_{y}X_{1})\phi \phi_{1}+X_{1}(\partial_{y}\phi)\phi_{1}+X_{1}\phi\partial_{y}\phi_{1},
		\end{equation*}
		which implies that 
		\begin{equation*}
			\begin{aligned}
				\left| \int_{\RR}\left(
				\partial_{y}^{2}\varepsilon-\varepsilon\right)
				\left(\partial_{y}(X_{1}\phi\phi_{1})\right)\dd y\right|&
				\lesssim\mathcal{N}_{B}(\varepsilon)
				\left(\|X_{1}\partial_{y}\phi\|_{L^{2}}
                +\|X_{1}\partial^{2}_{y}\phi\|_{L^{2}}
                \right)\\
                &\lesssim\mathcal{N}_{B}(\varepsilon)
				\left(
\|\partial_{y}\phi_{1}\|_{L^{1}}^{\frac{1}{2}}
                +\|\partial^{2}_{y}\phi_{1}\|_{L^{1}}^{\frac{1}{2}}
                \right)\\
				&+\mathcal{N}_{B}(\varepsilon)
				\left(1+\|\partial_{y}X_{1}\|_{L^{2}}
                +\|\partial^{2}_{y}X_{1}\|_{L^{2}}
                \right)
				\lesssim \frac{C_{0}}{s^{\frac{3}{2}}\log s}.
			\end{aligned}
		\end{equation*}
		Using a similar argument as above and Lemma~\ref{le:L2e}, we also find 
		\begin{equation*}
			\left|\int_{\RR}\left((V+\varepsilon)^{5}-V^{5}\right)
			\left(\partial_{y}(X_{1}\phi\phi_{1})\right)\dd y\right|\lesssim \frac{C_{0}}{s^{\frac{3}{2}}\log s}.
		\end{equation*}
		Next, from~\eqref{est:Boot},~\eqref{est:LambdaVenergy} and Lemma~\ref{le:controlpara},
		\begin{equation*}
			\begin{aligned}
				&\left(\frac{\dot{\lambda}}{\lambda}+b_{1}\right)\left((\Lambda V)\phi_{1},X_{1}\phi\right)
				+\left(\frac{\dot{x}_{1}}{\lambda}-1\right)
				\left((\partial_{y}V)\phi_{1},X_{1}\phi\right)
				\\
				&=-z\left(\frac{\dot{\lambda}}{\lambda}+b_{1}\right)((\partial_{y}Q_{2}))\phi_{1},X_{1}\phi)
				+O\left(\frac{C_{0}}{s^{\frac{3}{2}}\log s}\right).
			\end{aligned}
		\end{equation*}
		Similarly, from the definition of $\vec{\rm{Mod}}$ and $\vec{M}Q$ in~\eqref{equ:defMod}--\eqref{equ:defMQ},
		\begin{equation*}
			\left(  (\vec{\rm{Mod}}\cdot \vec{M}Q)\phi_{1},X_{1}\phi\right)=(\dot{z}-b_{1}z-\mu)((\partial_{y}Q_{2}))\phi_{1},X_{1}\phi)
			+O\left(\frac{C_{0}}{s^{\frac{3}{2}}\log s}\right).
		\end{equation*}
		It follows from Lemma~\ref{le:Psinorm} and Proposition~\ref{prop:approx} that 
		\begin{equation*}
			-(\Psi(V)\phi_{1},X\phi)=-(\dot{z}-b_{1}z-\mu)((\partial_{y}Q_{2}))\phi_{1},X\phi)
			+O\left(\frac{C_{0}}{s^{\frac{3}{2}}\log s}\right).
		\end{equation*}
		Based on the above two estimates and Lemma~\ref{le:controlpara}, we deduce that 
		\begin{equation*}
			\begin{aligned}
				&\left(\frac{\dot{\lambda}}{\lambda}+b_{1}\right)\left((\Lambda V)\phi_{1},X_{1}\phi\right)-(\Psi(V)\phi_{1},X_{1}\phi)\\
				& +\left(\frac{\dot{x}_{1}}{\lambda}-1\right)
				\left((\partial_{y}V)\phi_{1},X_{1}\phi\right)=O\left(\frac{C_{0}}{s^{\frac{3}{2}}\log s}\right).
			\end{aligned}
		\end{equation*}
		Then, from Lemma~\ref{le:controlpara} and the bootstrap assumption, 
		\begin{equation*}
			\begin{aligned}
				\bigg|\frac{\dot{\lambda}}{\lambda}\left((\Lambda \varepsilon)\phi_{1},X_{1}\phi\right)\bigg|&\lesssim
				\frac{\mathcal{N}_{B}(\varepsilon)}{s\log s}
				\left(\int_{\RR}(1+y^{2})X_{1}^{2}\phi_{1}\dd y\right)^{\frac{1}{2}}
				\lesssim \frac{C_{0}}{s^{\frac{3}{2}}\log s},
				\\
				\bigg|\left(\frac{\dot{x}_{1}}{\lambda}-1\right)
				\left((\partial_{y}\varepsilon)\phi_{1},X_{1}\phi\right)\bigg|
				&\lesssim \frac{C_{0}\mathcal{N}_{B}(\varepsilon)}{s^{\frac{3}{2}}\log s}
				\left(\int_{\RR}(1+y^{2})X_{1}^{2}\phi_{1}\dd y\right)^{\frac{1}{2}}
				\lesssim \frac{C_{0}}{s^{\frac{3}{2}}\log s}.
			\end{aligned}
		\end{equation*}
		Here, we use the fact that 
		\begin{equation*}
			\begin{aligned}
				\int_{\RR}(1+y^{2})X_{1}^{2}\phi_{1}\dd y
                &\lesssim s\int_{B(s^{\frac{1}{2}}+2)}^{\infty}
                \left(1+\left(\frac{y}{B}-s^{\frac{1}{2}}\right)^{2}\right)
                \phi_{1}\dd y\\
				&+\int_{-\infty}^{0}e^{-|y|}\dd y+\int_{0}^{B(s^{\frac{1}{2}}+2)}(1+y^{2})\dd y\lesssim s^{\frac{3}{2}}.
			\end{aligned}
		\end{equation*}
		Last, from the definition of $\phi$ and $\phi_{1}$ in~\eqref{equ:defphi} and~\eqref{equ:defphi1}, we find 
		\begin{equation*}
			|\partial_{s}\phi|\lesssim s^{-1}\textbf{1}_{[\frac{s}{2},\frac{5s}{2}]}\quad \mbox{and}\quad |\partial_{s}\phi_{1}|\lesssim s^{-\frac{1}{2}}\phi_{1},
		\end{equation*}
		which directly implies that 
		\begin{equation*}
			\left| \left(\varepsilon
			(\partial_{s}\phi_{1}),X_{1}\phi\right)\right|
			+\left|\left(\varepsilon\phi_{1},X_{1}\partial_{s}\phi\right)\right|\lesssim
			\mathcal{N}_{B}(\varepsilon)\lesssim \frac{C_{0}}{s^{\frac{3}{2}}\log s}.
		\end{equation*}
		Combining  the above estimates with~\eqref{est:J1J2}, we complete the proof of~\eqref{est:I4}.
		
		\smallskip
		\emph{Estimate on $\mathcal{I}_{5}$.}
		We claim that 
		\begin{equation}\label{est:I5}
			\mathcal{I}_{5}\simeq 
			-2\dot{\Theta}(\varepsilon\phi_{1},X_{1}\phi)+\frac{5}{2}b_{1}\dot{\mu}(\varepsilon\phi_{1},X_{1}\phi).
		\end{equation}
		Indeed, from Lemma~\ref{le:refinmu} and Lemma~\ref{le:refinedTheta}, we compute 
		\begin{equation*}
			\begin{aligned}
				\frac{\dd }{\dd s}\left(b_{1}
				\left(\mathcal{J}_{2}+
				\left(c_{0}+\frac{5}{4}\right)\mathcal{J}_{1}\right)\right)
				&=\dot{\Theta}-\frac{5}{4}b_{1}\dot{\mu}+\frac{5}{2}\alpha\frac{\dd}{\dd s}(ze^{-z})+2\Theta\left(c_{0}+\frac{5}{4}\right)b_{1}\\
				&+\dot{b}_{1}\left(\mathcal{J}_{2}+\left(c_{0}+\frac{5}{4}\right)\mathcal{J}_{1}\right)+O\left(\frac{1}{s^{3}\log s}\right).
			\end{aligned}
		\end{equation*}
		It follows from~Lemma~\ref{le:controlpara} and~\eqref{est:J1J2} that 
		\begin{equation*}
			\frac{\dd }{\dd s}\left(b_{1}
			\left(\mathcal{J}_{2}+
			\left(c_{0}+\frac{5}{4}\right)\mathcal{J}_{1}\right)\right)
			=\dot{\Theta}-\frac{5}{4}b_{1}\dot{\mu}+O\left(s^{-3}\right).
		\end{equation*}
		Combining the above estimate with~\eqref{est:ephi1Xphi}, we complete the proof of~\eqref{est:I5}.
		
		\smallskip
		We see that~\eqref{est:dsF2} follows directly from~\eqref{est:I3},~\eqref{est:I4} and~\eqref{est:I5}.
		
		\smallskip
		\textbf{Step 3.}
		Conclusion. Combining the estimate~\eqref{est:dsF1} with \eqref{est:dsF2}, we thus complete the proof of the estimate~\eqref{est:dsF}.
	\end{proof}
	
	\subsection{End of the proof for Proposition~\ref{prop:uni}}\label{SS:Proofuni}
	We are in a position to complete the proof of Proposition~\ref{prop:uni} via bootstrap argument. We start by closing all bootstrap estimates except the distance and position parameters $(z,x_{1})$. Then, we prove the existence of suitable final data $(z^{in},x_{1}^{in})$ using a topological argument.
	
	\begin{proof}[Proof of Proposition~\ref{prop:uni}]
		\textbf{Step 1.} Closing the estimate for $\varepsilon$.
		Integrating~\eqref{est:dsF} on $[s,s^{in}]$ for any $s\in [s^{*},s^{in}]$ and then using $\varepsilon(s^{in})=0$, we have 
		\begin{equation*}
			\mathcal{F}(s)\lesssim \frac{C_{0}^{2}}{s^{3}\log^{\frac{5}{2}}s}+\frac{C_{0}}{s^{3}\log^{2}s}.
		\end{equation*}
		Thus, from~\eqref{est:coerF}, for $s_{0}>1$ large enough (depending on $C_{0}$),
		\begin{equation*}
			\mathcal{N}_{B}^{2}(\varepsilon)
			\lesssim \mathcal{F}+s^{-\frac{13}{4}}\lesssim \frac{C_{0}^{2}}{s^{3}\log^{\frac{5}{2}}s}+\frac{C_{0}}{s^{3}\log^{2}s}.
		\end{equation*}
		This strictly improves the estimate on $\varepsilon$ in~\eqref{est:Boot} for $C_{0}$ large enough.
		
		\smallskip
		\textbf{Step 2.} Closing the estimate for $\Theta$. Note that, from the bootstrap assumption~\eqref{est:Boot} and~\eqref{est:bootz}, we find 
		\begin{equation}\label{est:alphae-z}
			\frac{z^{-1}}{3\alpha}e^{z}=s^{2}+O\left(\frac{s^{2}}{\log^{\frac{1}{2}}s}\right)
			\Longrightarrow \alpha e^{-z}=\frac{1}{6s^{2}\log s}+O\left(\frac{1}{s^{2}\log^\frac{3}{2}s}\right).
		\end{equation}
		From Lemma~\ref{le:controlpara}, the bootstrap assumption~\eqref{est:Boot},~\eqref{est:rho21}--\eqref{est:rho22} and~\eqref{est:alphae-z}, 
		\begin{equation*}
			c_{0}b_{1}\dot{\mu}-b_{1}\dot{\mathcal{J}}_{2}=c_{0}\frac{\dd }{\dd s}(b_{1}\mu)-\frac{\dd}{\dd s}(b_{1}\mathcal{J}_{2})+O\left(
            \frac{1}{s^{3}\log s}+
            \frac{C_{0}}{s^{\frac{7}{2}}\log^{\frac{3}{2}} s}\right).
		\end{equation*}
		Similarly, from~\eqref{est:Boot} and~\eqref{est:rho21}--\eqref{est:rho22}, we find 
		\begin{equation*}
			|b_{1}\mu|+|b_{1}\mathcal{J}_{2}|\lesssim \frac{1}{s^{2}\log s}+\frac{C_{0}}{s^{\frac{5}{2}}\log^{\frac{3}{2}}s}.
		\end{equation*}
		Combining the above estimates with Lemma~\ref{le:refinedTheta} and~\eqref{est:finialdatawell}, we find 
		\begin{equation}\label{est:improveTheta1}
			\Theta(s)=-\frac{5\alpha}{2} ze^{-z}+
			O\left(\frac{1}{s^{2}\log s}+\frac{C^{2}_{0}}{s^{\frac{5}{2}}\log^{\frac{3}{2}}s}\right).
		\end{equation}
		It follows from~\eqref{est:bootz} and~\eqref{est:alphae-z} that 
		\begin{equation}\label{est:improveTheta2}
			\Theta(s)=-\frac{5}{6s^{2}}+O\left(\frac{1}{s^{2}\log^{\frac{1}{2}}s}+\frac{C_{0}}{s^{\frac{5}{2}}\log^{\frac{3}{2}}s}\right).
		\end{equation}
		This strictly improves the estimate on $\Theta$ in~\eqref{est:Boot} for $s_{0}$ large enough.
		
		\smallskip
		\textbf{Step 3.} Closing the estimate for $(b_{1},b_{2},\mu)$.
		Recall that, from~\eqref{est:rho1},~\eqref{est:rho2} and the bootstrap assumption~\eqref{est:Boot}, we find 
		\begin{equation*}
			| \mathcal{J}_{1}|\lesssim \left(\int_{\RR}\varepsilon^{2}\rho_{1}\dd y\right)^{\frac{1}{2}}\left(\int_{\RR}\rho_{1}\dd y\right)^{\frac{1}{2}}\lesssim \frac{C_{0}}{s^{\frac{3}{2}}\log^{\frac{1}{2}}s}.
		\end{equation*}
		Then, from~\eqref{est:improveTheta1}, Lemma~\ref{le:controlpara} and Lemma~\ref{le:refinmu},
		\begin{equation*}
			\left| \frac{\dd}{\dd s}\left(\mu-5b_{1}z+\mathcal{J}_{1}\right)\right|
			\lesssim \frac{1}{s^{2}\log s}+\frac{C_{0}^{2}}{s^{\frac{5}{2}}\log s}.
		\end{equation*}
		Integrating the above estimate over $[s,s^{in}]$ and then using~\eqref{est:finialdatawell}, we find 
		\begin{equation}\label{est:improvedmu1}
			\mu(s)=5b_{1}(s)z(s)+O\left(\frac{1}{s\log s}+\frac{C^{2}_{0}}{s^{\frac{3}{2}}\log^{\frac{1}{2}} s}\right).
		\end{equation}
		Next, using again Lemma~\ref{le:controlpara}, we have 
		\begin{equation*}
			b_{1}=\frac{\dot{z}}{z}-\frac{\mu}{z}+O\left(\frac{C_{0}}{s^{\frac{3}{2}}\log s}\right).
		\end{equation*}
		Based on the above estimate and~\eqref{est:improvedmu1}, we find 
		\begin{equation*}
			\begin{aligned}
				b_{1}e^{-z}&=\frac{\dot{z}}{6}z^{-1}e^{-z}+O\left(\frac{1}{s^{3}\log^{3}s}+\frac{C_{0}^{2}}{s^{\frac{7}{2}}\log^{2} s}\right)\\
				&=-\frac{1}{6}\frac{\dd}{\dd s}(z^{-1}e^{-z})+O\left(\frac{1}{s^{3}\log^{3}s}+\frac{C_{0}^{2}}{s^{\frac{7}{2}}\log^{2} s}\right).
			\end{aligned}
		\end{equation*}
		It follows directly from Lemma~\ref{le:controlpara} and~\eqref{est:Boot} that 
		\begin{equation*}
			\dot{b}_{1}b_{1}=\frac{\alpha}{6}\frac{\dd }{\dd s }(z^{-1}e^{-z})+O\left(\frac{1}{s^{3}\log^{3}s}+\frac{C_{0}^{2}}{s^{\frac{7}{2}}\log^{2} s}\right).
		\end{equation*}
		Integrating the above estimate over $[s,s^{in}]$ and then using~\eqref{est:finialdatawell}, we find 
		\begin{equation}
			\begin{aligned}
				b^{2}_{1}(s)=\left(\frac{\alpha}{3}\right)z^{-1}e^{-z}+O\left(\frac{1}{s^{2}\log^{3}s}+\frac{C^{2}_{0}}{s^{\frac{5}{2}}\log^{2}s}\right).
			\end{aligned}
		\end{equation}
		Therefore, from~\eqref{est:alphae-z} and the bootstrap assumption~\eqref{est:Boot},
		\begin{equation}\label{est:improved1}
			\begin{aligned}
				\big|b_{1}(s)-\frac{1}{6s \log s}\big|&\lesssim \frac{1}{s\log^{\frac{3}{2}}s}+\frac{C^{2}_{0}}{s^{\frac{3}{2}}\log s},\\
				\big|b_{1}(s)-\left(\frac{\alpha}{3}\right)^{\frac{1}{2}}z^{-\frac{1}{2}}e^{-\frac{1}{2}z}\big|&\lesssim \frac{1}{s\log^{2}s}+\frac{C^{2}_{0}}{s^{\frac{3}{2}}\log s}.
			\end{aligned}
		\end{equation}
		This strictly improves the estimate on $b_{1}$ in~\eqref{est:Boot} for $s_{0}$ large enough. Moreover, using~\eqref{est:improveTheta2} and~\eqref{est:improved1}, we improve the estimate on $b_{2}$ in~\eqref{est:Boot} for $s_{0}$ large enough.
		
		\smallskip
		Combining the above estimate with~\eqref{est:improvedmu1}, we obtain
		\begin{equation}\label{est:improvemu}
			\mu(s)=\frac{5}{3s}+O\left(\frac{1}{s\log^{\frac{1}{2}}s}+\frac{C^{2}_{0}}{s^{\frac{3}{2}}}\right).
		\end{equation}
		This strictly improves the estimate on $\mu$ in~\eqref{est:Boot} for $s_{0}$ large enough.
		
		\smallskip
		\textbf{Step 4.} 
		Closing the estimate for $\lambda$. 
		From~\eqref{est:improved1} and Lemma~\ref{le:controlpara}, we have 
		\begin{equation*}
			\frac{\dd }{\dd s}\log \lambda(s)=-\frac{1}{6s\log s}
			+O\left(\frac{1}{s \log^\frac{3}{2}s}+\frac{C^{2}_{0}}{s^{\frac{3}{2}}\log s}\right).
		\end{equation*}
		Based on the above estimate and \eqref{est:finialdatawell}, we find 
		\begin{equation*}
			\log \lambda(s)=-\frac{1}{6}\log \log s+O\left(\frac{1}{\log^{\frac{1}{2}}s}+\frac{C^{2}_{0}}{s^{\frac{1}{2}}\log s}\right),
		\end{equation*}
		which directly implies that 
		\begin{equation}\label{est:improvelambda}
			\lambda(s)=\frac{1}{\log^{\frac{1}{6}}s}+O\left(\frac{1}{\log^{\frac{2}{3}}s}+\frac{C^{2}_{0}}{s^{\frac{1}{2}}\log s}\right).
		\end{equation}
		This strictly improves the estimate on $\lambda$ in~\eqref{est:Boot} for $s_{0}$ large enough.
		
		\smallskip
		\textbf{Step 5.} Estimate for $(z,x_{1})$.
		We need to adjust the final choice of $(z^{in},x_{1}^{in})$ via a topological argument (see~\cite[Section 3.4]{MartelRaphael} for a similar argument). For any $(\theta_{1},\theta_{2})\in \mathcal{B}_{\RR^{2}}(1)$,
		we consider the following choice of finial data
		\begin{equation}\label{equ:deffinalzx1}
			\begin{aligned}
				x_{1}^{in}-\frac{s^{in}}{\log^{\frac{1}{6}} s^{in}}&=\frac{\theta_{1}s^{in}}{\log^{\frac{1}{2}}s^{in}},\\
				\left(\frac{1}{3\alpha}\right)^{\frac{1}{2}}(z^{in})^{-\frac{1}{2}}e^{\frac{1}{2}z^{in}}-s^{in}&=\frac{\theta_{2}s^{in}}{\log^{\frac{1}{2}}s^{in}}.
			\end{aligned}
		\end{equation}
		We define $(\xi_{1},\xi_{2})$ the following two functions on $[s^{*},s^{in}]$:
		\begin{equation*}
			\xi_{1}(s)=x_{1}\log^{\frac{1}{6}}s\quad \mbox{and}\quad \xi_{2}(s)=\left(\frac{1}{3\alpha}\right)^{\frac{1}{2}}z^{-\frac{1}{2}}e^{\frac{1}{2}z}.
		\end{equation*}
		In addition, we denote 
		\begin{equation*}
			\eta(s)=\eta_{1}^{2}(s)+\eta_{2}^{2}(s),\quad \mbox{on}  \ [s^{*},s^{in}],
		\end{equation*}
		where
		\begin{equation*}
			\begin{aligned}
				\eta_{1}(s)&=(\xi_{1}(s)-s)s^{-1}\log^{\frac{1}{3}}s,\quad \mbox{on}  \ [s^{*},s^{in}],
				\\
				\eta_{2}(s)&=(\xi_{2}(s)-s)s^{-1}\log^{\frac{1}{2}}s,\quad \mbox{on}  \ [s^{*},s^{in}].
			\end{aligned}
		\end{equation*}
		We claim that, for any $s\in [s^{*},s^{in}]$ such that $\eta(s)=1$, the following transversality property holds,
		\begin{equation}\label{est:dseta}
			\frac{\dd }{\dd s}\eta(s)\le -s^{-1}.
		\end{equation}
		Indeed, from ~\eqref{est:improvelambda} and Lemma~\ref{le:controlpara}, we check that 
		\begin{equation*}
			\dot{\xi}_{1}(s)=1+O\left(\frac{1}{\log^{\frac{1}{2}}s}+\frac{C_{0}^{2}}{s^{\frac{1}{2}}\log^{\frac{1}{2}}s}\right).
		\end{equation*}
		It follows directly from the bootstrap assumption~\eqref{est:Boot} that 
		\begin{equation}\label{est:eta1}
			\frac{\dd }{\dd s}\left(\eta^{2}_{1}(s)\right)=-2s^{-1}\eta_{1}^{2}(s)+O\left(\frac{1}{s\log^{\frac{1}{6}}s}+\frac{C_{0}^{2}}{s^{\frac{3}{2}}\log^{\frac{1}{6}}s}\right).
		\end{equation}
		On the other hand, from~\eqref{est:improvedmu1},~\eqref{est:improved1} and Lemma~\ref{le:controlpara}, we check that 
		\begin{equation*}
			\dot{z}=6\left(\frac{\alpha}{3}\right)^{\frac{1}{2}}z^{\frac{1}{2}}e^{-\frac{1}{2}z}+O\left(\frac{1}{s\log s}+\frac{C^{2}_{0}}{s^{\frac{3}{2}}}+\frac{C^{2}_{0}}{s^{\frac{3}{2}}\log^{\frac{1}{2}}s}\right).
		\end{equation*}
		Based on the above estimate and~\eqref{est:alphae-z}, we find  
		\begin{equation*}
			\dot{\xi}_{2}=1+O\left(\frac{1}{\log s}+\frac{C^{2}_{0}}{s^{\frac{1}{2}}}+\frac{C^{2}_{0}}{s^{\frac{1}{2}}\log^{\frac{1}{2}}s}\right).
		\end{equation*}
		It follows from the bootstrap assumption~\eqref{est:Boot} that 
		\begin{equation}\label{est:eta2}
			\frac{\dd}{\dd s}(\eta_{2}^{2}(s))=-2s^{-1}\eta_{2}^{2}(s)+O\left(\frac{1}{s\log^{\frac{1}{2}}s}+\frac{C_{0}^{2}}{s^{\frac{3}{2}}}\log s\right).
		\end{equation}
		Combining~\eqref{est:eta1} and~\eqref{est:eta2}, we complete the proof of~\eqref{est:dseta} for $s_{0}$ large enough. The transversality relation~\eqref{est:dseta} is enough to justify the existence of at least a couple of $(\theta_{1},\theta_{2})\in \mathcal{B}_{\RR^{2}}(1)$ such that $s^{*}(\theta_{1},\theta_{2})=s_{0}$.
		
		\smallskip
		The proof is by contradiction, we assume that for all $(\theta_{1},\theta_{2})\in \mathcal{B}_{\RR^{2}}(1)$, it holds $s_{0}<s^{*}(\theta_{1},\theta_{2})$. Then, a contradiction follows from the following discussion (see for example~\cite[Page 730]{MartelRaphael} for more discussion).
		
		\smallskip
		\emph{Continuity of $s^{*}(\theta_{1},\theta_{2})$.}
		The above transversality relation~\eqref{est:dseta} implies that the map 
		\begin{equation*}
			(\theta_{1},\theta_{2})\in\bar{\mathcal{B}}_{\RR^{2}}(1)\mapsto 
			s^{*}(\theta_{1},\theta_{2}),
		\end{equation*}
		is continuous and 
		\begin{equation*}
			s^{*}(\theta_{1},\theta_{2})=s^{in},\quad \mbox{for any}\ (\theta_{1},\theta_{2})\in \mathcal{S}_{\RR^{2}}(1).
		\end{equation*}
		
		\smallskip
		\emph{Construction of a retraction.} We define 
		\begin{equation*}
			\begin{aligned}
				\mathcal{M}:(\theta_{1},\theta_{2})\in \mathcal{B}_{\RR^{2}}(1)\mapsto\left(\eta_{1}(s^{*}(\theta_{1},\theta_{2})),\eta_{2}(s^{*}(\theta_{1},\theta_{2}))\right)\in \mathcal{S}_{\RR^{2}}(1).
			\end{aligned}
		\end{equation*}
		From what precedes, $\mathcal{M}$ is continuous. Moreover, $\mathcal{M}$ restricted to $\mathcal{S}_{\RR^{2}}(1)$ is the identity. The existence of such a map is contradictory to the no retraction theorem for continuous maps from the ball to the sphere. Therefore, the existence of $(\theta_{1},\theta_{2})$ have proved and the uniform estimate~\eqref{est:unisolution} is a consequence of bootstrap estimate~\eqref{est:Boot},~\eqref{est:improvelambda} and Lemma~\ref{le:L2e}. The proof of Proposition~\ref{prop:uni} is complete.
	\end{proof}

	\section{Compactness argument}
	The objective of this section is to finish the construction of solution in 
    Theorem~\ref{thm:main} by passing to the limit on a sequence of solutions give by Proposition~\ref{prop:uni}.
	
	\begin{proof}
		[Proof of Theorem~\ref{thm:main}]
		\textbf{Step 1.} Construction of a sequence of backwards solutions.
		We claim that, there exists $T_{0}>1$ and a sequence of solution $u_{n}\in C([T_{0},T_{n}];H^{1})$ of~\eqref{equ:gKdV} where $T_{n}=n\in \mathbb{N}$ satisfying the following estimates on $[T_{0},T_{n}]$:
		\begin{equation}\label{est:unit}
			\left\{
			\begin{aligned}
				|z_{n}(t)-2\log t|&\lesssim \log \log t,\quad |b_{1n}(t)|\lesssim \frac{1}{t\log^{\frac{3}{2}}t},\\
				|x_{1n}(t)-t\log^{\frac{1}{3}}t|&\lesssim t\log^{\frac{1}{6}}t,\quad \|\varepsilon_{n}(t)\|_{H^{1}}\lesssim \frac{1}{t\log^{\frac{1}{2}}t},\\
				\bigg|\lambda_{n}(t)-\frac{1}{\log^{\frac{1}{6}}t}\bigg|&\lesssim \frac{1}{\log^{\frac{2}{3}}t},\quad |\mu_{n}(t)|\lesssim \frac{1}{t\log^{\frac{1}{2}}t},\quad |b_{2n}(t)|\lesssim \frac{1}{t\log^{\frac{3}{2}}t}.
			\end{aligned}
			\right.
		\end{equation}
		Here, $(\lambda,z,\mu,x_{1},b_{1},b_{2})\in (0,\infty)^{2}\times \RR^{4}$ are the parameters for the decomposition of $u_{n}$ given by Proposition~\ref{prop:modu}, that is, 
		\begin{equation}\label{def:ue}
			\begin{aligned}
				u_{n}(t,x)&=\frac{1}{\lambda^{\frac{1}{2}}_{n}}V\left(\frac{x-x_{1n}}{\lambda_{n}};(z_{n},\mu_{n},b_{1n},b_{2n})\right)+\frac{1}{\lambda^{\frac{1}{2}}_{n}}\varepsilon_{n}\left(t,\frac{x-x_{1n}}{\lambda_{n}}\right)\\
				&=\frac{1}{\lambda_{n}^{\frac{1}{2}}}
				\left(Q\left(\frac{x-x_{1n}}{\lambda_{n}}\right)-Q\left(\frac{x-x_{1n}}{\lambda_{n}}-z_{n}\right)\right)+O_{H^{1}}\left(\frac{1}{t^{\frac{1}{4}}}\right).
			\end{aligned}
		\end{equation}
		Indeed, from Proposition~\ref{prop:uni}, we know that there exists a solution $u_{n}(s)$ of~\eqref{equ:gKdV} on the time interval $[s_{0},s^{in}]$ and satisfies the uniform estimate~\eqref{est:unisolution} over such time interval. Recall that, 
		from~\eqref{equ:defs}
		we set 
		\begin{equation}\label{equ:defts}
			s(t)=s^{in}-\int_{t}^{T_{n}}\frac{\dd \sigma}{\lambda^{3}(\sigma)}
			\Longleftrightarrow
			t(s)=T_{n}-\int_{s}^{s^{in}}\lambda^{3}(\sigma)\dd \sigma.
		\end{equation}
		Denote 
		\begin{equation*}
			\eta_{3}(s)=(\xi_{3}(s)-s)^{2}s^{-2}\log^{\frac{1}{2}}s,\quad \mbox{with}\ \ 
			\xi_{3}(s)=t(s)\log^{\frac{1}{2}}s.
		\end{equation*}
		From~\eqref{est:unisolution} and~\eqref{equ:defts}, we check that 
		\begin{equation*}
			\eta_{3}(s)=1\Longrightarrow \frac{\dd \eta_{3}}{\dd s}\le -2s^{-1}+O\left(\frac{1}{s\log^{\frac{1}{4}}s}\right)\le -s^{-1}.
		\end{equation*}
		Therefore, using a similar topological argument as in the step 5 of the proof for Proposition~\ref{prop:uni}, we know that for any $T_{n}=n$ large enough, there exists a choice of $s^{in}>0$ satisfying~\eqref{est:defsin} such that 
		\begin{equation}\label{est:ts1}
			\eta_{3}(s)\le 1, \ \ \mbox{on}\ [s_{0},s^{in}]\Longrightarrow
			\left|t(s)-\frac{s}{\log^{\frac{1}{2}}s}\right|\le \frac{s}{\log^{\frac{3}{4}}s},\ \ \mbox{on}\ [s_{0},s^{in}].
		\end{equation}
		It follows directly that 
		\begin{equation}
			s(t)=t\log^{\frac{1}{2}}t\left(1+O\left(\frac{1}{\log^{\frac{1}{4}}t}\right)\right),\quad \mbox{on}\ [T_{0},T_{n}].
		\end{equation}
		Here, $T_{0}>0$ is a universal constant independent with the choice of $n\in \mathbb{N}$. Therefore, the uniform estimates~\eqref{est:unit} and~\eqref{def:ue} are standard consequence of Lemma~\ref{le:RAB}, Proposition~\ref{prop:uni} and the definition of $V$ in~\eqref{equ:defV}.
		
		\smallskip
		\textbf{Step 2.} Compactness argument. By~\eqref{est:unit} and~\eqref{def:ue}, the sequence $(u_{n}(T_0{}))_{n\in \mathbb{N}}$ is bounded in $H^{1}$. Therefore, there exists a subsequence of $(u_{n})_{n\in \mathbb{N}}$ (still denote by $(u_{n})_{n\in \mathbb{N}}$) and $u_{0}\in H^{1}$ such that 
		$ u_{n}(T_{0})\rightharpoonup u_{0}$ weakly in $H^{1}$. Let 
		$u(t)$ be the solution of~\eqref{equ:gKdV} corresponding to the initial data $u_{0}$. From the local Cauchy theory and weak continuity of the flow%
        \footnote{We refer to \cite[Lemma 2.10]{Comkdv} and \cite[Lemma 30]{MMJMPA} for the statement and proof of this property.}
        for~\eqref{equ:gKdV} and~\eqref{est:unit}, we know that $u\in C([T_{0},\infty);H^{1})$ and $u(t)$ satisfies~\eqref{equ:modu2} for all $t\in [T_{0},\infty)$. Moreover, the decomposition $(\mathcal{G},\varepsilon)$ of $u$ satisfies, for all $t\in [T_{0},\infty)$,
		\begin{equation}\label{est:weakL2}
			\mathcal{G}_{n}(t)\to \mathcal{G}(t)\quad \mbox{and}\quad \varepsilon_{n}(t)\rightharpoonup \varepsilon(t)\ \ \mbox{weakly in}\ H^{1}.
		\end{equation}
		In particular, for $t\in [T_{0},\infty)$, the solution $u(t)$ decompose as 
		\begin{equation}\label{est:defutx}
			u(t,x)=\frac{1}{\lambda^{\frac{1}{2}}}V\left(\frac{x-x_{1}}{\lambda_{}};(z,\mu,b_{1},b_{2})\right)+\frac{1}{\lambda^{\frac{1}{2}}}\varepsilon \left(t,\frac{x-x_{1}}{\lambda_{}}\right).
		\end{equation}
		From~\eqref{est:unit} and~\eqref{est:weakL2}, the functions $(\mathcal{G},\varepsilon)$ satisfy
		\begin{equation}\label{est:unit2}
			\left\{
			\begin{aligned}
				|z(t)-2\log t|&\lesssim \log \log t,\quad |b_{1}(t)|\lesssim \frac{1}{t\log^{\frac{3}{2}}t},\\
				|x_{1}(t)-t\log^{\frac{1}{3}}t|&\lesssim t\log^{\frac{1}{6}}t,\quad \|\varepsilon_{}(t)\|_{H^{1}}\lesssim \frac{1}{t\log^{\frac{1}{2}}t},\\
				\bigg|\lambda_{}(t)-\frac{1}{\log^{\frac{1}{6}}t}\bigg|&\lesssim \frac{1}{\log^{\frac{2}{3}}t},\quad |\mu_{}(t)|\lesssim \frac{1}{t\log^{\frac{1}{2}}t},\quad |b_{2}(t)|\lesssim \frac{1}{t\log^{\frac{3}{2}}t}.
			\end{aligned}
			\right.
		\end{equation}
		It follows directly from Lemma~\ref{le:RAB} that 
		\begin{equation*}
			\lambda^{-1}(t)\|V(t,y)-Q+Q(\cdot-z)\|_{H^{1}}\lesssim t^{-\frac{1}{4}}.
		\end{equation*}
		Therefore, we obtain the following expansion of $u$:
		\begin{equation*}
			\begin{aligned}
				u(t,x)&=\frac{1}{\lambda^{\frac{1}{2}}(t)}Q\left(\frac{x-x_{1}(t)}{\lambda(t)}\right)+\frac{1}{\lambda^{\frac{1}{2}}(t)}\varepsilon\left(t,\frac{x-x_{1}(t)}{\lambda(t)}\right)\\
				&-\frac{1}{\lambda^{\frac{1}{2}}(t)}Q\left(\frac{x-x_{1}(t)-\lambda(t)z(t)}{\lambda(t)}\right)+O_{H^{1}}(t^{-\frac{1}{4}}).
			\end{aligned}
		\end{equation*}
		Combining above estimate with~\eqref{est:unit2}, we complete the proof for Theorem~\ref{thm:main}.
	\end{proof}

\end{document}